\documentclass[a4paper,11pt]{amsart}
\usepackage[left=2.7cm,right=2.7cm,top=3.5cm,bottom=3cm]{geometry}

\usepackage{amsthm,amssymb,amsmath,amsfonts,mathrsfs,amscd,amsbsy,dsfont,verbatim}
\usepackage[latin1]{inputenc}
\usepackage[all,cmtip]{xy}
\usepackage{latexsym}
\usepackage{longtable}
\usepackage{mathtools}

\usepackage[pagebackref]{hyperref}

\mathtoolsset{showonlyrefs}

\usepackage{graphicx}
\newcommand{\Bmu}{\mbox{$\raisebox{-0.59ex}
  {$l$}\hspace{-0.18em}\mu\hspace{-0.88em}\raisebox{-0.98ex}{\scalebox{2}
  {$\color{white}.$}}\hspace{-0.416em}\raisebox{+0.88ex}
  {$\color{white}.$}\hspace{0.46em}$}{}}

\numberwithin{equation}{section}

\newfont{\cyr}{wncyr10 scaled 1100}
\newfont{\cyrr}{wncyr9 scaled 1000}

\theoremstyle{plain}
\newtheorem{theorem}{Theorem}[section]
\newtheorem*{ThmA}{Theorem A}
\newtheorem*{ThmB}{Theorem B}
\newtheorem*{ThmC}{Theorem C}
\newtheorem*{ThmD}{Theorem D}
\newtheorem{proposition}[theorem]{Proposition}
\newtheorem{lemma}[theorem]{Lemma}
\newtheorem{corollary}[theorem]{Corollary}
\newtheorem{conjecture}[theorem]{Conjecture}

\theoremstyle{definition}
\newtheorem{definition}[theorem]{Definition}
\newtheorem{assumption}[theorem]{Assumption}

\theoremstyle{remark}
\newtheorem{remark}[theorem]{Remark}
\newtheorem{notation}[theorem]{Notation}
\newtheorem{caveat}[theorem]{Caveat}
\newtheorem{remark/notation}[theorem]{Remark/Notation}
\newtheorem{notation/convention}[theorem]{Notation/Convention}

\newcommand{\Q}{\mathds Q}
\newcommand{\N}{\mathds N}
\newcommand{\Z}{\mathds Z}
\newcommand{\R}{\mathds R}
\newcommand{\C}{\mathds C}

\newcommand{\F}{\mathds F}
\newcommand{\T}{\mathds T}
\newcommand{\PP}{\mathds P}

\newcommand{\defeq}{\vcentcolon=}

\DeclareMathOperator{\Spec}{Spec}
\DeclareMathOperator{\Pic}{Pic}
\DeclareMathOperator{\End}{End}
\DeclareMathOperator{\Aut}{Aut}
\DeclareMathOperator{\Frob}{Frob}

\DeclareMathOperator{\Hom}{Hom}

\DeclareMathOperator{\Gal}{Gal}
\DeclareMathOperator{\GL}{GL}

\DeclareMathOperator{\Sel}{Sel}

\DeclareMathOperator{\BKK}{BK}
\DeclareMathOperator{\Nek}{Nek}

\DeclareMathOperator{\CH}{CH}
\DeclareMathOperator{\AJ}{AJ}

\DeclareMathOperator{\sspp}{sp}
\DeclareMathOperator{\bsspp}{\overline{sp}}
\DeclareMathOperator{\vol}{vol}
\DeclareMathOperator{\Tam}{Tam}

\DeclareMathOperator{\Fil}{Fil}
\DeclareMathOperator{\Symb}{Symb}
\DeclareMathOperator{\Heeg}{Heeg}
\DeclareMathOperator{\Bound}{Bound}
\DeclareMathOperator{\Proj}{Proj}
\DeclareMathOperator{\id}{id}

\DeclareMathOperator{\im}{im}

\DeclareMathOperator{\corank}{corank}

\DeclareMathOperator{\rk}{rk}
\DeclareMathOperator{\length}{length}
\DeclareMathOperator{\Comp}{Comp}
\DeclareMathOperator{\Tors}{Tors}
\DeclareMathOperator{\f}{\boldsymbol f}
\newcommand{\res}{\mathrm{res}}
\newcommand{\cores}{\mathrm{cores}}

\newcommand{\cyc}{{\rom{cyc}}}
\newcommand{\tr}{\mathrm{tr}}
\newcommand{\ord}{\mathrm{ord}}

\newcommand{\cont}{\mathrm{cont}}
\newcommand{\mot}{\mathrm{mot}}

\newcommand{\an}{\mathrm{an}}
\newcommand{\alg}{\mathrm{alg}}

\newcommand{\st}{\mathrm{st}}

\newcommand{\arith}{\mathrm{arith}}
\newcommand{\divv}{\mathrm{div}}
\newcommand{\reg}{\mathtt{reg}}

\newcommand{\rat}{\mathrm{rat}}
\newcommand{\hhom}{\mathrm{hom}}

\newcommand{\Sha}{\mbox{\cyr{X}}}

\usepackage[usenames]{color}
\definecolor{Indigo}{rgb}{0.2,0.1,0.7}
\definecolor{Violet}{rgb}{0.5,0.1,0.7}
\definecolor{White}{rgb}{1,1,1}
\definecolor{Green}{rgb}{0.1,0.9,0.2}

\newcommand{\longmono}{\mbox{\;$\lhook\joinrel\longrightarrow$\;}}

\newcommand{\longepi}{\mbox{\;$\relbar\joinrel\twoheadrightarrow$\;}}

\newcommand{\smallmat}[4]{\bigl(\begin{smallmatrix}#1&#2\\#3&#4\end{smallmatrix}\bigr)}

\newcommand{\E}{\mathcal E}

\newcommand{\bX}{\boldsymbol X}
\newcommand{\bp}{\boldsymbol p}

\newfont{\gotip}{eufb10 at 12pt}

\newcommand{\cO}{{\mathcal O}}

\newcommand{\m}{\mathfrak{m}}
\newcommand{\p}{\mathfrak{p}}
\newcommand{\fP}{\mathfrak{P}}

\newcommand{\XX}{\mathscr X}
\newcommand{\YY}{\mathscr Y}
\newcommand{\crys}{\mathrm{cris}}
\newcommand{\Sym}{\operatorname{Sym}}

\DeclareMathOperator{\Kol}{Kol}
\DeclareMathOperator{\GS}{GS}

\DeclareMathOperator{\Ta}{Ta}

\DeclareMathOperator{\Reg}{Reg}
\DeclareMathOperator{\Det}{Det}
\DeclareMathOperator{\Corr}{Corr}
\DeclareMathOperator{\Ext}{Ext}

\newcommand{\rom}{\mathrm}

\newcommand{\B}{\mathrm{B}}
\newcommand{\dR}{\mathrm{dR}}
\newcommand{\et}{\text{\'et}}
\newcommand{\MM}{\mathcal{M}}
\newcommand{\TT}{\mathbf{T}}
\newcommand{\KK}{\mathscr{K}}
\newcommand{\RR}{\mathbf{R}}
\newcommand{\uH}{\underline{H}}

\include{thebibliography}

\begin{document}

\title[TNC and Kolyvagin's conjecture for modular motives]{The Tamagawa number conjecture and Kolyvagin's conjecture for motives of modular forms}
%\today
%\date{}
\author{Matteo Longo and Stefano Vigni}

\thanks{The authors are partially supported by PRIN 2017 ``Geometric, algebraic and analytic methods in arithmetic''.}

\begin{abstract}
Assuming specific instances of two general conjectures in arithmetic algebraic geometry (bijectivity of $p$-adic regulator maps, injectivity of $p$-adic Abel--Jacobi maps), we prove several cases of the $p$-part of the Tamagawa number conjecture ($p$-TNC) of Bloch--Kato and Fontaine--Perrin-Riou for (homological) motives of modular forms of even weight $\geq4$ in analytic rank $1$. More precisely, we prove our results for a large class of newforms $f$ and prime numbers $p$ that are ordinary for $f$ and such that the weight of $f$ is congruent to $2$ modulo $2(p-1)$. Inspired by work of W. Zhang in weight $2$, the key ingredient in our strategy is an analogue for $p$-adic Galois representations attached to higher (even) weight newforms of Kolyvagin's conjecture on the $p$-indivisibility of derived Heegner points on elliptic curves, which we prove via a $p$-adic variation method exploiting the arithmetic of Hida families. Along the way, we also prove (under similar assumptions) the $p$-TNC for modular motives in analytic rank $0$ and the rationality conjecture of Beilinson and Deligne on the existence of zeta elements on the fundamental line in analytic ranks $0$ and $1$. Prior to this work, the only known results on (questions related to) the $p$-TNC for modular motives were in weight $2$ and analytic rank $\leq1$ and in even weight and analytic rank $0$. As further applications of our result on Kolyvagin's conjecture in higher weight, we deduce a structure theorem for Selmer groups, $p$-parity results, converse theorems and higher rank results for modular forms and modular motives. 
\end{abstract}

\address{Dipartimento di Matematica, Universit\`a di Padova, Via Trieste 63, 35121 Padova, Italy}
\email{mlongo@math.unipd.it}
\address{Dipartimento di Matematica, Universit\`a di Genova, Via Dodecaneso 35, 16146 Genova, Italy}
\email{stefano.vigni@unige.it}



\maketitle

\tableofcontents

\section{Introduction} \label{Intro}

The Tamagawa number conjecture (TNC, for short) of Bloch and Kato (\cite{BK}) predicts formulas for special values of $L$-functions of motives and represents a vast generalization of the analytic class number formula and of the Birch--Swinnerton-Dyer conjecture for abelian varieties. The conjecture of Bloch--Kato, which was originally expressed (by analogy with the theory of algebraic groups) in terms of Haar measures and Tamagawa numbers, was later reformulated and extended by Fontaine and Perrin-Riou (\cite{FPR}; \emph{cf.} also \cite{fontaine-bourbaki}) using the language of determinants of complexes and Galois cohomology; similar ideas were developed also by Kato (\cite{Kato-kodai}, \cite{Kato-Iwasawa}). The Tamagawa number conjecture was then generalized by Burns and Flach to an equivariant setting that covers the case of motives with not necessarily commutative coefficients (\cite{BF}, \cite{BF2}), thus giving birth to the so-called equivariant Tamagawa number conjecture.

The main result of the present paper is a proof of the $p$-part of the Tamagawa number conjecture ($p$-TNC) for the Grothendieck (\emph{i.e.}, homological) motive of a modular form $f$ in analytic rank $1$, under some very specific instances of two general conjectures in arithmetic algebraic geometry (bijectivity of $p$-adic regulator maps, injectivity of $p$-adic Abel--Jacobi maps) and some technical assumptions on $f$ and $p$. In the rest of this introduction we will describe our results; this will also give us an occasion to outline the structure of the article.

\subsection{A reformulation of $p$-TNC for modular motives}

The first result we describe is a reformulation of the $p$-part of the TNC for the motive of a higher, even weight modular form.

\subsubsection{Modular motives and their arithmetic invariants} \label{modular-arithmetic-subsubsec}

Let $N\geq1$ be an integer, let $k\geq4$ be an even integer and let $f\in S_k(\Gamma_0(N))$ be a normalized newform of weight $k$ and level $\Gamma_0(N)$, whose $q$-expansion will be denoted by $f(q)=\sum_{n\geq1}a_n(f)q^n$. Let $F\defeq\Q\bigl(a_n(f)\mid n\geq1\bigr)\subset\C$ be the totally real number field generated over $\Q$ by the Fourier coefficients of $f$ and let $\cO_F$ be its ring of integers. Put $F_\infty\defeq F\otimes_\Q\R$; moreover, for a prime number $p$ set also $F_p\defeq F\otimes_\Q\Q_p$ and $\cO_p\defeq\cO_F\otimes_\Z\Z_p$. We attach to $f$ (and a prime $p$) the following objects.

\begin{itemize}
\item The motive $\MM=(X,\Pi,k/2)$ of $f$. This is a Grothendieck (\emph{i.e.}, homological) motive defined over $\Q$ with coefficients in $F$, equipped with its \'etale realization $V_p$ for each prime number $p$ (which is an $F_p$-module), its Betti realization $V_\B$ and its de Rham realization $V_\dR$ (which are $F$-vector spaces), and comparison isomorphisms between these realisations. Here $X$ is the Kuga--Sato variety of level $N$ and weight $k$, while $\Pi$ is a projector on the ring of correspondences of $X$; see \S \ref{subsecmot} and \S \ref{realizations-subsec} for details. 
\item The (Bloch--Kato) Shafarevich--Tate group $\Sha_p^{\BKK}(\Q,\MM)$ of $\MM$ at $p$, which is defined as the quotient of the Bloch--Kato Selmer group of $V_p/T_p$ by its maximal $p$-divisible subgroup, where $T_p$ is a suitable Galois-stable $\cO_p$-lattice in $V_p$ (see \S \ref{STsubsec}). In our arguments, the interplay between the \emph{finite} group $\Sha_p^{\BKK}(\Q,\MM)$ and the Shafarevich--Tate group $\Sha_p^{\Nek}(\Q,\MM)$ of Nekov\'a\v{r}, which is the quotient of the Bloch--Kato Selmer group of $V_p/T_p$ by the image of a certain $p$-adic Abel--Jacobi map, will be crucial.
\item For every place $v$ of $\Q$ and prime $p$, a Tamagawa $\cO_p$-ideal $\Tam_v^{(p)}(\MM)$, whose definition is recalled in \S\ref{sectam} (in particular, $\Tam_v^{(p)}(\MM)=\cO_p$ for all but finitely many $v$). 
\item The $p$-torsion part $\Tors_p(\MM)$ of $\MM$ (see \S \ref{p-torsion-M-subsubsec}).
\item The period $\Omega_\MM\in(F\otimes_\Q\C)^\times$ coming from the comparison isomorphism between Betti and de Rham realizations (actually, in developing our arguments we work with a period $\Omega_\infty\in F_\infty^\times$, defined in \S\ref{periodmap}, that takes care of an appropriate twist in the Betti realization and is related to $\Omega_\MM$ by the equality $\Omega_\MM=\Omega_\infty/(2\pi i)^{k/2})$.
\item The motivic cohomology group $H^1_\mot(\Q,\MM)$, defined in \S \ref{motiviccohom}. This is a conjecturally finite-dimensional $F$-vector space; assuming this finite-dimensionality (see Conjecture \ref{finitenessconj}), we set 
\[ r_\alg(\MM)\defeq\dim_F\bigl(H^1_\mot(\Q,\MM)\bigr). \]
The $F_\infty$-module $H^1_\mot(\Q,\MM)\otimes_FF_\infty$ is equipped with a conjecturally non-degenerate height pairing in the vein of Gillet--Soul\'e (see \S \ref{GS-subsec}); we write $\Reg_{\mathscr B}(\MM)$ for the determinant of this pairing with respect to an $F$-basis $\mathscr B$ of $H^1_\mot(\Q,\MM)$, so that $\Reg_{\mathscr B}(\MM)\not=0$ if the pairing is non-degenerate. The $F_p$-module $H^1_\mot(\Q,\MM)\otimes_FF_p$ is endowed with a $p$-adic regulator map 
\[ \reg_p:H^1_\mot(\Q,\MM)\otimes_FF_p\longrightarrow H^1_f(\Q,V_p) \]
with values in the Bloch--Kato Selmer group $H^1_f(\Q,V_p)$ of $V_p$; this map is conjectured to be an isomorphism of $F_p$-modules (see Conjecture \ref{regpconj}).  
\item The \emph{completed} $L$-function $\Lambda(\MM,s)$ of $\MM$, which is an entire function on $\C$. We write $r_\an(\MM)$ (respectively, $\Lambda^*(\MM,0)$) for the order of vanishing (respectively, the leading term of the Taylor expansion) of $\Lambda(\MM,s)$ at $s=0$ (see \S \ref{Lfunctsec}). 
\end{itemize}
All these invariants will appear in our reformulation of the $p$-part of the TNC for $\MM$, which uses the language of determinants of (complexes of) projective modules, as proposed by Fontaine--Perrin-Riou in \cite{FPR} (at least when the field of coefficients is $\Q$, the formulation of Fontaine--Perrin-Riou is indeed equivalent to the one originally given by Bloch--Kato: see, \emph{e.g.}, \cite{fontaine-bourbaki} and \cite{Do} for details).

\subsubsection{A reformulation of $p$-TNC for $\MM$: assumptions} \label{assumptions-intro-subsubsec}

As above, $p$ is a prime number. To prove the result below, we work under the following assumptions, for precise statements of which we refer to later sections:

\begin{enumerate}
\item the Gillet--Soul\'e height pairing is non-degenerate (\emph{i.e.}, Conjecture \ref{nondegconj} holds true);
\item the rationality conjecture of Beilinson and Deligne on the existence of zeta elements on the fundamental line (Conjecture \ref{ratconj}) holds true;
\item the $p$-adic regulator $\reg_p$ is an isomorphism (\emph{i.e.}, Conjecture \ref{regpconj} over $\Q$ holds true);
%\item the Shafarevich--Tate group $\Sha_p(\Q,\MM)$ is finite (\emph{cf.} \S \ref{ST-subsubsec}).
\end{enumerate}
If $\reg_p$ is an isomorphism for some prime $p$, then $H^1_\mot(\Q,\MM)$ has finite dimension over $F$ (\emph{i.e.}, Conjecture \ref{finitenessconj} over $\Q$ holds true). Note that the non-degeneracy condition in (1) is imposed only to force $\Reg_{\mathscr B}(\MM)$ to be non-zero and thus can be removed once we know that, in the arithmetic situations we consider, $\Reg_{\mathscr B}(\MM)\not=0$. Significant advances on (let alone complete proofs of) any of the conjectures above in a general setting would represent major breakthroughs in arithmetic geometry: in this paper we have nothing new to say about them and simply content ourselves with assuming their validity in specific instances whenever needed. However, it is worthwhile to remark that in the low rank contexts we are interested in (\emph{i.e.}, when  $r_\an(\MM)\in\{0,1\}$) we know that $\Reg_{\mathscr B}(\MM)\not=0$ either by definition (if $r_\an(\MM)=0$, in which case $\Reg_{\mathscr B}(\MM)\defeq1$) or as a consequence of S.-W. Zhang's formula of Gross--Zagier type for higher weight modular forms (\cite{Zhang-heights}).
%\item $\Sha_p(\Q,\MM)$ is finite by a combination of work of Nekov\'a\v{r} on the arithmetic of Chow groups of Kuga--Sato varieties (\cite{Nek}) and analytic results by Bump--Friedberg--Hoffstein (\cite{BFH}), Murty--Murty (\cite{MM-derivatives}) and Waldspurger (\cite{Waldspurger}).
Furthermore, assuming the injectivity of certain $p$-adic Abel--Jacobi maps, we can also prove that if $r_\an(\MM)\in\{0,1\}$, then the rationality conjecture of Beilinson and Deligne is true (Theorems \ref{ratconj0thm} and \ref{ratconj1thm}).

\subsubsection{A reformulation of $p$-TNC for $\MM$: statement}

In the following lines, for a finitely generated $\cO_p$-module $M$ we denote by $\mathcal{I}(M)$ the $\cO_p$-ideal such that
\[ \ord_\p\bigl(\mathcal{I}(M)\bigr)=\mathrm{length}_{\cO_\p}(M) \]
for each prime $\p$ of $F$ above $p$, where $\cO_\p$ is the completion of $\cO_F$ at $\p$ and $\ord_\p$ is the $\p$-adic valuation.

\begin{ThmA}
Under the assumptions in \S \ref{assumptions-intro-subsubsec}, the $p$-part of the TNC for $\MM$ is equivalent to the equality
\[ \biggl(\frac{\Lambda^*(\MM,0)}{\Omega_\MM\cdot\Reg_{\mathscr B}(\MM)}\biggr)=\frac{\mathcal{I}\bigl(\Sha_p^{\BKK}(\Q,\MM)\bigr)\cdot\mathcal{I}_p(\gamma_f)\cdot\prod_{v\in S}\mathrm{Tam}_v^{(p)}(\MM)}{\bigl(\det(\mathtt{A})\bigr)^2\cdot\Tors_p(\MM)} \]
of fractional $\cO_p$-ideals.
\end{ThmA}

The reader is referred to \S \ref{reformulation-subsec} for the terms $\mathcal{I}_p(\gamma_f)$ and $\mathtt{A}\in\GL_{r_\alg(\MM)}(F_p)$, the latter being denoted by $\mathtt{A}_{\tilde{\mathscr{B}}}$ later in the text. To sketchily explain their roles, we observe that the definitions of some of the objects appearing in Theorem A, which were introduced in \S \ref{modular-arithmetic-subsubsec}, involve choices (not reflected in the notation above) of suitable bases; these are encoded in the terms $\mathcal{I}_p(\gamma_f)$ and $\mathtt{A}$, and then it can be checked that the validity of the resulting formula is independent of such choices. 

To the best of our knowledge, Theorem A, which corresponds to Theorem \ref{motivesthm}, offers the first reformulation of such an explicit kind of $p$-TNC for $\MM$ in arbitrary analytic rank; a similar interpretation when $r_\an(\MM)=0$ was proposed by Dummigan--Stein--Watkins (\cite{DSW}). 

\subsection{$p$-TNC for $\MM$ in analytic rank $1$}

We are now in a position to describe our main result on the $p$-TNC for the modular motive $\MM$.

\subsubsection{$p$-TNC for $\MM$ in analytic rank $1$: assumptions} \label{assumptions-intro-subsubsec2}

We prove our result under the following assumptions:
\begin{enumerate}
\item an integral variant of $\reg_p$ is an isomorphism (see \S \ref{integral-motivic-subsubsec2} and \S \ref{integral-conjecture-subsubsec});
\item certain $p$-adic Abel--Jacobi maps are injective.
\end{enumerate}
Moreover, we also assume all the conditions described in \S \ref{kolyvagin-assumptions-subsubsec} below (or, rather, minor variations thereof), so as to be able to apply Theorem C. Observe, in particular, that the square-freeness of $N$ forces $f$ not to be CM; in addition, $k$ and $p$ must satisfy the congruence $k\equiv2\pmod{2(p-1)}$. Here we are deliberately vague about hypothesis (2), as the actual injectivity properties of $p$-adic Abel--Jacobi maps that are needed are too technical to state in this introduction: we just remark that, while it seems to be a ``folklore'' conjecture that such maps are always injective, in this article we need to impose this injectivity condition only in very specific cases (\emph{cf.} Remark \ref{injectivity-AJ-rem} for further comments).

\subsubsection{$p$-TNC for $\MM$ in analytic rank $1$: statement}

As before, $r_\alg(\MM)$ (respectively, $r_\an(\MM)$) denotes the algebraic (respectively, analytic) rank of $\MM$.

\begin{ThmB}[$p$-TNC for $\MM$] \label{introTNC}
Suppose that $r_\an(\MM)=1$. Under the assumptions in \S \ref{assumptions-intro-subsubsec2}, the following results hold:
\begin{enumerate}
\item $r_\alg(\MM)=1$; 
\item $\Sha_p^{\BKK}(\Q,\MM)=\Sha_p^{\Nek}(\Q,\MM)$;  
\item the $p$-part of the TNC for $\MM$ is true.
\end{enumerate}
\end{ThmB}

This result corresponds to Theorem \ref{ThmTNC}. Albeit not available (as far as we know) in the literature, and never formulated for the motive $\MM$, parts (1) and (2) were essentially already known, thanks to a combination of work of Nekov\'a\v{r} on the arithmetic of Chow groups of Kuga--Sato varieties (\cite{Nek}) and analytic results by Bump--Friedberg--Hoffstein (\cite{BFH}), Murty--Murty (\cite{MM-derivatives}) and Waldspurger (\cite{Waldspurger}); thus, the novelty of Theorem B lies almost entirely in part (3). We prove the $p$-part of the TNC for $\MM$ by showing that, under the assumptions described above, the equality in Theorem A is satisfied. In doing so, a key role is played by our proof of a higher weight counterpart of a conjecture due to Kolyvagin about the non-triviality of his system of ``derived'' Galois cohomology classes built out of Heegner points on elliptic curves (\cite[Conjecture A]{kolyvagin-selmer}): in \S \ref{kolyvagin-intro-subsec}, we outline our arguments for proving this Kolyvagin-type conjecture. Among the several other ingredients that enter our proof of Theorem B, we would like to highlight fundamental results by Kato (\cite{Kato}) and by Skinner--Urban (\cite{SU}) on the Iwasawa theory of modular forms, which led us to a proof of an analogue of Theorem B (in particular, of the $p$-part of TNC for $\MM$) in analytic rank $0$ (Theorem \ref{skinner-urban-thm}).

\subsection{Kolyvagin's conjecture in higher weight} \label{kolyvagin-intro-subsec}

Inspired by work of W. Zhang in weight $2$ (\cite{zhang-selmer}), the key ingredient in our proof of Theorem B is an analogue for $p$-adic Galois representations attached to higher (even) weight newforms of Kolyvagin's conjecture on the $p$-indivisibility of derived Heegner points on rational elliptic curves, which we prove via a $p$-adic variation method exploiting the arithmetic of Hida families of modular forms.

\subsubsection{Kolyvagin's conjecture: assumptions} \label{kolyvagin-assumptions-subsubsec}

Let $\p$ be a prime of $F$ above the prime number $p$. Write $D_F$ for the discriminant of $F$ and $c_f$ for the index of the order $\Z\bigl[a_n(f)\mid n\geq1\bigr]$ in $\cO_F$. We prove Kolyvagin's conjecture under the following assumptions on the pair $(f,\p)$: 
\begin{enumerate}
\item $N\geq 3$ is square-free;
\item \label{p-f-ass} $p\nmid 6ND_Fc_f$;
\item $k\equiv 2\pmod{2(p-1)}$; 
\item $f$ is $p$-isolated, \emph{i.e.}, there are no non-trivial congruences modulo $p$ between $f$ and normalized eigenforms in $S_k(\Gamma_0(N))$;
\item \label{ordinary-ass} $a_p(f)\in\cO_\p^\times$; 
\item \label{non-cong-ass} $a_p(f)\not\equiv1\pmod{{\p}}$.
\end{enumerate}
We further require the $p$-adic Galois representation attached to $f$ to have big image and impose suitable irreducibility and ramification conditions on residual representations at primes dividing $N$ (\emph{cf.} \S \ref{kolyvagin-ass-subsubsec}). It turns out that condition (4) is satisfied for all but finitely many $p$. With the exception of (3), which we briefly comment upon in \S \ref{kolyvagin-strategy-intro-subsubsec}, these assumptions are analogous to those appearing in weight $2$ in \cite{zhang-selmer}: at least in principle, they could be relaxed (\emph{cf.} Remark \ref{non-ordinary-rem} for the ordinariness condition (5)), but doing so would add extra technicalities to the proofs, while bringing at the same time no significant novelty to the main arguments. Finally, observe that the other assumptions in \S \ref{assumptions-intro-subsubsec} and \S \ref{assumptions-intro-subsubsec2} play no role in the statement and proof of Kolyvagin's conjecture.

\subsubsection{Kolyvagin's conjecture: statement}

Let $p\nmid6ND_Fc_f$ be a prime number such that the $p$-adic representation attached to $f$ has big image and the residual representation at $\p$ is irreducible for each $\p\,|\,p$: this rules out only finitely many $p$. Choose an imaginary quadratic field $K$ where all the prime factors of $Np$ split. Fix a prime $\p$ of $F$ above $p$. Using Heegner cycles on Kuga--Sato varieties (\cite{Nek}) in place of Heegner points on elliptic curves, we mimic a recipe of Kolyvagin and define a set $\kappa_{f,\infty}$ of Kolyvagin-type ``derived'' Galois cohomology classes in $H^1(K,T_\p/p^MT_\p)$ for suitable integers $M$, where $T_\p\defeq T_p\otimes_{\cO_p}\!\cO_\p$. We call $\kappa_{f,\infty}$ the \emph{Kolyvagin set} associated with $f$, $K$, $\p$: see \S \ref{kolyvagin-integers-subsec} and \S \ref{kolyvagin-classes-subsec} for the detailed construction of $\kappa_{f,\infty}$. 

\begin{ThmC}[Kolyvagin's conjecture] \label{main-kolyvagin-thm}
Under the assumptions in \S \ref{kolyvagin-assumptions-subsubsec}, $\kappa_{f,\infty}\not=\{0\}$.
\end{ThmC}

Theorem C, which corresponds to Theorem \ref{kolyvagin-main-thm}, shows that the higher (even) weight counterpart of Kolyvagin's conjecture for elliptic curves that was first formulated in \cite[Conjecture A]{Masoero} holds true for ordinary primes $p$ satisfying the conditions described above. Actually, in Theorem \ref{kolyvagin-main-thm} we prove a stronger statement that implies Kolyvagin's conjecture. We remark that Kolyvagin's original conjecture was proved (at least in the ordinary case, under some technical assumptions) by W. Zhang for $p\nmid N$ and by Skinner--Zhang for $p\,\|N$. Other than a crucial role in our proof of Theorem B, Theorem C has also consequences on structure theorems for Selmer groups, $p$-parity results, converse theorems and higher rank results for modular forms and modular motives, some of which are outlined in \S \ref{consequences-intro-subsec}.

\subsubsection{Kolyvagin's conjecture: strategy of proof} \label{kolyvagin-strategy-intro-subsubsec}

Our strategy for proving Theorem C is based on a deformation-theoretic approach; in a nutshell, it goes as follows:
\begin{enumerate}
\item we take the $p$-adic Hida family $\f$ passing through our $p$-ordinary form $f$ (or, rather, through the $p$-stabilization of $f$);
\item we consider big Heegner points $\mathfrak X_n\in H^1(K_n,\T^\dagger)$ \emph{\`a la} Howard, where $K_n$ is the ring class field of $K$ of conductor $n$ and $\T^\dagger$ is the critical twist of Hida's ``big Galois representation'' attached to $\f$;
\item we define Kolyvagin-type classes $d(\f,n)\in H^1(K_n,\T^\dagger)$ built out of the $\mathfrak X_n$;
\item finally, we combine results of Zhang (\cite{zhang-selmer}) and Skinner--Zhang (\cite{SZ}) on Kolyvagin's conjecture for (modular) abelian varieties and specialization results of Howard (\cite{Howard-derivatives}), Castella (\cite{CasHeeg}) and Ota (\cite{Ota-JNT}) for big Heegner points to deduce, using the classes $d(\f,n)$, Kolyvagin's conjecture for $f$ from the corresponding statement in weight $2$.
\end{enumerate}
The need to exploit the specialization results alluded to in (4) is one of the reasons why we require the congruence $k\equiv2\pmod{2(p-1)}$ to hold. It would be interesting to give a direct proof of Theorem C by generalizing to higher weight the arguments in \cite{SZ} and \cite{zhang-selmer}: this would presumably allow one to drop the congruence condition above (see, \emph{e.g.}, \cite{wang} for partial results in this direction). Our motivations for this strategy towards Kolyvagin's conjecture in higher weight were at least two: first of all, we found it quite natural to use the results already available in weight $2$ as a ``bridge'' to the general case; on the other hand, in our main result on the $p$-TNC for $\MM$ we would need to impose a congruence assumption on $k$ and $p$ anyway, as such a congruence is required in the work of Skinner--Urban on the Iwasawa main conjecture for modular forms (\cite{SU}), which is of paramount importance for our arguments. 

To further elaborate on this point, for a given $f$ the congruence $k\equiv2\pmod{2(p-1)}$ is clearly satisfied only by finitely many primes $p$. However, by arguing as follows we can offer infinitely many examples of pairs $(f,p)$ fulfilling this condition. Let $f$ be a newform of weight $2$, level $\Gamma_0(N)$ and trivial character and let $p\nmid N$ be an ordinary prime for $f$, then take the Hida family passing through the $p$-stabilization of $f$. There are infinitely many cusp forms of weight $k$ such that $k\equiv 2\pmod{2(p-1)}$, level $\Gamma_0(Np)$ and trivial character appearing as specializations of the Hida family at $k$: these forms are ordinary $p$-stabilizations of newforms of weight $k$, level $\Gamma_0(N)$ and trivial character to which our results apply. 

\subsection{Other consequences of Theorem C} \label{consequences-intro-subsec}

As hinted at above, we deduce from Theorem C, in addition to the $p$-TNC for $\MM$, a structure theorem for Selmer groups (Theorem \ref{main-vanishing-thm}), a $p$-parity result (Theorem \ref{parity}), converse theorems (see, \emph{e.g.}, Theorem \ref{main-converse-Q-thm}) and higher rank results (Theorem \ref{main-higher-rank}). The next theorem is a sample of these results.
 
\begin{ThmD}
Under suitable assumptions on height pairings and $p$-adic Abel--Jacobi maps, the following statements are true: \begin{enumerate}
\item $r_\alg(\MM)\equiv r_\an(\MM)\pmod{2}$; 
\item if $r_\alg(\MM)=1$, then $r_\alg(\MM)=r_\an(\MM)$;
\item if $r_\an(\MM)>1$ is even, then $r_\alg(\MM)\geq2$;
\item if $r_\an(\MM)>1$ is odd, then $r_\alg(\MM)\geq3$.
\end{enumerate}
\end{ThmD}

Each result in Theorem D is proved under its own set of specific assumptions on the non-degeneracy of Gillet--Soul\'e height pairings and the injectivity of $p$-adic Abel--Jacobi maps: since their formulations are rather intricate, we do not attempt to describe these hypotheses here and simply refer to Sections \ref{structure-sec}--\ref{higher-sec} for all details.

\subsection{Relation to the existing literature}

Prior to this work, the only known results on (questions related to) the $p$-TNC for modular motives were in weight $2$ and analytic rank at most $1$ and in even weight and analytic rank $0$. More precisely, the $p$-part of the Birch and Swinnerton-Dyer formula for elliptic curves over $\Q$ (\emph{i.e.}, for weight $2$ newforms with rational Fourier coefficients) of analytic rank at most $1$ has been the subject, under different arithmetic assumptions and in various degrees of generality, of intense study in recent years. Here we would like to mention, in (rough) chronological order, the papers by Kobayashi (\cite{Kob}), Skinner--Urban (\cite{SU}), W. Zhang (\cite{zhang-selmer}), Skinner--Zhang (\cite{SZ}), Berti--Bertolini--Venerucci (\cite{BBV}), Jetchev--Skinner--Wan (\cite{JSW}), Castella (\cite{castella-BSD}). None of these articles is written in a motivic language or refers to the Tamagawa number conjecture explicitly, but all of them prove \emph{de facto} results on the $p$-TNC for the motives of elliptic curves over $\Q$, as it is known that the (complete) Birch--Swinnerton-Dyer conjecture for an elliptic curve is equivalent to the Tamagawa number conjecture for the corresponding motive (see, \emph{e.g.}, \cite{Kings} for a detailed explanation of this equivalence, which is highly non-trivial). 

As for motives of higher weight modular forms, work of Dummigan--Stein--Watkins (\cite{DSW}) deals with the analytic rank $0$ case. Results in a rank $0$ setting have been obtained also by Fouquet--Wan (\cite{FW}). More recently, some of the results in weight $2$ in the above-mentioned papers were partially extended to higher weights by Thackeray (\cite{Thackeray-JNT}). In particular, a formula was proved that relates the orders of Shafarevich--Tate groups to logarithms of generalized Heegner cycles \emph{\`a la} Bertolini--Darmon--Prasanna (\cite{BDP}); this formula might be linked to ours and could perhaps be used to deduce, following our approach, the $p$-TNC for $\MM$. Finally, we point out that, along a different line of investigation, Diamond--Flach--Guo studied the Tamagawa number conjecture for adjoint motives of modular forms (\cite{DFG-MRL}, \cite{DFG}).

\subsection{Notation and conventions} \label{notation-subsec}

We denote by $\bar\Q$ the algebraic closure of $\Q$ inside $\C$ and write $\bar\Z$ for the ring of integers in $\bar\Q$ (\emph{i.e.}, the integral closure of $\Z$ in $\bar\Q$). For every prime number $\ell$ we fix an algebraic closure $\bar\Q_\ell$ of $\Q_\ell$. 
%Moreover, for every prime $\ell$ and every number field $F$ we also fix field embeddings
%\footnote{\red{Matteo: Non sono sicuro che sia una buona scelta fissare l'embedding fin dall'inizio; possiamo farlo, ma poi per trattare i diversi $\p$ dobbiamo considerare diversi embeddings, quindi tanto vale non fissarli, no?}}
%\[ \iota_\ell:\bar\Q\longmono\bar\Q_\ell,\quad\iota_F:F\longmono\bar\Q. \]
%The map $\iota_\ell$ determines a prime ideal $\mathfrak L$ of $\bar\Z$ above $\ell$ that, in turn, induces a prime ideal $\mathfrak L_F\defeq\iota_F^{-1}\bigl(\mathfrak L\cap\iota_F(F)\bigr)$ of $F$ above $\ell$. In order to simplify our notation, when there is no risk of confusion we will often use alternative symbols to denote the ideal $\mathfrak L_F$ and related objects. For example, we write $F_\mathfrak L$ in place of $F_{\mathfrak L_F}$ for the completion of $F$ at the prime $\mathfrak L_F$.

For any number field $K$ we denote by $G_K\defeq\Gal(\bar K/K)$ the absolute Galois group of $K$, where $\bar K$ is a fixed algebraic closure of $K$. For any continuous $G_K$-module $M$ we write $H^i(K,M)$ for the $i$-th continuous cohomology group of $G_K$ with coefficients in $M$ in the sense of Tate (\cite[\S 2]{Tate}). Finally, if $K/F$ is an extension of number fields, then  
\[ \res_{K/F}:H^i(F,M)\longrightarrow H^i(K,M),\quad\cores_{K/F}:H^i(K,M)\longrightarrow H^i(F,M) \] 
denote the restriction and corestriction maps in cohomology, respectively.

\subsection*{Acknowledgements} 

Various incarnations of this project have been the subject of several talks over the past five years; this gave us the opportunity to collect precious feedback and suggestions from our audiences. In particular, we would like to thank Kazim B\"uy\"ukboduk, Henri Darmon, Jeffrey Hatley, Ming-Lun Hsieh, Chan-Ho Kim, Antonio Lei, Daniel Macias Castillo, Daniele Masoero and Rodolfo Venerucci for their interest in our work and helpful discussions on some of the topics of this paper. Our special gratitude goes to Congling Qiu for computing for us the self-intersection of a certain Heegner-type cycle that plays a crucial role in our arguments.

\section{The TNC for motives of modular forms} \label{secmot} 

We describe the Tamagawa number conjecture (TNC, for brevity) of Bloch--Kato (\cite[Conjecture 5.15]{BK}) in the case of motives of modular forms. As will be clear, our exposition follows \cite{FPR} and \cite{Kings} quite closely. We remark that the TNC for modular forms in analytic rank $0$ was also considered by Dummigan--Stein--Watkins in \cite{DSW}, while Diamond--Flach--Guo studied in \cite{DFG} the TNC for adjoint motives of modular forms. Results in rank $0$ have also been obtained by Fouquet--Wan in \cite{FW}.

Although we are chiefly interested in the TNC over $\Q$, we introduce some of the relevant notions (\emph{e.g.}, motivic cohomology, $L$-functions, Selmer groups, Shafarevich--Tate groups) over arbitrary number fields; in particular, we will eventually need to work over certain imaginary quadratic fields. For details on the TNC for more general motives, the reader may consult, \emph{e.g.}, \cite{BF}, \cite{Flach-tamagawa}, \cite{Kings}.

\subsection{Review of motives} \label{motives-subsec}

We briefly review the basic definitions in the theory of motives; for details, the reader is referred, \emph{e.g.}, to \cite[Ch. 4]{Andre}, \cite{manin-motives}, \cite{scholl-motives}. 

\subsubsection{Pure motives} \label{pure-subsubsec}

Let $K$ be a field and write $\mathcal V_K$ for the category of smooth projective schemes over $K$. Given an object $X$ of $\mathcal V_K$ and $d\in\N$, denote by $\mathcal Z^d(X)$ the group of cycles of codimension $d$ on $X$, \emph{i.e.}, the free abelian group generated by the irreducible subschemes of $X$ of codimension $d$. Let $\sim$ be an adequate equivalence relation on cycles (see, \emph{e.g.}, \cite[D\'efinition 3.1.1.1]{Andre}) and let $R$ be a commutative ring. Set $\mathcal Z^d_\sim(X{)}_R\defeq\bigl(\mathcal Z^d(X)\otimes_\Z R\bigr)\big/\sim$; it will also be convenient to put $\mathcal Z^d_\sim(X{)}_R=0$ for $d\in\Z_{<0}$. Let $X,Y$ be objects of $\mathcal V_K$ and suppose that $X$ is of pure dimension $d$; the group of correspondences modulo $\sim$ of degree $r$ from $X$ to $Y$ with coefficients in $R$ is 
\[ \Corr^r_\sim(X,Y{)}_R\defeq\mathcal Z^{d+r}_\sim(X\times_K Y{)}_R. \]
If $X$ is not of pure dimension, then $\Corr^r_\sim(X,Y{)}_R$ can be defined in terms of the irreducible components of $X$ as in \cite[\S 1.3]{scholl-motives}. Composition of correspondences furnishes, via intersection theory, a product structure
\[ \Corr^r_\sim(X,Y{)}_R\times\Corr^s_\sim(Y,Z{)}_R\longrightarrow\Corr^{r+s}_\sim(X,Z{)}_R. \]
In particular, $\Corr^0_\sim(X,X{)}_R$ inherits a ring structure for every object $X$ of $\mathcal V_K$. Now let $\mathcal{V}^0_{K,R}$ be the category whose objects are those of $\mathcal{V}_K$ and whose morphisms are given by degree $0$ correspondences modulo $\sim$ with coefficients in $R$. By definition, the category $\mathscr M_\sim(K{)}_R$ of \emph{pure $\sim$-motives (defined) over $K$ with coefficients in $R$} is the pseudo-abelian completion of $\mathcal{V}^0_{K,R}$; see, \emph{e.g.}, \cite[\S 4.1]{Andre}, \cite[Definition 1.1]{Kings} for details. More explicitly, a (pure) $\sim$-motive over $K$ with coefficients in $R$ is a triple
\[ \MM=(X,q,r) \] 
where $X$ is an object of $\mathcal V_K$, $q\in\Corr^0_\sim(X,X{)}_R$ is an idempotent and $r\in\Z$. If $r=0$, then $\MM$ is said to be \emph{effective}. Furthermore, if $\MM_1=(X_1,q_1,r_1)$ and $\MM_2=(X_2,q_2,r_2)$ are motives, then
\[ \Hom_{\mathscr M_\sim(K{)}_R}(\MM_1,\MM_2)=q_2\cdot\Corr^{r-s}_\sim(X,Y{)}_R\cdot q_1\subset\Corr^{r-s}_\sim(X,Y{)}_R. \]
Given motives $\MM_i=(X_i,q_i,r_i)$ for $i=1,2$, the \emph{product} of $\MM_1$ and $\MM_2$ is  
\[ \MM_1\otimes_K\MM_2\defeq(X_1\times_KX_2,q_1\times_Kq_2,r_1+r_2), \]
One can also define the \emph{direct sum} of two motives (see, \emph{e.g.}, \cite[\S 1.14]{scholl-motives}), and it turns out that $\mathscr M_\sim(K{)}_R$ is an additive, $R$-linear, pseudo-abelian category (\cite[Theorem 1.6]{scholl-motives}).

The \emph{dual} of a motive $\MM=(X,q,r)$ is 
\[ \MM^\vee\defeq\bigl(X,q^t,\dim(X)-r\bigr), \]
where $q^t$ is the image of the idempotent $q$ under the map that interchanges the factors of $X\times_K X$ (\emph{cf.} \cite[Definition 1.2]{Kings}). A motive $\MM$ is \emph{self-dual} if $\MM\simeq\MM^\vee(1)$, where $\star(1)$ denotes Tate twist (see, \emph{e.g.}, \cite[\S 4.1.5]{Andre}). Finally, if $X$ is an object of $\mathcal V_K$ and $\Delta_X$ is the diagonal in $X\times_KX$, then the effective $\sim$-motive $(X,\Delta_X,0)$ is the \emph{$\sim$-motive of $X$}. 

\subsubsection{Chow motives} \label{chow-subsubsec}

Taking $\sim$ to be rational equivalence (see, \emph{e.g.}, \cite[\S 3.2.2]{Andre}), we obtain the category $\mathscr M_\rat(K{)}_R$ of \emph{rational} (or \emph{Chow}) \emph{motives over $K$ with coefficients in $R$}. As we will see in \S \ref{subsecmot}, the motive of modular forms of given weight and level is an object of $\mathscr M_\rat(\Q{)}_R$ for $R$ a suitable Hecke algebra. However, in order to define the motive of a single modular form one needs to pass to the category of Grothendieck (\emph{i.e.}, homological) motives, which we introduce below.

\subsubsection{Grothendieck motives} \label{grothendieck-subsubsec}

With notation as in \S \ref{pure-subsubsec}, the category of \emph{homological} \emph{motives over $K$ with coefficients in $R$} is $\mathscr M_\hhom(K{)}_R$, where ``hom'' indicates homological equivalence (see, \emph{e.g.}, \cite[\S 3.3.4]{Andre}). Following Scholl (\cite[\S1.2.3]{Scholl}), we shall call the objects of $\mathscr M_\hhom(K{)}_R$ \emph{Grothendieck motives}. Tensor products, duals and self-duality of Grothendieck motives are defined formally as for Chow motives. 

Rational equivalence is finer than homological equivalence (in fact, rational equivalence is the finest of all adequate equivalence relations; see, \emph{e.g.}, \cite[Lemme 3.2.2.1]{Andre}), so there is a natural functor
\begin{equation} \label{chow-functor-eq}
\mathscr F_K:\mathscr M_\rat(K{)}_R\longrightarrow\mathscr M_\hhom(K{)}_R
\end{equation}
that is the identity on objects and allows one to view a Chow motive as a Grothendieck motive.

\subsection{Motives of modular forms}\label{subsecmot} 

Let $N\geq3$ be an integer and $k\geq4$ be an even integer. We want to introduce, following Scholl (\cite{Scholl}), the (Grothendieck) motive of a fixed modular form of weight $k$ and level $N$, whose realizations will be carefully described in \S \ref{realizations-subsec}. In order to do this, we need to introduce first the (Chow) motive of all modular forms of weight $k$ and level $N$.

\subsubsection{Anaemic Hecke algebras} \label{anaemic-hecke-subsubsec}

Write $\mathfrak{H}_k(\Gamma(N))\subset\End_{\bar\Q}\bigl(S_k(\Gamma(N),\bar\Q)\bigr)$ for the $\Z$-algebra generated by the Hecke operators $T_n$ with $(n,N)=1$; it is often called the ``anaemic'' Hecke algebra of weight $k$ and level $\Gamma(N)$. Analogously, denote by $\mathfrak{H}_k(\Gamma_0(N))$ the anaemic Hecke algebra of weight $k$ and level $\Gamma_0(N)$. There is a natural surjection
\begin{equation} \label{hecke-map-eq}
\mathfrak H_k(\Gamma(N))\longepi\mathfrak{H}_k(\Gamma_0(N))
\end{equation}
of $\Z$-algebras that is induced by the inclusion $S_k(\Gamma_0(N))\subset S_k(\Gamma(N))$. Finally, for any $\Z$-algebra $A$, set $\mathfrak{H}_k(\Gamma_0(N))_A\defeq\mathfrak{H}_k(\Gamma_0(N))\otimes_\Z A$. 

\subsubsection{The motive of modular forms of weight $k$ and level $N$} \label{motive-modular-form-subsubsec}

Denote by $\tilde{\mathcal{E}}_N^{k-2}$ the \emph{Kuga--Sato variety} of level $N$ and weight $k$, \emph{i.e.}, the smooth projective $\Q$-scheme defined as the canonical desingularization of the $(k-2)$-fold product $\mathcal{E}_N^{k-2}$ of the universal generalized elliptic curve $\pi:\mathcal{E}_N\rightarrow X(N)$ over the compact modular curve $X(N)$ of level $\Gamma(N)$ (see, \emph{e.g.}, \cite[\S1.2.0]{Scholl}). An idempotent $\Pi_\epsilon$ in the ring of correspondences of degree $0$ of $\tilde{\mathcal{E}}_N^{k-2}$ can be constructed as follows. Recall that if we set 
\begin{equation} \label{Gamma-k-2-eq}
\Gamma_{k-2}\defeq \bigl((\Z/N\Z)^2\rtimes\{\pm1\}\bigr)^{k-2}\rtimes\mathfrak{S}_{k-2}, 
\end{equation}
where $\mathfrak{S}_{k-2}$ is the symmetric group on $k-2$ letters, then there is a canonical action of $\Gamma_{k-2}$ on $\tilde{\mathcal{E}}_N^{k-2}$ (\cite[\S1.1.0, \S1.1.1]{Scholl}); we define $\Pi_\epsilon$ to be the projector associated with the character $\epsilon:\Gamma_{k-2}\rightarrow\{\pm1\}$ that is the sign character on $\mathfrak{S}_{k-2}$, the trivial character on $(\Z/N\Z)^{2(k-2)}$ and the product character on $\{\pm1\}^{k-2}$; see \cite[\S1.1.2]{Scholl} for details (note, in particular, that in order to define $\Pi_\epsilon$ we need to invert $2N(k-2)!$). Moreover, write $\Pi_B$ for the idempotent attached to the quotient $\Gamma_0(N)/\Gamma(N)$, whose order $t_N$ we need to invert in order to define $\Pi_B$. Let 
\[ \MM_k(N)\defeq\bigl(\tilde{\mathcal{E}}_N^{k-2},\Pi_B\Pi_\epsilon,k/2\bigr) \] 
be the Chow motive of modular forms of weight $k$ and level $N$ (\cite[\S1.2.2]{Scholl}). The motive $\MM _k(N)$ is an object of $\mathscr M_\rat(\Q{)}_{\mathfrak H_k(\Gamma(N))}$, where $\mathfrak H_k(\Gamma(N))$ is the Hecke algebra acting on modular forms of weight $k$ and level $\Gamma(N)$ (\emph{cf.} \cite[Definition 1.4]{Kings} and \cite[Proposition 4.1.3]{Scholl}). Note that $\MM _k(N)$ is self-dual. 

\subsubsection{The motive of a modular form}

Let $f\in S_k(\Gamma_0(N))$ be a normalized newform of weight $k$ and level $\Gamma_0(N)$, whose $q$-expansion will be denoted by $f(q)=\sum_{n\geq 1}a_n(f)q^n$. Let $F\defeq \Q\bigl(a_n(f)\mid n\geq1\bigr)$ be the Hecke field of $f$, which is a totally real number field, and let $\cO_F$ be its ring of integers. There is set-theoretic inclusion $F\subset\R$. Let $\MM (f)$ be the Grothendieck motive over $\Q$ with coefficients in $F$ attached to $f$ by Scholl (\cite[Theorem 1.2.4]{Scholl}). To construct $\MM (f)$, consider the projector $\Psi_f$ associated with $f$ (\cite[\S4.2.0]{Scholl}) and define $\MM (f)$ to be the submotive of $\MM _k(N)$ that is the kernel of $\Psi_f$; in other words, $\MM(f)$ is the object of $\mathscr M_\hhom(\Q{)}_F$ given by 
\[ \MM(f)\defeq \Bigl(\tilde{\mathcal{E}}_N^{k-2},(1-\Psi_f)\circ(\Pi_B\Pi_\epsilon\otimes1),k/2\Bigr). \] 
Like $\MM _k(N)$, the motive $\MM(f)$ is self-dual.  

\begin{remark}
The crux in Scholl's construction of $\MM(f)$ is a decomposition of $\MM_k(N)$ under the action of the Hecke algebra. This decomposition takes place in $\mathscr M_\hhom(\Q{)}_F$; that is, one needs to replace $\MM_k(N)$ with its image via the functor $\mathscr F_K$ introduced in \eqref{chow-functor-eq}. As is pointed out in \cite[Remark 1.2.6]{Scholl}, it is reasonable to expect that the splitting of $\MM_k(N)$ can be performed already in the category $\mathscr M_\rat(\Q{)}_F$ (as in the $k=2$ case), but this seems very hard to achieve without assuming Grothendieck's standard conjectures.
\end{remark}

\begin{remark} \label{motive-f^K-rem}
In this article, we will need to consider also the motive $\MM(f^K)$ of the twist $f^K$ of $f$ by the Dirichlet character associated with a suitable imaginary quadratic field $K$.
\end{remark}

\subsection{Notation} \label{notation-subsec}

To simplify our notation, from now on we set $X\defeq\tilde{\mathcal{E}}_N^{k-2}$, $\bar{X}\defeq  X\times_\Q\bar\Q$, $\Pi\defeq(1-\Psi_f)\circ(\Pi_B\Pi_\epsilon\otimes1)$, $\MM\defeq\MM(f)$. Thus, $\MM$ is the Grothendieck motive 
\[ \MM=(X,\Pi,k/2) \] 
defined over $\Q$ with coefficients in $F$. For any number field $K$, we also write 
\[ \MM/K\defeq (X\times_\Q K,\Pi\times_\Q K,k/2) \] 
for the Grothendiek motive over $K$ with coefficients in $F$ obtained from $\MM$ by extension of scalars, \emph{i.e.}, base change (see, \emph{e.g.}, \cite[\S 4.2.3]{Andre}, \cite[Remark 1.7]{Kings}).

\subsection{Realizations of $\MM$} \label{realizations-subsec} 

Here we consider only the $(k-1)$-st realizations of the motive $\MM$, which correspond to the choice $i=2r-1$ in \cite[Ch. 1, \S 2]{Kings}. We also briefly describe the realizations of $\MM/K$ obtained from those of $\MM$ by base change. 

\subsubsection{$f$-isotypic submodules} \label{isotypic-subsubsec}

With notation as in \S \ref{anaemic-hecke-subsubsec}, let $\theta_f:\mathfrak{H}_k(\Gamma_0(N))\rightarrow\cO_F$ be the ring homomorphism associated with $f$; explicitly, $\theta_f(T_n)\defeq a_n(f)$. Composing with \eqref{hecke-map-eq}, $\theta_f$ yields a ring homomorphism (denoted by the same symbol) $\theta_f:\mathfrak H_k(\Gamma(N))\rightarrow\cO_F$.

For every $\mathfrak H_k(\Gamma(N))$-module $M$, let us set 
\begin{equation} \label{isotypic-eq}
M[\theta_f]\defeq \bigl\{m\in M\mid\text{$T\cdot m=0$ for all $T\in\ker(\theta_f)$}\bigr\}. 
\end{equation}
We call $M[\theta_f]$ the \emph{$f$-isotypic submodule} of $M$.

\subsubsection{Betti realization} \label{Betti-subsubsec}

Fix a subring $R\subset\C$ and define $R(k/2)\defeq (2\pi i)^{k/2}\cdot R$, which we view as an $R$-submodule of $\C$. In the following, we view $R(k/2)$ as a locally constant sheaf on $X(\C)$. 
Define 
\[ T_\B(R)\defeq\Pi\cdot H^{k-1}\bigl(X(\C),R(k/2)\bigr)=\Bigl(\Pi_B\Pi_\epsilon\cdot H^{k-1}\bigl(X(\C),R(k/2)\bigr)\!\Bigr)[\theta_f], \]
viewed as an $R$-submodule of $T_\B(\C)$. In particular, with $t_N$ as in \S \ref{motive-modular-form-subsubsec}, if $2Nt_N(k-2)!$ is not invertible in $R$, then $T_\B(R)$ may not be contained in $H^{k-1}\bigl(X(\C),R(k/2)\bigr)$. 

Recall that the field $F$ arises naturally as a subfield of $\C$.

\begin{definition}
The \emph{Betti} (or \emph{singular}) \emph{realization} of $\MM$ is the $F$-vector space 
$V_\B\defeq T_\B(F)$. 
\end{definition}

The $F$-vector space $V_\B$ is also the Betti realization of the motive $\MM/K$ for any number field $K$. Set 
\begin{equation} \label{T_B-eq}
T_\B\defeq T_\B(\cO_F). 
\end{equation}
Complex conjugation $\tau$ induces involutions $\iota_\infty:X(\C)\rightarrow X(\C)$ and $F(k/2)\rightarrow F(k/2)$, the latter being given by multiplication by $(-1)^{k/2}$. Denote by
$\phi_\infty:V_\B\rightarrow V_\B$ the composition of these two involutions and write
\begin{equation} \label{V-plus-eq}
V_\B^+\defeq  V_\B^{\phi_\infty=1} 
\end{equation}
for the $F$-subspace of $V_\B$ on which $\phi_\infty$ acts trivially. 
%\red{We may also consider 
%the integral singular cohomology group $H^{k-1}(X(\C),\mathcal{O}_F(k/2))$ where as before 
%$\mathcal{O}_F(k/2)=(2\pi i)^{k/2}\mathcal{O}_F$ and still denote $H^{k-1}(X(\C),\mathcal{O}(k/2))$ its image in $H^{k-1}(X(\C),F(k/2))$ with an abuse of notation (this operation kills the torsion). 
%Then form the $\mathcal{O}_F$-submodule $\Pi\cdot H^{k-1}(X(\C),\mathcal{O}_F(k/2))$ of 
%$V_\B$; as before, in general this is not contained in $H^{k-1}(X(\C),\mathcal{O}_F(k/2))$. 
%Define 
%\begin{equation} \label{T_B-eq}T_\B=\Pi\cdot H^{k-1}\bigl(X(\C),\mathcal{O}_F(k/2)\bigr)=\left((\Pi_B\Pi_\epsilon)\cdot H^{k-1}(X(\C),\mathcal{O}_F(k/2))\right)[\theta_f].\end{equation}}

Now we describe the Betti realization more explicitely. Let $Y(N)$ be the open modular curve of level $\Gamma(N)$. We still denote by $\pi:\mathcal{E}_N\rightarrow Y(N)$ the universal elliptic curve over $Y(N)$ and, with $R$ as before, set 
%\begin{equation} \label{A-sheaf-eq}
\[ \mathcal{F}_\B(R)\defeq\RR^1\pi_*(R(k/2)),\quad\mathcal{F}_\B^{k-2}(R)\defeq\Sym^{k-2}\mathcal{F}_\B(R), \]
%\end{equation} 
where $R$ stands for the corresponding constant sheaf on $\mathcal{E}_N$, which we regard as a sheaf on $X(N)$ by pushforward via the canonical embedding $Y(N)\hookrightarrow X(N)$. Write 
\[ \mathds{T}_\B(R)\defeq H^1_\mathrm{par}\Bigl(Y(N),\mathcal{F}_\B^{k-2}(R)\Bigr) \] 
for the parabolic cohomology group that is the image of the natural map
\[ H^1_\mathrm{cpt}\Bigl(Y(N),\mathcal{F}_\B^{k-2}(R)\Bigr)\longrightarrow H^1\Bigl(Y(N),\mathcal{F}_\B^{k-2}(R)\Bigr), \]
where $H^1_\mathrm{cpt}(\heartsuit,\diamondsuit)$ denotes compactly supported cohomology. Set 
\[ \mathds{V}_\B\defeq\T_\B (F),\quad\T_\B\defeq\T_\B (\cO_F). \]
By \cite[Theorem 1.2.1]{Scholl}, $V_\B$ is related to parabolic cohomology by a canonical isomorphism
\begin{equation}\label{Betti-iso-V}
V_\B\simeq\Pi\cdot\mathds{V}_\B=(\Pi_B\Pi_\epsilon\cdot\mathds{V}_\B)[\theta_f]
\end{equation} 
(note that $\Pi_B$ and $\Pi_\epsilon$ act canonically on $\mathds{V}_\B$). Since $f$ is a newform, Eichler--Shimura theory guarantees that $V_\B$ (respectively, $V_\B^+$) has dimension $2$ (respectively, $1$) over $F$. Isomorphism \eqref{Betti-iso-V} induces an isomorphism 
%\begin{equation} \label{Betti-iso-T} 
\[ T_\B\simeq\Pi\cdot\T_\B=(\Pi_B\Pi_\epsilon\cdot\T_\B)[\theta_f], \]
%\end{equation} 
where the right-hand side is, in general, not contained in $\T_\B$, but only in $\mathds{V}_\B$. 

\subsubsection{\'Etale realization} \label{etale-subsubsec}

Let $p$ be a prime number and $R$ a $\Z_p$-algebra. For every $r\in\Z$, we denote by $R(r)$ the $r$-fold Tate twist of $R$ and define the $R$-module 
\begin{equation} \label{V_p-eq}
T_p(R)\defeq \Pi \cdot H^{k-1}_\et\bigl(\bar{X},R(k/2)\bigr)=\Bigl(\Pi_B\Pi_\epsilon\cdot H^{k-1}_\et\bigl(\bar{X},R(k/2)\bigr)\!\Bigr)[\theta_f], 
\end{equation}
which is equipped with a natural action of $G_\Q$. As before, we view $T_p(R)$ as an $R$-submodule of $ T_p(\bar\Q_p)$: if $2Nt_N(k-2)!$ is not invertible in $R$, then $T_p(R)$ is not necessarily contained in $H^{k-1}_\et\bigl(\bar X,R(k/2)\bigr)$.

From here on, set $F_p\defeq F\otimes_\Q\Q_p$.

\begin{definition} 
The \emph{\'etale realization} of $\MM$ at $p$ is the $F_p$-module $V_p\defeq T_p(F_p)$. 
\end{definition}

By \cite[Theorem 1.2.4]{Scholl}, $V_p$ is equivalent to the self-dual twist $V_{f,p}^\dagger\defeq V_{f,p}(k/2)$ of the representation $V_{f,p}$ attached by Deligne to $f$ and $p$ (\cite{Del-Bourbaki}); in particular, $V_p$ is free of rank $2$ over $F_p$. There is an identification $F_p=\prod_{\p\mid p}F_\p$, where the product ranges over all primes $\p$ of $F$ above $p$ and $F_\p$ is the completion of $F$ at $\p$. As a consequence, there is a splitting $V_p=\prod_{\p\mid p}V_\p$, where each $V_\p$ is a $2$-dimensional $F_\p$-vector space on which $G_\Q$ acts. As above, $V_\p$ is equivalent to the self-dual twist $V_{f,\p}^\dagger=V_{f,\p}(k/2)$ of the $\p$-adic representation $V_{f,\p}$ attached to $f$.
Similarly, set $\cO_p\defeq\cO_F\otimes_\Z\Z_p$ (notice that $F_p$ is the total ring of fractions of $\cO_p$) and define 
\begin{equation} \label{T_p-eq}
T_p\defeq T_p(\cO_p).
\end{equation} Denoting by $\cO_\p$ the completion of $\cO_F$ at a prime $\p$, there is a splitting $T_p=\prod_{\p\mid p}T_\p$, where each $T_\p=T_p(\cO_\p)$ is a $G_\Q$-stable $\cO_\p$-lattice inside $V_\p$, \emph{i.e.}, a free $\cO_\p$-submodule of rank $2$ of $V_\p$ on which $G_\Q$ acts. It follows that $T_p$ is a $G_\Q$-stable self-dual $\cO_p$-lattice inside $V_p$. Again, $T_p\simeq T_{f,p}^\dagger\defeq T_{f,p}(k/2)$ for the distinguished $G_\Q$-stable $\cO_p$-lattice $T_{f,p}\subset V_{f,p}$ that is obtained from \eqref{V_p-eq} by omitting the $k/2$-fold Tate twist; an analogous isomorphism holds for each $T_\p$. Finally, set 
\begin{equation} \label{A-eq}
A_p\defeq V_p/T_p 
\end{equation}
and notice that $A_p$ is (canonically isomorphic to) the Pontryagin dual of $T_p$ as a $\Z_p$-module.

One can also describe the \'etale realizations $V_p$ and $T_p$ as follows. Define the $p$-adic sheaves 
\[ \mathcal{F}_\et(R)\defeq\pi_*\bigl(R(k/2)\bigr),\quad \mathcal{F}_\et^{k-1}(R)\defeq\mathrm{Sym}^{k-2}\bigl(\mathcal{F}_\et(R)\bigr). \] 
on $Y(N)$; as before, we still denote by $\mathcal{F}^{k-2}_\et(R)$ the sheaf on $X(N)$ obtained by pushforward via $Y(N)\hookrightarrow X(N)$. Set 
\[ \T_\et(R)\defeq H^1_\et\bigl(X(N),\mathcal{F}^{k-2}_\et(R)\bigr),\quad\T_\et\defeq\T_\et(\cO_p),\quad\mathds{V}_\et\defeq\mathds{T}_\et(F_p). \] 
The projector $\Pi$ acts on $\mathds{V}_\et$ and there is a canonical isomorphism 
\begin{equation} \label{etale-iso-V}
V_p\simeq\Pi\cdot\mathds{V}_\et=(\Pi_B\Pi_\epsilon\cdot\mathds{V}_\et)[\theta_f]
\end{equation}
(see, again, \cite[Theorem 1.2.1]{Scholl}). Isomorphism \eqref{etale-iso-V} induces an isomorphism 
%\begin{equation}\label{etale-iso-T} 
\[ T_p\simeq\Pi\cdot\T_\et=(\Pi_B\Pi_\epsilon\cdot\T_\et)[\theta_f]; \]
%\end{equation} 
here, as above, the right-hand side is, in general, not contained in $\T_\et$, but only in $\mathds{V}_\et$. 

\subsubsection{de Rham realization} \label{de-Rham-subsec}

Let $R$ be a $\Z[1/N]$-algebra. We view $X(N)$ and $X$ as schemes over $\Z[1/N]$. Define the $R$-module 
\[ T_\dR(R)\defeq\Pi\cdot\Bigl(H^{k-1}_\dR(X)\otimes_{\Z[1/N]}R\Bigr)=\Bigl(\Pi_B\Pi_\epsilon\cdot \bigl(H^{k-1}_\dR(X)\otimes_{\Z[1/N]}R\bigr)\!\Bigr)[\theta_f]. \]
Here $H^{i}_\dR(X)\defeq\mathds H^i\bigl(\Omega_{X}^\bullet\bigr)$ is the $i$-th hypercohomology group of the de Rham complex $\Omega_{X}^\bullet$ of $X$; thus, $T_\dR(R)$ is a finitely generated $R$-module. Recall that $H^{i}_\mathrm{dR}(X)$ is equipped with the filtration 
\[ \Fil^n\bigl(H^{i}_\mathrm{dR}(X)\bigr)\defeq\im\Bigl(\mathds H^{i}\bigl(\Omega^{\geq n}_{X}\bigr)\longrightarrow\mathds H^{i}\bigl(\Omega^\bullet_{X}\bigr)=H^{i}_\mathrm{dR}(X )\Bigr). \] 
Define a filtration $\Fil^n(V_{\dR})$ on $V_{\dR}$ by setting 
\[ \Fil^n(T_{\dR}(R))\defeq\Bigl(\Pi\cdot\Fil^{n+k/2}\bigl(H^{k-1}_\mathrm{dR}(X )\bigr)\otimes_{\Z[1/N]} R\Bigr)[\theta_f]. \]
Finally, set $V_\dR\defeq T_\dR(F)$.

\begin{definition}
The \emph{de Rham realization} of $\MM$ is the filtered $F$-vector space given by the pair $\Bigl(V_\dR,\bigl(\Fil^n(V_{\dR})\bigr)_{n\in\N}\Bigr)$.
\end{definition}

The \emph{tangent space} of $\MM$ is the $F$-vector space   
\begin{equation} \label{tangent-eq}
t(\MM)\defeq  V_{\dR}\big/\Fil^0(V_{\dR}).
\end{equation}  
Following \cite[\S1.2.4]{DFG}, we offer an alternative description of de Rham cohomology, as we did for the \'etale and Betti realizations. We write $\bar{\pi}:\bar{\mathcal{E}}_N\rightarrow X(N)$ for the generalized universal elliptic curve over $X(N)$ and denote by $\omega^1_\star$ the sheaf of relative logarithmic differentials of $\star$ (see, \emph{e.g.}, \cite[\S 1.7]{Kato-logarithmic}). Put $T\defeq\Spec(R)$. By \cite[Theorem 3.5 and Proposition 3.12]{Kato-logarithmic}, there is a short exact sequence of coherent locally free $\cO_{\bar{\mathcal{E}}_N}$-modules 
\[ 0\longrightarrow \bar{\pi}^*\bigl(\omega^1_{X(N)/T}\bigr)\longrightarrow\omega^1_{\bar{\mathcal{E}}_N/T}\longrightarrow\omega^1_{\bar{\mathcal{E}}_N/X(N)}\longrightarrow0. \]
Let us consider the locally free sheaves 
\[ \mathcal{F}_\dR(R)\defeq\mathbf{R}^1\bar{\pi}_*\bigl(\omega^\bullet_{\bar{\mathcal{E}}_N/X(N)}\bigr),\quad\mathcal{F}^{k-2}_\dR(R)\defeq\mathrm{Sym}^{k-2}\bigl(\mathcal{F}_\dR(R)\bigr) \] 
of $\cO_{X(N)}$-modules on $X(N)$, where $\omega^\bullet_{\bar{\mathcal{E}}_N/X(N)}$ is the complex $d:\cO_{\bar{\mathcal{E}}_N}\rightarrow\omega^1_{\bar{\mathcal{E}}_N/X(N)}$. Set $\omega(R)\defeq\bar{\pi}_*\bigl(\omega^1_{\bar{\mathcal{E}}_N/X(N)}\bigr)$; the sheaf $\mathcal F_\dR(R)$ has a decreasing filtration with $\Fil^2\mathcal F_\dR(R)=0$, $\Fil^1\mathcal F_\dR(R)=\omega$, $\Fil^0\mathcal F_\dR(R)=\mathcal F_\dR(R)$. In turn, this filtration produces a filtration on $\mathcal F_\dR^{k-2}(R)$. 
Now define
\[ \T_\dR(R)\defeq\mathds H^1\bigl(X(N),\omega^\bullet(\mathcal F_\dR^{k-2}(R)\bigr),\quad\T_\dR\defeq\T_\dR\bigl(\cO_F[1/N]\bigr),\quad\mathds V_\dR\defeq\T_\dR(F), \]
where $\mathds H^i(\heartsuit,\diamondsuit)$ denotes hypercohomology and $\omega^\bullet\bigl(\mathcal F_\dR^{k-2}(R)\bigr)$ is the complex associated with $\mathcal F_\dR^{k-2}(R)$ equipped with the logarithmic connection that is induced by the (logarithmic) Gauss--Manin connection $\nabla:\mathcal F_\dR(R)\rightarrow\mathcal F_\dR(R)\otimes_{\cO_{X(N)}}\!\omega^1_{X(N)/T}$. The filtration on $\mathcal F_\dR^{k-2}(R)$ yields a filtration on $\T_\dR(R)$ whose graded pieces can be described (up to isomorphism) as
\[ \mathrm{gr}^{i}\bigl(\T_\dR(R)\bigr)\simeq\begin{cases} H^0\bigl(X(N),\omega^{k-2}(R)
\otimes\omega_{X(N)/T}^1\bigr) & \text{if $i=k-1$,}\\[2mm]
H^1\bigl(X(N),\omega^{2-k}(R)\bigr) & \text{if $i=k/2$,}\\[2mm]
0 & \text{otherwise}.\end{cases} \] 
Let $\mathcal H\defeq\bigl\{z\in\C\mid\Im(z)>0\bigr\}$ be the complex upper half-plane. Pulling back to $\coprod_{t\in(\Z/N\Z)^\times}\mathcal{H}$ and trivializing by the differential form $(2\pi i)^{k-1}dz\wedge dz_1\wedge\dots\wedge dz_{k-2}$, where $z$ is the coordinate on $\mathcal{H}$ and $z_i$ is the coordinate on the $i$-th copy of the universal elliptic curve, we obtain an isomorphism 
\[ \mathrm{Fil}^{k-1}\bigl(\T_\dR(\C)\bigr)\overset\simeq\longrightarrow\bigoplus_{t\in(\Z/N\Z)^\times}M_k(\Gamma(N)), \]
where $M_k(\Gamma(N))$ is the $\C$-vector space of modular forms of weight $k$ and level $\Gamma(N)$. By the $q$-expansion principle, the map $\bigoplus_{t\in(\Z/N\Z)^\times}M_k(\Gamma(N))\rightarrow \C[\![q^{1/N}]\!]$ sending $g(e^{2\pi i \tau/N})$ to $g(q^{1/N})$ identifies $\mathrm{Fil}^{k-1}(\T_\dR)$ with the subset of $\bigoplus_{t\in(\Z/N\Z)^\times}M_k(\Gamma(N))$ consisting of those modular forms whose $q$-expansion at $\infty$ has coefficients in $\cO_F[1/N,\Bmu_N]$, where $\Bmu_N\subset\bar\Q^\times$ is the group of $N$-th roots of unity.

By \cite[Theorem 1.2.1]{Scholl} (see also \cite[\S2.1]{BDP} and \cite{Scholl-deRham}), there is an isomorphism 
\begin{equation} \label{deRham-iso-V}
V_\dR\simeq\Pi\cdot\mathds{V}_\dR=(\Pi_B\Pi_\epsilon\cdot\mathds V_\dR)[\theta_f],\end{equation} 
with filtration on the left-hand side obtained by shifting by $k/2$ the obvious filtration on the right-hand side. Then $\Fil^{k/2-1}(V_\dR)$ is spanned by (the image of) the differential form on $X$ given by 
\begin{equation} \label{omega_f-eq}
\omega_f\defeq (2\pi i)^{k-1}\cdot f(z)dz\wedge dz_1\wedge\dots\wedge dz_{k-2}. 
\end{equation}
See, \emph{e.g.}, \cite[\S1.2.1]{DFG} and \cite[Appendix 1]{Katz} for details; \emph{cf.} also the explanation offered in \cite[Appendix]{BrHa} for the factor $(2\pi i)^{k-1}$ appearing in the right-hand term of \eqref{omega_f-eq}. Isomorphism \eqref{deRham-iso-V} induces an isomorphism 
%\begin{equation}\label{deRham-iso-T} 
\[ T_\dR\simeq\Pi\cdot\T_\dR=(\Pi_B\Pi_\epsilon\cdot\T_\dR)[\theta_f], \]
%\end{equation}
where, as before, the right-hand side is, in general, not contained in $\T_\dR$, but only in $\mathds{V}_\dR$.  

\subsection{The period map} \label{periodmap}

From the comparison isomorphism between singular and de Rham cohomology we obtain a comparison isomorphism of $(F\otimes_\Q\C)$-modules 
\[ \Comp_{\B,\dR}: V_\B\otimes_\Q\C\overset\simeq\longrightarrow 
V_{\dR}\otimes_\Q\C. \] 
If $\phi_\infty$ is the involution from \S \ref{Betti-subsubsec}, then this isomorphism is equivariant with respect to the actions of $\phi_\infty\otimes\tau$ on the left and of 
$1\otimes\tau$ on the right, so it induces an isomorphism of $\R$-vector spaces
\begin{equation} \label{comp-R-eq}
\Comp_{\B,\dR}:( V_\B\otimes_\Q\C)^{\phi_\infty\otimes\tau=1}\overset\simeq\longrightarrow V_\dR\otimes_\Q\R. 
\end{equation}

\subsubsection{Period map}

Let $F_\infty\defeq  F\otimes_\Q\R$ and set 
\[ V_{\B,\infty}^+\defeq  V_\B^+\otimes_\Q\R=V_\B^+\otimes_F F_\infty,\quad V_{\dR,\infty}\defeq  V_\dR\otimes_\Q\R=V_\dR\otimes_F F_\infty. \]
Moreover, write ${t(\MM)}_\infty$ for $t(\MM)\otimes_\Q\R=t(\MM)\otimes_F F_\infty$. The \emph{period map} is the isomorphism 
\begin{equation} \label{period-eq}
\alpha_\MM:V^+_{\B,\infty}\overset\simeq\longrightarrow{t(\MM)}_\infty 
\end{equation}
of free $F_\infty$-modules of rank $1$ that is obtained by composing the natural inclusion $V^+_{\B,\infty}\hookrightarrow(V_\B\otimes_\Q\C)^{\phi_\infty\otimes\tau=1}$ with isomorphism $\Comp_{\B,\dR}$ from \eqref{comp-R-eq} and the map $V_{\dR,\infty}\rightarrow{t(\MM)}_\infty$ defined by tensoring with $\R$ over $\Q$ the canonical projection $V_\dR\twoheadrightarrow t(\MM)$. 

\begin{remark}
In the language of motives, isomorphism \eqref{period-eq} says that $\MM$ is \emph{critical} (see, \emph{e.g.}, \cite[\S 2.5]{Venjakob}).
\end{remark}

\subsubsection{Embeddings and periods} \label{emb-subsubsec}

Let $\Sigma$ be the set of embeddings of $F$ into $\R$ (equivalently, since $F$ is totally real, into $\C$). For each $\sigma\in\Sigma$, let $f^\sigma(q)\defeq \sum_{n\geq 1}a_n(f^\sigma)q^n\in S_k(\Gamma_0(N))$ be the newform of weight $k$ and level $N$ such that $a_n(f^\sigma)=\sigma\bigl(a_n(f)\bigr)$ for all $n\geq1$ (the form $f^\sigma$ is the \emph{$\sigma$-conjugate} of $f$). Clearly, the Hecke field of $f^\sigma$ is $\sigma(F)$. We write $\iota_F:F\hookrightarrow\R$ for the distinguished embedding corresponding to the inclusion $F\subset\R$.

\begin{remark} \label{deligne-rem2}
There is a canonical isomorphism of $\R$-algebras $F_\infty\simeq\R^\Sigma=\prod_{\sigma\in\Sigma}(F\otimes_{F,\sigma}\R)$ that induces an isomorphism of $\C$-algebras $F\otimes_\Q\C\simeq\C^{\Sigma}$. From here on, we shall usually not distinguish between $F_\infty$ and $\R^\Sigma$; we consider the embedding
\[ \iota_\Sigma:F\longmono F_\infty=\R^\Sigma,\quad x\longmapsto\bigl(\sigma(x)\bigr)_{\sigma\in\Sigma} \]
and identify any $x\in F$ with $\iota_\Sigma(x)$. With this convention in mind, when we claim that an element $\alpha\in F_\infty$ belongs to $F$ we really mean that there exists $x\in F$ such that $\iota_\Sigma(x)=\alpha$.
\end{remark}

Observe that ${t(\MM)}_\infty$ is spanned by the image of $T_\dR$. Fix $\gamma\in V_\B^+\smallsetminus\{0\}$.

\begin{definition} \label{period-def}
The \emph{period of $f$ relative to $\gamma$} is the determinant $\Omega_{\infty}^{(\gamma)}$ of $\alpha_\MM$ computed with respect to the basis $\{\gamma\}$ of $V_\B^+$ and the image of $\omega_f$ in ${t(\MM)}_\infty$.
\end{definition}
%
%\red{\begin{remark}
%The reason of the factor $(2\pi i)^{-k/2}$ appearing in the definition 
%above is crucial to express the Tamagawa Number Conjecture in terms of the \emph{completed} $L$-function; more details, and the connection with other more common periods appearing in the literature, will be discussed later in \S \ref{section-PMF}. 
%\end{remark}}

\begin{remark} 
For our later purposes of investigating the $p$-part of the Tamagawa number conjecture for $\MM$, an integral choice $\gamma\in T_\B^+\smallsetminus\{0\}$, which is in line with an analogous choice in \cite[\S4]{DSW}, will be preferable. 
\end{remark}

\subsection{Motivic cohomology} \label{motiviccohom}

Kuga--Sato varieties do not possess, in general, proper, flat, regular models over $\Z$. Therefore, the definition of motivic cohomology in low degrees that is given in \cite[Definition 1.19]{Kings} does not apply in our case. When the relevant varieties do not admit such a ``nice'' integral model, (integral) motivic cohomology is defined in \cite{Scholl2} in terms of alterations in the sense of de Jong (\cite{deJong}). In order to better stress the relation with \cite{Zhang-heights}, here we introduce (the first two groups of) motivic cohomology in a concise, utilitarian way; it would be interesting to compare our definition with the one in \cite{Scholl2}, but in this paper we do not pursue this matter any further.

\subsubsection{Motivic cohomology of $\MM$} \label{motivic-subsubsec}

Let $K$ be a number field. Denote by $\CH^{k/2}_0(X/K)$ the abelian group of codimension $k/2$ homologically trivial cycles on $X$ defined over $K$. Let us write $\CH^{k/2}_\arith(X/K)$ for the subgroup of $\Pi_B\Pi_\epsilon\cdot\CH^{k/2}_0(X/K)$ consisting of (the images of) the classes of those cycles that admit an integral model having trivial intersection with all the cycles of dimension $k$ supported on special fibers; here $\Pi_B\Pi_\epsilon\cdot\CH^{k/2}_0(X/K)$, and hence $\CH^{k/2}_\arith(X/K)$ too, should be viewed as a subgroup of $\CH^{k/2}_0(X/K)\otimes_\Z\Q$ (or, rather, should be identified with its natural image in $\CH^{k/2}_0(X/K)\otimes_\Z\Q$).

\begin{remark}
Since we are considering null-homologous cycles only, every cycle (class) in $\CH^{k/2}_\arith(X/K)$ has trivial image in $H^k\bigl(X(\C),\C\bigr)$; in other words, the cycles in $\CH^{k/2}_\arith(X/K)$ satisfy conditions (a) and (b) in \cite[\S 1.3]{Zhang-heights}, so there is a Gillet--Soul\'e height on $\CH^{k/2}_\arith(X/K)$ (\emph{cf.} \S \ref{GS-subsec}). This is the main reason for considering, in this article, $\CH^{k/2}_\arith(X/K)$ instead of the larger group $\Pi_B\Pi_\epsilon\cdot\CH^{k/2}_0(X/K)$. 
\end{remark}

For any $\Z$-algebra $R$ and $\star\in\{0,\text{arith}\}$, set
\begin{equation} \label{R-chow-eq}
{\CH^{k/2}(X/K)}_R\defeq \CH^{k/2}(X/K)\otimes_\Z R,\quad{\CH^{k/2}_\star(X/K)}_R\defeq \CH^{k/2}_\star(X/K)\otimes_\Z R.
\end{equation}
With notation as in \eqref{isotypic-eq} and $i=0,1$, let us define the \emph{$i$-th motivic cohomology group of $\MM$ over $K$} as
\begin{equation} \label{motivic-def-eq}
H^i_\mot(K,\MM)\defeq \begin{cases} 0 & \text{if $i=0$},\\[3mm]{\CH^{k/2}_\arith(X/K)}_F[\theta_f] & \text{if $i=1$}. \end{cases} 
\end{equation}
As these are the only motivic cohomology groups that play a role in our paper, we do not introduce $H^i_\mot(K,\MM)$ for $i\geq2$. 

\subsubsection{A finiteness conjecture} \label{finiteness-subsubsec}

The finiteness conjecture we are about to state is a special case of a classical conjecture predicting that Chow groups of smooth projective varieties over number fields are finitely generated (\emph{cf.} Conjecture \ref{chow-conj}).  

\begin{conjecture}[Finiteness of $H^1_\mot$] \label{finitenessconj}
$H^1_\mot(K,\MM)$ has finite dimension over $F$. 
\end{conjecture} 

In our route to the Tamagawa number conjecture for $\MM$ we shall assume the validity of Conjecture \ref{finitenessconj}. Thus, we can give

\begin{definition} \label{algebraic-rank-def}
The \emph{algebraic rank of $\MM$ over $K$} is $r_\alg(\MM/K)\defeq\dim_F\bigl(H^1_\mot(K,\MM)\bigr)$.
\end{definition}

For simplicity, we also set $r_\alg(\MM)\defeq  r_\alg(\MM/\Q)$ and call it the \emph{algebraic rank of $\MM$}.

\subsubsection{Some remarks on Chow groups} \label{chow-finiteness-subsubsec}

We offer some motivation for Conjecture \ref{finitenessconj}, which, as we pointed out in \S \ref{finiteness-subsubsec}, is essentially a byproduct of a more general conjecture on Chow groups of projective varieties over number fields. 

Let $Y$ be a smooth projective variety defined over a number field $L$. As usual, write $\CH^n(Y)$ for the Chow group of codimension $n$ algebraic cycles on $Y$. The following well-known conjecture is wide open.

\begin{conjecture} \label{chow-conj}
The abelian groups $\CH^n(Y)$ are finitely generated for all $n\in\N$.
\end{conjecture}

As a consequence, the groups $\CH^n(Y/L)$ of cycles on $Y$ of codimension $n$ that are defined over $L$ are conjecturally finitely generated for all $n\in\N$. Moreover, the groups $\CH^n_0(Y)$ and $\CH^n_0(Y/L)$ of codimension $n$ homologically trivial cycles, the conjecture of whose finite generation is attributed to Swinnerton-Dyer (\cite[Conjecture 5.0]{Beilinson}), will be finitely generated as well. Clearly, Conjecture \ref{chow-conj} is stronger than Conjecture \ref{finitenessconj}.

Little is known, as far as we are aware of, about Conjecture \ref{chow-conj}. From a broader point of view, it is a special case of the generalized (\emph{i.e.}, motivic) Bass conjecture; for the convenience of the reader, we briefly recall why this is true. Let $Y$ and $L$ be as above. Let $\Theta$ be the finite set of primes of $L$ at which $Y$ has bad reduction (see, \emph{e.g.}, \cite[Proposition A.9.1.6, (i)]{HS}) and let $\cO_{L,\Theta}$ be the ring of $\Theta$-integers of $L$. By \cite[Proposition A.9.1.6, (ii)]{HS}), there is a smooth model $\mathscr Y$ of $Y$ over $\cO_{L,\Theta}$. The ring $\cO_{L,\Theta}$ is regular, so $\mathscr Y$ is regular. Since $\mathscr Y$ is a $\Z$-scheme of finite type, the motivic Bass conjecture (see, \emph{e.g.}, \cite[Conjecture 37, b)]{kahn}) predicts that the motivic cohomology groups $H^i_{\mathscr M}\bigl(\mathscr Y,\Z(n)\bigr)$ of Suslin--Voevodsky (\cite[Lecture 3]{MVW}) are finitely generated for all $i,n\in\N$. On the other hand, if $\CH^n(\mathscr Y,i)$ denote, for $n\in\N$ and $i\in\Z$, Bloch's higher Chow groups (\cite{Bloch}), then
\begin{itemize}
\item $\CH^n(\mathscr Y,0)=\CH^n(\mathscr Y)$ for all $n\in\N$ (\cite[p. 268]{Bloch});
\item $H^i_{\mathscr M}\bigl(\mathscr Y,\Z(n)\bigr)\simeq\CH^n(\mathscr Y,2n-i)$ for all $i,n\in\N$ (\cite[Theorem 19.1]{MVW}).
\end{itemize}
It follows that $H^{2n}_{\mathscr M}\bigl(\mathscr Y,\Z(n)\bigr)\simeq\CH^n(\mathscr Y)$, therefore the motivic Bass conjecture predicts, in particular, that the Chow groups $\CH^n(\mathscr Y)$ are finitely generated for all $n\in\N$. Finally, taking scheme-theoretic closures of codimension $n$ cycles on $Y$, one checks that for all $n\in\N$ the natural group homomorphism $\CH^n(\mathscr Y) \rightarrow\CH^n(Y)$ is surjective, which completes the argument.

\subsection{The Gillet--Soul\'e height pairing for $\MM$} \label{GS-subsec}

We introduce the height pairing on our motivic cohomology groups and then define a regulator in terms of this pairing.

\subsubsection{Gillet--Soul\'e height pairings}

Let $K$ be a number field. Denote by
\begin{equation} \label{GS-eq}
{\langle\cdot,\cdot\rangle}_{\GS}:\CH^{k/2}_\arith(X/K)\times\CH^{k/2}_\arith(X/K)\longrightarrow\R 
\end{equation}
the height pairing defined by S.-W. Zhang (\cite[\S 1.3]{Zhang-heights}) using arithmetic intersection theory \emph{\`a la} Gillet--Soul\'e (\cite{GS-1}). For each $\sigma\in\Sigma$, pairing \eqref{GS-eq} yields an $F$-bilinear pairing
\begin{equation} \label{GS-eq3}
{\langle\cdot,\cdot\rangle}_{\GS,\sigma}:H^1_\mot(K,\MM)\times H^1_\mot(K,\MM)\longrightarrow\R=F\otimes_{F,\sigma}\R 
\end{equation}
Let us consider the $F_\infty$-module 
\[ \begin{split}
   {H^1_\mot(K,\MM)}_\infty&\defeq  H^1_\mot(K,\MM)\otimes_\Q\R\\&=H^1_\mot(K,\MM)\otimes_F F_\infty
   =\prod_{\sigma\in\Sigma}H^1_\mot(K,\MM)\otimes_{F,\sigma}\R, 
   \end{split} \] 
which is free of rank $r_\alg(\MM/K)$. We can define an $F_\infty$-bilinear height pairing
\begin{equation} \label{GS-eq2}
{\langle\cdot,\cdot\rangle}_{\GS,\infty}:{H^1_\mot(K,\MM)}_\infty\times{H^1_\mot(K,\MM)}_\infty\longrightarrow F_\infty 
\end{equation}
by the rule
\[ \bigl({(x_\sigma)}_{\sigma\in\Sigma},{(y_\sigma)}_{\sigma\in\Sigma}\bigr)\longmapsto\bigl({\langle x_\sigma,y_\sigma\rangle}_{\GS,\sigma}\bigr)_{\sigma\in\Sigma}, \]
where ${(x_\sigma)}_{\sigma\in\Sigma},{(y_\sigma)}_{\sigma\in\Sigma}$ belong to $H^1_\mot(K,\MM)\otimes_{F,\sigma}\R$ and each pairing ${\langle\cdot,\cdot\rangle}_{\GS,\sigma}$ in \eqref{GS-eq3} has been extended $\R$-linearly over $F$ (with respect to $\sigma$). 

\begin{remark}
Given $x,y\in H^1_\mot(K,\MM)$ and $\sigma\in\Sigma$, it is easy to check that ${\langle x,y\rangle}_{\GS,\iota_F}=0$ if and only if ${\langle x,y\rangle}_{\GS,\sigma}=0$. It follows that ${\langle\cdot,\cdot\rangle}_{\GS,\iota_F}$ is non-degenerate if and only if ${\langle\cdot,\cdot\rangle}_{\GS,\sigma}$ is non-degenerate.
\end{remark}

\subsubsection{A non-degeneracy conjecture}

The validity of the next conjecture is predicted by the arithmetic analogues of the standard conjectures proposed by Gillet--Soul\'e (\cite{GS-2}); it is also a special case of general conjectures of Beilinson (\cite{Beilinson}) and Bloch (\cite{bloch-height}) on positive definiteness of height pairings.

\begin{conjecture}[Non-degeneracy of height] \label{nondegconj} 
The pairing ${\langle\cdot,\cdot\rangle}_{\GS,\infty}$ is non-degenerate.
\end{conjecture}

In this paper, we often assume this conjecture (or variants thereof) to hold true; notice, however, that no condition of this sort will be needed in our main result on the $p$-part of the Tamagawa number conjecture for $\MM$ (Theorem \ref{ThmTNC}). 

\begin{remark}
Of course, Conjecture \ref{nondegconj} is true if and only if the $\R$-linear extension (with respect to $\sigma$) of ${\langle\cdot,\cdot\rangle}_{\GS,\sigma}$ is non-degenerate for each $\sigma\in\Sigma$.
\end{remark}

\subsubsection{The $\mathscr B$-regulator of $\MM$ over $K$}

Set $r\defeq r_\alg(\MM/K)$. If $r>0$, then fix a basis $\mathscr B=\{t_1,\dots,t_r\}$ of $H^1_\mot(K,\MM)$ over $F$; clearly, $\mathscr B$ is also a basis of ${H^1_\mot(K,\MM)}_\infty$ over $F_\infty$. 

\begin{definition} \label{defreg}
If $r>0$, then the \emph{Gillet--Soul\'e $\mathscr B$-regulator of $\MM$ over $K$} is
\[ \Reg_\mathscr B(\MM/K)\defeq\det\bigl({\langle t_i,t_j\rangle}_{\GS,\infty}\bigr)_{1\leq i,j\leq r}\in F_\infty. \]
If $r=0$, then $\Reg(\MM/K)\defeq1$.
\end{definition}

For simplicity, we put $\Reg_\mathscr B(\MM)\defeq\Reg_\mathscr B(\MM/\Q)$ and call it the \emph{$\mathscr B$-regulator of $\MM$}.

\begin{remark}
If Conjecture \ref{nondegconj} holds true, then $\Reg_\mathscr B(\MM/K)\in F_\infty^\times$.
\end{remark}

\begin{remark} \label{bases-rem}
If $\mathscr B$ and $\mathscr B'$ are two bases of $H^1_\mot(K,\MM)$ over $F$, then $\Reg_\mathscr B(\MM/K)$ and $\Reg_{\mathscr B'}(\MM/K)$ differ by multiplication by the square of the determinant of the transition matrix from $\mathscr B$ to $\mathscr B'$. 
\end{remark}

Assume $r>0$. For each $\sigma\in\Sigma$, it is convenient to define
\begin{equation} \label{reg-sigma-eq}
\Reg_{\mathscr B}^\sigma(\MM)\defeq\det\bigl({\langle t_i,t_j\rangle}_{\GS,\sigma}\bigr)_{1\leq i,j\leq r}\in \R,
\end{equation}
so that we can write $\Reg_{\mathscr B}(\MM)=\bigl(\Reg_{\mathscr B}^\sigma(\MM)\bigr)_{\sigma\in\Sigma}$.

\subsection{$p$-adic Galois representations} \label{p-adic-subsec}

For every prime $v$ of $K$, let $K_v$ be the completion of $K$ at $v$, fix an algebraic closure $\bar K_v$ of $K_v$ and write $K_v^\mathrm{unr}\subset\bar K_v$ for the maximal unramified extension of $K_v$. Furthermore, set $G_{K_v}\defeq \Gal(\bar K_v/K_v)$, denote by $I_v\subset G_{K_v}$ the inertia subgroup and let $\Frob_v\in\Gal(K_v^\mathrm{unr}/K_v)\simeq G_{K_v}/I_v$ be the arithmetic Frobenius. Finally, for every prime number $p$, fix a field embedding $\bar\Q\hookrightarrow\bar\Q_p$.

\subsubsection{Dieudonn\'e modules}

Let $v$ be a prime of $K$ above $p$. If $V$ is a $p$-adic representation of $G_{K_v}$, by which we mean that $V$ is a finite-dimensional $\Q_p$-vector space that is equipped with a continuous action of $G_{K_v}$, then we denote by
\[ \mathbf{D}_\crys(V)\defeq  H^0\bigl(G_{K_v},V\otimes_{\Q_p}\mathbf{B}_\crys\bigr) \]
and 
\[ \mathbf{D}_\dR(V)\defeq  H^0\bigl(G_{K_v},V\otimes_{\Q_p}\mathbf{B}_\dR\bigr) \] 
the crystalline and de Rham Dieudonn\'e modules of $V$, where $\mathbf{B}_\crys$ (respectively, $\mathbf{B}_\dR$) is Fontaine's crystalline (respectively, de Rham) ring of periods ($\mathbf{B}_\dR$ is, in fact, a field). Recall that $\mathbf{D}_\dR(V)$ is equipped with a filtration $\Fil^i\bigl(\mathbf{D}_\dR(V)\bigr)$ and $\mathbf{D}_\crys(V)$ is endowed with a distinguished endomorphism $\phi$, the so-called Frobenius endomorphism. 

\subsubsection{Tangent space}

The \emph{tangent space of $V$} is 
\begin{equation} \label{V-tangent-eq}
t(V)\defeq\mathbf{D}_\mathrm{dR}(V)\big/\Fil^0\bigl(\mathbf{D}_\mathrm{dR}(V)\bigr), 
\end{equation} 
which is a finite-dimensional $\mathbf{B}_\dR$-vector space.

\begin{remark} \label{fontaine-rem}
Since $\mathbf{B}_\crys$ is a subring of $\mathbf{B}_\dR$, there is a natural injection $\mathbf{D}_\crys(V)\hookrightarrow\mathbf{D}_\dR(V)$.
\end{remark}

\subsection{$L$-functions of $\MM$} \label{Lfunctsec}

Let $p$ be a prime number and let $K$ be a number field. We can regard the $G_\Q$-representation $V_p$ from \eqref{V_p-eq} as a $G_K$-representation, and hence as a $G_{K_v}$-representation for every prime $v$ of $K$. 

\subsubsection{Euler factors and the $L$-function}

For every prime $v$ of $K$, we define the Euler factor
\[ L_v(V_p,x)\defeq \begin{cases} \det_{F_p}\bigl(\mathrm{id}-\Frob_v^{-1}x,V_p ^{I_v}\bigr)& \text{if $v\nmid p$},\\[3mm]
 \det_{F_p}\bigl(\mathrm{id}-\phi x,\mathbf{D}_\crys(V_p)\bigr)& \text{if $v\,|\,p$}.
\end{cases} \] 
The polynomial $L_v( V_p ,x)$ has coefficients in $F$ and is independent of $p$. Denote by $q_v$ the cardinality of the residue field of $K_v$ and set
\[ L_v(V_p,s)\defeq L_v(V_p,q_v^{-s}). \]

\begin{definition} \label{L-function-def}
The \emph{$L$-function of $\MM$ over $K$} is the formal Euler product 
\[ L(\MM/K ,s)\defeq \prod_v L_v(V_p,s)^{-1}. \]
\end{definition}

As a shorthand, we write $L(\MM,s)$ instead of $L(\MM/\Q,s)$. As in \cite[Remark 7]{BF}, for every $s\in\C$ we regard $L_v(V_p,s)$ as an element of $F\otimes_\Q\C$. The function $L(\MM/K ,s)$ of the complex variable $s$ admits a holomorphic continuation to $\C$, which takes values in $F\otimes_\Q\C$. Furthermore, by \cite[Lemma 8]{BF}, $L(\MM/K,s)\in F_\infty$ if $s\in\R$.

\begin{remark} \label{splitting-L-rem}
Let $\mathcal K$ be a quadratic field, let $\epsilon_{\mathcal K}$ be the Dirichlet character attached to $\mathcal K$ and, as in Remark \ref{motive-f^K-rem}, denote by $f^{\mathcal K}$ the twist of $f$ by $\epsilon_{\mathcal K}$, which is a newform satisfying $a_n(f^{\mathcal K})=\epsilon_{\mathcal K}(n)\cdot a_n(f)$ for all $n\geq1$. If $\MM\bigl(f^{\mathcal K}\bigr)$ is the motive of $f^{\mathcal K}$, then 
\[ L(\MM/\mathcal K,s)=L(\MM,s)\cdot L\bigl(\MM(f^{\mathcal K}),s\bigr). \]
This equality is a special case of a factorization that holds over any abelian extension of $\Q$.
\end{remark}

\subsubsection{The completed $L$-function} \label{completed-subsubsec} 

Let $\Gamma$ be the classical complex $\Gamma$-function. The \emph{completed $L$-function of $f$} is the complex-valued function 
\[ \Lambda(f,s)\defeq\biggl(\frac{\sqrt{N}}{2\pi}\biggr)^{s}\cdot\Gamma(s)\cdot L(f,s). \]
It satisfies a functional equation (usually referred to as the functional equation for $L(f,s)$) of the form
\[ \Lambda(f,s)=\varepsilon(f)\cdot\Lambda(f,k-s), \]
where $\varepsilon(f)\in\{\pm1\}$ is the \emph{root number} of $f$. In particular, $\Lambda(f,s)$ is holomorphic on (or, rather, can be holomorphically continued to) the whole complex plane, and the same is true of $L(f,s)$ (see, \emph{e.g.}, \cite[p. 141]{Koblitz}). As a shorthand, put $L_\infty(f,s)\defeq\bigl(\sqrt{N}/2\pi\bigr)^s\cdot\Gamma(s)$, so that $\Lambda(f,s)=L_\infty(f,s)\cdot L(f,s)$. We find it convenient to introduce the normalization of $\Lambda(f,s)$ given by 
\begin{equation} \label{tilde-Lambda-eq}
\tilde{\Lambda}(f,s)\defeq\frac{{\Lambda}(f,s)}{\bigl(i\sqrt{N}\bigr)^{k/2}}%=(i\sqrt{N})^{s-k/2}\Gamma(s)\frac{L(f,s)}{(2\pi i)^s}
.
\end{equation}
It is well known that $\Gamma(n)=(n-1)!$ for every integer $n\geq1$, so if $g^{(i)}$ denotes the $i$-th derivative of a complex function $g$, then there is an equality 
\[ \tilde{\Lambda}^{(r)}(f,k/2)=\frac{(k/2-1)!\cdot L^{(r)}(f,k/2)}{(2\pi i)^{k/2}}, \] 
where $r\in\N$ is the order of vanishing of $L(f,s)$ at $s=k/2$ (from \S \ref{leading-subsubsec} onwards, this integer will be called the analytic rank of $g$, \emph{cf.} Definition \ref{analytic-rank-def}). Our next goal is to define the completed $L$-function of $\MM$; to do this, we introduce the archimedean factor 
\[ L_\infty(\MM,s)\defeq\biggl(\frac{\sqrt{N}}{2\pi}\biggr)^{s+k/2}\cdot\frac{\Gamma(s+k/2)}{{\bigl(i\sqrt{N}\bigr)^{k/2}}}=\frac{\bigl(i\sqrt{N}\bigr)^{s}\cdot\Gamma(s+k/2)}{(2\pi i)^{s+k/2}}, \]
where the equality on the right follows from a trivial computation.

\begin{definition} \label{completed-L-function-def}
The \emph{completed $L$-function of $\MM$ over $\Q$} is 
\[ \Lambda(\MM,s)\defeq L_\infty(\MM,s)\cdot L(\MM,s). \]
\end{definition}

This is the $L$-function in terms of which we shall prove our main results. Like $L(\MM,s)$, the completed $L$-function $\Lambda(\MM ,s)$ has an $F\otimes_\Q\C$-valued holomorphic continuation to $\C$; moreover, $\Lambda(\MM,s)\in F_\infty$ for all $s\in\R$. In particular, there is an equality 
\[ \Lambda^{(r)}(\MM,0)=\frac{(k/2-1)!\cdot L^{(r)}(\MM,0)}{(2\pi i)^{k/2}}, \]
where now $r\in\N$ is the order of vanishing of $L(\MM,s)$ at $s=0$.

\begin{remark}
By carefully keeping track of $\Gamma$-factors, one can define $\Lambda(\MM/\mathcal K,s)$ for any number field $\mathcal K$. Since we will have no use for it, we refrain from introducing this more general notion. Rather, in the special case where $\mathcal K$ is a quadratic field we set 
\[ \Lambda(\MM/\mathcal K,s)\defeq\Lambda(\MM,s)\cdot\Lambda\bigl(\MM(f^{\mathcal K}),s\bigr). \]
This equality, which we take as the definition of the left-hand term, is in fact a special case of a factorization that holds over any abelian extension of $\Q$ (\emph{cf.} Remark \ref{splitting-L-rem}).
\end{remark}

\subsubsection{Analytic ranks and leading terms} \label{leading-subsubsec}

Recall the $L$-functions $L(\MM/K,s)$ and $\Lambda(\MM,s)$ from Definitions \ref{L-function-def} and \ref{completed-L-function-def}, respectively.

\begin{definition} \label{analytic-rank-def2}
\begin{enumerate}
\item The \emph{analytic rank $r_\an(\MM/K)$ of $\MM$ over $K$} is the order of vanishing of $L(\MM/K,s)$ at $s=0$, \emph{i.e.}, $r_\an(\MM/K)\defeq \ord_{s=0}L(\MM/K,s)$.
\item The \emph{leading term $L^*(\MM/K,0)$ of $L(\MM/K ,s)$ at $s=0$} is the leading term of the Taylor expansion of $L(\MM/K ,s)$ at $s=0$, \emph{i.e.}, $L^*(\MM/K,0)\defeq \lim\limits_{s\rightarrow0}s^{-r_\an(\MM/K)}L(\MM/K,s)$. 
\item The \emph{leading term $\Lambda^*(\MM,0)$ of $\Lambda(\MM ,s)$ at $s=0$} is the leading term of the Taylor expansion of $\Lambda(\MM ,s)$ at $s=0$, \emph{i.e.}, $\Lambda^*(\MM,0)\defeq\lim\limits_{s\rightarrow0}s^{-r_\an(\MM/\Q)}\Lambda(\MM,s)$. 
\end{enumerate}
\end{definition}

Observe that $L_\infty(\MM,0)\not=0$, so the orders of vanishing of $L(\MM,s)$ and $\Lambda(\MM,s)$ at $s=0$ are equal: this justifies part (3) of Definition \ref{analytic-rank-def2}. 

In line with notation that was introduced earlier, we set $L^*(\MM,0)\defeq L^*(\MM/\Q,0)$ and $r_\an(\MM)\defeq r_\an(\MM/\Q)$. Therefore, there is an equality 
\begin{equation} \label{completedL}
\Lambda^*(\MM,0)=\frac{(k/2-1)!\cdot L^*(\MM,0)}{(2\pi i)^{k/2}}
\end{equation} 
that establishes a relation between the leading terms of $L(\MM,s)$ and $\Lambda(\MM,s)$.

\begin{remark} \label{leading-term-rem}
For each $\sigma\in\Sigma$, write $L^*(f^\sigma,k/2)$ for the leading term of $L(f^\sigma,s)$ at $s=k/2$, where $L(f^\sigma,s)$ is the $L$-function of $f^\sigma$. Under the identification of Remark \ref{deligne-rem2}, the leading term $L^*(\MM,0)$ corresponds to $\bigl(L^*(f^\sigma,k/2)\bigr)_{\sigma\in\Sigma}$. It follows that $L^*(\MM/K,0)\in F^\times_\infty$. 

A similar remark applies to $\Lambda^*(\MM,0)$. More precisely, denote by $\tilde{\Lambda}^*(f^\sigma,k/2)$ the leading term of $\tilde{\Lambda}(f^\sigma,s)$ at $s=k/2$, where $\tilde\Lambda(f^\sigma,s)$ is the normalized completed $L$-function of $f^\sigma$ as in \eqref{tilde-Lambda-eq}. Then $\Lambda^*(\MM,0)$ corresponds to $\bigl(\tilde{\Lambda}^*(f^\sigma,k/2)\bigr)_{\sigma\in\Sigma}$ and $\Lambda^*(\MM,0)\in F^\times_\infty$. 
\end{remark}

Recall the algebraic rank $r_\alg(\MM/K)$ from Definition \ref{algebraic-rank-def}. The following conjecture can be seen as the rank part of the Beilinson--Bloch--Kato conjecture for the motive $\MM$ over $K$.

\begin{conjecture}[Equality of ranks] \label{orderconj}
$r_\mathrm{an}(\MM/K)=r_\alg(\MM/K)$. 
\end{conjecture}

Quite generally, let $g$ be an eigenform of weight $k$ and level $\Gamma_0(N)$; as usual, let $L(g,s)$ be the (complex) $L$-function of $g$.

\begin{definition} \label{analytic-rank-def}
The \emph{analytic rank of $g$} is $r_\an(g)\defeq\ord_{s=\frac{k}{2}}L(g,s)\in\N$.
\end{definition}

In an analogous way, given a number field $K$, one can define the \emph{analytic rank $r_\an(f/K)$ of $g$ over $K$}. As in \S \ref{completed-subsubsec}, if $g$ is a newform, then write $\varepsilon(g)\in\{\pm1\}$ for the root number of $g$, \emph{i.e.}, the sign of the functional equation for $L(g,s)$. The root number controls the parity of $r_\an(g)$, in the sense that $\varepsilon(g)={(-1)}^{r_\an(g)}$. Equivalently, there is a congruence
\begin{equation} \label{r-parity-eq}
r_\an(g)\equiv\frac{1-\varepsilon(g)}{2}\pmod{2}. 
\end{equation}
With notation as in \S \ref{periodmap}, recall the $\sigma$-conjugate $f^\sigma$ of $f$.

\begin{remark} \label{deligne-rem}
Under the isomorphism in Remark \ref{deligne-rem2}, the $F\otimes_\Q\C$-valued function $L(\MM,s)$ corresponds to the $\C^\Sigma$-valued function $\bigl(L(f^\sigma,s+k/2)\bigr)_{\sigma\in\Sigma}$ of the complex variable $s$. This is the point of view of Deligne in \cite[\S 2.2]{deligne-valeurs} (\emph{cf.} also \cite[\S 1]{Panchishkin}); it offers, in particular, a more explicit interpretation of the analytic rank $r_\an(\MM)$ from part (2) of Definition \ref{analytic-rank-def2}. Namely, there is an equality
\begin{equation} \label{r-min-eq}
r_\an(\MM)=\min\bigl\{r_\an(f^\sigma)\mid\sigma\in\Sigma\bigr\}. 
\end{equation}
It is conjectured that $r_\an(f^\sigma)$ is constant as $\sigma$ varies in $\Sigma$ (which, if true, would imply that $r_\an(\MM)=r_\an(f)$), but we will not need this property in the rest of the paper. Later on, it will sometimes be convenient to identify $L(\MM,s)$ with $\bigl(L(f^\sigma,s+k/2)\bigr)_{\sigma\in\Sigma}$ and view $L(\MM,s)$ as taking values in $\C^\Sigma$.
\end{remark}

\begin{remark}
With notation as in \S \ref{completed-subsubsec}, $L_\infty(f,k/2)\not=0$, so $r_\an(f)$ is equal to the order of vanishing of $\Lambda(f,s)$ at $s=k/2$.
\end{remark}

For later use, we record an auxiliary result. 

\begin{lemma} \label{low-rank-lemma}
Let $r\in\{0,1\}$. If $r_\an(\MM)=r$, then $r_\an(f^\sigma)=r$ for all $\sigma\in\Sigma$. Conversely, if $r_\an(f^\sigma)=r$ for some $\sigma\in\Sigma$, then $r_\an(\MM)=r$.
\end{lemma}

\begin{proof} In light of Remark \ref{deligne-rem}, the lemma follows from \cite[Corollary 0.3.5]{Zhang-heights}. \end{proof}

\subsubsection{Leading term and periods}

In the statement below, which should be interpreted as explained in Remark \ref{deligne-rem2}, the period $\Omega^{(\star)}_\infty$ is the one defined in \S \ref{emb-subsubsec}.

\begin{lemma} \label{rationality-lemma}
Let $\gamma,\gamma\in V_\B^+\smallsetminus\{0\}$ and let $\mathscr B,\mathscr B'$ be bases of $H^1_\mot(\Q,\MM)$ over $F$. Then $L^*(\MM,0)\big/\bigl(\Omega^{(\gamma)}_\infty\cdot\Reg_\mathscr B(\MM)\bigr)\in F^\times$ if and only if $L^*(\MM,0)\big/\bigl(\Omega^{(\gamma')}_\infty\cdot\Reg_{\mathscr B'}(\MM)\bigr)\in F^\times$.
\end{lemma}

\begin{proof} Since the $F$-vector space $V_\B^+$ is $1$-dimensional, $\gamma$ and $\gamma'$ differ by multiplication by an element of $F^\times$, and then the same is true of the periods $\Omega^{(\gamma)}_\infty$ and $\Omega^{(\gamma')}_\infty$. On the other hand, as was pointed out in Remark \ref{bases-rem}, the regulators $\Reg_{\mathscr B}(\MM)$ and $\Reg_{\mathscr B'}(\MM)$ differ by multiplication by an element of $F^\times$ as well, and the lemma is proved. \end{proof}

\subsection{The fundamental line of $\MM$} \label{fundamental-subsec}

We introduce the ``fundamental line'' of $\MM$ in the formulation of Fontaine--Perrin-Riou (\cite{FPR}), using the theory of determinants described in Appendix \ref{determinants-subsec}, to which the reader is referred for details. In this case, the projective modules that play a role are vector spaces over $F$. 

In the definition that follows, $V^+_B$ is the $F$-subspace of $V_B$ from \eqref{V-plus-eq}, $t(\MM)$ is the tangent space of $\MM$ introduced in \eqref{tangent-eq} and $H^1_\mot(\Q,\MM)$ is the cohomology group from \eqref{motivic-def-eq}.

\begin{definition} \label{fundline}
The \emph{fundamental line of $\MM$} is 
\[ \Delta(\MM)\defeq \Det_{F}^{-1}\bigl(H^1_\mot(\Q,\MM)\bigr)\cdot\Det_{F}\bigl(H^1_\mot(\Q,\MM)^*\bigr)\cdot\Det_{F}\bigl(t(\MM)\bigr)\cdot\Det_{F}^{-1}\bigl(V_\B^+\bigr). \]
\end{definition} 

Note that, by construction, the $F$-vector space underlying $\Delta(\MM)$ is $1$-dimensional.

\begin{remark} 
In order to compare Definition \ref{fundline} with \cite[Definition 2.4]{Kings}, recall from \eqref{motivic-def-eq} that $H^0_\mot(\Q,\MM)=0$ and keep in mind that $\MM^\vee(1)\simeq\MM$.
\end{remark}

\subsection{Rationality conjecture} \label{rationality-subsec}

For now, let us assume that
\begin{itemize}
\item Conjecture \ref{nondegconj} holds true.
\end{itemize}
Define the $\R$-vector space
\[ {\Delta(\MM)}_\infty\defeq \Delta(\MM)\otimes_\Q\R; \] 
then ${\Delta(\MM)}_\infty\simeq\Delta(\MM)\otimes_FF_\infty$ is a free $F_\infty$-module of rank $1$. Conjecture \ref{nondegconj}, which we are assuming, ensures that the Gillet--Soul\'e height pairing ${\langle\cdot,\cdot\rangle}_{\GS,\infty}$ from \eqref{GS-eq2} induces an isomorphism of $F_\infty$-modules
\[ {H^1_\mot(\Q,\MM)}_\infty\overset\simeq\longrightarrow{H^1_\mot(\Q,\MM)}^*_\infty. \]
Combining the base change formula \eqref{det-basechange-eq} for determinants, the multiplicativity \eqref{det-multiplicativity-eq} of determinants in short exact sequences and isomorphism \eqref{period-eq}, we obtain an isomorphism 
\[ \theta_\infty:{\Delta(\MM)}_\infty\overset\simeq\longrightarrow(F_\infty,0) \]
of $F_\infty$-modules. 

\subsubsection{Rationality conjecture} \label{ratconj-subsubsec}

The following conjecture is essentially due to Beilinson (\cite{Beilinson-higher}) and Deligne (\cite[Conjecture 1.8]{deligne-valeurs}).

\begin{conjecture}[Rationality conjecture] \label{ratconj}
There exists $\zeta_f\in\Delta(\MM)$ such that the equality
\[ \theta_\infty(\zeta_f)=L^*(\MM,0)^{-1} \]
holds in $F_\infty^\times$.
\end{conjecture}

The element $\zeta_f$ is called a \emph{zeta element} and $\{\zeta_f\}$ is, of course, a basis of $\Delta(\MM)$ over $F$. Moreover, since $\theta_\infty$ is an isomorphism, such a $\zeta_f$ is unique if it exists. Under some technical conditions, later in this paper we will prove Conjecture \ref{ratconj} when $r_\an(\MM)\in\{0,1\}$ (Theorems \ref{ratconj0thm} and \ref{ratconj1thm}).

\subsubsection{A variant of the rationality conjecture} \label{ratconj2-subsubsec}

Now we offer an alternative formulation of Conjecture \ref{ratconj} that involves the completed $L$-function $\Lambda(\MM,s)$. In light of equality \eqref{completedL}, a straightforward computation shows that Conjecture \ref{ratconj} is equivalent to

\begin{conjecture}[Rationality conjecture, second version] \label{ratconj2}
There exists $\zeta_f^*\in\Delta(\MM)$ such that the equality
\[ \theta_\infty(\zeta_f^*)=\Bigl((2\pi i)^{k/2}\Lambda^*(\MM,0)\Bigr)^{-1} \]
holds in $F_\infty^\times$.
\end{conjecture}

One can switch between Conjecture \ref{ratconj} and Conjecture \ref{ratconj2} by means of the relation $\zeta_f^*=\zeta_f\big/(k/2-1)!$.   

\subsubsection{A reformulation of the rationality conjecture}

The term ``rationality'' in Conjecture \ref{ratconj} is justified by the reformulation below, which involves the leading term $L^*(\MM,0)$ and the Gillet--Soul\'e $\mathscr B$-regulator $\Reg_{\mathscr B}(\MM)$.

\begin{remark} \label{det-rem}
Let $R$ be a ring, let $M$ be a free $R$-module of finite rank, say $r$, and write $M^*\defeq\Hom_R(M,R)$ for the $R$-linear dual of $M$. Let $\langle\cdot,\cdot\rangle:M\times M\rightarrow R$ be an $R$-bilinear pairing and let $f:M\rightarrow M^*$ be the $R$-linear map given by $t\mapsto\langle t,\cdot\rangle$. Choose a basis $\mathscr B=\{t_1,\dots,t_r\}$ of $M$ over $R$, let $\mathscr B^*$ be the dual basis of $M^*$, set $A\defeq\bigl(\langle t_i,t_j\rangle\bigr)_{1\leq i,j\leq r}$ and denote by $\det(f)^{\mathscr B}_{\mathscr B^*}$ the determinant of $f$ computed with respect to $\mathscr B$ and $\mathscr B^*$. A straightforward calculation shows that $\det(f)^{\mathscr B}_{\mathscr B^*}=\det(A)$.
\end{remark}

Like Lemma \ref{rationality-lemma}, the result we are about to state should be understood in terms of the embedding $\iota_\Sigma$, as explained in Remark \ref{deligne-rem2}.

\begin{proposition} \label{rationality-prop}
Conjecture $\ref{ratconj}$ is equivalent to $L^*(\MM,0)\big/\bigl(\Omega^{(\gamma)}_\infty\cdot\Reg_\mathscr B(\MM)\bigr)\in F^\times$ for all $\gamma\in V^+_\B\smallsetminus\{0\}$ and all bases $\mathscr B$ of $H^1_\mot(\Q,\MM)$ over $F$. 
\end{proposition}

\begin{proof} Thanks to Lemma \ref{rationality-lemma}, it is enough to prove the claim for fixed $\gamma$ and $\mathscr B$ as above. Thus, let $\mathscr B=\{t_1,\dots,t_r\}$ be a basis of $H^1_\mathrm{mot}(\Q,\MM)$ over $F$, where $r=r_\alg(\MM)$, and let $\mathscr B^*=\{t_1^*,\dots,t_r^*\}$ be the dual basis of $H^1_\mathrm{mot}(\Q,\MM)^*$. Define
\[ \underline{t}_\mathscr B\defeq t_1\wedge\cdots\wedge t_r,\quad\underline{t}_\mathscr B^*\defeq t_1^*\wedge \cdots\wedge t_r^*, \] 
so that $\bigl\{\underline{t}_\mathscr B\bigr\}$ and $\bigl\{\underline{t}_\mathscr B^*\bigr\}$ are $F$-bases of $\bigwedge^r H^1_\mathrm{mot}(\Q,\MM)$ and $\bigwedge^r H^1_\mathrm{mot}(\Q,\MM)^*$, respectively. Pick $\gamma\in V_\B^+\smallsetminus\{0\}$ and set
\[ \zeta^\gamma_\mathscr B\defeq\underline{t}_\mathscr B^{-1}\otimes\underline{t}_\mathscr B^*\otimes\gamma^{-1}\otimes\omega_f, \] 
where $\omega_f\in {t(\MM)}_\infty$ is the differential form in \eqref{omega_f-eq}. Then $\bigl\{\zeta^\gamma_\mathscr B\bigr\}$ is a basis of $\Delta(\MM)$ over $F$ and, in light of Remark \ref{det-rem}, there is an equality 
\begin{equation} \label{beta-eq}
\theta_\infty\bigl(\zeta^\gamma_\mathscr B\bigr)=\bigl(\Omega^{(\gamma)}_\infty\cdot\Reg_\mathscr B(\MM)\bigr)^{-1}
\end{equation}
(see, \emph{e.g.}, \cite[Example 1.30]{Kings} for the computation of determinants). On the other hand, since any $\zeta_f$ as in Conjecture \ref{ratconj} differs from $\zeta^\gamma_\mathscr B$ by multiplication by an element of $F^\times$, Conjecture \ref{ratconj} is equivalent to the assertion that
\begin{equation} \label{beta-eq2}
a\theta_\infty\bigl(\zeta^\gamma_\mathscr B\bigr)=L^*(\MM,0)^{-1}
\end{equation}
for some $a\in F^\times$. The desired result follows by combining \eqref{beta-eq} and \eqref{beta-eq2}. \end{proof}

Since Conjectures \ref{ratconj} and \ref{ratconj2} are equivalent, Proposition \ref{rationality-prop} offers a reformulation of Conjecture \ref{ratconj2} as well.

\subsection{Local Galois cohomology} \label{local-subsec}

Let $K$ be a number field, $p$ a prime number, $v$ a place of $K$ and $V$ a $p$-adic representation of $G_{K_v}$. For a continuous $G_{K_v}$-module $M$ we write 
\[ \RR\Gamma(K_v,M)\defeq\mathcal{C}^\bullet(G_{K_v},M) \]
for the complex of continuous cochains of $G_{K_v}$ with values in $M$. Let $t(V)$ be as in \eqref{V-tangent-eq} and consider the complex 
\[ \RR\Gamma_f(K_v,V)\defeq\begin{cases}
\Bigl(\,\mathbf{D}_\crys(V)\xrightarrow{(1-\phi,\mathrm{pr})}\mathbf{D}_\crys(V)\oplus t(V)\Bigr)&\text{if $v\,|\,p$},\\[3mm]
\RR\Gamma(K_v,V) & \text{if $v\,|\,\infty$,} \\[3mm]
\Bigl(V^{I_v}\xrightarrow{1-\Frob_v}V^{I_v}\Bigr) & \text{if $v\nmid p\infty$},
\end{cases} \]
where $\phi$ is, as above, the Frobenius of $\mathbf{D}_\crys(V)$ and $\mathrm{pr}:\mathbf{D}_\crys(V)\rightarrow t(V)$ is the canonical map (\emph{cf.} Remark \ref{fontaine-rem}). Note that if $v\neq\infty$, then $\RR\Gamma_f(K_v,V)$ is concentrated in degrees $0$ and $1$. Denote by $H^\bullet_f(K_v,V)$ the cohomology of $\RR\Gamma_f(K_v, V)$. In particular, if $v\nmid p\infty$, then
\begin{equation} \label{h^0_f-eq}
H^0_f(K_v,V)=H^0(K_v,V)
\end{equation}
and
\begin{equation} \label{h^1_f-eq}
H^1_f(K_v,V)=H^1_\mathrm{unr}(K_v,V)\defeq H^1\bigl(\Gal(K_v^\mathrm{unr}/K_v),V^{I_v}\bigr).
\end{equation}
We also set  
\[ H^1_s(K_v,V)\defeq H^1(K_v,V)\big/H^1_f(K_v, V) \]
and call it the \emph{singular part of $V_p$ at $v$.} The complex $\RR\Gamma_f(K_v,V)$ is quasi-isomorphic to a subcomplex of the complex $\RR\Gamma(K_v,V)$, and we define $\RR\Gamma_s(K_v,V)$ to be the cokernel of the corresponding inclusion map. 

\begin{remark}
If $V$ is a $p$-adic representation of $G_K$, then we call $H^j_f(K_v,V)$ the \emph{$j$-th finite cohomology group of $V$ at $v$}.
\end{remark}

%
%We fix now a $G_{\Q_v}$-stable $\Z_p$-lattice $T_p$ in $ V_p$. Define 
%$H^0_f(\Q_v,T_p)=H^0(\Q_v,T_p)$, $H^1_f(\Q_v,T_p)$ the inverse image of $H^1_f(\Q_v, V_p)$ under the canonical map 
%$H^1(\Q_v,T_p)\rightarrow H^1(\Q_v, V_p)$ and $H^2(\Q_v,T_p)=0$. 
%Define $H^i_s(\Q_v,T_p)=H^1(\Q_v,T_p)/H^1_f(\Q_v,T_p)$.
 
\subsection{Global Galois cohomology} \label{global-subsec} 

Let $K$ be a number field, write $\mathscr P_K$ for the set of (archimedean and non-archimedean) primes of $K$ and let $p$ be a prime number. The set
\begin{equation} \label{S-def-eq}
S\defeq\bigl\{v\in\mathscr P_K\mid\text{$v$ divides $p\infty$}\bigr\}\cup\bigl\{v\in\mathscr P_K\mid\text{$V_p$ is ramified at $v$}\bigr\}
\end{equation}
is clearly finite. Let us write $G_{K,S}$ for the Galois group over $K$ of the maximal extension of $K$ unramified outside $S$. Finally, for any continuous $G_{K,S}$-module $M$ denote by
\[ \RR\Gamma(G_{K,S},M)\defeq\mathcal{C}^\bullet(G_{K,S},M) \]
the complex of continuous cochains of $G_{K,S}$ with values in $M$.  

\begin{remark} \label{S-rem}
For our later arguments, it would be equally fine to fix, in place of the set $S$ defined in \eqref{S-def-eq}, any subset of $\mathscr P_K$ containing $S$.
\end{remark}

\subsubsection{Finite cohomology} \label{finite-subsubsec}

The \emph{finite complex of $V_p$} is
\[ \RR\Gamma_f(K,V_p)\defeq\mathrm{Cone}\biggl(\RR\Gamma(G_{K,S},V_p)\longrightarrow \bigoplus_{v\in S}\RR\Gamma_s(K_v,V_p)\biggr)[-1]. \] 
We denote by $H^\bullet_f(K,V_p)$ the cohomology of $\RR\Gamma_f(K,V_p)$ and call it the \emph{finite} (or \emph{unramified}) \emph{cohomology of $V_p$} over $K$. In particular, by \cite[Lemma 19]{BF}, we have $H^j_f(\Q, V_p)=0$ for $j\notin\{0,1,2,3\}$ and there are isomorphisms
\begin{equation} \label{Seldual}
H^j_f(\Q, V_p)\simeq H^{3-j}_f(\Q, V_p)^*,
\end{equation} 
where, as before, $(\cdot)^*$ denotes the $\Q_p$-linear dual (note that we are implicitly using the fact that $V_p\simeq V_p^*(1)$).

\subsubsection{Cohomology with compact support} \label{compactcone}

We introduce cohomology with compact support only for $K=\Q$; for simplicity, let us set $G_S\defeq G_{\Q,S}$. See, \emph{e.g.}, \cite[\S 5.3]{Nek-Selmer} for the case of a general global field.

Let $M$ be a continuous $G_S$-module. The \emph{compact complex of $M$} is 
\[ \RR\Gamma_c(G_S,M)=\mathrm{Cone}\biggl(\RR\Gamma(G_S,M)\longrightarrow \bigoplus_{v\in S}\RR\Gamma(\Q_v,M)\biggr)[-1] \]
We denote by $H^\bullet_c(\Q,M)$ the cohomology of $\RR\Gamma_c(G_S,M)$ and call it the \emph{cohomology with compact support of $M$}. Observe that there is a triangle
\begin{equation} \label{triangle}
\RR\Gamma_c(G_S,V_p)\longrightarrow\RR\Gamma_f(\Q,V_p)\longrightarrow\bigoplus_{v\in S}\RR\Gamma_f(\Q_v, V_p) 
\end{equation}
that can be made into a true triangle (see \cite[\S 3.2]{BF}).

\subsection{The $p$-adic \'etale regulator of $\MM$} \label{p-adic-reg-subsec}

We introduce the $p$-adic \'etale regulator of the modular motive $\MM$ over a number field. 

\subsubsection{Anaemic splittings in \'etale cohomology} \label{anaemic-splitting-subsubsec}

Let $p$ be a prime number, fix a prime $\p$ of $F$ above $p$ and let $\iota_\p:\bar\Q\hookrightarrow\bar\Q_p$ be an embedding that induces $\p$. With $\Pi$ as in \S \ref{notation-subsec}, define
\[ {W}_{\bar\Q_p}\defeq\Pi\cdot H^{k-1}_\et\bigl(\bar{X},\bar\Q_p(k/2)\bigr). \]
If we set $\theta_\p\defeq\iota_\p\circ\theta$ and let $\theta$ range over all homomorphisms $\theta:\mathfrak{H}_k(\Gamma_0(N))_{\bar\Q}\rightarrow\bar\Q$ of $\bar\Q$-algebras, then $\theta_\p$ varies over all homomorphisms $\mathfrak{H}_k(\Gamma_0(N))_{\bar\Q}\rightarrow\bar\Q_p$ of $\bar\Q$-algebras. It follows that there is an ``anaemic'' splitting
\begin{equation} \label{split-1}
{W}_{\bar\Q_p}=\bigoplus_\theta{W}_{\bar\Q_p}[\theta_\p],
\end{equation} 
where $W_{\bar\Q_p}[\theta_\p]$ is the $\theta_\p$-eigenspace of $W_{\bar\Q_p}$ under the action of $\mathfrak{H}_k(\Gamma_0(N))_{\bar\Q}$. Recall from \S \ref{isotypic-subsubsec} that $\theta_f$, which arises as a map $\mathfrak H_k(\Gamma_0(N))\rightarrow\cO_F$, can also be viewed as a map $\mathfrak H_k(\Gamma(N))\rightarrow\cO_F$ via the surjection $\mathfrak H_k(\Gamma(N))\twoheadrightarrow\mathfrak H_k(\Gamma_0(N))$. In particular, restriction gives a map $\theta_f:\mathfrak H_{k,\bar\Q}^{(N)}\rightarrow\bar\Q$. Set  
\[ {W}_\p\defeq\Pi\cdot H^{k-1}_\et\bigl(\bar{X},F_\p(k/2)\bigr), \]
so that $V_\p=W_\p[\theta_f]$. By a slight abuse of notation, we adopt the same symbol for $\theta_f$ and $\iota_\p\circ\theta_f$, which allows us to view ${W}_{\bar\Q_p}[\theta_f]$ as one of the direct summands appearing in \eqref{split-1}. There is a canonical injection ${W}_\p\hookrightarrow{W}_{\bar\Q_p}$ that gives rise to a commutative square 
\begin{equation} \label{commutative-regulator-eq}
\xymatrix{{W}_\p\ar@{^(->}[r]&{W}_{\bar\Q_p}\ar@{->>}[d]\\
V_\p\ar@{^(->}[r]\ar@{^(->}[u]& {W}_{\bar\Q_p}[\theta_f],} 
\end{equation}
where the right vertical arrow is the projection induced by \eqref{split-1} and the other maps are the obvious injections. Thus, we obtain from \eqref{commutative-regulator-eq} a canonical surjection $\pi_{f,\p}:W_\p\twoheadrightarrow V_\p$. Finally, set
\[ {W}_p\defeq\bigoplus_{\p\mid p}{W}_\p. \]
Taking sums over all $\p\,|\,p$, the maps $\pi_{f,\p}$ yield a canonical surjection $\pi_{f,p}:{W}_p\twoheadrightarrow V_p$.

\subsubsection{$p$-adic \'etale regulator} \label{regulator-subsubsec}

Let $K$ be a number field. For $\star\in\{p\}\cup\{\p\,|\,p\}$, set
\begin{equation} \label{H^1-mot-star-eq}
H^1_\mot(K,\MM{)}_\star\defeq H^1_\mot(K,\MM)\otimes_FF_\star.
\end{equation}
Denote by
\begin{equation} \label{AJ}
\AJ_{K,\Z_p}:\CH^{k/2}_0(X/K)\longrightarrow H^1\Bigl(K,H^{k-1}_\text{\'et}\bigl(\bar X,\Z_p(k/2)\bigr)\!\Bigr) 
\end{equation}
the (integral) $p$-adic Abel--Jacobi map induced by the $p$-adic cycle class map (see, \emph{e.g.}, \cite[\S 4]{Nek}, \cite[\S 1]{Nek2}, \cite[\S 1]{Nek3}). With notation as in \eqref{R-chow-eq}, for each prime $\p$ of $F$ above $p$ we obtain a map 
%\begin{equation} \label{AJ-1}
\[ \AJ_{K,F_\p}:{\CH^{k/2}_0(X/K)}_{F_\p}\longrightarrow H^1\Bigl(K,H^{k-1}_\text{\'et}\bigl(\bar X,F_\p(k/2)\bigr)\!\Bigr). \]
%\end{equation}
Therefore, taking sums over all $\p\,|\,p$, we get from $\AJ_{K,F_\p}$ a map 
\begin{equation} \label{AJ-2*}
\AJ_{K,F_p}:{\CH^{k/2}_0(X/K)}_{F_p}\longrightarrow H^1\Bigl(K,H^{k-1}_\text{\'et}\bigl(\bar X,F_p(k/2)\bigr)\!\Bigr). 
\end{equation}
Finally, applying the $F_p$-linear extension of $\Pi$ to \eqref{AJ-2*}, restricting the resulting map to $H^1_\mot(K,\MM)_p$ and applying the map induced by $\pi_{f,p}$ to its target, we get a map 
\begin{equation} \label{p-reg-eq}
\reg_{K,p}:H^1_\mot(K,\MM)_p\longrightarrow H^1(K,V_p)
\end{equation}
that is called the \emph{$p$-adic \'etale regulator} (or simply the \emph{$p$-adic regulator}) \emph{of $\MM$ over $K$}. We also set $\reg_p\defeq\reg_{\Q,p}$. By construction, there is a splitting $\reg_{K,p}=\bigoplus_{\p|p}\reg_{K,\p}$, where
\begin{equation} \label{pp-reg-eq}
\reg_{K,\p}:H^1_\mot(K,\MM)_\p\longrightarrow H^1(K,V_\p) 
\end{equation}
is the \emph{$\p$-adic regulator of $\MM$ over $K$}. Again, we may set $\reg_\p\defeq\reg_{\Q,\p}$.

\begin{remark} \label{saito-rem}
As a consequence of work of Saito on the weight-monodromy conjecture for compactified Kuga--Sato varieties (\cite{saito}, \cite{saito2}) and of results of Nekov\'a\v{r} (\cite{Nek3}) and Nizio\l\ (\cite{niziol}) on $p$-adic regulators, we know that $\im(\reg_{K,p})\subset H^1_f(K, V_p)$, where $H^1_f(K, V_p)$ is the finite cohomology group from \S \ref{finite-subsubsec}. See, \emph{e.g.}, \cite[Theorem 2.4]{LV} for details.
\end{remark}

The next conjecture predicts a deep relation between motivic cohomology and (global) unramified cohomology.

\begin{conjecture}[$p$-adic regulator] \label{regpconj}
For all primes $p$ and all number fields $K$, the $p$-adic regulator in \eqref{p-reg-eq} induces an isomorphism 
\begin{equation} \label{regpconj-eq}
\reg_{K,p}:H^1_\mot(K,\MM)_p\overset\simeq\longrightarrow H^1_f(K, V_p) 
\end{equation}
of $F_p$-modules. 
\end{conjecture}

We refer to the statement of Conjecture \ref{regpconj} for fixed $p$ and $K$ as the \emph{$p$-part of Conjecture \ref{regpconj} over $K$}. Later on, we shall need to assume that (the $p$-part of) Conjecture \ref{regpconj} holds true over certain number fields.

\begin{remark} \label{reg-MM-rem}
Let $K$ be a number field and suppose that the $p$-part of Conjecture \ref{regpconj} over $K$ holds true for a prime number $p$. It is well known (essentially a consequence of results of Tate, \emph{cf.} \cite[\S 2]{Tate}) that $H^1_f(K, V_p)$ is finitely generated over $F_p$, so isomorphism \eqref{regpconj-eq} implies that $H^1_\mot(K,\MM)_p$ is finitely generated over $F_p$ as well. It follows that $H^1_\mot(K,\MM)$ is finite-dimensional as an $F$-vector space: we conclude that Conjecture \ref{regpconj} for some prime $p$ implies Conjecture \ref{finitenessconj}. Note that isomorphism \eqref{regpconj-eq} ensures, in fact, that $H^1_f(K, V_p)$ is free (of finite rank) over $F_p$.
\end{remark}

\begin{remark}
For the counterpart of Conjecture \ref{regpconj} for motives of elliptic curves, the reader is referred to \cite[Example 2.16]{Kings}.
\end{remark}

\subsection{Projective $\cO$-structures in $\MM$} \label{sec1.15}

In the definition that follows, $\cO$ is an order of $F$. Furthermore, given a prime $p$, we consider the semilocal ring $\cO\otimes_\Z\Z_p$. Denote by 
\begin{equation} \label{betti-etale-eq}
\Comp_{\B,\et}:V_\B\otimes_FF_p\overset\simeq\longrightarrow V_p 
\end{equation}
the comparison isomorphism between Betti and \'etale cohomology.

The following notion was introduced in \cite[\S 3.3, Definition 1]{BF}.

\begin{definition} \label{projective-def}
A \emph{projective $\cO$-structure} in $\MM$ is a finitely generated projective $\cO$-module $T_\B\subset  V_\B$ such that
\begin{enumerate}
\item $T_\B\otimes_\cO F\simeq  V_\B$; 
\item $\Comp_{\B,\et}\bigl(T_\B\otimes_\cO(\cO\otimes_\Z\Z_p)\bigr)$ is a Galois-stable $\cO_p$-lattice in $V_p$ for all primes $p$. 
\end{enumerate}
\end{definition}

The $\cO_F$-module $T_\B$ that was defined in \S\ref{Betti-subsubsec} is a projective $\cO_F$-structure in $\MM$. Moreover, if $T_p\subset V_p$ is the $G_\Q$-stable $\cO_p$-lattice introduced in \eqref{T_p-eq}, then the integrality properties of $\Comp_{\B,\et}$ (see, \emph{e.g.}, \cite[Exp. XI, Th\'eor\`eme 4.4, (iii)]{SGA4}) ensure that
\begin{equation} \label{B-et-eq}
\Comp_{\B,\et}(T_\B\otimes_{\cO_F}\cO_p)=T_p. 
\end{equation}
We highlight this equality for future use.

\subsection{The Tamagawa number conjecture for $\MM$} \label{TNC-subsec}

We formulate the Tamagawa number conjecture of Bloch--Kato (\cite{BK}) and Fontaine--Perrin-Riou (\cite{FPR}) in the case of the motive $\MM$.

\subsubsection{The isomorphism $\theta_{p,S}$}

Define the $F_p$-module
\[ t(\MM)_p\defeq t(\MM)\otimes_\Q\Q_p=t(\MM)\otimes_FF_p. \] 
The comparison isomorphism between de Rham and \'etale cohomology induces an isomorphism 
\[ \Comp_{\dR,\et}: t(V_p)\overset\simeq\longrightarrow t(\MM)_p. \] 
Let $v$ be place of $\Q$. Note that 
\begin{equation} \label{Det}
\Det_{F_p}^{-1}\bigl(\RR\Gamma_f(\Q_v, V_p)\bigr)\simeq
\begin{cases}
(F_p,0) & \text{if $v\notin\{p,\infty\}$},\\[3mm]
\Det_{F_p}\bigl(t(\MM)_p\bigr) & \text{if $v=p$},\\[3mm]
\Det_{F_p}^{-1}\bigl(H^0(\R,V_p)\bigr) & \text{if $v=\infty$}.
\end{cases}
\end{equation}
Combining the comparison isomorphism from \eqref{betti-etale-eq} with the multiplicativity of $\Det_{F_p}$ applied to \eqref{triangle} and with \eqref{Det}, we obtain an isomorphism
\[ \Det_{F_p}\bigl(\RR\Gamma_c(G_S, V_p)\bigr)\simeq\Det_{F_p}\bigl(\RR\Gamma_f(\Q,V_p)\bigr)\cdot\Det_{F_p}\bigl(t(\MM)_p\bigr)\cdot\Det_{F_p}^{-1}(V_\B^+). \]
Define the $F_p$-module 
\[ \Delta(\MM)_p\defeq\Delta(\MM)\otimes_\Q\Q_p=\Delta(\MM)\otimes_FF_p. \]
Using Conjecture \ref{regpconj}, the definition of the fundamental line $\Delta(\MM)$ (Definition \ref{fundline}) and \eqref{Seldual}, we get a conjectural isomorphism of $F_p$-modules
\begin{equation} \label{theta-isom-eq}
\theta_{p,S}:\Delta(\MM)_p\overset\simeq\longrightarrow\Det_{F_p}\bigl(\RR\Gamma_c(G_S, V_p)\bigr).
\end{equation}

\subsubsection{TNC for $\MM$}

We formulate the Tamagawa number conjecture (TNC, for short) for the motive $\MM$ over $\Q$. Our previous notation is in force: $S$ is the set of primes that was fixed in \eqref{S-def-eq}, $\theta_{p,S}$ is the isomorphism in \eqref{theta-isom-eq} and $T_p$ is the $\cO_p$-lattice in \eqref{T_p-eq}. Recall the zeta elements $\zeta_f$ and $\zeta_f^*$ appearing in Conjectures \ref{ratconj} and \ref{ratconj2}.

\begin{conjecture}[TNC for $\MM$] \label{TNC} 
Assume Conjectures \ref{ratconj} and \ref{regpconj}. Let $T_\B$ be a projective $\cO$-structure in $\MM$ for some order $\cO$ of $F$. For every prime number $p$ there is an equality
\begin{equation} \label{TNC-eq}
\theta_{p,S}(\zeta_f^*)\cdot\mathcal{O}_p=\Det_{\cO_p}\bigl(\RR\Gamma_c(G_S,T_p)\bigr) 
\end{equation}
of $\cO_p$-submodules of $\Det_{F_p}\bigl(\RR\Gamma_c(G_S,V_p)\bigr)$.  
\end{conjecture}

Henceforth, equality \eqref{TNC-eq} for a given $p$ will be referred to as the \emph{$p$-part of the TNC} for $\MM$; we will sometimes indicate it as $p$-TNC.

\begin{remark}
As in \cite[Remark 2.21]{Kings}, one can show that Conjecture \ref{TNC} is independent of the choice of the $\cO$-projective structure $T_\B$; moreover, keeping Remark \ref{S-rem} in mind, it can also be checked that Conjecture \ref{TNC} does not depend on the choice of $S$, in the sense explained in \cite[Remark 2.22]{Kings}. 
\end{remark}

\subsection{Bloch--Kato Selmer groups} \label{secBK} 

Let $K$ be a number field and let $p$ be a prime number. Let $V$ be a $p$-adic representation of $G_K$ and let $T$ be a $\Z_p$-lattice in $V$. Set $A\defeq V/T$. If $T$ is endowed with a $\Z_p$-linear action of an order $\cO$ of $F$, then $V$ inherits a structure of an $F_p$-module, while both $T$ and $A$ are equipped with a structure of $\cO_p$-modules.

\subsubsection{Finite local conditions}

Let $v$ be a place of $K$. The \emph{finite local conditions} $H^\bullet_f(K_v,T)$ and $H^\bullet_f(K_v,A)$ 
at $v$ are defined by propagation from the cohomology groups $H^\bullet_f(K_v,V)$ in \S \ref{local-subsec} using the canonical maps $T\hookrightarrow V$ and $V\twoheadrightarrow A$ (see, \emph{e.g.}, \cite[\S 1.1]{MR}). In particular, it follows from \eqref{h^0_f-eq} that if $v\nmid p\infty$, then
\[ H^0_f(K_v,T_p)=H^0(K_v,T_p),\quad H^0_f(K_v,A_p)=H^0(K_v,A_p). \]
We denote by $H^\bullet_s(K_v,T)$ (respectively, $H^\bullet_s(K_v,A)$) the quotients of $H^\bullet(K_v,T)$ (respectively, $H^\bullet(K_v,A)$) by $H^\bullet_f(K_v,T)$ (respectively, $H^\bullet_f(K_v,A)$).

\subsubsection{Bloch--Kato Selmer groups} \label{p-modules-subsec}

In the following definition, let $M\in\{V,T,A\}$. 

\begin{definition} \label{Bloch-Kato-def} 
The \emph{Bloch--Kato Selmer group of $M$ over $K$} is 
\[ H^1_f(K,M)\defeq\ker\biggl(H^1(K,M)\longrightarrow\prod_{v}H^1_s(K_v,M)\biggr), \]
where the product is taken over all places $v$ of $K$. 
\end{definition}

One can check (see \cite[Lemma 2.15]{Kings} or \cite[Lemma 5.1]{LV-Iwasawa}) that 
\begin{equation} \label{selmer-alternative-eq}
H^1_f(K,M)=\ker\biggl(H^1(G_{K,S},M)\longrightarrow\bigoplus_{v\in S}H^1_s(K_v,M)\biggr),
\end{equation}
where $S$ is the finite set of places of $K$ that was fixed in \S \ref{global-subsec}.

As in \S \ref{etale-subsubsec}, for all primes $\p$ of $F$ above $p$ we set $T_\p\defeq T_p\otimes_{\cO_p}\!\cO_\p$, which is an $\cO_\p$-lattice inside $V_\p$; there is a splitting $T_p=\prod_{\p\mid p}T_\p$. We also put $A_\p\defeq V_\p/T_\p$; if $A_p$ is defined as in \eqref{A-eq}, then $A_p=\prod_{\p\mid p}A_\p$. There is a splitting
\begin{equation} \label{splittings-eq}
H^1_f(K,M_p)=\bigoplus_{\p\mid p}H^1_f(K,M_\p), 
\end{equation}
where the direct sum is taken over all primes $\p$ of $F$ above $p$.

\subsection{Shafarevich--Tate groups of $\MM$} \label{STsubsec} 

Let $p$ be a prime number. From now on, for a $p$-primary abelian group $G$ we denote by $G_\divv$ the maximal $p$-divisible subgroup of $G$. We introduce Shafarevich--Tate groups \emph{\`a la} Bloch--Kato.

\subsubsection{Shafarevich--Tate groups} \label{ST-subsubsec}

Let $K$ be a number field. For any prime $\p$ of $F$ above $p$, recall the Bloch--Kato Selmer group $H^1_f(K,A_\p)$ of $A_\p$ over $K$ from \S \ref{p-modules-subsec}. The following definition of Shafarevich--Tate group is due to Bloch--Kato (\cite[Remark 5.15.2]{BK}; \emph{cf.} also \cite{Flach}).
   
\begin{definition} \label{Sha-def}
\begin{enumerate}
\item The (\emph{Bloch--Kato}) \emph{Shafarevich--Tate group of $\MM$ over $K$ at $\p$} is
\[ \Sha_\p^{\BKK}(K,\MM)\defeq H^1_f(K,A_\p)\big/H^1_f(K,A_\p{)}_\divv. \]
\item The (\emph{Bloch--Kato}) \emph{Shafarevich--Tate group of $\MM$ over $K$ at $p$} is
\[ \Sha_p^{\BKK}(K,\MM)\defeq H^1_f(K, A_p)\big/H^1_f(K, A_p{)}_\divv. \]
\item The (\emph{Bloch--Kato}) \emph{Shafarevich--Tate group of $\MM$ over $K$} is
\[ \Sha^{\BKK}(K,\MM)\defeq\bigoplus_p\Sha_p^{\BKK}(K,\MM), \]
where $p$ varies over all prime numbers.
\end{enumerate}
\end{definition}

There is a splitting $\Sha_p^{\BKK}(K,\MM)=\bigoplus_{\p\mid p}\Sha_\p^{\BKK}(K,\MM)$, where the direct sum is taken over all primes $\p$ of $F$ above $p$. Therefore, we can write 
\[ \Sha^{\BKK}(K,\MM)=\bigoplus_\lambda\Sha_\lambda^{\BKK}(K,\MM)=\bigoplus_\ell\Sha_\ell^{\BKK}(K,\MM), \]
where $\lambda$ (respectively, $\ell$) varies over all primes of $F$ (respectively, all prime numbers). Notice that $\Sha_\lambda^{\BKK}(K,\MM)$ is finite for every $\lambda$, and then the same is true of $\Sha_\ell^{\BKK}(K,\MM)$ for every $\ell$. We remark that in \S \ref{nekovar-subsubsec} we will introduce also Shafarevich--Tate groups $\Sha_\p^{\Nek}(K,\MM)$ \emph{\`a la} Nekov\'a\v{r}: the interplay between $\Sha_\p^{\BKK}(K,\MM)$ and $\Sha_\p^{\Nek}(K,\MM)$ will be crucial for our arguments.

\subsubsection{A finiteness conjecture}

By analogy with a classical conjecture for Shafarevich--Tate groups of abelian varieties over global fields, it is natural to propose

\begin{conjecture}[Finiteness of $\Sha$] \label{Shaconj} 
For all number fields $K$, the group $\Sha^{\BKK}(K,\MM)$ is finite. 
\end{conjecture} 

Clearly, Conjecture \ref{Shaconj} (which will play no explicit role in the paper) is equivalent to the prediction that, for all number fields $K$, the group $\Sha_p^{\BKK}(K,\MM)$ is trivial for all but finitely many $p$. 

\begin{remark}
To be in line with terminology and notation introduced in Definition \ref{Sha-def} for Shafarevich--Tate groups, we could alternatively set $\Sel_p(K,\MM)\defeq H^1_f(K,A_p)$ and call it the \emph{Bloch--Kato Selmer group of $\MM$ over $K$ at $p$}. However, later on we shall reserve a symbol of this kind (at least when $K$ varies over the finite layers of the cyclotomic $\Z_p$-extension of $\Q$) for Selmer groups in the sense of Greenberg (\emph{cf.} \S \ref{SUsec}), so here we chose to adopt the notation that was originally used by Bloch and Kato in \cite{BK}.
\end{remark}

\subsection{Local finite cohomology groups} \label{local-finite-subsec}

We collect some basic facts on local cohomology groups of $p$-adic Galois representations. 

\subsubsection{Local Tate duality}

Let $V$, $T$, $A$ be as in \S\ref{secBK}. Define 
\[T^*\defeq\Hom_{\Z_p}(T,\Z_p) \] 
and recall that if $v$ is a place of $\Q$, then under the local Tate duality pairing
\[ {(\cdot,\cdot)}_v:H^1(\Q_v,T)\times H^1\bigl(\Q_v,T^*\otimes_{\Z_p}\!(\Q_p/\Z_p)(1)\bigr)\longrightarrow \Q_p/\Z_p \] 
the subgroups $H^1_f(\Q_v,T)$ and $H^1_f\bigl(\Q_v,T^*\otimes_{\Z_p}\!(\Q_p/\Z_p)(1)\bigr)$ are exact annihilators of each other (\cite[Proposition 3.8]{BK}). We remark that $H^1_f\bigl(\Q_v,T^*\otimes_{\Z_p}\!(\Q_p/\Z_p)(1)\bigr)$ is defined by propagation from the corresponding local conditions for the representation $V^*(1)$, where $V^*\defeq\Hom_{\Q_p}(V,\Q_p)$. Now we assume that there is an isomorphism $V^*(1)\simeq V$ under which $T^*(1)\simeq T$. It follows that $A\simeq T^*\otimes_{\Z_p}\!(\Q_p/\Z_p)(1)$, so $H^1_f\bigl(\Q_v,T^*\otimes_{\Z_p}\!(\Q_p/\Z_p)(1)\bigr)$ is isomorphic to $H^1_f(\Q_v,A)$. Then the local Tate pairing at $v$ yields a perfect pairing 
\begin{equation} \label{local-tate-eq}
{(\cdot,\cdot)}_v:H^1(\Q_v,T)\times H^1(\Q_v,A)\longrightarrow \Q_p/\Z_p
\end{equation}
under which the subgroups $H^1_f(\Q_v,T)$ and $H^1_f(\Q_v,A)$ are exact annihilators of each other. Since ${(\cdot,\cdot)}_v$ is perfect, this means that there are isomorphisms 
\begin{equation} \label{varphi}
\varphi_v:H^1_f(\Q_v,T)\overset\simeq\longrightarrow H^1_s(\Q_v,A)^\vee
\end{equation}
and
\begin{equation} \label{psi}
\psi_v:H^1_s(\Q_v,T)\overset\simeq\longrightarrow H^1_f(\Q_v,A)^\vee,
\end{equation}
where for a $\Z_p$-module $M$ we let 
\begin{equation} \label{pontryagin-def-eq}
M^\vee\defeq\Hom_\cont(M,\Q_p/\Z_p)
\end{equation}
be the Pontryagin dual of $M$.

\subsubsection{The case of modular motives}

In the case of motives of modular forms, there is an isomorphism $V_p^*(1)\simeq V_p$ under which $T_p^*(1)\simeq T_p$, so the results above apply with $V=V_p$, $T=T_p$, $A=A_p$.

\begin{lemma} \label{lemma1.15}
There is a commutative diagram
\[ \xymatrix{
0\ar[r]&
\bigoplus_{v\in S}H^1_f(\Q_v, T_p )\ar[r]\ar[d]_-\simeq^-{\varphi_S} &
\bigoplus_{v\in S}H^1(\Q_v, T_p )\ar[r]\ar[d] & 
\bigoplus_{v\in S}H^1_s(\Q_v, T_p )\ar[r] \ar[d]& 
0\\
& 
\bigoplus_{v\in S}H^1_s(\Q_v, A_p )^\vee\ar[r]& 
H^1(G_S, A_p )^\vee\ar[r] & 
H^1_f(\Q, A_p )^\vee\ar[r] & 
0
} \]
with exact rows.
\end{lemma}

\begin{proof} The top row is a direct consequence of the definitions of the groups involved, while the bottom row is obtained by taking Pontryagin duals of the exact sequence
\[ 0\longrightarrow H^1_f(\Q, A_p)\longrightarrow H^1(G_S,A_p)\longrightarrow\bigoplus_{v\in S} H^1_s(\Q_v,A_p) \] 
induced by \eqref{selmer-alternative-eq} with $M=A_p$. The vertical isomorphism on the left is defined by setting $\varphi_S\defeq\bigoplus_{v\in S}\varphi_v$, with $\varphi_v$ as in \eqref{varphi}. On the other hand, the middle vertical arrow is the composition of the map $\bigoplus_{v\in S}H^1(\Q_v,T_p)\rightarrow \bigoplus_{v\in S}H^1(\Q_v,A_p)^\vee$ induced by \eqref{local-tate-eq} and the dual of the map $H^1(G_S,A_p)\rightarrow\bigoplus_{s\in S}H^1(\Q_v,A_p)$ given by restriction in cohomology. Finally, the right vertical map is the composition of the map $\oplus_{v\in S}\psi_v$, where $\psi_v$ is as in \eqref{psi}, and the dual of the map $H^1_f(\Q,A_p)\rightarrow\bigoplus_{s\in S}H^1_f(\Q_v,A_p)$ defined by restriction in cohomology. The commutativity of the diagram is immediate by construction. \end{proof}

In the rest of this article, we let $p$ be a prime number and work under the following

\begin{assumption} \label{motass}
\begin{enumerate}
\item $p\nmid N$;
%\vskip 1mm
\item $V_p$ is ramified at the primes dividing $N$; 
%\vskip 1mm
\item $\ell$ prime, $\ell^2\,|\,N\Rightarrow V_p^{I_\ell}=0$. 
\end{enumerate}
\end{assumption}

In part (3) above, $I_\ell\subset G_{\Q_\ell}$ is the inertia subgroup at $\ell$. As a consequence of Assumption \ref{motass}, the (finite) set of places of $\Q$ from \S \ref{global-subsec} is explicitly given by
\begin{equation} \label{S-eq}
S\defeq\bigl\{\text{$\ell$ prime}\mid\text{$\ell$ divides $Np$}\bigr\}\cup\{\infty\}. 
\end{equation}

\begin{lemma} \label{lemmaram}
If $\ell\not=p$ is a prime number, then $H^0_f(\Q_\ell,V_p)=H^1_f(\Q_\ell,V_p)=0$.
\end{lemma}

\begin{proof} Let $S$ be as in \eqref{S-eq} and let $\ell$ be a prime number. First assume that $\ell\not\in S$. Since $ V_p$ is unramified at $\ell$, we have $ V_p^{I_\ell}= V_p$. It follows from \eqref{h^1_f-eq} and \cite[Lemma 1.3.2, (i)]{Rubin-ES} that there is an isomorphism
\begin{equation} \label{Frob-eq}
H^1_f(\Q_\ell,V_p)\simeq V_p\big/(\Frob_\ell-1)V_p.
\end{equation}
In this case, $\Frob_\ell$ acts with eigenvalues $\alpha$, $\beta$ such that $|\alpha|=|\beta|=\ell^{k/2+1}$;  we deduce that $\Frob_\ell-1: V_p\rightarrow V_p$ is an isomorphism, and the result follows from \eqref{Frob-eq}. 

Assume now that $\ell\,|\,N$. If $\ell^2\,|\,N$, then by part (3) of Assumption \ref{motass} we have $V_p^{I_\ell}=0$, so again the result follows from \eqref{h^1_f-eq}. Finally, assume that $\ell\|N$. The restriction of $V_p$ to $G_{\Q_\ell}$ is isomorphic to $\bigl(\begin{smallmatrix}\chi_\mathrm{cyc}&c\\0&1\end{smallmatrix}\bigr)$, where $\chi_\mathrm{cyc}$ is the $p$-adic cyclotomic character and $c:G_{\Q_\ell}\rightarrow V_p$ is a $1$-cocycle. Moreover, $I_\ell$ acts via the map $g\mapsto\bigl(\begin{smallmatrix}1&c(g)\\0&1\end{smallmatrix}\bigr)$, so, since $V_p$ is ramified at $\ell$, we have $c\neq0$ and $V_p^{I_\ell}\simeq F_p(1)$. In particular, $\Frob_\ell-1$ is an isomorphism of $V_p^{I_\ell}$, and by \eqref{h^1_f-eq} the lemma is proved. \end{proof}

\begin{lemma} \label{lemmap}
$H^0(\Q_p,V_p)=0$. 
\end{lemma}

\begin{proof} By part (1) of Assumption \ref{motass}, $V_p$ is a crystalline, hence de Rham, representation, so $\mathbf{D}_\dR(V_p)=\mathbf{D}_\crys(V_p)$. By \cite[Theorem 4.1, (ii)]{BK}, the Bloch--Kato exponential map gives an isomorphism 
\[ \exp_\mathrm{BK}:t(V_p)\overset{\simeq}\longrightarrow H^1_f(\Q_p,V_p), \]
and then \cite[Corollary 3.8.4]{BK} implies that $H^0(\Q_p, V_p)=0$. \end{proof}

\subsection{On the cohomology of $T_\star$, $V_\star$, $A_\star$} \label{cohomology-subsec}

Let $K$ be a number field. Recall that we assume throughout that Conjecture \ref{finitenessconj} is true, \emph{i.e.}, $H^1_\mot(K,\MM)$ has finite dimension, denoted by $r_\alg(\MM/K)$, over $F$. For notational convenience, set $r\defeq r_\alg(\MM/K)$. Furthermore, assume also that
\begin{itemize}
\item the $p$-part of Conjecture \ref{regpconj} over $K$ holds true.
\end{itemize}
This condition will essentially be in force until the end of the article. Therefore, the $p$-adic regulator map from \eqref{p-reg-eq} is an isomorphism
\[ \reg_{K,p}:H^1_\mot(K,\MM)_p\overset\simeq\longrightarrow H^1_f(K,V_p) \]
of $F_p$-modules. As a consequence, $H^1_f(K, V_p)$ is free of rank $r$ over $F_p$. Since $F_p=\prod_{\p|p}F_\p$, it follows from \eqref{splittings-eq} with $M=V$ that the $F_\p$-vector space $H^1_f(L,V_\p)$ has dimension $r$ for all primes $\p$ of $F$ above $p$. 

\subsubsection{Cohomology and divisible submodules}

In the following lines, let $\star\in\{p\}\cup\{\p\,|\,p\}$. Set
\[ \uH^1_f(K,T_\star)\defeq\im\Bigl(H^1_f(K,T_\star)\longrightarrow H^1_f(K,V_\star)\Bigr). \] 
Since $H^1_f(K,T_\p)$ is finitely generated over $\cO_\p$ and there is a canonical isomorphism
\[ H^1_f(L,T_\p)\otimes_{\cO_\p}\!F_\p\overset\simeq\longrightarrow H^1_f(L,V_\p) \]
of $F_\p$-vector spaces (\emph{cf.} \cite[Proposition B.2.4]{Rubin-ES} and  \cite[Proposition 2.3]{Tate}), we conclude that $\uH_f^1(K,T_\p)$ is a free $\cO_\p$-submodule of $H^1_f(K,V_\p)$ of rank $r$; in other words, $\uH_f^1(K,T_\p)$ is an $\cO_\p$-lattice inside $H^1_f(K,V_\p)$. Again by \eqref{splittings-eq}, $\uH^1_f(K,T_p)=\bigoplus_{\p|p}\uH^1_f(K,T_\p)$, so $\uH^1_f(K,T_p)$ is a free $\cO_p$-submodule of $H^1_f(K,V_p)$ of rank $r$.

It can be checked that there is an exact sequence
\begin{equation} \label{T-V-A-eq}
H^1_f(K,T_\star)\longrightarrow H^1_f(K,V_\star)\longrightarrow H^1_f(K,A_\star), 
\end{equation}
which induces an exact sequence
\begin{equation} \label{T-V-A-eq2}
0\longrightarrow\uH^1_f(K,T_\star)\longrightarrow H^1_f(K,V_\star)\longrightarrow H^1_f(K,A_\star). 
\end{equation}
The group $H^1_f(K,V_\star)$ is a vector space over a field of characteristic $0$, so it is divisible, and then the rightmost map in \eqref{T-V-A-eq} gives a map
\begin{equation} \label{V-A-div-eq}
\Upsilon_\star:H^1_f(K,V_\star)\longrightarrow H^1_f(K,A_\star{)}_\divv.
\end{equation}
Clearly, $\Upsilon_p=\bigoplus_{\p|p}\Upsilon_\p$.

\begin{proposition} \label{upsilon-onto-prop}
The map $\Upsilon_\star$ is surjective.
\end{proposition}

\begin{proof} The surjectivity of $\Upsilon_p$ is equivalent to that of $\Upsilon_\p$ for all $\p\,|\,p$, which is well known (\emph{cf.} \cite{Flach}, \cite[\S 2.1.3]{NP}). \end{proof}

\subsubsection{On Pontryagin duals}

Let $M\in\{T,V,A\}$. As is pointed out in Remark \ref{pontryagin-rem2}, there is an identification
\begin{equation} \label{pontryagin-eq}
\Hom_\cont\bigl(H^1_f(L,M_\p),\Q_p/\Z_p\bigr)=\Hom_\cont\bigl(H^1_f(L,M_\p),F_\p/\cO_\p\bigr),
\end{equation}
which provides an alternative description of the Pontryagin dual $H^1_f(L,M_\p)^\vee$ of $H^1_f(L,M_\p)$. Let us also define
\[ H^1_f(K,M_p)^\vee\defeq\Hom_\cont\bigl(H^1_f(K,M_p),F_p/\cO_p) \]
and call this $\cO_p$-module the \emph{Pontryagin dual} of $H^1_f(K,M_p)$. It follows from \eqref{pontryagin-eq} and the splitting $F_p/\cO_p=\prod_{\p|p}F_\p/\cO_\p$ that
\begin{equation} \label{M-dual-splitting-eq}
H^1_f(K,M_p)^\vee=\bigoplus_{\p|p}H^1_f(K,M_\p)^\vee. 
\end{equation}
Let us write $\corank_{\cO_\star}\!H^1_f(K,A_\star)$ for the corank of $H^1_f(K,A_\star)$ over $\cO_\star$, \emph{i.e.}, the rank as an $\cO_\star$-module of its Pontryagin dual $H^1_f(L,A_\star)^\vee$. 

\begin{corollary} \label{upsilon-onto-coro}
$\corank_{\cO_\star}\!H^1_f(K,A_\star)=r$.
\end{corollary}

\begin{proof} To begin with, $\corank_{\cO_\star}\!H^1_f(K,A_\star)=\corank_{\cO_\star}\!H^1_f(K,A_\star{)}_\divv$. On the other hand, it follows from \eqref{T-V-A-eq2} and Proposition \ref{upsilon-onto-prop} that there is an isomorphism of $\cO_\star$-modules
\[ H^1_f(K,A_\star{)}_\divv\simeq H^1_f(K,V_\star)\big/\uH^1_f(K,T_\star)\simeq(F_\star/\cO_\star)^r, \]
whence the desired equality. \end{proof}

\subsubsection{Dual bases}

In light of Proposition \ref{upsilon-onto-prop}, taking Pontryagin duals of the map $\Upsilon_\star$ in \eqref{V-A-div-eq} gives an injection of $\cO_\star$-modules
\begin{equation} \label{Upsilon-dual-eq}
\Upsilon_\star^\vee:H^1_f(K,A_\star{)}^\vee_\divv\longmono H^1_f(K,V_\star)^\vee. 
\end{equation}
Again, $\Upsilon_p^\vee=\bigoplus_{\p|p}\Upsilon_\p^\vee$. 

Now pick an $F_p$-basis $\tilde{\mathscr{B}}=\{x_1,\dots,x_r\}$ of $H^1_f(K,V_p)$; it generates a rank $r$ free $\cO_p$-submodule $\Lambda_{\tilde{\mathscr{B}}}$ of $H^1_f(K,V_p)$. Fix an isomorphism
\[ \varphi_{\tilde{\mathscr{B}}}:\Lambda_{\tilde{\mathscr{B}}}\overset\simeq\longrightarrow\uH^1_f(K,T_p) \]
of $\cO_p$-modules and set $\xi_i\defeq\varphi_{\tilde{\mathscr{B}}}(x_i)$ for $i=1,\dots,r$. Thus, $\{\xi_1,\dots,\xi_r\}$ is an $\cO_p$-basis of $\uH^1_f(K,T_p)$; it is also an $F_p$-basis of $H^1_f(K,V_p)$. The elements $\xi_1,\dots,\xi_n$ give rise to the dual basis $\{\xi_1^*,\dots,\xi_n^*\}$ of the $\cO_p$-linear dual of $\uH^1_f(K,T_p)$ by the recipe $\xi_i^*(\xi_j)\defeq\delta_{ij}$, where $\delta_{ij}$ is the Kronecker delta; this gives also a basis, which will be denoted in the same way, of the $F_p$-linear dual of $H^1_f(K,V_p)$. Composing the $F_p$-linear maps $\xi_i^*$ with the canonical projection $F_p\twoheadrightarrow F_p/\cO_p$, we obtain elements $\xi_i^\vee$ of the Pontryagin dual $H^1_f(K,V_p)^\vee$.

\begin{lemma} \label{pontryagin-independent-lemma}
The elements $\xi_1^\vee,\dots,\xi_r^\vee$ are linearly independent over $\cO_p$.
\end{lemma}

\begin{proof} In light of splitting \eqref{M-dual-splitting-eq} for $M=V$, we may
\begin{itemize}
\item assume that $\xi_1^\vee,\dots,\xi_r^\vee$ belong to $H^1_f(K,V_\p)^\vee$ for a prime $\p$ of $F$ above $p$,
\item replace $\cO_p$ with $\cO_\p$.
\end{itemize}
Now the lemma follows from Lemma \ref{pontryagin-indep-lemma} with, in the notation of \S \ref{vector-duals-subsec}, $\KK=F_\p$, $\mathscr O=\cO_\p$, $v_i=\xi_i$, $V=H^1_f(K,V_\p)$, $T_{\mathscr B}=\uH^1_f(K,T_\p)$. \end{proof}

Denote by $\Xi_{\tilde{\mathscr{B}}}$ the $\cO_p$-submodule of $H^1_f(K,V_p)^\vee$ spanned by $\xi_1^\vee,\dots,\xi_r^\vee$; by Lemma \ref{pontryagin-independent-lemma}, the $\cO_p$-rank of $\Xi_{\tilde{\mathscr{B}}}$ is $r$.

\begin{lemma} \label{equality-lattices-dual-lemma}
If $\tilde{\mathscr B}$ and $\tilde{\mathscr B}'$ are $F_p$-bases of $H^1_f(K,V_p)$, then $\Xi_{\tilde{\mathscr{B}}}=\Xi_{\tilde{\mathscr{B}}'}$.
\end{lemma}

\begin{proof} With self-explaining notation, $\{\xi_1^*,\dots,\xi_r^*\}$ and $\bigl\{(\xi'_1)^*,\dots,(\xi'_r)^*\bigr\}$ are bases of the $\cO_p$-linear dual of $\uH^1_f(K,T_p)$. In particular, for every $i\in\{1,\dots,r\}$ there are $a_{i,1},\dots,a_{i,r}\in\cO_p$ such that $(\xi_i')^*=a_{i,1}\xi_1^*+\dots+a_{i,r}\xi_r^*$. On the other hand, the $\cO_p$-linearity of the projection $F_p\twoheadrightarrow F_p/\cO_p$ ensures that
\[ (\xi_i')^\vee=a_{i,1}\xi_1^\vee+\dots+a_{i,r}\xi_r^\vee, \]
so $\Xi_{\tilde{\mathscr{B}}'}\subset\Xi_{\tilde{\mathscr{B}}}$. Analogously, the inclusion $\Xi_{\tilde{\mathscr{B}}}\subset\Xi_{\tilde{\mathscr{B}}'}$ holds as well. \end{proof}

In light of Lemma \ref{equality-lattices-dual-lemma}, from here on we set
\[ \Xi_p\defeq\Xi_{\tilde{\mathscr{B}}}\subset H^1_f(K,V_p)^\vee \]
for any $F_p$-basis $\tilde{\mathscr{B}}$ of $H^1_f(K,V_p)$. Now recall the injection of $\cO_p$-modules $\Upsilon_p^\vee$ from \eqref{Upsilon-dual-eq}.

\begin{proposition} \label{image-upsilon-dual-prop}
The image of $\Upsilon_p^\vee$ is $\Xi_p$.
\end{proposition}

\begin{proof} Since $\Upsilon_p^\vee=\bigoplus_{\p|p}\Upsilon_\p^\vee$, as in the proof of Lemma \ref{pontryagin-independent-lemma} we may work with a fixed prime $\p\,|\,p$ in place of $p$. Taking Pontryagin duals in the short exact sequence 
\[ 0\longrightarrow\uH^1_f(K,T_\p)\overset{\iota_\p}\longrightarrow H^1_f(K,V_\p)\overset{\Upsilon_\p}\longrightarrow H^1_f(K,A_\p{)}_\divv\longrightarrow0, \]
where $\iota_\p$ is simply inclusion, gives a short exact sequence of $\cO_\p$-modules
\[ 0\longrightarrow H^1_f(K,A_\p{)}^\vee_\divv\xrightarrow{\Upsilon_\p^\vee} H^1_f(K,V_\p)^\vee\overset{\iota_\p^\vee}\longrightarrow\uH^1_f(K,T_\p)^\vee\longrightarrow0. \]
Therefore, if $\Xi_\p$ denotes the analogue of $\Xi_p$ at $\p$, then we need to show that $\ker(\iota_\p^\vee)=\Xi_\p$. Let $\varphi\in\ker(\iota_\p^\vee)$. This means that $\varphi\in H^1_f(K,V_\p)^\vee$ and ${\varphi|}_{\uH^1_f(K,T_\p)}=0$, which is equivalent to $\varphi\in\Xi_\p$ by Proposition \ref{pontryagin10} with, in the notation of \S \ref{vector-duals-subsec}, $\KK=F_\p$, $\mathscr O=\cO_\p$, $V=H^1_f(K,V_\p)$, $T_{\mathscr B}=\uH^1_f(K,T_\p)$, $\Xi_{\mathscr B}=\Xi_\p$. \end{proof}

\begin{corollary} \label{image-upsilon-dual-coro}
The set $\{\xi_1^\vee,\dots,\xi_r^\vee\}$ is an $\cO_p$-basis of the image of $\Upsilon_p^\vee$.
\end{corollary}

\begin{proof} By definition of $\Xi_p$, this follows from Proposition \ref{image-upsilon-dual-prop}. \end{proof}

In light of the fact that $\Upsilon_p^\vee$ is the dual of the map induced by the canonical projection $V_p\twoheadrightarrow A_p$, from now on we shall usually identify $H^1_f(K,A_p{)}^\vee_\divv$ with its image under $\Upsilon_p^\vee$; in particular, we shall regard $\{\xi_1^\vee,\dots,\xi_r^\vee\}$ as an $\cO_p$-basis of $H^1_f(K,A_p{)}^\vee_\divv$.
 
\subsection{Tamagawa ideals of $\MM$} \label{sectam}

For each prime $\p$ of $F$ above $p$, the completion $\cO_\p$ of $\cO_F$ at $\p$ is a PID (actually, a DVR), so we can apply the definition of determinants in \S \ref{pid-prod-subsubsec} to the case where $R=\cO_p=\prod_{\p\mid p}\cO_\p$. Therefore, if $T$ is a finite $\cO_p$-module and $T=\bigoplus_{\p\mid p}T_\p$ is its splitting as a product of $\cO_\p$-modules, then 
\[ \mathcal{I}_{\cO_p}(T)=\prod_{\p\mid p}\mathcal I_{\cO_\p}(T_\p)\subset F_p. \] 
Since we shall essentially work with $\cO_p$-modules only, from now on we simply set 
\begin{equation} \label{notational-convention-eq}
\mathcal{I}(T)\defeq\mathcal{I}_{\cO_p}(T),\quad\mathcal{I}^{-1}(T)\defeq\mathcal{I}^{-1}_{\cO_p}(T) 
\end{equation}
for every finite $\cO_p$-modules $T$, unless confusion may arise.

\subsubsection{Finite primes $\ell\neq p$} 

Let $\ell\neq p$ be a prime number and for any $G_{\Q_\ell}$-module $M$ set
\[ H^1_\mathrm{unr}(\Q_\ell,M)\defeq H^1\bigl(\Gal(\Q_\ell^\mathrm{unr}/\Q_\ell),M^{I_\ell}\bigr)=\ker\bigl(H^1(\Q_\ell,M)\longrightarrow H^1(I_\ell,M)\bigr), \]
where the equality on the right is a consequence of the inflation-restriction exact sequence (see, \emph{e.g.}, \cite[Proposition B.2.5]{Rubin-ES}). In particular, observe that $H^1_f(\Q_\ell, A_p )\subset H^1_\mathrm{unr}(\Q_\ell, A_p )$ and the inclusion has finite index.  By definition, $H^1_f(\Q_\ell, V_p)=H^1_\mathrm{unr}(\Q_\ell, V_p)$ and there is an exact sequence 
\[ 0\longrightarrow H^0_f(\Q_\ell, V_p)\longrightarrow V_p^{I_\ell}\xrightarrow{\Frob_\ell-1}V_p^{I_\ell}
\longrightarrow H^1_f(\Q_\ell, V_p)\longrightarrow0. \]
Thus, we obtain a chain of two isomorphisms 
\[ \vartheta_\ell:\Det_{F_p}\bigl(H^0_f(\Q_\ell,V_p)\bigr)\cdot\Det_{F_p}^{-1}\bigl(H^1_f(\Q_\ell,V_p)\bigr)\overset\simeq\longrightarrow\Det_{F_p}^{-1}\bigl(H^1_f(\Q_\ell,V_p)\bigr)\overset\simeq\longrightarrow(F_p,0), \]
the former being a consequence of Lemma \ref{lemmaram}.

\begin{definition} \label{deftamell}
The \emph{$p$-part of the Tamagawa ideal of $\MM$ at $\ell$} is 
\begin{equation} \label{tam-def-eq}
\Tam_\ell^{(p)}(\MM)\defeq\vartheta_\ell\Bigl(\Det_{\cO_p}^{-1}\bigl(H^1_f(\Q_\ell,T_p)\bigr)\!\Bigr). 
\end{equation}
\end{definition}

Notice that, in light of the definition given in \eqref{det-splitting-eq}, there is a splitting
\begin{equation} \label{tam-splitting-eq}
\Tam_\ell^{(p)}(\MM)=\prod_{\p|p}\Tam_\ell^{(\p)}(\MM),
\end{equation}
where the $\p$ varies over all primes of $F$ above $p$ and $\Tam_\ell^{(\p)}(\MM)$, the \emph{$\p$-part of the Tamagawa ideal of $\MM$ at $\ell$}, is defined as in \eqref{tam-def-eq} by replacing $p$ with $\p$.

To lighten our notation, set $\mathscr G_\ell\defeq\Gal(\Q_\ell^\mathrm{unr}/\Q_\ell)\simeq\hat\Z$. The auxiliary result below will be used in the proof of Proposition \ref{Tamv}.

\begin{lemma} \label{H^2-lemma}
$H^2\bigl(\mathscr G_\ell,T_p^{I_\ell}\bigr)=H^2\bigl(\mathscr G_\ell,V_p^{I_\ell}\bigr)=0$.
\end{lemma}

\begin{proof} The $\mathscr G_\ell$-module $V_p^{I_\ell}$, which is a vector space over a field of characteristic $0$, is divisible, so $H^2\bigl(\mathscr G_\ell,V_p^{I_\ell}\bigr)=0$ by \cite[Ch. XIII, Proposition 2]{serre-local}. As for the other vanishing, note that $T_p^{I_\ell}=\varprojlim_n(T_p/p^nT_p)^{I_\ell}$ and $T_p/p^nT_p$ is finite for all $n\in\N$. On the other hand, \cite[Lemma 1.3.2, (i)]{Rubin-ES} gives
\[ H^1\bigl(\mathscr G_\ell,(T_p/p^nT_p)^{I_\ell}\bigr)\simeq(T_p/p^nT_p)^{I_\ell}\big/(\Frob_\ell-1)(T_p/p^nT_p)^{I_\ell}, \]
so $H^1\bigl(\mathscr G_\ell,(T_p/p^nT_p)^{I_\ell}\bigr)$ is finite for all $n\in\N$. By \cite[Corollary 2.2]{Tate}, it follows that there is an isomorphism
\begin{equation} \label{H^2-eq}
H^2\bigl(\mathscr G_\ell,T_p^{I_\ell}\bigr)\simeq\varprojlim_nH^2\bigl(\mathscr G_\ell,(T_p/p^nT_p)^{I_\ell}\bigr).
\end{equation}
Finally, since $(T_p/p^nT_p)^{I_\ell}$ is finite, hence torsion, \cite[Ch. XIII, Proposition 2]{serre-local} ensures that
\begin{equation} \label{H^2-eq2}
H^2\bigl(\mathscr G_\ell,(T_p/p^nT_p)^{I_\ell}\bigr)=0
\end{equation} 
for all $n\in\N$. Combining \eqref{H^2-eq} and \eqref{H^2-eq2}, we conclude that $H^2\bigl(\mathscr G_\ell,T_p^{I_\ell}\bigr)=0$. \end{proof}

Recall that $H^1_\mathrm{unr}(\Q_\ell,T_p)=H^1\bigl(\mathscr G_\ell,T_p ^{I_\ell}\bigr)$ and that Assumption \ref{motass} is in force. 

\begin{proposition} \label{unr-finite-prop}
$\#H^1_\mathrm{unr}(\Q_\ell,T_p)<\infty$.
\end{proposition}

\begin{proof} The kernel of the natural map $H^1\bigl(\mathscr G_\ell,T_p^{I_\ell}\bigr)\rightarrow H^1\bigl(\mathscr G_\ell,V_p^{I_\ell}\bigr)$ is isomorphic to a quotient of $H^0\bigl(\mathscr G_\ell,V_p^{I_\ell}/T_p^{I_\ell}\bigr)$, so it is torsion over $\cO_p$. By Lemma \ref{lemmaram}, $H^1\bigl(\mathscr G_\ell,V_p^{I_\ell}\bigr)=0$, and then $H^1\bigl(\mathscr G_\ell,T_p^{I_\ell}\bigr)$ is $\cO_p$-torsion. Since $H^1\bigl(\mathscr G_\ell,T_p^{I_\ell}\bigr)=\prod_{\p\mid p}H^1\bigl(\mathscr G_\ell,T_\p^{I_\ell}\bigr)$ as $\cO_p$-modules, we deduce that $H^1\bigl(\mathscr G_\ell,T_\p^{I_\ell}\bigr)$ is torsion over $\cO_\p$ for all $\p\,|\,p$. Furthermore, since $T_\p$ is finitely generated over $\cO_\p$, it follows from \cite[Proposition B.2.7]{Rubin-ES} that $H^1(G_{\Q_\ell},T_\p)$ is finitely generated over $\cO_\p$ for all $\p\,|\,p$, and then the same is true of its $\cO_\p$-submodule $H^1\bigl(\mathscr G_\ell,T_\p ^{I_\ell}\bigr)$. We conclude that $H^1\bigl(\mathscr G_\ell,T_\p ^{I_\ell}\bigr)$ is finite for all $\p\,|\,p$, which implies that $H^1_\mathrm{unr}(\Q_\ell,T_p)=H^1\bigl(\mathscr G_\ell,T_p ^{I_\ell}\bigr)$ is finite. \end{proof}

Let $S$ be the finite set of places of $\Q$ that was fixed in \eqref{S-eq} and recall the notational convention from \eqref{notational-convention-eq}. The next result is a slight refinement, in our modular context, of \cite[Ch. I, Proposition 4.2.2]{FPR}. 

\begin{proposition} \label{Tamv}
\begin{enumerate}
\item $\Tam_\ell^{(p)}(\MM)\simeq\mathcal{I}\bigl(H^1_\mathrm{unr}(\Q_\ell,A_p)\bigr)$.
\item If $\ell\not\in S$, then $\Tam_\ell^{(p)}(\MM)\simeq\cO_p$. 
\end{enumerate}  
\end{proposition}

\begin{proof} The inflation-restriction exact sequence yields a commutative diagram with exact rows
\begin{equation} \label{tamv-diagram-eq}
\xymatrix{
0\ar[r]& H^1\bigl(\mathscr G_\ell,T_p ^{I_\ell}\bigr)\ar[r]\ar[d]&H^1(G_{\Q_\ell},T_p)\ar[r]\ar[d]&H^1(I_\ell,T_p)^{G_{\Q_\ell}}\ar[r]\ar[d]&0\\
0\ar[r]& H^1\bigl(\mathscr G_\ell, V_p^{I_\ell}\bigr)\ar[r]&H^1(G_{\Q_\ell},V_p)\ar[r]&H^1(I_\ell,V_p)^{G_{\Q_\ell}}\ar[r]&0
}
\end{equation} 
in which the surjectivity of the right non-trivial arrows on both rows follows from Lemma \ref{H^2-lemma}. As a consequence of Lemma \ref{lemmaram}, $H^1_f(\Q_\ell,T_p)$ is the kernel of the middle vertical map in diagram \eqref{tamv-diagram-eq}. By \cite[Proposition 2.3]{Tate}, the kernels of the middle and the right vertical arrows in \eqref{tamv-diagram-eq} are $H^1(G_{\Q_\ell},T_p{)}_\mathrm{tors}$ and $H^1(I_\ell,T_p)^{\!G_{\Q_\ell}}_\mathrm{tors}$, respectively. By applying the snake lemma to \eqref{tamv-diagram-eq}, we get a short exact sequence 
\begin{equation} \label{exseq}
0\longrightarrow H^1_\mathrm{unr}(\Q_\ell,T_p)\longrightarrow H^1_f(\Q_\ell,T_p)\longrightarrow H^1(I_\ell,T_p)^{\!G_{\Q_\ell}}_\mathrm{tors}\longrightarrow0.
\end{equation}
The exact sequence 
\[ 0\longrightarrow H^0\bigl(\mathscr G_\ell,T_p ^{I_\ell}\bigr)\longrightarrow T_p ^{I_\ell}\xrightarrow{\Frob_\ell-1}T_p ^{I_\ell}\longrightarrow H^1_\mathrm{unr}(\Q_\ell,T_p)\longrightarrow 0 \]
shows that $\Det_{\cO_p}\bigl(H^1_\mathrm{unr}(\Q_\ell,T_p)\bigr)\simeq\Det_{\cO_p}^{-1}\bigl(H^0\bigl(\mathscr G_\ell,T_p ^{I_\ell}\bigr)\bigr)$. Now $H^0\bigl(\mathscr G_\ell,T_p ^{I_\ell}\bigr)\subset H^0(\Q_\ell,V_p)$, and $H^0(\Q_\ell,V_p)=0$ by Lemma \ref{lemmaram}, so $\Det_{\cO_p}\bigl(H^1_\mathrm{unr}(\Q_\ell,T_p)\bigr)\simeq\cO_p$ and exact sequence \eqref{exseq} yields an isomorphism
\begin{equation}\label{torseq}
\Det_{\cO_p}^{-1}\bigl(H^1_f(\Q_\ell,T_p)\bigr)\simeq\Det_{\cO_p}\Bigl(H^1(I_\ell,T_p{)}_\mathrm{tors}^{\!G_{\Q_\ell}}\Bigr).
\end{equation}
Now set $\mathcal{W}\defeq A_p ^{I_\ell}\big/(A_p ^{I_\ell}{)}_\divv$. There is an exact sequence of $\cO_p$-modules 
\[ 0\longrightarrow\mathcal{W}^{\Frob_\ell=1}\longrightarrow\mathcal{W}\xrightarrow{\Frob_\ell-1}\mathcal{W}\longrightarrow\mathcal{W}\big/(\Frob_\ell-1)\mathcal{W}\longrightarrow0 \]
showing that 
\begin{equation} \label{eqW}
\Det_{\cO_p}\bigl(\mathcal{W}^{\Frob_\ell=1}\bigr)\simeq\Det_{\cO_p}\bigl(\mathcal{W}/(\Frob_\ell-1)\mathcal{W}\bigr).
\end{equation} 
Furthermore, by \cite[Lemma 1.3.5, (iii)]{Rubin-ES}, there are isomorphisms of $\cO_p$-modules 
\[ H^1_\mathrm{unr}(\Q_\ell,A_p)\big/H^1_f(\Q_\ell,A_p)\simeq \mathcal{W}\big/(\Frob_\ell-1)\mathcal{W} \]
and
\[ H^1_f(\Q_\ell,T_p)\big/H^1_\mathrm{unr}(\Q_\ell,T_p)\simeq \mathcal{W}^{\Frob_\ell=1}. \] 
By Lemma \ref{lemmaram}, $H^1_f(\Q_\ell,V_p)=0$, so $H^1_f(\Q_\ell,A_p)=0$ and $H^1_\mathrm{unr}(\Q_\ell,A_p)\simeq\mathcal{W}/(\Frob_\ell-1)\mathcal{W}$. Combining Definition \ref{deftamell} with \eqref{exseq}, \eqref{torseq} and 
\eqref{eqW}, we obtain an isomorphism
\[ \Tam_\ell^{(p)}(\MM)\simeq\Det_{\cO_p}\bigl(H^1_\mathrm{unr}(\Q_\ell,A_p)\bigr), \]
which, in light of the definition of $\mathcal{I}\bigl(H^1_\mathrm{unr}(\Q_\ell,A_p)\bigr)$, concludes the proof. 
\end{proof}

\subsubsection{The prime $p$}  \label{prime-p-subsec}

Now we consider the case $\ell=p$. There is an exact sequence of $F_p$-modules
\begin{equation} \label{l=p-eq}
0\longrightarrow H^0_f(\Q_p,V_p)\longrightarrow \mathbf{D}_\crys(V_p)\xrightarrow{(\varphi-1,\mathrm{pr})}\mathbf{D}_\crys(V_p)\oplus t(\MM)_p\longrightarrow H^1_f(\Q_p,V_p)\longrightarrow0. 
\end{equation}
By \cite[Corollary 3.8.4]{BK}, $H^0_f(\Q_p,V_p)=H^0(\Q_p,V_p)$, so $H^0_f(\Q_p,V_p)=0$ by Lemma \ref{lemmap}. Taking determinants in \eqref{l=p-eq}, we obtain an isomorphism 
\[\vartheta_p:\Det^{-1}_{F_p}\bigl(H^1_f(\Q_p, V_p)\bigr)\overset\simeq\longrightarrow\Det_{F_p}^{-1}\bigl(t_p( V_p)\bigr). \]
Define the $\cO_p$-submodule 
\[ \Lambda_p\defeq\vartheta_p\Bigl(\Det^{-1}_{\cO_p}\bigl(H^1_f(\Q_p,T_p)\bigr)\!\Bigr) \]
of $\Det_{F_p}^{-1}\bigl(t_p( V_p)\bigr)$. Recall that $t_p( V_p)$ is a free $F_p$-module of rank $1$ and fix an $F_p$-generator $\omega$ of $t(\MM)_p$. Then $\omega$ is a generator of the free $F_p$-module $\Det_{F_p}\bigl(t(\MM)_p\bigr)$ of rank $1$; moreover, $\Det_{F_p}^{-1}\bigl(t(\MM)_p\bigr)$ is free of rank $1$ over $F_p$, and we let $\omega^{-1}$ denote the generator of $\Det_{F_p}^{-1}\bigl(t(\MM)_p\bigr)$ corresponding to $\omega$, which is characterized by $\omega^{-1}(\omega)=1$. Now $\Lambda_p$ is an $\cO_p$-submodule of free $F_p$-module $\Det_{F_p}^{-1}\bigl(t(\MM)_p\bigr)=F_p\cdot\omega^{-1}$, so there exists an $\cO_p$-ideal $\Tam_{p,\omega}(A_p)$ such that 
\begin{equation} \label{Tam-A-eq}
\Lambda_p=\Tam_{p,\omega}(A_p)\cdot\omega^{-1}.
\end{equation}
With $T_\B$ as in \eqref{T_B-eq} and $\phi_\infty$ the involution from \S \ref{Betti-subsubsec}, we define $T_\B ^+\defeq T_\B^{\phi_\infty=1}$ and $T_p^+\defeq H^0(\R,T_p)$. There are comparison isomorphisms 
\begin{equation} \label{comparison-isomorphisms-eq}
\Comp_{\B,\et}:T_\B^+\otimes_\Q\Q_p\overset\simeq\longrightarrow T_p ^+\otimes_{\Z_p}\Q_p,\quad\Comp_{\B,\dR}:V_\B^+\otimes_\Q\Q_p\overset\simeq\longrightarrow t(V_p). 
\end{equation}
We deduce from \eqref{B-et-eq} that there is an induced isomorphism of $\cO_p$-modules
\[ \Comp_{\B,\et}:T_\B^+\otimes_{\cO_F}\cO_p\overset\simeq\longrightarrow T_p^+. \]
Choose $\delta_f\in T_\B^+\smallsetminus\{0\}$ and set $\omega_{\delta_f}\defeq\Comp_{\B,\dR}(\delta_f)\in t(V_p)$. 
%In what follows, $\omega_{\gamma_f}^*$ is defined as in \S \ref{de-Rham-subsec}.

\begin{definition} \label{Tamp}
The \emph{$p$-part of the Tamagawa ideal of $\MM$ at $p$} is 
\[ \Tam_{p}^{(p)}(\MM)\defeq\Tam_{p,\omega_{\delta_f}}(A_p), \]
where $\Tam_{p,\omega_{\delta_f}}(A_p)$ is defined as in \eqref{Tam-A-eq}.
\end{definition}

Now assume that $\Comp_{\B,\et}(\delta_f)$ is an $\cO_p$-generator of $T_p^+$ (later on, this condition will be satisfied for a suitable choice of $p$). Then, since $V_p$ is crystalline, we actually have
\begin{equation} \label{Tam-p-eq}
\Tam_{p}^{(p)}(\MM)=\cO_p. 
\end{equation}
This follows from the proof in \cite[Theorem 4.2.2]{BB} of Conjecture $C_{EP,\Q_p}(V)$ from \cite[\S4.5.4]{FPR} (\emph{cf.} also \cite[Proposition II.2]{Berger-preprint}, and \cite[Proposition 4.2.5]{PR} in the ordinary case, which will be the setting of interest for us in subsequent sections).

\subsubsection{The archimedean prime} 

In the archimedean case, we introduce (the $p$-part of) the Tamagawa ideal of $\MM$ as follows. 

\begin{definition} \label{deftaminfty}
The \emph{$p$-part of the Tamagawa ideal of $\MM$ at $\infty$} is 
\[ \Tam_\infty^{(p)}(\MM)\defeq\Det_{\cO_p}^{-1}\bigl(H^1_f(\R,T_p)\bigr). \]
\end{definition}

This definition completes the list of the $p$-parts of Tamagawa ideals of $\MM$. When $p$ is odd, it can be checked that
\begin{equation} \label{Tam-infty-eq}
\Tam_\infty^{(p)}(\MM)=\cO_p
\end{equation}
(see, \emph{e.g.}, \cite[p. 708]{DFG}).

\subsection{Compact cohomology} 

Our present goal is to calculate $\Det_{\cO_p}\bigl(\RR\Gamma_c(G_S,T_p)\bigr)$ in the sense of \eqref{complex-determinant-eq}, where $\RR\Gamma_c(G_S,T_p)$ is the compact complex from \S \ref{compactcone}. 

\subsubsection{A vanishing lemma}

We begin with a basic vanishing result for cohomology with compact support.

\begin{lemma} \label{H0c} 
$H^0_c(G_S,T_p)=0$.
\end{lemma}

\begin{proof} By \eqref{triangle}, the group $H^0_c(G_S,T_p)$ consists of the elements of $H^0_f(\Q,T_p)$ whose image in $H^0_f(\Q_v,T_p)$ is zero for all $v\in S$. Now $H^0_f(\Q, T_p )$ injects into 
$H^0_f(\Q_p, T_p)$ and $H^0_f(\Q_p,T_p)$ injects into $H^0_f(\Q_p, V_p)$; since $p\in S$ and 
$H^0_f(\Q_p, V_p)=0$ by Lemma \ref{lemmap}, we conclude that $H^0_f(\Q, T_p )=0$, as was 
completing the proof.\end{proof}

\subsubsection{Computing $\Det_{\cO_p}\bigl(\RR\Gamma_c(G_S, T_p)\bigr)$}

The next result computes $\Det_{\cO_p}\bigl(\RR\Gamma_c(G_S, T_p)\bigr)$. 

\begin{proposition} \label{RGamma_c} 
There is a canonical isomorphism 
\[ \begin{split}
   \Det_{\cO_p}\bigl(\RR\Gamma_c(G_S,T_p)\bigr)\simeq&
\Det_{\cO_p}^{-1}\bigl(H^1_f(\Q, T_p)\bigr)\cdot\Det_{\cO_p}\bigl(H^1_f(\Q, A_p)^\vee\bigr)\cdot\Det^{-1}_{\cO_p}\bigl(H^0(G_S,A_p)^\vee\bigr)\\
&\cdot\prod_{v\in S}\Det_{\cO_p}\bigl(H^1_f(\Q_v, T_p)\bigr)\cdot\Det_{\cO_p}^{-1}(T_p ^+).
   \end{split} \]
\end{proposition}

\begin{proof} Fix an integer $n\geq1$ and write $A_p [p^n]$ for the $p^n$-torsion subgroup of $A_p$. Since $\MM$ is self-dual, there is a canonical isomorphism $A_p[p^n]\simeq\Hom_{\Z_p}\bigl(A_p [p^n],(\Q_p/\Z_p)(1)\bigr)$, so the Poitou--Tate exact sequence (see, \emph{e.g.}, \cite[\S 5.1.6]{Nek-Selmer}) gives a long exact sequence 
\[ \begin{split}
   H^1\bigl(G_S,A_p[p^n]\bigr)&\longrightarrow\bigoplus_{v\in S}H^1\bigl(\Q_v,A_p[p^n]\bigr)\longrightarrow H^1\bigl(G_S, A_p [p^n]\bigr)^\vee\\
   & \longrightarrow H^2\bigl(G_S, A_p [p^n]\bigr)\longrightarrow\bigoplus_{v\in S}H^2\bigl(\Q_v, A_p [p^n]\bigr)\longrightarrow H^0\bigl(G_S, A_p [p^n]\bigr)^\vee\longrightarrow 0.
   \end{split} \]
Since these groups are finite, passing to inverse limits on $n$ produces an exact sequence 
\[ \begin{split}
   H^1(G_S, T_p )&\longrightarrow\bigoplus_{v\in S}H^1(\Q_v,T_p)\longrightarrow 
H^1(G_S, A_p )^\vee\\
&\longrightarrow H^2(G_S, T_p )\longrightarrow
\bigoplus_{v\in S}H^2(\Q_v, T_p )\longrightarrow H^0(G_S, A_p )^\vee\longrightarrow 0.
   \end{split} \]
Using Lemma \ref{lemma1.15}, we get an exact sequence
\begin{equation} \label{pt-eq1}
 \begin{split}
    0\longrightarrow H^1_f(\Q,T_p)&\longrightarrow H^1(G_S,T_p)\longrightarrow 
\bigoplus_{v\in S}H^1_s(\Q_v, T_p )\longrightarrow H^1_f(\Q,A_p)^\vee\\
& \longrightarrow H^2(G_S, T_p )\longrightarrow
\bigoplus_{v\in S}H^2(\Q_v, T_p )\longrightarrow H^0(G_S, A_p )^\vee\longrightarrow 0.
   \end{split}
\end{equation}
On the other hand, using Lemmas \ref{lemmaram}, \ref{lemmap} and \ref{H0c}, the definition of cohomology with compact support (\emph{cf.} \S \ref{compactcone}) yields an exact sequence
\begin{equation} \label{pt-eq2}
\begin{split}
   0&\longrightarrow H^0(\R,T_p)\longrightarrow H^1_c(G_S,T_p)\longrightarrow H^1(G_S,T_p)\longrightarrow\bigoplus_{s\in S}H^1(\Q_v,T_p)\\
  &\longrightarrow H^2_c(G_S, T_p )\longrightarrow H^2(G_S, T_p )\longrightarrow
\bigoplus_{s\in S}H^2(\Q_v, T_p )\longrightarrow H^3_c(G_S, T_p )\longrightarrow0.
\end{split} 
\end{equation}
Taking $\Det_{\cO_p}$ of \eqref{pt-eq1} and \eqref{pt-eq2}, and using the multiplicativity of determinants, gives the desired result. \end{proof}

\subsection{A reformulation of $p$-TNC} \label{reformulation-subsec}

We want to reformulate the $p$-part of Conjecture \ref{TNC} in a convenient way. %Assume that $p\nmid 2N(k-2)!$. 

\subsubsection{$p$-torsion of $\MM$} \label{p-torsion-M-subsubsec}

Let $H^1(\Q, T_p{)}_\mathrm{tors}$ be the torsion submodule of $H^1(\Q, T_p)$. 

\begin{definition} \label{deftors}
The \emph{$p$-torsion part of $\MM$} is 
\[ \Tors_p(\MM)\defeq\mathcal{I}^{-1}\bigl(H^0(G_S,A_p)^\vee\bigr)\cdot\mathcal{I}^{-1}\bigl(H^1(\Q, T_p{)}_\mathrm{tors}\bigr). \]
\end{definition}

The $\cO_p$-module $\Tors_p(\MM)$ will play a role in Theorem \ref{motivesthm}.

\subsubsection{Some linear algebra of lattices} \label{linear-algebra-lattices-subsubsec}

Let $\mathscr B=\{t_1,\dots,t_r\}$ be a basis of $H^1_\mot(\Q,\MM)$ as an $F$-vector space; this is also a basis of $H^1_\mot(\Q,\MM)_p$ over $F_p$. Recall that $\reg_p$ is a shorthand for $\reg_{\Q,p}$ and for each $i\in\{1,\dots,r\}$ put $x_i\defeq\reg_p(t_i)$. We are assuming the $p$-part of Conjecture \ref{regpconj} over $\Q$, so $\tilde{\mathscr B}\defeq\{x_1,\dots,x_r\}$ is a basis of $H^1_f(\Q,V_p)$ as an $F_p$-module. As in \S \ref{cohomology-subsec}, write $\Lambda_{\tilde{\mathscr{B}}}$ for the free $\cO_p$-submodule of $H^1_f(\Q,V_p)$ of rank $r$ generated by $\tilde{\mathscr B}$. Let $\{\xi_1,\dots,\xi_r\}$ be an $\cO_p$-basis of $\uH^1_f(\Q,T_p)$ and let $\mathtt{A}_{\tilde{\mathscr B}}\in\GL_r(F_p)$ be the transition matrix from the $F_p$-basis $\tilde{\mathscr B}$ to the $F_p$-basis $\{\xi_1,\dots,\xi_r\}$ of $H^1_f(\Q,V_p)$. Let $\xi_1^\vee,\dots,\xi_r^\vee\in H^1_f(\Q,V_p)^\vee$ be the dual elements from \S \ref{cohomology-subsec} (\emph{cf.} also \S \ref{vector-duals-subsec}) and define $x_1^\vee,\dots,x_r^\vee$ in an analogous way. 

\begin{remark} \label{B-rem}
If $\Lambda_{\tilde{\mathscr B}}=\uH^1_f(\Q,T_p)$, then $\mathtt{A}_{\tilde{\mathscr B}}\in\GL_r(\cO_p)$. In general, $\det(\mathtt{A}_{\tilde{\mathscr B}})$ depends on the choice of an $\cO_p$-basis of $\uH^1_f(\Q,T_p)$ only up to multiplication by elements of $\cO_p^\times$, which shows that the principal fractional $\cO_p$-ideal $\bigl(\det(\mathtt{A}_{\tilde{\mathscr B}})\bigr)$ is independent of the choice of a basis of $\uH^1_f(\Q,T_p)$ over $\cO_p$. On the other hand, keeping $\{\xi_1,\dots,\xi_r\}$ fixed, if $\mathscr B'$ is another $F$-basis of $H^1_\mot(\Q,\MM)$, then $\det(\mathtt{A}_{\tilde{\mathscr B}})$ and $\det(\mathtt{A}_{\tilde{\mathscr B}'})$ differ by multiplication by the determinant of the transition matrix from $\mathscr B$ to $\mathscr B'$.
\end{remark}

\subsubsection{Reformulation of $p$-TNC} \label{reformulation-subsubsec}

Recall the comparison isomorphism $\Comp_{\B,\et}$ from \eqref{comparison-isomorphisms-eq}. In particular, $\Comp_{\B,\et}^{-1}(T_p ^+)$ is an $\cO_p$-submodule of the free $F_p$-module 
$V_\B^+\otimes_\Q\Q_p$ of rank $1$. As in \S \ref{prime-p-subsec}, choose $\gamma_f\in T_\B^+\smallsetminus\{0\}$ and let $\Omega_\infty\defeq\Omega_\infty^{(\gamma_f)}\in F_\infty^\times$ be the period from \S \ref{emb-subsubsec}. Let us consider the $\cO_p$-submodule $\Lambda_{\gamma_f}$ of $T_p^+$ generated by $\Comp_{\B,\et}(\gamma_f)$ and set $\mathcal{I}_p(\gamma_f)\defeq\mathcal{I}\bigl(T_p^+/\Lambda_{\gamma_f}\bigr)$. Let us define the period
\begin{equation} \label{Omega-eq}
\Omega_\MM\defeq\frac{\Omega_\infty}{(2\pi i)^{k/2}}\in(F\otimes_\Q\C)^\times.
\end{equation}
From here on, in order to simplify some formulas, for two invertible $\cO_p$-ideals $\mathfrak{a}$ and $\mathfrak{b}$ we shall write $\frac{\mathfrak{a}}{\mathfrak{b}}$ in place of $\mathfrak{a}\cdot\mathfrak{b}^{-1}$. With all our current notation in force (in particular, recall the Shafarevich--Tate group $\Sha_p^{\BKK}(\Q,\MM)$ from part (1) of Definition \ref{Sha-def} and the matrix $\mathtt{A}_{\tilde{\mathscr B}}\in\GL_r(F_p)$ in \S \ref{linear-algebra-lattices-subsubsec}), we can state the reformulation of the $p$-part of Conjecture \ref{TNC} we are interested in.

\begin{theorem} \label{motivesthm}
Assume that 
\begin{enumerate}
\item Conjecture \ref{nondegconj} holds true;
\item Conjecture \ref{ratconj} (or, equivalently, Conjecture \ref{ratconj2}) holds true;
\item the $p$-part of Conjecture \ref{regpconj} for $K=\Q$ holds true;
%\item $\Sha_p^{\BKK}(\Q,\MM)$ is finite.
\end{enumerate}
The $p$-part of Conjecture \ref{TNC} is equivalent to the equality 
\begin{equation} \label{eqBK}
\biggl(\frac{\Lambda^*(\MM,0)}{\Omega_\MM\cdot\Reg_{\mathscr B}(\MM)}\biggr)=\frac{\mathcal{I}\bigl(\Sha_p^{\BKK}(\Q,\MM)\bigr)\cdot\mathcal{I}_p(\gamma_f)\cdot\prod_{v\in S}\mathrm{Tam}_v^{(p)}(\MM)}{\bigl(\det(\mathtt{A}_{\tilde{\mathscr{B}}})\bigr)^2\cdot\Tors_p(\MM)} \tag{$p$-$\mathrm{TNC}_{\mathscr B}$}
\end{equation}
of fractional $\cO_p$-ideals.  
\end{theorem}

Observe that assumption (3) in the statement of the theorem ensures that Conjecture \ref{finitenessconj} for $K=\Q$ holds true (\emph{cf.} Remark \ref{reg-MM-rem}). In order to make sense of equality \eqref{eqBK}, in particular of how the left-hand side is seen as a principal fractional $\cO_p$-ideal, the reader is referred to Remark \ref{deligne-rem2}. With notation as in \eqref{reg-sigma-eq}, it is also useful to bear in mind that $\Reg_{\mathscr B}(\MM)=\bigl(\Reg_{\mathscr B}^\sigma(\MM)\bigr)_{\sigma\in\Sigma}$. 

\begin{remark} \label{non-degeneracy-rem}
Assumption (1) in Theorem \ref{motivesthm} is imposed only to force $\Reg_{\mathscr B}(\MM)$ to be non-zero: if we know that $\Reg_{\mathscr B}(\MM)\not=0$, we can remove this non-degeneracy requirement.
\end{remark}

\begin{remark} \label{BK-invariance-rem}
Suppose that $\mathscr B$ and $\mathscr B'$ are bases of $H^1_\mot(\Q,\MM)$ over $F$. Combining Remarks \ref{bases-rem} and \ref{B-rem}, one sees that \eqref{eqBK} holds if and only if ($p$-$\mathrm{TNC}_{\mathscr B'}$) holds.
\end{remark}

\begin{remark} \label{B-rem2}
By Remark \ref{B-rem}, the $\cO_p$-ideal $\bigl(\det(\mathtt{A}_{\tilde{\mathscr B}})\bigr)$ on the right-hand side of \eqref{eqBK} is independent of the choice of an $\cO_p$-basis of $\uH^1_f(\Q,T_p)$. Moreover, if $\Lambda_{\tilde{\mathscr B}}=\uH^1_f(\Q,T_p)$, then $\bigl(\det(\mathtt{A}_{\tilde{\mathscr B}})\bigr)=\cO_p$; in accord with this fact, one can check that the left-hand side of \eqref{eqBK} does not depend on $\mathscr B$ if $\mathscr B$ varies over all $F$-bases of $H^1_\mot(\Q,\MM)$ such that $\Lambda_{\tilde{\mathscr B}}=\uH^1_f(\Q,T_p)$.
\end{remark}

\begin{proof}[Proof of Theorem \ref{motivesthm}] If $\gamma_p$ is a generator of $T_p^+$, then $\Det_{\cO_p}(T_p^+)=\mathcal{I}_p(\gamma_f)\cdot\gamma_p$. As before, given a ring $R$, an $R$-module $M$ and $m_1,\dots,m_r\in M$, set $\underline{m}\defeq m_1\wedge\dots\wedge m_r$. Keeping Remark/Notation \ref{rem-not} in mind, there is an isomorphism of $\cO_p$-modules
\[ \Det_{\cO_p}^{-1}\bigl(H^1_f(\Q,T_p)\bigr)\simeq\mathcal{I}^{-1}\bigl(H^1_f(\Q,T_p{)}_\mathrm{tors}\bigr)\cdot\underline{\xi}^{-1}=\det(\mathtt{A}_{\tilde{\mathscr{B}}})\cdot\mathcal{I}^{-1}\bigl(H^1_f(\Q,T_p{)}_\mathrm{tors}\bigr)\cdot\underline{x}^{-1}. \] 
By Corollary \ref{image-upsilon-dual-coro} and the convention introduced at the end of \S \ref{cohomology-subsec}, $\{\xi_1^\vee,\dots,\xi_r^\vee\}$ is an $\cO_p$-basis of $H^1_f(\Q,A_p)_\divv^\vee$. Therefore, setting $\underline\xi^\vee\defeq\xi_1^\vee\wedge\dots\wedge\xi_r^\vee$ and $\underline{x}^\vee\defeq x_1^\vee\wedge\dots\wedge x_r^\vee$, there is an isomorphism of $\cO_p$-modules
\[ \Det_{\cO_p}\bigl(H^1_f(\Q,A_p)^\vee\bigr)\simeq\mathcal{I}^{-1}\bigl(\Sha_p^{\BKK}(\Q,\MM)\bigr)\cdot\underline{\xi}^\vee=\det(\mathtt{A}_{\tilde{\mathscr{B}}})\cdot\mathcal{I}^{-1}\bigl(\Sha_p^{\BKK}(\Q,\MM)\bigr)\cdot\underline{x}^{\vee}. \] 
Recall the isomorphism of $F_p$-modules $\theta_{p,S}$ from \eqref{theta-isom-eq}. Combining Proposition \ref{Tamv} with Proposition \ref{RGamma_c} and Definitions \ref{Tamp}, \ref{deftaminfty} and \ref{deftors}, we get an isomorphism of $\cO_p$-modules
\begin{equation} \label{eqdet}
\Det_{\cO_p}\bigl(\RR\Gamma_c(G_S,T_p)\bigr)\simeq\frac{\det(\mathtt{A}_{\tilde{\mathscr{B}}})^2\cdot\Tors_p(\MM)}{\mathcal{I}_p(\gamma_f)\cdot\mathcal{I}\bigl(\Sha_p^{\BKK}(\Q,\MM)\bigr)\cdot\prod_{v\in S}\mathrm{Tam}^{(p)}_v(\MM)}\cdot\theta_{p,S}(\beta_f),
\end{equation}
where, as above, $\beta_f\defeq\underline{x}^{-1}\otimes\underline{x}^\vee\otimes\gamma_f^{-1}\otimes\omega_f^\vee\in\Delta(\MM)$. Let $\zeta_f\in\Delta(\MM)$ be as in Conjecture \ref{ratconj} and write $\zeta_f=a\beta_f$ for some $a\in F^\times$. 
By Proposition \ref{rationality-prop} and its proof, there is an equality
\begin{equation} \label{equiv-eq1}
\frac{1}{a}=\frac{L^*(\MM,0)}{\Omega_\infty\cdot\Reg_{\mathscr B}(\MM)},
\end{equation}
which should be understood as in Remark \ref{deligne-rem2}. By \eqref{eqdet} there is an isomorphism
\begin{equation} \label{equiv-eq2}
\begin{split}
   \frac{\theta_{p,S}(\zeta_f)}{a}\cdot\cO_p\simeq&\Det_{\cO_p}\bigl(\RR\Gamma_c(G_S,T_p)\bigr)\cdot\det(\mathtt{A}_{\tilde{\mathscr{B}}})^{-2}\cdot\mathcal{I}_p(\gamma_f)\\
   &\cdot\Tors_p(\MM)^{-1}\cdot\mathcal{I}\bigl(\Sha_p^{\BKK}(\Q,\MM)\bigr)\cdot\prod_{v\in S}\mathrm{Tam}_v^{(p)}(\MM).
\end{split} 
\end{equation}
Combining \eqref{equiv-eq1}, \eqref{equiv-eq2}, the relation $\zeta_f^*=\zeta_f/(k/2-1)!$, formula \eqref{completedL} and the definition of $\Omega_\MM$ given in \eqref{Omega-eq} shows that the equality $\theta_{p,S}(\zeta_f^*)\cdot\cO_p=\Det_{\cO_p}\bigl(\RR\Gamma_c(G_S,T_p)\bigr)$ predicted by the $p$-part of Conjecture \ref{TNC} is equivalent to equality \eqref{eqBK}, as claimed. \end{proof}

\begin{remark} \label{BK-bis-rem}
Equality \eqref{eqBK} is equivalent to the equality
\begin{equation} \label{eqBKbis}
\biggl(\frac{(k/2-1)!\cdot L^*(\MM,0)}{\Omega_\infty\cdot\Reg_{\mathscr B}(\MM)}\biggr)=\frac{\mathcal{I}\bigl(\Sha_p^{\BKK}(\Q,\MM)\bigr)\cdot\mathcal{I}_p(\gamma_f)\cdot\prod_{v\in S}\mathrm{Tam}_v^{(p)}(\MM)}{\bigl(\det(\mathtt{A}_{\tilde{\mathscr{B}}})\bigr)^2\cdot\Tors_p(\MM)} \tag{$p$-$\mathrm{TNC}_{\mathscr B}$-$\mathrm{bis}$}
\end{equation} 
of fractional $\cO_p$-ideals: this follows immediately from \eqref{completedL}.
\end{remark}

To the best of our knowledge, Theorem \ref{motivesthm} offers the first reformulation of this form of (the $p$-part of) the TNC for $\MM$ in arbitrary analytic rank; the reader is referred to \cite{DSW} for a similar interpretation in analytic rank $0$.

\section{Kolyvagin's conjecture for modular forms} \label{kolyvagin-sec}

Our goal in this section is to state and prove Kolyvagin's conjecture for a large class of higher (even) weight modular forms. 

Let $\bar\Q$ denote the algebraic closure of $\Q$ in $\C$. As in Section \ref{secmot}, let $f\in S_k(\Gamma_0(N))$ be a newform of weight $k\geq4$ and level $N$, with $q$-expansion $f(q)=\sum_{n\geq1}a_n(f)q^n$. From here on we assume, as in the introduction, that 
\begin{itemize}
\item $f$ has no complex multiplication in the sense of \cite[p. 34, Definition]{ribet}.
\end{itemize}
As before, let $F$ be the Hecke field of $f$; by construction, it is naturally a subfield of $\bar\Q$. We write $\cO_F$ (respectively, $D_F$) for the ring of integers (respectively, the discriminant) of $F$. Finally, let $\cO_f\defeq\Z\bigl[a_n(f)\mid n\geq1\bigr]$ be the order of $\cO_F$ generated over $\Z$ by the Fourier coefficients $a_n(f)$ and let $c_f\defeq[\cO_F:\cO_f]$ be the index of $\cO_f$ in $\cO_F$.

\begin{remark}
A sufficient condition for $f$ not to have complex multiplication is that $N$ be square-free (\emph{cf.} \cite[p. 34]{ribet}), which will be assumed in due course.
\end{remark}

\subsection{Big image and irreducibility assumptions}

We collect two results on the Galois representations attached to $f$.

\subsubsection{Big image}

Let $p$ be a prime number. Denote by
\[ \rho_p:G_\Q\longrightarrow\Aut_{\cO_p}(T_p)\simeq\GL_2(\cO_p) \] 
the $p$-adic Galois representation attached to $f$ and $p$. We say that $\rho_p$ has \emph{big image} if there is an inclusion
\[ \bigl\{g\in\GL_2(\cO_p)\mid\det(g)\in(\Z_p^\times)^{k-1}\bigr\}\subset\im(\rho_p). \]

\begin{lemma} \label{big-lemma}
The representation $\rho_p$ has big image for all but finitely many $p$.
\end{lemma}

\begin{proof} Since $f$ is not CM, the lemma follows from \cite[Theorem 3.1]{ribet2}. \end{proof}

\subsubsection{Residual irreducibility}

With notation as above, if $\p$ is a prime of $F$ above $p$, then we denote by
\[ \rho_{\p}:G_\Q\longrightarrow\Aut_{\cO_\p}(T_\p)\simeq\GL_2(\cO_\p) \]
the Galois representation associated with $T_\p$. Reducing modulo the maximal ideal of $\cO_{\p}$, we obtain a residual representation 
\[ \bar\rho_{\p}:G_\Q\longrightarrow\Aut_{\F_\p}(T_\p/\p T_\p)\simeq\GL_2(\F_\p), \] 
where $\F_\p\defeq\cO_\p/\p\cO_\p$ is the residue field of $F_\p$. 

\begin{lemma} \label{residual-irreducibility-lemma}
For all but finitely many prime numbers $p$, the representation $\bar\rho_\p$ is irreducible for every prime $\p$ of $F$ above $p$.
\end{lemma}

\begin{proof} Since $f$ is not CM, this is \cite[Theorem 2.1, (a)]{ribet2}. \end{proof}

To state Kolyvagin's conjecture, we work under

\begin{assumption} \label{p-ass}
The prime number $p$ satisfies the following conditions:
\begin{enumerate}
\item  $p\nmid 6ND_Fc_f$; 
%\vskip 1mm
\item $\rho_p$ has big image;
%\vskip 1mm
\item $\bar\rho_\p$ is irreducible for each $\p\,|\,p$.
\end{enumerate}
\end{assumption}

By Lemmas \ref{big-lemma} and \ref{residual-irreducibility-lemma}, all but finitely many primes $p$ satisfy Assumption \ref{p-ass}.

\subsection{$p$-isolation} \label{sec3.2} \label{isolated-subsec}

As in \S \ref{anaemic-hecke-subsubsec}, let us write $\mathfrak{H}_k(\Gamma_0(N))$ for the anaemic Hecke algebra of weight $k$ and level $\Gamma_0(N)$. 

\subsubsection{$p$-isolation of $f$}

Let $g(q)=\sum_{n\geq1}a_n(g)q^n\in S_k(\Gamma_0(N))$ be a normalized eigenform for $\mathfrak{H}_k(\Gamma_0(N))$ and let $L\defeq F\bigl(a_n(g)\mid n\geq1\bigr)$ be the composite of $F$ and the Hecke field of $g$. Let $p$ be a prime number and pick a prime $\p$ of $F$ above $p$. The form $f$ is said to be \emph{congruent to $g$ modulo $\p$} if 
\[ a_n(f)\equiv a_n(g)\pmod{\mathfrak P} \]
for some prime $\mathfrak P$ of $L$ above $\p$ and for all $n\geq1$. In this case, we write $f\equiv g\pmod\p$.
  
\begin{definition} \label{p-isolated-def}
The form $f$ is \emph{$p$-isolated} if there is no normalized eigenform $g\in S_k(\Gamma_0(N))$ other than $f$ such that $f\equiv g\pmod\p$ for some prime $\p$ of $F$ above $p$. 
\end{definition}

The next result tells us that, for a given $f$, the existence of congruences modulo $p$ is an exception.

\begin{theorem}[Ribet] \label{isolated-ribet-thm}
The modular form $f$ is $p$-isolated for all but finitely many $p$.
\end{theorem}

\begin{proof} This follows from \cite[Theorem 1.4]{ribet-congruences}. \end{proof}

\subsubsection{Congruence ideal of $f$ at $\p$}

With notation as in \S \ref{isotypic-subsubsec}, let $\theta_{f,\p}:\mathfrak{H}_k(\Gamma_0(N))_{\cO_\p}\twoheadrightarrow\cO_\p$ be the (surjective) $\cO_\p$-linear extension of $\theta_f$ and let $\mathrm{Ann}_{\mathfrak{H}_k(\Gamma_0(N))_{\cO_\p}}\!\bigl(\ker(\theta_{f,\p})\bigr)$ be the annihilator ideal of $\ker(\theta_{f,\p})$ in $\mathfrak{H}_k(\Gamma_0(N))_{\cO_\p}$. The \emph{congruence ideal of $f$ at $\p$} is the ideal of $\cO_\p$ given by 
\[ \eta_{f,\p}\defeq\theta_{f,\p}\Bigl(\mathrm{Ann}_{\mathfrak{H}_k(\Gamma_0(N))_{\cO_\p}}\!\bigl(\ker(\theta_{f,\p})\bigr)\!\Bigr). \]
Let us set $\eta_{f,p}\defeq\prod_{\p\mid p}\eta_{f,\p}\subset\cO_p$. It is well known that, under Assumption \ref{ass-p-iso}, $\eta_{f,p}=\cO_p$. Furthermore, since the $\cO_\p$-algebra $\mathfrak{H}_k(\Gamma_0(N))_{\cO_\p}$ is flat, for each $\p\,|\,p$ the (tautological) short exact sequence of $\mathfrak{H}_k(\Gamma_0(N))_{\cO_\p}$-modules
\[ 0\longrightarrow\ker(\theta_{f,\p})\longrightarrow\mathfrak{H}_k(\Gamma_0(N))_{\cO_\p}\xrightarrow{\theta_{f,\p}}\cO_\p\longrightarrow 0 \] 
splits canonically. Thus, there is a canonical isomorphism, which we regard as an equality, of $\mathfrak{H}_k(\Gamma_0(N))_{\cO_\p}$-modules
\begin{equation} \label{lenstra}
\mathfrak{H}_k(\Gamma_0(N))_{\cO_\p}=\ker(\theta_{f,\p})\oplus\cO_\p.
\end{equation}
See, \emph{e.g.}, \cite{lenstra} for details.

\subsection{On $p$-adic Abel--Jacobi maps} \label{AJ-subsec}

Following \cite[Ch. II, (6.5)]{Nek2}, with $T_p$ and $T_\p$ as in \S \ref{etale-subsubsec}, let us define 
\[ J_p\defeq\Pi_B\Pi_\epsilon\cdot H^1_\et\bigl(\bar{X},\cO_p(k/2)\bigr), \]
which we view as a subgroup of $H^1_\et\bigl(\bar{X},F_p(k/2)\bigr)$; then $T_p=J_p[\theta_f]$. There is a splitting $J_p=\prod_{\p\mid p}J_\p$, where
\begin{equation} \label{J-split-eq}
J_\p\defeq\Pi_B\Pi_\epsilon\cdot H^1_\et\bigl(\bar{X},\cO_\p(k/2)\bigr)
\end{equation}
is regarded as a subgroup of $H^1_\et\bigl(\bar{X},F_\p(k/2)\bigr)$. Notice that 
\[ T_\p=J_\p[\theta_{f,\p}]=J_\p\otimes_{\mathfrak{H}_k(\Gamma_0(N))_{\cO_\p}}\!\cO_\p. \]
As is pointed out in \cite[\S 3]{Nek}, there is a surjection $\varpi_p:J_p\twoheadrightarrow T_p$ whose restriction to $T_p$ is the multiplication-by-$p^m$ map for some $m\in\N$. Notation being as in \eqref{R-chow-eq}, for any number field $L$ the Abel--Jacobi map in \eqref{AJ} yields a map  
\begin{equation} \label{AJ-map-p-2} 
\AJ_{L,p}:{\CH^{k/2}_0(X/L)}_{\cO_p}\longrightarrow H^1_f(L,J_p)\xrightarrow{\varpi_{p,*}}H^1_f(L,T_p), 
\end{equation}
where $\varpi_{p,*}$ is induced by $\varpi_p$ functorially. Furthermore, both the source and the target of \eqref{AJ-map-p-2} split over the primes $\p$ of $F$ above $p$, and we let  
\begin{equation} \label{AJ-map*}
\AJ_{L,\p}:{\CH^{k/2}_0(X/L)}_{\cO_\p}\longrightarrow H^1_f(L,T_\p)
\end{equation}
be the $\p$-component of \eqref{AJ-map-p-2}. Finally, set  
\begin{equation} \label{Lambda-eq}
\Lambda_\p(L)\defeq\im(\AJ_{L,\p})\subset H^1_f(L,T_\p). 
\end{equation}
The $\cO_\p$-module $H^1_f(L,T_\p)$ is finitely generated, so $\Lambda_\p(L)$ is finitely generated over $\cO_\p$.

\begin{remark} \label{AJ-factorization-rem}
By construction, $\AJ_{L,p}$ and $\AJ_{L,\p}$ factor through $\Pi_B\Pi_\epsilon\cdot{\CH^{k/2}_0(X/L)}_{\cO_p}$ and $\Pi_B\Pi_\epsilon\cdot{\CH^{k/2}_0(X/L)}_{\cO_\p}$, respectively (\emph{cf.} \cite[p. 105]{Nek}). In particular, $\AJ_{L,p}$ and $\AJ_{L,\p}$ induce Abel--Jacobi maps on ${\CH^{k/2}_\arith(X/L)}_{\cO_p}$ and ${\CH^{k/2}_\arith(X/L)}_{\cO_p}$, respectively, to be denoted by the same symbols.
\end{remark}

%\begin{remark} \label{AJ-Q-rem}
%Later in this paper (\emph{cf.} the proof of Theorem \ref{ThmTNC}), we will need to use the $\Q$-linear $\p$-adic Abel--Jacobi map 
%\[ \AJ_{L,\p}:{\CH^{k/2}_0(X/L)}_\Q\longrightarrow\Lambda_\p(L)\otimes_\Z\Q=\Lambda_\p(L)\otimes_{\cO_\p}\!F_\p, \]
%still denoted in the same way as the map above.
%\end{remark}

\subsection{$p$-integral motivic cohomology} 

From now on, we work under

\begin{assumption} \label{ass-p-iso}
The eigenform $f$ is $p$-isolated.
\end{assumption} 

By Theorem \ref{isolated-ribet-thm}, Assumption \ref{ass-p-iso} rules out only finitely many primes $p$. This condition will be used to split $J_\p$ over the Hecke algebra.

\subsubsection{Splitting $J_\p$}

In the present situation, one can take $\varpi_p$ so that its restriction to $T_p$ is the identity (in other words, one can take $m=0$ in \S \ref{AJ-subsec}). More precisely, if $J_\p$ is the $\mathfrak{H}_k(\Gamma_0(N))_{\cO_\p}$-module from \eqref{J-split-eq}, then the splitting in \eqref{lenstra} produces a splitting
\begin{equation} \label{J-split-eq2}
J_\p=\Bigl(J_\p\otimes_{\mathfrak{H}_k(\Gamma_0(N))_{\cO_\p}}\!\cO_\p\Bigr)\textstyle{\bigoplus}\Bigl(J_\p\otimes_{\mathfrak{H}_k(\Gamma_0(N))_{\cO_\p}}\!\ker(\theta_{f,\p})\Bigr). 
\end{equation}
On the other hand, we already observed that $T_\p=J_\p\otimes_{\mathfrak{H}_k(\Gamma_0(N))_{\cO_\p}}\!\cO_\p$, so from \eqref{J-split-eq2} we get a (projection) map $\varpi_\p:J_\p\twoheadrightarrow T_\p$ for each $\p\,|\,p$. Taking sums over all $\p\,|\,p$ gives the desired surjection $\varpi_p:J_p\twoheadrightarrow T_p$. 

\subsubsection{$\p$-integral motivic cohomology of $\MM$} \label{integral-motivic-subsubsec}

Let $L$ be a number field. With notation as in \eqref{R-chow-eq} and using again \eqref{lenstra}, we can consider the splitting 
\begin{equation} \label{CH-split-eq}
\resizebox{.9\hsize}{!}{${\CH^{k/2}_{\arith}(X/L)}_{\cO_\p}=\Bigl({\CH^{k/2}_{\arith}(X/L)}_{\cO_\p}\otimes_{\mathfrak{H}_k(\Gamma_0(N))_{\cO_\p}}\!\cO_\p\Bigr)\textstyle{\bigoplus}\Bigl({\CH^{k/2}_{\arith}(X/L)}_{\cO_\p}\otimes_{\mathfrak{H}_k(\Gamma_0(N))_{\cO_\p}}\!\ker(\theta_{f,\p})\Bigr).$} 
\end{equation}
For each $\p\,|\,p$, we define the \emph{first $\p$-integral motivic cohomology group of $\MM$ over $L$} to be 
\begin{equation} \label{p-integral-motivic-eq}
\begin{split}
H^1_\mot(L,\MM{)}_\text{$\p$-int}&\defeq\CH^{k/2}_{\arith}(X/L)\otimes_{\mathfrak{H}_k(\Gamma_0(N))}\cO_\p\\
&\;={\CH^{k/2}_{\arith}(X/L)}_{\cO_\p}\otimes_{\mathfrak{H}_k(\Gamma_0(N))_{\cO_\p}}\!\cO_\p,
\end{split}
\end{equation}
where the $\mathfrak{H}_k(\Gamma_0(N))$-algebra structure on $\cO_\p$ is induced by composing $\theta_f$ with the natural injection $\cO_F\hookrightarrow\cO_\p$ and the bottom identification is a standard canonical isomorphism.

\subsubsection{$p$-integral motivic cohomology of $\MM$} \label{integral-motivic-subsubsec2}

We define the \emph{first $p$-integral motivic cohomology group of $\MM$ over $L$} as 
\begin{equation} \label{p-int-mot-eq}
H^1_\mot(L,\MM{)}_\text{$p$-int}\defeq\bigoplus_{\p\mid p}H^1_\mot(L,\MM{)}_\text{$\p$-int}.
\end{equation}
It follows that the $p$-adic \'etale regulator map from \eqref{p-reg-eq} yields maps
\begin{equation} \label{AJ-map}
\reg_{L,\p}:H^1_\mot(L,\MM{)}_\text{$\p$-int}\longrightarrow H^1_f(L,T_\p) 
\end{equation}
for each $\p\,|\,p$ and
\begin{equation} \label{AJ-map-p}
\reg_{L,p}:H^1_\mot(L,\MM{)}_\text{$p$-int}\longrightarrow H^1_f(L,T_p) 
\end{equation}
that satisfy $\reg_{L,p}=\bigoplus_{\p\mid p}\reg_{L,\p}$. Note that, since ${\CH^{k/2}_{\arith}(X/L)}_{\cO_p}=\bigoplus_{\p|p}{\CH^{k/2}_{\arith}(X/L)}_{\cO_\p}$, for $\star\in\{p\}\cup\{\p\,|\,p\}$ there is a commutative triangle 
\begin{equation} \label{CH-p-triangle-eq}
\xymatrix@R=30pt@C=5pt{{\CH^{k/2}_{\arith}(X/L)}_{\cO_\star}\ar[rr]^-{\AJ_{L,\star}}\ar@{->>}[rd]^-{\Pi_{\MM,L,\star}}&& H^1_f(L,T_\star)\\
&H^1_\mot(L,\MM{)}_{\text{$\star$-int}}\ar[ru]^-{\reg_{L,\star}}} 
\end{equation}
in which, bearing \eqref{p-integral-motivic-eq} in mind, $\Pi_{\MM,L,\star}$ is the projection induced by \eqref{CH-split-eq} if $\star=\p$ or the direct sum of such projections over all $\p\,|\,p$ if $\star=p$, whereas $\AJ_{L,\star}$ is (the restriction of) the map in \eqref{AJ-map-p-2} if $\star=p$ or in \eqref{AJ-map*} if $\star=\p$.

\begin{remark} \label{linear-extension-rem}
Let $\star\in\{p\}\cup\{\p\,|\,p\}$. Of course, extending $\reg_{L,\star}$ in \eqref{AJ-map} and \eqref{AJ-map-p} $F_\star$-linearly we recover the $p$-adic regulator map in \eqref{p-reg-eq} if $\star=p$ or the $\p$-adic regulator map in \eqref{pp-reg-eq} if $\star=\p$, which justifies the slight abuse of notation. 
\end{remark}

\subsubsection{A conjecture on $\reg_{L,p}$} \label{integral-conjecture-subsubsec}

The following conjecture on the $p$-adic regulator map $\reg_{L,p}$ in \eqref{AJ-map-p} is a stronger, integral version of Conjecture \ref{regpconj}. It predicts that if $f$ is $p$-isolated, then $\reg_{L,p}$ is not only injective (as implied by Conjecture \ref{regpconj}), but also surjective. 

\begin{conjecture} \label{cong-int} 
Let $L$ be a number field. The map $\reg_{L,p}$ in \eqref{AJ-map-p} is an isomorphism of $\cO_p$-modules for all but finitely many primes $p$ at which $f$ is $p$-isolated. 
\end{conjecture}

In particular, we conjecture $\reg_{L,p}$ to be an isomorphism of $\cO_p$-modules for all but finitely many $p$ (\emph{cf.} Theorem \ref{isolated-ribet-thm}). At a certain point, we will assume that Conjecture \ref{cong-int} is true for a specific choice of $L$ (namely, $L=\Q$ or $L$ a suitable imaginary quadratic field).

\begin{remark}
Conjecture \ref{cong-int} implies Conjecture \ref{regpconj} for all primes $p$ at which $f$ is $p$-isolated (\emph{cf.} Remark \ref{linear-extension-rem}).
\end{remark}

\begin{remark}
Our main motivation for proposing Conjecture \ref{cong-int} is the following. Let $\mathcal{T}$ be an abelian tensor category and let $\mathcal H=\bigl(H^\bullet(\cdot,\star),H_\bullet(\cdot,\star)\bigr)$ be a $\mathcal{T}$-valued twisted Poincar\'e duality theory with weights (see, \emph{e.g.}, \cite[\S6]{Jannsen}). Furthermore, let $X$ be a smooth proper variety of dimension $d$ over a field. As explained, \emph{e.g.}, in \cite[\S9.1]{Jannsen}, there is an Abel--Jacobi map 
\[ r_\mathcal{H}:{Z^j(X)}_0\longrightarrow \mathrm{R}^1\Gamma H^{2j-1}(X,j) \]
for all integers $0\leq j\leq d$, where ${Z^j(X)}_0$ is the group of cycles of codimension $j$ on $X$ that are homologically equivalent to $0$ and 
\[ \mathrm{R}^1\Gamma H^{2j-1}(X,j)=\Ext\bigl(1,H^{2j-1}(X,j)\bigr). \]
Now let $\mathcal{H}_\B$ be the Betti cohomology theory with coefficients in $\bar\Q$. The image of $r_{\mathcal H_\B}$ and integral Betti cohomology induce integral structures on $\Ext\bigl(1,H^{2j-1}(X,j)\bigr)$ that we expect to coincide after localization at a prime $p$ for all but finitely many $p$. Therefore, the comparison isomorphism between Betti and \'etale cohomology suggests that the $p$-adic Abel--Jacobi map is surjective when regulator maps can be defined (using an assumption of $p$-isolation on $f$), and this led us to Conjecture \ref{cong-int}. Admittedly, at present we are at a loss to provide a more convincing and less vague motivation for this conjecture.
\end{remark}
  
\subsection{Heegner cycles} \label{cycles-subsec}

We recall the definition of (classical) Heegner cycles in the sense of Nekov\'a\v{r} (\cite{Nek}, \cite{Nek2}).

\subsubsection{Heegner hypothesis}

Let $K$ be an imaginary quadratic field of discriminant $D_K$ such that 
\begin{itemize}
\item all the prime factors of $N$ split in $K$.
\end{itemize}
In other words, $K$ satisfies the \emph{Heegner hypothesis} relative to $N$. By virtue of this condition, if $\cO_K$ is the ring of integers of $K$, then we can fix an $N$-cyclic ideal of $\cO_K$, \emph{i.e.}, an ideal $\mathcal N\subset\cO_K$ such that $\mathcal O_K/\mathcal N\simeq \Z/N\Z$. Let us choose once and for all an embedding $K\hookrightarrow\C$. For every integer $n\geq1$ prime to $NpD_K$ let $\mathcal O_n\defeq \Z+n\mathcal O_K$ be the order of $K$ of conductor $n$. The isogeny $\C/\cO_n\rightarrow \C/(\cO_n\cap\mathcal N)^{-1}$ of complex tori defines a Heegner point $x_n\in X_0(N)$ that, by the theory of complex multiplication, is rational over the ring class field $K_n$ of $K$ of conductor $n$ (in particular, $K_1$ is the Hilbert class field of $K$). 

\subsubsection{Heegner cycles} \label{heegner-cycles-subsubsec}

Write $\pi_N:X(N)\rightarrow X_0(N)$ for the map induced by the inclusion $\Gamma(N)\subset\Gamma_0(N)$ and choose $\tilde x_n\in\pi_N^{-1}(x_n)$. The elliptic curve $E_n$ corresponding to $\tilde x_n$ has complex multiplication by $\cO_n$. Fix the unique square root $\xi_n=\sqrt{-n^2D_K}$ of the discriminant of $\cO_n$ with positive imaginary part under the chosen embedding $K\hookrightarrow\C$. For any $a\in\cO_n$ let $\Gamma_{n,a}\subset E_n\times E_n$ denote the graph of $a$ and let $i_{\tilde x_n}:\tilde\pi_{k-2}^{-1}(\tilde x_n)=E_n^{k-2}\hookrightarrow X$ be the canonical inclusion (recall that $X=\tilde{\mathcal E}_N^{k-2}$). We will frequently write the same symbol $Z$ for a cycle $Z$ and the class $[Z]$ of $Z$ in the Chow group. Put 
\begin{equation}\label{cycle-eq0}
Z_k(\tilde{x}_n)\defeq\Gamma_{n,\xi_n}\smallsetminus\bigl[(E_n\times\{0\})\cup(\{0\}\times E_n)\bigr] \end{equation}
and, with notation as in \eqref{R-chow-eq}, define  
\begin{equation} \label{cycle-eq1}
\tilde\Gamma_n\defeq\Pi_B\Pi_\epsilon\cdot{(i_{\tilde x_n})}_*\bigl(Z_k(\tilde{x}_n)^{(k-2)/2}\bigr)\in\Pi_B\Pi_\epsilon\cdot{\CH}^{k/2}(X/K_n),
\end{equation}
where $\Pi_B\Pi_\epsilon\cdot{\CH}^{k/2}(X/K_n)$ is to be viewed as a subgroup of ${{\CH}^{k/2}(X/K_n)}_\Q$. As explained in \cite[p. 105]{Nek}, there is an equality
\[ \Pi_\epsilon\cdot{{\CH}^{k/2}(X/K_n)}_{\Z_p}=\Pi_\epsilon\cdot{{\CH}^{k/2}_0(X/K_n)}_{\Z_p} \] 
inside ${{\CH}^{k/2}_0(X/K_n)}_{\Q_p}$. We call the image
\begin{equation} \label{cycle-eq2}
\Gamma_{n,p}\in\Pi_B\Pi_\epsilon\cdot{{\CH}^{k/2}_0(X/K_n)}_{\cO_p}\subset{{\CH}^{k/2}_0(X/K_n)}_{F_p}
%H^1_\mot(K_n,\MM{)}_{\text{$p$-int}}
\end{equation} 
of $\tilde\Gamma_{n}$ 
%$H^1_\mot(K_n,\MM{)}_{\text{$p$-int}}$ 
the \emph{geometric Heegner cycle} (\emph{at $p$}) \emph{of conductor $n$}. Moreover, for each $\p\,|\,p$ we write 
\begin{equation} \label{cycle-eq3}
\Gamma_{n,\p}\in\Pi_B\Pi_\epsilon\cdot{{\CH}^{k/2}_0(X/K_n)}_{\cO_\p}\subset{{\CH}^{k/2}_0(X/K_n)}_{F_\p}
\end{equation}
for the image of $\Gamma_{n,p}$ (here we are implicitly using the splitting $\Pi_B\Pi_\epsilon\cdot{{\CH}^{k/2}_0(X/K_n)}_{\cO_p}=\bigoplus_{\p\mid p}\Pi_B\Pi_\epsilon\cdot{{\CH}^{k/2}_0(X/K_n)}_{\cO_\p}$). The \emph{arithmetic Heegner cycle} (\emph{at $p$}) \emph{of conductor $n$} is then the image
\[ y_{n,p}\defeq\AJ_{K_n,p}(\Gamma_{n,p})\in H^1_f(K_n,T_p) \]
of the cycle in \eqref{cycle-eq2} via the $p$-adic Abel--Jacobi map from \eqref{AJ-map-p-2}, which factors through $\Pi_B\Pi_\epsilon\cdot{{\CH}^{k/2}_0(X/K_n)}_{\cO_p}$ (\emph{cf.} Remark \ref{AJ-factorization-rem}). With notation as in \eqref{Lambda-eq}, for each $\p\,|\,p$ we also set
\[ y_{n,\p}\defeq\AJ_{K_n,\p}(\Gamma_{n,\p})\in\Lambda_\p(K_n), \]
where $\AJ_{K_n,\p}$ is the $\p$-adic Abel--Jacobi map in \eqref{AJ-map*} and, as in \eqref{Lambda-eq}, $\Lambda_\p(K_n)$ is its image. In other words, $y_{n,\p}$ is the natural image of $y_{n,p}$ in $H^1_f(K_n,T_\p)$. It turns out that $y_{n,p}$ and $y_{n,\p}$ are independent of the choice of $\tilde x_n$ (\cite[p. 107]{Nek}). In the rest of this paper, the expression ``Heegner cycle of conductor $n$'' will always refer to $y_{n,\p}$ for a fixed $\p$ as above.

Finally, a crucial role in our arguments will be played by the cycle
\begin{equation} \label{y_K-eq}
y_{K,\p}\defeq\cores_{K_1/K}(y_{1,\p})\in\Lambda_\p(K);
\end{equation}
here we exploit the Galois-equivariance of the maps $\AJ_{\star,\p}$, which implies that the square
\begin{equation} \label{AJ-cores-eq}
\xymatrix@C=50pt@R=30pt{\Pi_B\Pi_\epsilon\cdot{\CH^{k/2}_0(X/K_1)}_{\cO_\p}\ar[d]^-{\tr_{K_1/K}}\ar[r]^-{\AJ_{K_1,\p}}&\Lambda_\p(K_1)\ar[d]^-{\cores_{K_1/K}}\\\Pi_B\Pi_\epsilon\cdot{\CH^{k/2}_0(X/K)}_{\cO_\p}\ar[r]^-{\AJ_{K,\p}}&\Lambda_\p(K)}
\end{equation}
commutes (as the notation suggests, $\tr_{K_1/K}$ is the Galois trace map on Chow groups).
%The cycle $y_{K,\p}$ is a higher weight counterpart of the $K$-rational Heegner point that appears, for example, in the Gross--Zagier formula for $L$-functions of elliptic curves (\cite{GZ}), in the work of Kolyvagin on the Euler system of Heegner points and the arithmetic of modular abelian varieties (\cite{kolyvagin-finiteness}, \cite{Kol-Euler}, \cite{Kol-Log}) and in the work of Zhang (\cite{zhang-selmer}) and of Skinner--Zhang (\cite{SZ}) on Kolyvagin's conjecture and the $p$-part of the Birch and Swinnerton-Dyer formula for elliptic curves in analytic rank one.

\subsection{Kolyvagin integers} \label{kolyvagin-integers-subsec} 

Recall that, by Assumption \ref{p-ass}, the prime $p$ is unramified in $F$, hence $p$ is a local uniformizer for $F$ at $\p$. Let $v_\p$ be the valuation of $F_\p$ normalized so that $v_\p(p)=1$. 

\subsubsection{Kolyvagin primes}

A prime number $\ell$ is a \emph{Kolyvagin prime} for the data $(f,{\p},K)$ if
\begin{enumerate}
\item $\ell\nmid Np$;
\item $\ell$ is inert in $K$; 
\item $M(\ell)\defeq \min\bigl\{v_{\p}(\ell+1),v_{\p}\bigl(a_\ell(f)\bigr)\bigr\}>0$. 
\end{enumerate} 
Denote by $\mathcal{P}_\mathrm{Kol}(f)$ the set of Kolyvagin primes for $(f,{\p},K)$. 

\subsubsection{Kolyvagin integers}

Let us write 
\[ \Lambda_{\Kol}(f)\defeq\bigl\{\text{square-free products of primes in $\mathcal{P}_{\Kol}(f)$}\bigr\} \] 
for the set of \emph{Kolyvagin integers} for $(f,{\p},K)$. If we need to specify the data $(\p,K)$ we also write $\mathcal{P}_{\Kol}(f,{\p},K)$ and  $\Lambda_{\Kol}(f,{\p},K)$ for $\mathcal{P}_{\Kol}(f)$ and  $\Lambda_{\Kol}(f)$, respectively. 
Notice that $1\in\Lambda_{\Kol}(f)$. Finally, for every $n\in\Lambda_{\Kol}(f)$ define
\begin{equation} \label{M(n)-eq}
M(n)\defeq \begin{cases} \min\bigl\{M(\ell)\mid\ell\,|\,n\bigr\}&\text{if $n\geq2$},\\[3mm]\infty&\text{if $n=1$}. \end{cases}
\end{equation}
The integer $M(n)$ is the \emph{Kolyvagin index} of $n$. 

\subsection{Kolyvagin classes and Kolyvagin's conjecture} \label{kolyvagin-classes-subsec}

We attach to our newform $f$ a systematic supply of Galois cohomology classes, which we call \emph{Kolyvagin classes}, that are indexed by Kolyvagin integers and take values in quotients of $T_\p$ (or, equivalently, in torsion submodules of $A_\p$, \emph{cf.} \S \ref{p-modules-subsec}). Our strategy for producing these classes, which are defined in terms of the Heegner cycles of \S \ref{cycles-subsec}, follows the recipe proposed by Kolyvagin for modular abelian varieties (see, \emph{e.g.}, \cite[\S 4]{Gross} and \cite[\S 3.7]{zhang-selmer}). In order to fix notation that we will use in the rest of the paper, and for the convenience of the reader, we describe the construction of Kolyvagin classes in our higher weight setting.

\subsubsection{Kolyvagin derivatives} \label{derivatives-subsubsec}

For all $n\in\Lambda_{\Kol}(f)$, let us set $G_n\defeq \Gal(K_n/K_1)$ and $\mathcal G_n\defeq \Gal(K_n/K)\simeq\Pic(\mathcal O_n)$, so that, by class field theory, $G_n=\prod_{\ell|n}G_\ell$ with $G_\ell$ cyclic of order $\ell+1$. For all $\ell\in\mathcal P_{\Kol}$ choose a generator $\sigma_\ell$ of $G_\ell$; define Kolyvagin derivative operators as
\[ D_\ell\defeq \sum_{i=1}^\ell i\sigma_\ell^i\in\Z[G_\ell],\quad D_n\defeq \prod_{\ell|n}D_\ell\in\Z[G_n]. \]
In particular, $D_1$ is the identity operator. Fix $n\in\Lambda_{\Kol}(f)$, let $\mathcal G$ be a system of representatives for $\mathcal G_n/G_n$ and set
\[ z_{n,\p}\defeq \sum_{\sigma\in\mathcal G}\sigma\bigl(D_n(y_{n,\p})\bigr)\in\Lambda_\p(K_n). \]

\begin{remark} \label{z_1-rem}
Since $G_1$ is trivial and $D_1$ is the identity operator, $z_{1,\p}=\sum_{\sigma\in\mathcal G_1}\sigma(y_{1,\p})\in\Lambda_\p(K_1)$. A direct computation shows that
\begin{equation} \label{y_K-z_1-eq}
\res_{K_1/K}(y_{K,\p})=z_{1,\p},
\end{equation}
where $y_{K,\p}\in\Lambda_\p(K)$ is the cycle defined in \eqref{y_K-eq}.
\end{remark}

\subsubsection{Kolyvagin classes}

As a consequence of \cite[Corollary 2.7, (3)]{LV} and \cite[Proposition 2.8]{LV}, for any number field $L$ and every integer $M\geq1$ there is a natural Galois-equivariant injection
\begin{equation} \label{iota_L-eq}
\iota_{L,M}:\Lambda_\p(L)\big/p^M\Lambda_\p(L)\longmono H^1\bigl(L,A_\p[p^M]\bigr). 
\end{equation}
This map should the thought of as a higher weight avatar of the usual Kummer map in the Galois cohomology of abelian varieties over number fields.
The extension $K_n/\Q$, which is generalized dihedral, is solvable, so \cite[Lemma 3.10, (2)]{LV} ensures that $H^0\bigl(K_n,A_\p[p^M]\bigr)=0$. It follows that restriction induces an isomorphism
\begin{equation} \label{res-iso-eq} 
\res_{K_n/K}:H^1\bigl(K,A_\p[p^M]\bigr)\overset\simeq\longrightarrow H^1\bigl(K_n,A_\p[p^M]\bigr)^{\mathcal G_n}. 
\end{equation}
Moreover, one can easily check that if ${[z_{n,\p}]}_M$ denotes the class of $z_{n,\p}$ modulo $p^M$ (analogous notation will be used, below, for $y_{K,\p}$) and $M(n)$ is the Kolyvagin index of $n$ from \eqref{M(n)-eq}, then
\[ M\leq M(n)\;\Longrightarrow\;{[z_{n,\p}]}_M\in\bigl(\Lambda_\p(K_n)\big/p^M\Lambda_\p(K_n)\bigr)^{\mathcal G_n}, \]
whence
\[ M\leq M(n)\;\Longrightarrow\;\iota_{K_n,M}\bigl({[z_{n,\p}]}_M\bigr)\in H^1\bigl(K_n,A_\p[p^M]\bigr)^{\mathcal G_n}. \]
For notational convenience, and also for later reference, let us define
%\begin{equation} \label{d-classes-eq}
\[ d_M(f,n)\defeq\iota_{K_n,M}\bigl({[z_{n,\p}]}_M\bigr)\in H^1\bigl(K_n,A_\p[p^M]\bigr)^{\mathcal G_n}. \] 
%\end{equation}
Keeping isomorphism \eqref{res-iso-eq} in mind, for all $M\leq M(n)$ we can define the \emph{Kolyvagin class} $c_M(f,n)$ as
\[ c_M(f,n)\defeq \res_{K_n/K}^{-1}\bigl(d_M(f,n)\bigr)\in H^1\bigl(K,A_\p[p^M]\bigr). \]
In particular, one has 
\begin{equation} \label{c-d-eq}    
c_M(f,n)=0\;\Longleftrightarrow\;d_M(f,n)=0.
\end{equation}
The \emph{Kolyvagin set} associated with $(f,{\p},K)$ is 
\begin{equation} \label{Kolsys}
\kappa_{f,{\p},\infty}^{(K)}\defeq\bigl\{c_M(f,n)\mid n\in\Lambda_{\Kol}(f),\;1\leq M\leq M(n)\bigr\}. 
\end{equation}
If $\p$ and $K$ are clear from the context (which will usually be the case), then we shall write $\kappa_{f,\infty}$ in place of $\kappa^{(K)}_{f,\p,\infty}$
%When $f$, ${\p}$ or (more typically) $K$ are unambiguously understood from the context, we may drop some of them from the notation. Thus, the reader will sometimes find symbols like, for example, $\kappa_\infty$ in place of $\kappa_{f,{\p},\infty}^{(K)}$.

\subsubsection{Kolyvagin's conjecture in higher weight}

The following conjecture was first proposed, with a slightly different formalism, in \cite[Conjecture A]{Masoero}.

\begin{conjecture}[Kolyvagin's conjecture, higher weight] \label{kolyvagin-conj}
$\kappa_{f,\infty}\neq\{0\}$.
\end{conjecture}

This is a higher (even) weight counterpart of a conjecture for rational elliptic curves due to Kolyvagin (\cite[Conjecture A]{kolyvagin-selmer}).

It is convenient to introduce some more terminology. The \emph{strict Kolyvagin set} attached to $(f,{\p},K)$ is
\begin{equation} \label{Kolsys-strict}
\kappa^{(K),\st}_{f,\p,\infty}\defeq\bigl\{c_1(f,n)\mid n\in\Lambda_{\Kol}(f)\bigr\}.
\end{equation} 
As above, we shall write $\kappa^{\st}_{f,\infty}$ in place of $\kappa^{(K),\st}_{f,\p,\infty}$ if no confusion is likely to arise. 

\begin{conjecture}[Kolyvagin's conjecture, higher weight, strong form] \label{strong-kolyvagin-conj}
$\kappa^{\st}_{f,\infty}\neq\{0\}$.
\end{conjecture}

Clearly, there is an inclusion $\kappa^{\st}_{f,\infty}\subset\kappa_{f,\infty}$, so Conjecture \ref{strong-kolyvagin-conj} is stronger than Conjecture \ref{kolyvagin-conj}. 

\subsection{The Kolyvagin classes $c_M(f,1)$}

Of special interest will be the classes $c_M(f,1)$ for $M\geq1$; in the next result, we describe them more explicitly. 

\begin{proposition} \label{c_M(1)-prop}
$c_M(f,1)=\iota_{K,M}\bigl({[y_{K,\p}]}_M\bigr)$.
\end{proposition}

\begin{proof} As is explained, \emph{e.g.}, in \cite[\S A.9.17]{BG}, there is a natural base change map
\begin{equation} \label{base-change-eq}
\CH^{k/2}_0(X/K)\longrightarrow\CH^{k/2}_0(X/K_1)^{\mathcal G_1}. 
\end{equation}
By composing \eqref{base-change-eq} with the (restriction of the) Abel--Jacobi map $\Phi_{K_1}$, we get a map
\[ \psi_{K,K_1}:\CH^{k/2}_0(X/K)\longrightarrow{H^1(K_1,T_\p)}^{\mathcal G_1}. \]
On the other hand, since the extension $K_1/\Q$ is solvable, by \cite[Lemma 3.10, (2)]{LV} one has $H^0(K_1,A_\p[p^m])=0$ for all $m\geq1$, so $H^0(K_1,T_\p)=\varprojlim_mH^0\bigl(K_1,A_\p[p^m]\bigr)=0$. 
It follows that restriction induces an isomorphism
\[ \res_{K_1/K}:H^1(K,T_\p)\overset\simeq\longrightarrow{H^1(K_1,T_\p)}^{\mathcal G_1}. \] 
Now it turns out (see, \emph{e.g.}, the proof of \cite[Proposition 6]{Charles}) that the composition
\[ \res_{K_1/K}^{-1}\circ\psi_{K,K_1}: \CH^{k/2}_0(X/K)\longrightarrow H^1(K,T_\p) \]
coincides with the Abel--Jacobi map. Since $\Lambda_\p(K)=\mathrm{im}(r_{K,\p})$ and $\Lambda_\p(K_1)=\mathrm{im}(r_{K_1,\p})$, we obtain a natural injection
$\Lambda_\p(K)\hookrightarrow{\Lambda_\p(K_1)}^{\mathcal G_1}$ given by $x\mapsto\res_{K_1/K}(x)$. 
This map, in turn, induces a map $\eta_{K,K_1}$ that is given by the composition
\[ \Lambda_{\p}(K)\big/p^M\Lambda_{\p}(K)\longrightarrow{\Lambda_{\p}(K_1)}^{\mathcal G_1}\big/p^M{\Lambda_{\p}(K_1)}^{\mathcal G_1}\longmono\bigl(\Lambda_{\p}(K_1)\big/p^M\Lambda_{\p}(K_1)\bigr)^{\mathcal G_1} \]
and fits into the commutative square
%\begin{equation} \label{AJ-square-eq}
\[ \xymatrix@R=32pt@C=38pt{\Lambda_{\p}(K)\big/p^M\Lambda_{\p}(K)\ar@{^(->}[r]^-{\iota_{K,M}}\ar[d]^-{\eta_{K,K_1}}&H^1\bigl(K,A_\p[p^M]\bigr)\ar[d]^-{\res_{K_1/K}}_-\simeq\\\bigl(\Lambda_{\p}(K_1)\big/p^M\Lambda_{\p}(K_1)\bigr)^{\mathcal G_1}\ar@{^(->}[r]^-{\iota_{K_1,M}}&H^1\bigl(K_1,A_\p[p^M]\bigr)^{\mathcal G_1}.} \]
%\end{equation}
Finally, one has $\eta_{K,K_1}\bigl({[y_{K,\p}]}_M\bigr)=\bigl[\res_{K_1/K}(y_{K,\p})\bigr]={[z_{1,\p}]}_M$, where the second equality is a consequence of formula \eqref{y_K-z_1-eq}, so
\[ \iota_{K,M}\bigl({[y_{K,\p}]}_M\bigr)=\res_{K_1/K}^{-1}\Bigl(\iota_{K_1,M}\bigl(\eta_{K_1/K}({[y_{K,\p}]}_M)\bigr)\!\Bigr)=\res_{K_1/K}^{-1}\bigl(\iota_{K_1,M}({[z_{1,\p}]}_M)\bigr)=c_M(f,1), \]
as desired. \end{proof}

\subsection{Towards a proof of Kolyvagin's conjecture: assumptions} \label{kol-proof-I-subsec}

Our goal in the next sections is to prove Conjecture \ref{kolyvagin-conj} for a large class of modular forms. As hinted at in the introduction, our strategy is based on a deformation-theoretic approach via Hida theory. 

\subsubsection{$\p$-ordinariness of $f$}

Recall that $f$ is \emph{${\p}$-ordinary} if $a_p(f)\in\cO_\p^\times$. Furthermore, the semisimplification $\bar\rho_{{\p}}^{ss}$ of $\bar\rho_{{\p}}$ is \emph{$p$-distinguished} if its restriction to $G_{\Q_p}$ can be put in the shape ${\bar\rho_{{\p}}^{ss}|}_{G_{\Q_p}}=\bigl(\begin{smallmatrix}\varepsilon_1&*\\0&\varepsilon_2\end{smallmatrix}\bigr)$ for characters $\varepsilon_1\not=\varepsilon_2$ (see, \emph{e.g.}, \cite[\S 2]{ghate}). Moreover, by \cite[Theorem 2.1, (a)]{ribet2}, if $p$ is sufficiently large (\emph{i.e.}, if $p$ lies outside a suitable finite set of prime numbers), then $\bar\rho_{f,{\p}}$ is irreducible for all ${\p}$ as above. 

\begin{proposition} \label{p-distinguished-prop}
If $f$ is ${\p}$-ordinary, then $\bar\rho_{{\p}}^{ss}$ is $p$-distinguished.
\end{proposition}

\begin{proof} This is a consequence of \cite[Theorem 2.1.4]{Wiles}: see, \emph{e.g.}, \cite[Lemma 4.12]{Vigni} for details (\emph{cf.} also \cite[\S 2.3]{LV-Iwasawa} for related computations). \end{proof}

\subsubsection{Assumptions} \label{kolyvagin-ass-subsubsec}

We work under the following assumption, which includes Assumptions \ref{p-ass} and \ref{ass-p-iso} stated before.

\begin{assumption} \label{main-assumption} \label{ass}
The pair $(f,\p)$ satisfies the following conditions: 
\begin{enumerate}
\item $N\geq 3$ is square-free;
\item \label{p-f-ass} $p\nmid 6ND_Fc_f$;
\item \label{weight-cong-ass} $k\equiv 2\pmod{2(p-1)}$; 
\item $f$ is $p$-isolated;
\item \label{ordinary-ass} $a_p(f)\in\cO_\p^\times$; 
\item \label{non-cong-ass} $a_p(f)\not\equiv1\pmod{{\p}}$;
\item $\rho_p$ has big image;  
\item $\bar\rho_{\mathfrak q}$ is irreducible and ramified for each prime $\mathfrak q$ of $F$ dividing $N$.
\end{enumerate}
\end{assumption}

See Definition \ref{p-isolated-def} for the notion of $p$-isolated form. By Proposition \ref{p-distinguished-prop}, if $f$ and ${\p}$ satisfy Assumption \ref{ass}, then $\bar\rho_{\p}^{ss}$ is $p$-distinguished.

\begin{remark}
In light of results of Serre on eigenvalues of Hecke operators (\cite[\S 7.2]{Serre-Cheb}), it seems reasonable to expect that condition \eqref{ordinary-ass}, which is an ordinariness property for $f$ at $\p$, holds for infinitely many $p$. In fact, questions of this sort appear to lie in the circle of ideas of the Lang--Trotter conjectures on the distribution of traces of Frobenius acting on elliptic curves (\cite{LT}) and of their extensions to higher weight modular forms (see, \emph{e.g.}, \cite{MM}, \cite{MMS}).
\end{remark}

\begin{remark}
It can be checked that $f$ is $p$-distinguished also as a consequence of the fact that, by condition \eqref{weight-cong-ass} in Assumption \ref{ass}, $k$ is not congruent to $1$ modulo $p-1$: see, \emph{e.g.}, \cite[Remark 7.2.7]{KLZ} and the reference therein.
\end{remark}

We record the following consequence of Assumption \ref{ass}. 

\begin{proposition} \label{no torsion lemma}
Let $L$ be a number field such that the extension $L/\Q$ is solvable.
\begin{enumerate}
\item The $\cO_{\p}$-module $H^1(L,T_{{\p}})$ is torsion-free.
\item The $\cO_{\p}$-modules $H^1_f(L,T_{{\p}})$ and $\Lambda_{{\p}}(L)$ are free of finite rank.
\end{enumerate} 
\end{proposition}

\begin{proof} As $p$ is a uniformizer for $\cO_{\p}$, in order to show that $H^1(L,T_\p)$ is torsion-free over $\cO_{\p}$ it is enough to check that the $p$-torsion of $H^1(L,T_{\p})$ is trivial. Since $T_{\p}$ is free (hence torsion-free) over $\cO_{\p}$, we can consider the short exact sequence of Galois modules  
\[ 0\longrightarrow T_{\p}\overset{p\cdot}\longrightarrow T_{\p}\longrightarrow T_{{\p}}\big/pT_{{\p}}=A_{{\p}}[p]\longrightarrow0, \]
where the first non-trivial arrow is the multiplication-by-$p$ map and the equality denotes a canonical identification. Passing to cohomology, we see that the $p$-torsion subgroup of $H^1(L,T_{{\p}})$ is a quotient of $H^0\bigl(L,A_{{\p}}[p]\bigr)$. On the other hand, $H^0\bigl(L,A_{{\p}}[p]\bigr)$ is trivial by \cite[Lemma 3.10, (2)]{LV}, and part (1) is proved. Finally, the $\cO_{\p}$-submodules $H^1_f(L,T_{{\p}})$ and $\Lambda_{{\p}}(L)$ of $H^1(L,T_{{\p}})$ are finitely generated, so part (2) follows from part (1). \end{proof}

\subsection{Hida families of modular forms} \label{hida-subsec}

We sketch the basics of Hida's theory of families of modular forms; see, \emph{e.g.}, \cite{hida86b}, \cite{hida86a}, \cite[Ch. 7]{hida-elementary} for details and proofs.

\subsubsection{$p$-stabilization of $f$}

Let $f\in S_k(\Gamma_0(N))$ be the newform fixed above. Let us write $f^\sharp(q)=\sum_{n\geq 1}a_n(f^\sharp)q^n\in S_k(\Gamma_0(Np))$ for the ordinary $p$-stabilization of $f$ (see, \emph{e.g.}, \cite[p. 410]{GS} or \cite[\S 2.4]{Vigni}). The cusp form $f^\sharp$ can be characterized as the unique (normalized) $p$-ordinary eigenform of weight $k$ and level divisible by $p$ with the property that $a_n(f^\sharp)=a_n(f)$ except for those $n$ divisible by $p$ (\cite[Lemma 3.3]{hida-measure}).

\subsubsection{Arithmetic primes}

Set $\Gamma\defeq 1+p\Z_p$, choose a finite extension $L$ of $\Q_p$ with valuation ring $\cO_L$ and form the Iwasawa algebra $\Lambda_L\defeq\cO_L[\![\Gamma]\!]$ of $\Gamma$ with coefficients in $\cO_L$; in the following, we will take $L=F_{\p}$. Let $A$ be a finitely generated commutative $\Lambda$-algebra. As in \cite[Definition 2.1]{Howard-Inv}, an $\cO_L$-algebra homomorphism $\kappa :A\rightarrow\bar\Q_p$ is said to be \emph{arithmetic} if the composition 
\[ \Gamma\longrightarrow A^\times\overset\kappa\longrightarrow\bar\Q_p^\times \]
with the canonical map $\Gamma\rightarrow A^\times$ has the form $\gamma\mapsto\psi(\gamma)\gamma^{k-2}$ for some integer $k\geq2$ and some finite order character $\psi$ of $\Gamma$. A prime ideal of $A$ that is the kernel of an arithmetic homomorphism is an \emph{arithmetic prime} of $A$; we write $\mathcal X^\arith(A)$ for the set of such primes. If $\wp$ is an arithmetic prime of $A$ and $A_\wp$ is the localization of $A$ at $\wp$, then the residue field $L_\wp\defeq A_\wp/\wp A_\wp$ is a finite extension of $L$. The composition $\Gamma\rightarrow A^\times\rightarrow L_\wp^\times$ has the form $\gamma\mapsto\psi_\wp(\gamma)\gamma^{r_\wp-2}$ for a finite order character $\psi_\wp:\Gamma\rightarrow L_\wp^\times$ and an integer $r_\wp\geq2$; we call $\psi_\wp$ and $r_\wp$ the \emph{wild character} and the \emph{weight} of $\wp$, respectively. 

\subsubsection{$p$-adic Hida family through $f$}

Let $\mathcal{R}$ denote the branch of the $p$-adic Hida family $\f$ passing through $f$ (or, rather, through $f^\sharp$; \emph{cf.} \cite[\S 2.1]{Howard-Inv}); we briefly explain the terminology, referring to, \emph{e.g.}, \cite{Howard-Inv} for details. The ring $\mathcal R$ is a complete local noetherian domain that is a finite flat $\Lambda_L$-algebra; we write $\m_{\mathcal{R}}$ for the maximal ideal of $\mathcal R$, $\F_{\mathcal R}\defeq\mathcal{R}/\m_\mathcal{R}$ for its residue field, which is finite of characteristic $p$, and $\mathcal{F}\defeq\mathrm{frac}(\mathcal{R})$ for its quotient field. Without loss of generality, we may assume that $\F_{\mathcal{R}}$ is equal to the residue field $\F_L\defeq\cO_L/\pi_L\cO_L$ of $L$, where $\cO_L$ is the valuation ring of $L$ and $\pi_L\in\cO_L$ is a uniformizer. As above, for every $\wp\in\mathcal X^\arith(\mathcal R)$ let $L_\wp\defeq\mathcal R_\wp/\wp\mathcal R_\wp$; moreover, set $s_\wp\defeq\max\bigl\{1,\mathrm{cond}(\psi_\wp)\bigr\}$, where $\mathrm{cond}(\psi_\wp)$ is the conductor of $\psi_\wp$. There exists a formal power series 
\[ \f=\sum_{n\geq1}a_n(\f)q^n\in\mathcal R[\![q]\!] \] 
such that
\begin{itemize}
\item for every $\wp=\ker(\kappa)\in\mathcal X^\arith(\mathcal R)$, the power series 
\[ f_\wp(q)\defeq\sum_{n\geq1}\kappa\bigl(a_n(\f)\bigr)q^n\in L_\wp[\![q]\!] \]
is the $q$-expansion of a cusp form of weight $r_\wp$, level $\Gamma_1(Np^{s_\wp})$ and character $\psi_\wp\omega^{2-r_\wp}$, where $\omega$ is the Teichm\"uller character;
\item there is $\wp _{f^\sharp}\in\mathcal X^\arith(\mathcal R)$ such that $f_{\wp_{f^\sharp}}=f^\sharp$.
\end{itemize}
Finally, denote by $\mathfrak{h}_N^\ord$ Hida's $p$-ordinary Hecke algebra of tame level $N$. The homomorphism $\mathfrak{h}^\ord_N\rightarrow\mathcal{R}$ associated with $\f$ corresponds to a minimal prime ideal $\mathfrak{a}$ of $\mathfrak{h}^\ord_N$, and then $\mathcal{R}$ is the integral closure of $\mathfrak{h}^\ord_N/\mathfrak{a}$ in its quotient field. In particular, $\mathcal R$ is a module over $\mathfrak h_N^\ord$.

\subsection{Big Galois representations} \label{big-galois-subsec}

As in \S \ref{notation-subsec}, let $\bar\Z$ be the ring of integers of $\bar\Q$ and choose a prime ideal $\mathfrak P$ of $\bar\Z$ such that $\mathfrak P\cap\cO_F=\p$. Let us denote by $F_g$ the Hecke field of a given eigenform $g$. Furthermore, notation being as in \S \ref{notation-subsec}, we write $V_g$ for the representation of $G_\Q$ attached to $g$ and the prime $\fP\cap F_g$, then let $V_g^\dagger=V_g(k/2)$ be the self-dual twist of $V_g$, where $k$ is the weight of $g$.

\subsubsection{The representation $\T$}

Let $\T$ denote the representation of $G_\Q$ attached to the Hida family $\f$ from \S \ref{hida-subsec}; this ``big'' Galois representation was constructed by Hida in \cite{hida86b} (\emph{cf.} also \cite[Proposition 2.1.2]{Howard-Inv}). Namely, for all $s\geq1$ let $J_1(Np^s)$ be the Jacobian variety of the (compact) modular curve $X_1(Np^s)$. By Albanese (\emph{i.e.}, covariant) functoriality, the degeneracy maps $X_1(Np^{s+1})\rightarrow X_1(Np^s)$ yield maps $J_1(Np^{s+1})\rightarrow J_1(Np^s)$, which in turn give maps $\Ta_p\bigl(J_1(Np^{s+1})\bigr)\rightarrow\Ta_p\bigl(J_1(Np^s)\bigr)$ between $p$-adic Tate modules. As in \ref{hida-subsec}, let $\mathfrak{h}_N^\ord$ be Hida's $p$-ordinary Hecke algebra of tame level $N$, then define
\begin{equation} \label{big-T-eq}
\T\defeq\biggl(\varprojlim_s\Bigl(\Ta_p\bigl(J_1(Np^s)\bigr)\otimes_{\Z_p}\cO_L\Bigr)^{\!\ord}\biggr)\otimes_{\mathfrak h_N^\ord}\mathcal R, 
\end{equation}
where the superscript ``ord'' indicates ordinary parts, which are cut out by Hida's ordinary projector (see, \emph{e.g.}, \cite[p. 551]{hida86b} and \cite[\S 7.2, Lemma 1]{hida-elementary}). The $\mathcal R$-module $\T$ is equipped with a natural action of $G_\Q$.

\subsubsection{Basic properties of $\T$}

Under standard assumptions on residual representations (\emph{cf.} \S \ref{residual-subsec}), $\T$ is a free $\mathcal R$-module of rank $2$. It satisfies the following crucial property: for every arithmetic prime $\wp$ of $\mathcal R$, the \emph{specialization} $\T_\wp/\wp\T_\wp$ \emph{of $\T$ at $\wp$} is equivalent over $\bar\Q_p$ (\emph{i.e.}, after a finite base change) to the dual (\emph{i.e.}, contragredient) representation $V_{f_\wp}^*$ of the $p$-adic representation $V_{f_\wp}$ of $G_\Q$ attached to $f_\wp$ (see, \emph{e.g.}, \cite[(1.5.5)]{NP}). The representation
\[ \rho_{\f}:G_\Q\longrightarrow\GL(\T)\simeq\GL_2(\mathcal R) \]
is unramified outside $Np$ and 
\[ \tr\bigl(\rho_{\f}(\Frob_\ell)\bigr)=a_\ell(\f) \]
for all prime numbers $\ell\nmid Np$, where $\Frob_\ell$ denotes the conjugacy class in $G_\Q$ of an arithmetic Frobenius at $\ell$ (see, \emph{e.g.}, \cite[Theorem 4.3]{CasHeeg}).

Now let $v$ be a place of $\bar\Q$ above $p$, let $G_v\subset G_\Q$ be the decomposition group at $v$ and let $I_v\subset G_v$ be the inertia subgroup; by \cite[Proposition 2.4.1]{Howard-Inv}, there is a short exact sequence of (left) $\mathcal{R}[G_v]$-modules 
\[ 0\longrightarrow F^+_v(\T)\longrightarrow\T\longrightarrow F^-_v(\T)\longrightarrow0 \] 
in which both $F^+_v(\T)$ and $F^-_v(\T)$ are free of rank $1$ over $\mathcal R$. The group $G_v$ acts on $F^-_v(\T)$ via the unamified character $\eta_v:G_{F_v}/I_v\rightarrow \mathcal{R}^\times$ taking the arithmetic Frobenius to $U_p$, while it acts on $F_v^+(\T)$ via $\eta_v^{-1}\varepsilon_\mathrm{cyc}[\varepsilon_\mathrm{cyc}]$, where $\varepsilon_\mathrm{cyc}:G_\Q\rightarrow\Z_p^\times$ is the $p$-adic cyclotomic character and $z\mapsto[z]$ denotes the inclusion $\Z_p^\times\hookrightarrow\Z_p[\![\Z_p^\times]{\!]}^\times$ of group-like elements (here recall that $\Z_p^\times\simeq\Bmu_{p-1}\times\Gamma$, where $\Bmu_{p-1}\subset\bar\Q_p^\times$ is the group of $(p-1)$-st roots of unity).

\subsection{Critical twist and residual representations} \label{residual-subsec}

Rather than in $\T$ itself, we will be interested in a suitable twist $\T^\dagger$ of $\T$. 

\subsubsection{Critical character and critical twist}

Let us fix a \emph{critical character} $\Theta:G_\Q\rightarrow\Lambda^\times$ as in \cite[Definition 2.1.3]{Howard-Inv}; with notation and terminology as in \cite{Howard-Inv}, in our case $k\equiv2\pmod{2(p-1)}$ and $j=0$, so only the ``wild'' part of $\Theta$ plays a role. We remark that the choice of $\Theta$ amounts to the choice of a square root of $\omega^{k-2}$, \emph{i.e.}, an integer $a$ modulo $p-1$ such that $2a=k-2$; let us fix such a choice once and for all (\emph{cf.} \cite[Remark 2.1.4]{Howard-Inv}). Let $\mathcal R^\dagger$ denote $\mathcal R$ viewed as a module over itself but with $G_\Q$ acting via $\Theta^{-1}$, then define th \emph{critical twist} of $\T$ to be 
\[ \T^\dagger\defeq\T\otimes_{\mathcal R}\mathcal R^\dagger. \]
The twist $\T^\dagger$ has the property that for every arithmetic prime $\wp$ of $\mathcal R$ of weight $k_\wp\equiv2\pmod{2(p-1)}$ and trivial character the specialization $\T^\dagger_\wp\big/\wp\T^\dagger_\wp$ of $\T^\dagger$ at $\wp$ is equivalent to $V^\dagger_{f_\wp}$ after a finite base change (see, \emph{e.g.}, \cite[(3.2.4)]{NP}). As a consequence, there is a specialization map
\[ \T^\dagger\longrightarrow\T^\dagger_\wp\big/\wp\T^\dagger_\wp\simeq V^\dagger_{f_\wp}, \]
which in turn induces specialization maps in cohomology. Summing up, $\T$ and $\T^\dagger$ enjoy the following interpolation properties, up to a finite base change:
\begin{itemize}
\item the specialization of $\T$ at an arithmetic prime $\wp$ of $\mathcal R$ is equivalent to $V_{f_\wp}^*$;
\item if $k_\wp\equiv2\pmod{2(p-1)}$, then the specialization of $\T^\dagger$ at $\wp$ is equivalent to $V_{f_\wp}^\dagger$.
\end{itemize}

\subsubsection{Residual representations}

Define
\[ \bar\T\defeq \T/\mathfrak m_\mathcal R\T=\T\otimes_{\mathcal R}\F_{\mathcal R}, \]
which is a two-dimensional representation of $G_\Q$ over $\F_{\mathcal R}$. As above, let $\wp$ be an arithmetic prime of weight $k_\wp\equiv2\pmod{2(p-1)}$ and trivial character; as in \S \ref{realizations-subsec}, let $T_{f_\wp}$ be the lattice realizing the $\fP$-adic representation attached to $f_\wp$ and let $T^\dagger_{f_\wp}$ be the self-dual twist of $T_{f_\wp}$. As explained, \emph{e.g.}, in \cite[\S 2.2]{Vigni}, we have reduced representations
\[ \bar\rho_{f_\wp}:G_\Q\longrightarrow\GL(\bar T_{f_\wp}),\quad\bar\rho^\dagger_{f_\wp}:G_\Q\longrightarrow\GL\bigl(\bar T^\dagger_{f_\wp}\bigr) \]
and their semi-simplifications
\[ \bar\rho^{\,ss}_{f_\wp}:G_\Q\longrightarrow\GL(\bar T^{\,ss}_{f_\wp}),\quad\bar\rho^{\dagger,ss}_{f_\wp}:G_\Q\longrightarrow\GL\bigl(\bar T^{\dagger,ss}_{f_\wp}\bigr). \]
It turns out that if $\wp$ and $\wp'$ are two arithmetic primes, then 
\begin{equation} \label{semi-simple-eq}
\bar\rho^{\,ss}_{f_\wp}\simeq\bar\rho^{\,ss}_{f_{\wp'}} 
\end{equation}
after a finite base change (\cite[p. 251]{hida86a}). If $\bar\rho_{f_\wp}$ (equivalently, $\bar\rho^{\,ss}_{f_\wp}$) is irreducible and $p$-distinguished for one (hence for every) arithmetic prime $\wp$, then 
\begin{itemize}
\item $\T$ is free of rank $2$ over $\mathcal R$ (\cite[Th\'eor\`eme 7]{mazur-tilouine});
\item $\bar\T\simeq\bar\rho_{f_\wp}$ after a finite base change for all such $\wp$ (see, \emph{e.g.}, \cite[Proposition 5.4]{LV-MM}).
\end{itemize}
Notice that property \eqref{semi-simple-eq} no longer holds unconditionally once $\bar\rho^{\,ss}_{f_\wp}$ and $\bar\rho^{\,ss}_{f_{\wp'}}$ have been replaced by $\bar\rho^{\dagger,ss}_{f_\wp}$ and $\bar\rho^{\dagger,ss}_{f_{\wp'}}$, respectively. However, it is true that if $\wp$ and $\wp'$ are arithmetic primes such that $k_\wp\equiv k_{\wp'}\equiv2\pmod{2(p-1)}$, then $\bar\rho^{\dagger,ss}_{f_\wp}\simeq\bar\rho^{\dagger,ss}_{f_{\wp'}}$ after a finite base change (\emph{cf.} \cite[Remark 2.7]{Vigni}).

Finally, set $\bar\T^\dagger\defeq \T^\dagger/\m_\mathcal R\T^\dagger$. An easy computation shows that $\Theta$ is trivial modulo $\m_\mathcal R$, so there is a canonical identification $\bar\T=\bar\T^\dagger$
of representations of $G_\Q$ over $\F_\mathcal R$. We write $\pi_\mathcal R:\T^\dagger\twoheadrightarrow\bar\T$ for the surjection determined by the identification between $\bar\T$ and $\bar\T^\dagger$, and
\begin{equation} \label{pi-R-L-eq}
\pi_{\mathcal R,L}:H^1\bigl(L,\T^\dagger\bigr)\longrightarrow H^1(L,\bar\T)
\end{equation} 
for the map in cohomology induced functorially by $\pi_{\mathcal R}$, where $L$ is a given number field.

\subsection{Abelian varieties and Kummer maps in weight $2$} \label{abelian-subsec}

Let $f_2$ be the specialization of $\f$ of weight $2$ and trivial character; the cusp form $f_2$ is a $p$-ordinary, $p$-stabilized newform (in the sense of \cite[Definition 2.5]{GS}) of level $Np$ and conductor either $N$ or $Np$. If $f_2$ has conductor $N$, then $f_2$ is the $p$-stabilization of a newform $g$ of weight $2$ and level $N$, otherwise we set $g\defeq f_2$. In both cases, denote by $N_g$ the level (\emph{i.e.}, conductor) of the newform $g$ and write $\sum_{n\geq1}a_n(g)q^n$ for the $q$-expansion of $g\in S_2(\Gamma_0(N_g))$.

\subsubsection{The abelian variety $A_g$}

As in \S \ref{big-galois-subsec}, let $F_g=\Q\bigl(a_n(g)\mid n\geq1\bigr)$ be the Hecke field of $g$; denote by $A_g$ the abelian variety over $\Q$ of conductor $N_g$ attached to $g$ via the Eichler--Shimura construction. Then $A_g$, which arises as a quotient of the Jacobian $J_0(N_g)$ of the modular curve $X_0(N_g)$, has dimension equal to the degree of $F_g$ and is of $\GL_2$-type. Furthermore, the endomorphism ring of $A_g$ is (isomorphic to) the ring of integers $\cO_g$ of $F_g$ and all endomorphisms of $A_g$ are defined over $\Q$.

\subsubsection{Galois representations attached to $g$}

Let $\mathfrak P$ be the prime ideal of $\bar\Z$ from \S \ref{big-galois-subsec}. Set $\bp\defeq \mathfrak P\cap\cO_g$, which is a prime of $F_g$ above $p$, and let $F_{g,\bp}$ be the completion of $F_g$ at $\bp$, whose valuation ring will be denoted by $\cO_{g,\bp}$. Let $\Ta_{\bp}(A_g)$ be the $\bp$-adic Tate module of $A_g$ and let $V_{\bp}(A_g)\defeq \Ta_{\bp}(A_g)\otimes\Q$ be the associated $F_{g,\bp}$-linear representation of $G_\Q$. If $V_{g,\bp}$ denotes the $F_{g,\bp}$-linear Galois representation attached to $g$ then there is an identification
\begin{equation} \label{weight-2-dual-eq}
V_{g,\bp}=H^1_{\text{\'et}}(A_g,F_{g,\bp})\simeq{V_{\bp}(A_g)}^*, 
\end{equation}
where ${V_{\bp}(A_g)}^*\defeq \Hom_{F_{g,\bp}}\bigl(V_{\bp}(A_g),F_{g,\bp}\bigr)$ is the $F_{g,\bp}$-linear dual of $V_{\bp}(A_g)$ equipped with its contragredient $G_\Q$-action (the isomorphism in \eqref{weight-2-dual-eq} follows by combining \cite[Theorem 15.1, (a)]{Milne-AV} with the $G_\Q$-equivariant splitting $\Ta_p(A_g)=\bigoplus_{\pi\mid p}\Ta_\pi(A_g)$, with $\pi$ varying over all the primes of $F_g$ above $p$). By taking the self-dual twist, we obtain
\begin{equation} \label{V-2-eq}
V_{g,\bp}^\dagger=V_{g,\bp}(1)\simeq V_{\bp}(A_g)^*(1)\simeq V_{\bp}(A_g), 
\end{equation}
where the rightmost isomorphism, which we fix once and for all, is a consequence of the Weil pairing. Let us choose a $G_\Q$-stable $\cO_{g,\bp}$-lattice $T_{g,\bp}\subset V_{g,\bp}$ whose Tate twist
\[ T^\dagger_{g,\bp}\defeq T_{g,\bp}\otimes_{\cO_{g,\bp}}\cO_{g,\bp}(1)\subset V^\dagger_{g,\bp} \] 
corresponds to $\Ta_{\bp}(A_g)$ under isomorphism \eqref{V-2-eq}. For any number field $L$, let 
\begin{equation} \label{xi-eq}
\xi_{g,L}:H^1\bigl(L,T^\dagger_{g,\bp}\bigr)\overset\simeq\longrightarrow H^1\bigl(L,\Ta_{\bp}(A_g)\bigr)
\end{equation}
be the isomorphism induced by \eqref{V-2-eq} functorially.

\begin{notation}
From here on, we drop dependence on $\bp$ from our notation and simply write $T^\dagger_g$ (respectively, $V^\dagger_g$) in place of $T^\dagger_{g,\bp}$ (respectively, $V^\dagger_{g,\bp}$).
\end{notation}

\subsubsection{Kummer maps}

Now let $A_g[\bp]$ be the $\bp$-torsion $\cO_g$-submodule of $A_g(\bar\Q)$. For any number field $L$ and every integer $M\geq1$ let 
\begin{equation} \label{pi-g-L-eq}
\pi_{g,L,M}:H^1\bigl(L,\Ta_{\bp}(A_g)\bigr)\longrightarrow H^1\bigl(L,A_g[\bp^M]\bigr) 
\end{equation}
be the map induced by the surjection $\Ta_{\bp}(A_g)\twoheadrightarrow A_g\bigl[\bp^M\bigr]$ and let 
\begin{equation} \label{pi-g-L-eq2}
\delta_{g,L}:A_g(L)\longrightarrow H^1\bigl(L,\Ta_{\bp}(A_g)\bigr),\quad\pi_{g,L,M}\circ\delta_{g,L}: A_g(K)\longrightarrow H^1\bigl(L,A_g[\bp^M]\bigr) 
\end{equation}
be the Kummer maps in Galois cohomology (see, \emph{e.g.}, \cite[Appendix A.1]{GP}). In turn, the second map gives an injection
\begin{equation} \label{delta-bar-eq}
\bar\delta_{g,L,M}:A_g(L)\big/\bp^M A_g(L)\longmono H^1\bigl(L,A_g[\bp^M]\bigr)
\end{equation}
that, in our context, is the weight $2$ counterpart of the map $\iota_{L,M}$ introduced in \eqref{iota_L-eq}. From now on, we shall simply set $\pi_{g,L}\defeq \pi_{g,L,1}$ and $\bar\delta_{g,L}\defeq \bar\delta_{g,L,1}$.

\subsection{Kolyvagin classes in weight $2$} \label{kolyvagin-class-2-subsec}

Let $g=\sum_{n\geq1}a_n(g)q^n\in S_2(\Gamma_0(N_g))$ be the newform from \S \ref{abelian-subsec}. Assume that $p$ splits in $K$, so that the imaginary quadratic field $K$ satisfies the Heegner hypothesis with respect to both $N$ and $Np$.

\subsubsection{Kolyvagin integers}

Let $v_{\bp}$ be the valuation of $F_{g,\bp}$ normalized by declaring that $v_{\bp}$ takes the value $1$ at a uniformizer of $\cO_{g,\bp}$. By analogy with the definition given in \S \ref{kolyvagin-integers-subsec}, using $v_{\bp}$ in place of $v_\p$, one can introduce the sets $\mathcal P_{\Kol}(g)$ of Kolyvagin primes and $\Lambda_{\Kol}(g)$ of Kolyvagin integers for the data $(g,\bp,K)$. Moreover, complex multiplication allows one to define Heegner points $\alpha_c\in A_g(K_c)$ indexed by integers $c\geq1$ coprime to $N$. These points arise by modularity from the Heegner points $x_c\in X_0(N)(K_c)$ appearing in \S \ref{cycles-subsec}. For details, the reader is referred to \cite{Gross}, \cite{Kol-Euler}, \cite{Kol-Log}.

\subsubsection{Weight $2$ Kolyvagin classes}

Using the points $\alpha_n$ instead of the cycles $y_{n,\p}$, and the maps $\bar\delta_{g,K_n,M}$ from \eqref{delta-bar-eq} in place of the maps $\iota_{K_n,M}$, the recipe in \S \ref{kolyvagin-classes-subsec} yields Kolyvagin cohomology classes 
\[ c_M(g,n)\in H^1\bigl(K,A_g[\bp^M]\bigr),\quad d_M(g,n)\in H^1\bigl(K_n,A_g[\bp^M]\bigr)^{\mathcal G_n} \]
for $n\in\Lambda_{\Kol}(g)$ and $M\leq M(g,n)$, and then a Kolyvagin set
%\begin{equation} \label{Kolsys-2}
\[ \kappa_{g,\infty}\defeq \bigl\{c_M(g,n)\mid n\in\Lambda_{\Kol}(g),\;1\leq M\leq M(g,n)\bigr\} \]
%\end{equation}
attached to $(g,\bp,K)$. This is the supply of Galois cohomology classes whose non-triviality was predicted (at least when $A_g$ is an elliptic curve) by Kolyvagin in \cite{kolyvagin-selmer} and then confirmed (in a stronger form, under some assumptions) by Zhang in \cite{zhang-selmer} and by Skinner--Zhang in \cite{SZ}. For more details on the construction of $\kappa_{g,\infty}$, see, \emph{e.g.}, \cite[\S 3.7]{zhang-selmer}.

\begin{remark}
By analogy with what was done in \S \ref{kolyvagin-classes-subsec}, we should write $\kappa^{(K)}_{g,\bp,\infty}$ instead of $\kappa_{g,\infty}$. However, since this family of Kolyvagin classes will play only an auxiliary role, we prefer to keep our notation as light as possible.
\end{remark}

\begin{remark}
The results of Zhang on Kolyvagin's conjecture for weight $2$ newforms have recently been extended by Sweeting in \cite{sweeting}.
\end{remark}

In particular, one has 
\begin{equation} \label{c-d-eq2}
c_M(g,n)=0\;\Longleftrightarrow\;d_M(g,n)=0.
\end{equation}
For later use, we prove

\begin{lemma} \label{integers-lemma}
$\Lambda_{\Kol}(g)=\Lambda_{\Kol}(f)$.
\end{lemma}

\begin{proof} Let $\ell$ be a prime number. Clearly, $v_\p(\ell+1)>0$ if and only if $v_{\bp}(\ell+1)>0$. On the other hand, there is a congruence
\begin{equation} \label{cong-fourier-eq}
a_\ell(f)\equiv a_\ell(g)\pmod{\mathfrak P} 
\end{equation}
(\emph{cf.} \cite[p. 251]{hida86a}), which immediately implies that $v_\p\bigl(a_\ell(f)\bigr)>0$ if and only if $v_{\bp}\bigl(a_\ell(g)\bigr)>0$. This shows that $\mathcal P_{\Kol}(f)=\mathcal P_{\Kol}(g)$, and the lemma is proved. \end{proof}

\subsection{Distinguished specialization maps} \label{distinguished-subsec}

Let $\wp_f$ be the arithmetic prime of $\mathcal R$ such that $f_{\wp_f}=f^\sharp$. Similarly, with notation as in \S \ref{abelian-subsec}, let $\wp_g$ be the arithmetic prime of $\mathcal R$ such that $f_{\wp_g}=g$. Let $\star\in\{f,g\}$. 

\subsubsection{Specializations and reductions}

Following Ota, we fix once and for all the $G_\Q$-equivariant specialization map
\begin{equation} \label{specialization-eq2}
\sspp^\star_0:\T^\dagger\longrightarrow T^\dagger_\star
\end{equation}
that is described in \cite[\S 2.6]{Ota-JNT}. This map factors through the surjection $\T^\dagger\twoheadrightarrow\T^\dagger/\wp_\star\T^\dagger$ and induces, possibly after a finite base change, an isomorphism
\begin{equation} \label{specialization-V-eq}
\sspp^\star_0:\T^\dagger_{\wp_\star}\big/\wp_\star\T^\dagger_{\wp_\star}\overset\simeq\longrightarrow V^\dagger_\star 
\end{equation}
of representations of $G_\Q$. Note that $\T^\dagger/\wp_\star\T^\dagger$ sits as an $\mathcal R/\wp_\star$-lattice inside $\T^\dagger_{\wp_\star}\big/\wp_\star\T^\dagger_{\wp_\star}$. As explained, \emph{e.g.}, in \cite[\S 5.1]{Vigni}, it is not restrictive to assume that isomorphism \eqref{specialization-V-eq} determines, up to finite base change (\emph{i.e.}, over $\bar\F_p$), an isomorphism
\begin{equation} \label{specialization-eq3}
\bsspp^\star_0:\bar\T\overset\simeq\longrightarrow\bar T_\star^\dagger
\end{equation}
of representations of $G_\Q$ that makes the square
\begin{equation} \label{specialization-square}
\xymatrix@C=35pt{\T^\dagger\ar[r]^-{\sspp^\star_0}\ar@{->>}[d]^-{\pi_{\mathcal R}}&T^\dagger_\star\ar@{->>}[d]\\
                   \bar\T\ar[r]^-{\bsspp^\star_0}_-\simeq&\bar T_\star^\dagger}
\end{equation}
commute (\emph{cf.} also \cite[\S 7.3]{KLZ}); here $\pi_{\mathcal R}$ is, as before, given by reduction modulo $\mathfrak m_\mathcal R$ and the right vertical arrow is the canonical surjection. 

\subsubsection{Specializations and cohomology}

By functoriality, for any number field $L$ the maps in \eqref{specialization-eq2} and \eqref{specialization-eq3} determine maps
\begin{equation} \label{specialization-eq5}
\sspp^\star_{0,L}:H^1(L,\T^\dagger)\longrightarrow H^1\bigl(L,T^\dagger_\star\bigr),\quad\bsspp^\star_{0,L}:H^1(L,\bar\T)\overset\simeq\longrightarrow H^1\bigl(L,\bar T^\dagger_\star\bigr),
\end{equation}
and then the square in \eqref{specialization-square} gives a commutative square
\begin{equation} \label{specialization-square2}
\xymatrix@R=32pt@C=37pt{H^1(L,\T^\dagger)\ar[r]^-{\sspp^\star_{0,L}}\ar[d]^-{\pi_{\mathcal R,L}}&H^1\bigl(L,T^\dagger_\star\bigr)\ar[d]^-{\varpi_{\star,L}}\\
                   H^1(L,\bar\T)\ar[r]^-{\bsspp^\star_{0,L}}_-\simeq&H^1\bigl(L,\bar T_\star^\dagger\bigr)}
\end{equation}
in Galois cohomology, where $\pi_{\mathcal R,L}$ is the map in \eqref{pi-R-L-eq} and $\varpi_{\star,L}$ is induced functorially by the surjection $T^\dagger_\star\twoheadrightarrow\bar T_\star^\dagger$.

\subsection{Big Heegner points in Hida families} \label{big-heegner-subsec}

Let $K$ be the imaginary quadratic field from \S \ref{kolyvagin-class-2-subsec}; the ring class fields of $K$ will be denoted as in \S \ref{cycles-subsec}. 

\subsubsection{Heegner points in towers of modular curves}

Following a recipe described by Castella in \cite{CasHeeg}, which refines a construction originally proposed by Howard in \cite{Howard-Inv}, we consider the tower of modular curves $\tilde X_s\defeq X_1(Np^s)$ for integers $s\geq1$. Let $n\geq1$ be an integer coprime to $Np$. As in \cite[\S 4.2]{CasHeeg}, one can define Heegner points $P_{n,s}\in\tilde X_s\bigl(\tilde{L}_{np^s}(\Bmu_{p^s})\bigr)$, where $\Bmu_{p^s}$ is the set of $p^s$-th roots of unity (in a fixed algebraic closure of $K$) and $\tilde{L}_{np^s}$ is the compositum of $K_{np^s}$ and the ray class field of $K$ of conductor $\mathcal{N}$; here $\mathcal{N}$ is the ideal of $\cO_K$ with $\cO_K/\mathcal{N}\simeq\Z/N\Z$ appearing in the construction of those Heegner points (\emph{cf.} \cite[\S 2.5]{CasHeeg}): we fix it as in \S \ref{cycles-subsec}. 

\subsubsection{Big Heegner points}

As in \S \ref{hida-subsec}, let $\mathfrak{h}_N^\ord$ be Hida's ordinary Hecke algebra of tame level $N$. Let $\Ta_p^\ord\bigl(\tilde J_s\bigr)$ be the ordinary part of the $p$-adic Tate module of the Jacobian variety 
$\tilde J_s$ of $\tilde X_s$ and define $\T_s^\dagger\defeq \Ta_p^\ord(\tilde J_s)\otimes_{\mathfrak{h}_N^\ord}\mathcal{R}^\dagger$. For any number field $L$ denote by $\mathfrak G_L$ the Galois group over $L$ of the maximal extension of $L$ unramified outside the primes above $Np$. Applying the twisted Kummer map defined in \cite[p. 101]{Howard-Inv} to $P_{n,s}$ gives rise to a cohomology class in $H^1\bigl(\mathfrak G_{\tilde{L}_{np^s}(\Bmu_{p^s})},\T^\dagger_s\bigr)$; inflating and then corestricting from $\tilde{L}_{np^s}(\Bmu_{p^s})$ to $K_n$, we get a class $\mathcal{P}_{n,s}\in H^1\bigl({K_{n}},\T^\dagger_s\bigr)$. These classes satisfy the compatibility relation 
\[ \alpha_{s,*}(\mathcal{P}_{n,s})=U_p\cdot\mathcal{P}_{n,s-1} \]
for all $s\geq2$, where $\alpha_s:\tilde X_s\rightarrow\tilde X_{s-1}$ is the natural covering map and the map $\alpha_{s,*}$ is induced functorially by $\alpha_s$ in cohomology. The \emph{big Heegner point of conductor $n$} is then
\begin{equation} \label{big-heegner-eq}
\mathfrak{X}_n\defeq\varprojlim_{s}U_p^{-s}\cdot\mathcal{P}_{n,s}\in H^1(K_n,\T ^\dagger); 
\end{equation}
observe that this expression does indeed make sense because the Hecke operator $U_p$ acts invertibly on ordinary submodules.

\subsection{Specializations of big Heegner points} \label{big-specialization-subsec}

Let $c\geq1$ be an integer coprime to $ND_Kp$. Let $\alpha_c$ and $y_c$ be the Heegner point and the Heegner cycle of conductor $c$ introduced in \S \ref{kolyvagin-class-2-subsec} and \S \ref{cycles-subsec}, respectively. Finally, let $\mathfrak X_c$ be the big Heegner point of conductor $c$ from \eqref{big-heegner-eq}. We identify (big) Heegner points and Heegner cycles with their images via the natural group homomorphisms
\[ H^1(L,M)\longrightarrow H^1\bigl(L,M\otimes\bar\Z_p\bigr),\quad H^1(L,M')\longrightarrow H^1\bigl(L,M'\otimes\bar\F_p\bigr) \]
where $M\in\bigl\{\Ta_{\bp}(A_g),T^\dagger_g,\T^\dagger,T^\dagger_f\bigr\}$, $M'\in\bigl\{A_g[\bp],\bar T^\dagger_g,\bar\T,\bar T^\dagger_f\bigr\}$ and $L$ is either $K$ or a ring class field of $K$. Recall the maps
\[ \sspp^f_{0,K_c}:H^1(K_c,\T^\dagger)\rightarrow H^1\bigl(K_c,T^\dagger_\star\bigr),\quad\xi_{g,K_c}\circ\sspp^g_{0,K_c}:H^1(K_c,\T^\dagger)\rightarrow H^1\bigl(K_c,\Ta_{\bp}(A_g)\bigr) \]
from \S \ref{abelian-subsec} and \S \ref{distinguished-subsec}. It is natural to consider $\sspp^f_{0,K_c}(\mathfrak X_c)$ and $\xi_{g,K_c}\bigl(\sspp^g_{0,K_c}(\mathfrak X_c)\bigr)$ and compare them to Heegner cycles and to Heegner points, respectively. The results we are interested in are due to Howard (in weight $2$, \cite{Howard-derivatives}), Castella (\cite{CasHeeg}) and Ota (\cite{Ota-JNT}). 

\subsubsection{Weight $2$ specialization: the case $N_g=N$} \label{weight-2-subsubsec1}

This is the case covered by the results of Castella (\cite{CasHeeg}) and of Ota (\cite{Ota-JNT}): see part (2) of Theorem \ref{c-o-thm}.

\subsubsection{Weight $2$ specialization: the case $N_g=Np$} \label{weight-2-subsubsec2}

In this case, the comparison between $\xi_{g,K_c}\bigl(\sspp^g_{0,K_c}(\mathfrak X_c)\bigr)$ and 
Heegner points is easier; it is carried out, albeit somewhat in disguise, by Howard in \cite[Section 3]{Howard-derivatives}. To explain this result, including the fact that the big Heegner points considered in \cite{CasHeeg} differ slightly from those originally defined in \cite{Howard-Inv}, we proceed as follows (the actual relation between Howard's classes and Castella's is immaterial for our goals, so we shall be quite brief and simply outline the arguments). 

Notation from \S \ref{abelian-subsec} is in force. For all $s\geq1$ set $X_{0,1}(N,p^s)\defeq X\bigl(\Gamma_0(N)\cap\Gamma_1(p^s)\bigr)$ and write $J_{0,1}(N,p^s)$ for the Jacobian variety of $X_{0,1}(N,p^s)$; there are obvious degeneracy maps $X_{0,1}(N,p^{s+1})\rightarrow X_{0,1}(N,p^s)$ that yield maps between $p$-adic Tate modules of Jacobians. Set
\[ \widetilde\T\defeq\biggl(\varprojlim_s\Bigl(\Ta_p\bigl(J_{0,1}(N,p^s)\bigr)\otimes_{\Z_p}\cO_L\Bigr)^{\!\ord}\biggr)\otimes_{\mathfrak h_N^\ord}\mathcal R. \]
This is the big Galois representation considered by Howard in \cite{Howard-Inv}; its critical twist $\widetilde\T^\dagger$ is defined as in \S \ref{residual-subsec}. The canonical degeneracy maps $X_1(Np^s)\rightarrow X_{0,1}(N,p^s)$ induce maps between $p$-adic Tate modules of Jacobians and a natural map $\T^\dagger\rightarrow\widetilde\T^\dagger$ of representations of $G_\Q$, where $\T$ is defined in \eqref{big-T-eq}. By functoriality, for any number field $\mathcal K$ we get a map
\begin{equation} \label{big-T-comparison-eq}
\Xi_{\mathcal K}:H^1(\mathcal K,\T^\dagger)\longrightarrow H^1\bigl(\mathcal K,\widetilde\T^\dagger\bigr). 
\end{equation}
Let
\[ \tilde\xi_{g,\mathcal K}\circ\widetilde\sspp^g_{0,\mathcal K}:H^1\bigl(\mathcal K,\widetilde\T^\dagger\bigr)\longrightarrow H^1\bigl(\mathcal K,\Ta_{\bp}(A_g)\bigr) \]
be the analogue of $\xi_{g,\mathcal K}\circ\sspp^g_{0,\mathcal K}$ for the representation $\widetilde\T^\dagger$; then
\begin{equation} \label{kummer-comp-eq}
\xi_{g,\mathcal K}\circ\sspp^g_{0,\mathcal K}=\tilde\xi_{g,\mathcal K}\circ\widetilde\sspp^g_{0,\mathcal K}\circ\Xi_{\mathcal K}. 
\end{equation}
Given $c\geq1$ coprime to $Np$, denote by $\mathfrak{X}^\mathrm{How}_c\in H^1\bigl(K_c,\widetilde\T^\dagger\bigr)$ the image under inflation of Howard's original big Heegner point of conductor $c$ (\cite[Definition 2.2.3]{Howard-Inv}). As remarked in the proof of \cite[Proposition 4.4]{CasHeeg}, the Heegner points used in \cite{CasHeeg}, which live on the tower of modular curves $X_1(Np^s)$, project to those from \cite{Howard-Inv}, which live on the tower of modular curves $X_{0,1}(N,p^s)$. On the other hand, the constructions in \cite{CasHeeg} and \cite{Howard-Inv} are compatible with projections, and then one can check that 
\begin{equation} \label{big-heegner-comp-eq}
\Xi_{K_c}(\mathfrak X_c)=\mathfrak{X}^\mathrm{How}_c, 
\end{equation}
where $\Xi_{K_c}$ is the map in \eqref{big-T-comparison-eq} with $\mathcal K=K_c$.

Now, for all $s\geq1$ embed $X_0(N_g)$ into its Jacobian $J_0(N_g)$ by sending the cusp $\infty$ to $0$. Since $g$ has trivial character, the degeneracy maps $X_1(N_g)\rightarrow X_{0,1}(N_g)\overset\vartheta\longrightarrow X_0(N_g)$ induce by covariant functoriality a commutative diagram
\[
\xymatrix@C=50pt@R=22pt{J_1(N_g)\ar[rd]\ar[d]&\\
J_{0,1}(N,p)\ar[r]\ar[d]^-{\vartheta_*}& A_g,\\
J_0(Ng)\ar[ru]}
\]
where the horizontal arrow is defined as the composition of the maps in the lower triangle. Denote by $\tilde{x}_{cp}\in X_{0,1}(N,p)\bigl(K_{cp}(\Bmu_p)\bigr)$ the Heegner point that appears in \cite[p. 98]{Howard-Inv} in the construction of $\mathfrak{X}^\mathrm{How}_c$. If $x_{cp}\in X_0(N_g)(K_{cp})$ is the Heegner point of conductor $cp$ from \S \ref{cycles-subsec} and \S \ref{kolyvagin-class-2-subsec}, then it is not restrictive to assume that $x_{cp}=\vartheta(\tilde x_{cp})$. It turns out that, since $g$ has trivial character, $\tilde\xi_{g,K_c}\bigl(\widetilde\sspp^g_{0,K_c}(\mathfrak{X}^\mathrm{How}_c)\bigr)$ is the image under the Kummer map of the trace 
\begin{equation} \label{big-trace-eq}
\sum_{\sigma\in\Gal(K_{cp}(\Bmu_p)/K_c)}\bigl([\tilde{x}_{cp}]-[\infty]\bigr)^\sigma 
\end{equation}
(\emph{cf.} \cite[eq. (7)]{Howard-Inv}). On the other hand, $x_{cp}$ is rational over $K_{cp}$, so the image in $J_0(N_g)$ via $\vartheta_\ast$ of the divisor class in \eqref{big-trace-eq} is 
\begin{equation} \label{big-trace-eq2}
(p-1)\cdot\sum_{\sigma\in\Gal(K_{cp}/K_c)}\bigl([x_{cp}]-[\infty]\bigr)^\sigma. 
\end{equation}
As explained, \emph{e.g.}, in \cite[\S2.4]{BD}, there is an equality
\[ \sum_{\sigma\in\Gal(K_{cp}/K_c)}\bigl([x_{cp}]-[\infty]\bigr)^\sigma=(U_p-1)\bigl([x_c]-[\infty]\bigr), \]
where $x_c\in X_0(N_g)(K_c)$ is the Heegner point of conductor $c$ introduced in \S \ref{cycles-subsec} and $U_p$ is the usual Hecke operator. It follows that the image of \eqref{big-trace-eq2} in $A_g$ is 
\[ \bigl(a_p(g)-1\bigr)\cdot(p-1)\cdot\alpha_c. \]
By part (\ref{non-cong-ass}) of Assumption \ref{ass}, $a_p(f)\not\equiv1\pmod{\p}$, so $a_p(g)\not\equiv1\pmod{\bp}$, by congruence \eqref{cong-fourier-eq}. Finally, combining all these facts with \eqref{kummer-comp-eq} for $\mathcal K=K_c$ and \eqref{big-heegner-comp-eq}, a computation shows that
\begin{equation} \label{weight-2-specialization-eq}
\xi_{g,K_c}\bigl(\sspp^g_{0,K_c}(\mathfrak X_c)\bigr)=d\cdot\delta_{g,K_c}(\alpha_c) 
\end{equation}
for a suitable $d=d(c)\in\bar\Z_p^\times$, where $\delta_{g,K_c}$ is the Kummer map from \eqref{pi-g-L-eq2}.

\subsubsection{The specialization theorem}

We keep notation from \eqref{pi-g-L-eq2} in force. The next result will play a crucial role in our proof of Kolyvagin's conjecture for $f$.

\begin{theorem}[Castella, Howard, Ota] \label{c-o-thm}
There exist $d=d(c)\in\bar\Z_p^\times$ and $e=e(c)\in\bar\Z_p^\times$ such that 
\begin{enumerate}
\item $\xi_{g,K_c}\bigl(\sspp^g_{0,K_c}(\mathfrak X_c)\bigr)=d\cdot\delta_{g,K_c}(\alpha_c)$;
%\vskip 1mm
\item $\sspp^f_{0,K_c}(\mathfrak X_c)=e\cdot y_{c,\p}$.
\end{enumerate}   
\end{theorem}

\begin{proof} Part (2) is \cite[Theorem 6.5]{CasHeeg}, while part (1) follows by the same arguments if $N_g=N$ (\emph{cf.} \cite[Remark 6.6]{CasHeeg}) and by the arguments sketched in    \S \ref{weight-2-subsubsec2} if $N_g=Np$ (\emph{cf.} equality \eqref{weight-2-specialization-eq}). Note that, in our setting, the constant appearing in \cite[Theorem 6.5]{CasHeeg} is a $p$-adic unit. See also \cite[Theorem 1.2]{Ota-JNT} for a refinement of the main result of \cite{CasHeeg} that works under our assumptions, which are slightly weaker than those in \cite{CasHeeg}. \end{proof}
 
\subsection{Big Kolyvagin classes} \label{big-kolyvagin-subsec}

For every integer $c\geq1$ coprime to $N$ let $\mathfrak X_c\in H^1(K_c,\T^\dagger)$ be the big Heegner point of conductor $c$ from \eqref{big-heegner-eq}. With notation as in \S \ref{kolyvagin-classes-subsec}, for every $n\in\Lambda_{\Kol}(f)$ define the \emph{big Kolyvagin class} of conductor $n$ as 
\begin{equation} \label{big-class-eq}
d(\f,n)\defeq \sum_{\sigma\in\mathcal G}\sigma\bigl(D_n(\mathfrak X_n)\bigr)\in H^1(K_n,\T^\dagger). 
\end{equation}
As we shall see in \S \ref{kolyvagin-proof-subsec}, the classes $d(\f,n)$ will be crucially used to prove, under suitable assumptions, Conjecture \ref{strong-kolyvagin-conj}. 

\subsection{Proof of Kolyvagin's conjecture} \label{kolyvagin-proof-subsec}

Let $\kappa^\st_{f,\infty}$ be the strict Kolyvagin set that was attached to $f$, $\p$, $K$ in \eqref{Kolsys-strict}. We are in a position to establish (under our running assumptions) Conjecture \ref{strong-kolyvagin-conj} for $f$; as a consequence, we deduce Conjecture \ref{kolyvagin-conj} for $f$, thus proving Theorem C.

\begin{theorem} \label{kolyvagin-main-thm}
There exists $n_0\in\Lambda_{\Kol}(f)$ such that $c_1(f,n_0)\neq0$. In particular, $\kappa^\st_{f,\infty}\not=\{0\}$.
\end{theorem}

\begin{proof} For any number field $L$, square \eqref{specialization-square2} gives a commutative diagram
\begin{equation} \label{commutative-eq}
\xymatrix@R=32pt@C=42pt{H^1\bigl(L,\Ta_{\bp}(A_g)\bigr)\ar[d]^-{\pi_{g,L}}&H^1\bigl(L,T^\dagger_g\bigr)\ar[d]^-{\varpi_{g,L}}\ar[l]^-\simeq_-{\xi_{g,L}}&H^1(L,\T^\dagger)\ar[r]^-{\sspp^f_{0,L}}\ar[l]_-{\sspp^g_{0,L}}\ar[d]^-{\pi_{\mathcal R,L}}&H^1\bigl(L,T^\dagger_f\bigr)\ar[d]^-{\varpi_{f,L}}\\H^1\bigl(L,A_g[\bp]\bigr)&H^1\bigl(L,\bar T^\dagger_g\bigr)\ar[l]^-\simeq_-{\bar\xi_{g,L}}&
                   H^1(L,\bar\T)\ar[l]_-{\bsspp^g_{0,L}}^-\simeq&H^1\bigl(L,\bar T_f^\dagger\bigr)\ar[l]_-{\bigl(\bsspp^f_{0,L}\bigr)^{-1}}^-\simeq} 
\end{equation}
in which $\xi_{g,L}$ is as in \eqref{xi-eq}, while the central and right horizontal maps are those appearing in \eqref{specialization-eq5} and the vertical ones are as in \eqref{pi-g-L-eq}, \eqref{pi-R-L-eq} and \S \ref{distinguished-subsec}. We remark that, although our notation does not reflect this, all characteristic $0$ (respectively, residual) representations in \eqref{commutative-eq} are base changed to $\bar\Z_p$  (respectively, $\bar\F_p$).

Recall from Lemma \ref{integers-lemma} that $\Lambda_{\Kol}(f)=\Lambda_{\Kol}(g)$. Let $n\in\Lambda_{\Kol}(f)$ and let $d(\f,n)\in H^1(K_n,\T^\dagger)$ be the big Kolyvagin class of conductor $n$ introduced in \eqref{big-class-eq}. By part (2) of Theorem \ref{c-o-thm}, $\sspp^f_{0,K_n}(\mathfrak X_n)=e\cdot y_{n,\p}$ for some $e\in\bar\Z_p^\times$. Since 
\begin{itemize}
\item there is a canonical identification $\bar T^\dagger_f=A^\dagger_f[p]$, 
\item the square
\[ \xymatrix@C=36pt@R=30pt{\Lambda_\p(K_n)\ar@{^{(}->}[r]\ar@{->>}[d]&H^1\bigl(K_n,T^\dagger_f\bigr)\ar^-{\varpi_{f,K_n}}[d]\\ \Lambda_\p(K_n)\big/p\Lambda_\p(K_n)\ar@{^{(}->}^-{\iota_{K_n,1}}[r]&H^1\bigl(K_n,A^\dagger_f[p]\bigr),} \]
where the top horizontal arrow is the set-theoretic inclusion, is commutative, 
\item all the maps in \eqref{commutative-eq} are Galois-equivariant,
\end{itemize}
it follows that
\[ \sspp^f_{0,K_n}\bigl(d(\f,n)\bigr)=\sum_{\sigma\in\mathcal G}\sigma\Bigl(D_n\bigl(\sspp^f_{0,K_n}(\mathfrak X_n)\bigr)\!\Bigr)=e\cdot\sum_{\sigma\in\mathcal G}\sigma\bigl(D_n(y_{n,\p})\bigr)=e\cdot z_{n,\p}, \]
whence
\begin{equation} \label{reduction-eq1}
\varpi_{f,K_n}\Bigl(\sspp^f_{0,K_n}\bigl(d(\f,n)\bigr)\!\Bigr)=\bar e\cdot\iota_{K_n,1}\bigl({[z_{n,\p}]}_1\bigr)=\bar e\cdot d_1(f,n)
\end{equation}
with $\bar e\in\bar\F_p^\times$. On the other hand, $\xi_{g,K_n}\bigl(\sspp^g_{0,K_n}(\mathfrak X_n)\bigr)=d\cdot\delta_{g,K_n}(\alpha_n)$ for some $d\in\bar\Z_p^\times$, by part (1) of Theorem \ref{c-o-thm}. Thus, if ${[\alpha_n]}_1$ denotes the image of $\alpha_n$ in $A_g(K_n)/\bp A_g(K_n)$, then
\[ \begin{split}
 \pi_{g,K_n}\Bigl(\xi_{g,K_n}\bigl(\sspp^g_{0,K_n}(\mathfrak X_n)\bigr)\!\Bigr)&=\bar d\cdot\pi_{g,K_n}\bigl(\delta_{g,K_n}(\alpha_n)\bigr)\\
 &=\bar d\cdot\bar\delta_{g,K_n}\bigl({[\alpha_n]}_1\bigr)
\end{split} \] 
with $\bar d\in\bar\F^\times_p$. By definition of $d_1(g,n)$, this immediately implies that
\begin{equation} \label{reduction-eq3}
\pi_{g,K_n}\bigg(\xi_{g,K_n}\Bigl(\sspp^g_{0,K_n}\bigl(d(\f,n)\bigr)\!\Bigr)\!\bigg)=\bar d\cdot d_1(g,n).
\end{equation}
For simplicity, set
\[ \psi_n\defeq \bar\xi_{g,K_n}\circ\bsspp^g_{0,K_n}\circ\,\bigl(\bsspp^f_{0,K_n}\bigr)^{-1}:H^1\bigl(K_n,\bar T_f^\dagger\bigr)\overset\simeq\longrightarrow H^1\bigl(K_n,A_g[\bp]\bigr). \]
In light of \eqref{reduction-eq1} and \eqref{reduction-eq3}, the commutativity of \eqref{commutative-eq} with $L=K_n$ ensures that
\begin{equation} \label{reduction-eq4}
   \psi_n\bigl(d_1(f,n)\bigr)=\bar e^{-1}\cdot\psi_n\bigg(\varpi_{f,K_n}\Bigl(\sspp^f_{0,K_n}\bigl(d(\f,n)\bigr)\!\Bigr)\!\bigg)=\bar e^{-1}\bar d\cdot d_1(g,n), 
\end{equation} 
with $\bar e^{-1}\bar d\in\bar\F_p^\times$. Now observe that, by \cite[Theorem 1.1]{zhang-selmer} (respectively, \cite[Theorem 1.3]{SZ}) if $N_g=N$ (respectively, $N_g=Np$), there is $n_0\in\Lambda_{\Kol}(g)=\Lambda_{\Kol}(f)$ such that $c_1(g,n_0)\neq0$. By \eqref{c-d-eq2}, it follows that $d_1(g,n_0)\not=0$, and then $d_1(f,n_0)\neq0$ by \eqref{reduction-eq4}. Finally, by \eqref{c-d-eq} we conclude that $c_1(f,n_0)\neq0$, so $\kappa^\st_{f,\infty}\not=\{0\}$. \end{proof}

\begin{remark} \label{non-ordinary-rem}
We expect that a similar deformation-theoretic strategy can be adopted to prove an analogue of Theorem \ref{kolyvagin-main-thm} for newforms of finite slope at $p$ (\emph{i.e.}, for newforms $g$ such that $a_p(g)\not=0$), replacing the results due to Castella and Ota on specializations of big Heegner points in Hida families with those by B\"uy\"ukboduk--Lei (\cite{BL}) and Jetchev--Loeffler--Zerbes (\cite{JLZ}) on the interpolation of Heegner points and Heegner cycles in Coleman families (\emph{cf.} \cite{PPV} for applications of the main results of \cite{BL} to the study of algebraic ranks, analytic ranks and Shafarevich--Tate groups when the modular forms they are attached to vary in a Coleman family).
\end{remark}

\begin{remark}
It is likely that Conjecture \ref{strong-kolyvagin-conj} can also be proved by directly mimicking the strategy proposed by Zhang in weight $2$ (\cite{zhang-selmer}), which would presumably allow one to get rid of the condition $k\equiv2\pmod{2(p-1)}$: see, \emph{e.g.}, \cite{wang} for partial results in this direction. Note, however, that this assumption on $k$ and $p$ would still appear in our main result on the $p$-part of the TNC for $\MM$ (Theorem B), as such a congruence is needed in the work by Skinner--Urban (\cite{SU}) that is crucial for our arguments. 
\end{remark}

\section{The $p$-part of the TNC for $\MM$}

Our goal is to prove, under suitable assumptions on $f$ and $p$, the $p$-part of the TNC for the motive $\MM$ when the analytic rank of $\MM$ is $1$; this will be done in \S \ref{main-proof-subsubsec}.

\subsection{Heegner modules} \label{heegner-modules-subsec}

We review the construction of Heegner cycles made in \cite{Zhang-heights} and compare it with that from \cite{Nek}. Recall that we use the same symbol for an algebraic cycle and for its class in the corresponding Chow group. We point out that the assumption that the level $N$ of $f$ be square-free is not needed in \S \ref{heegner-modules-subsec}, \S \ref{zhang-formula-subsec}, \S \ref{p-primary-Sha-subsec}. The freedom to work with newforms whose level is not necessarily square-free will be important in \S \ref{higher-subsec}, when collecting some of the arithmetic consequences of our main results.

\subsubsection{Zhang's cycles}

Fix an imaginary quadratic field $K$ in which all the primes dividing $Np$ split. Let $x_n\in X_0(N)$ be a Heegner point of conductor $n$ and let $\tilde{x}_n\in X(N)$ be the lift of $x_n$ that was chosen in \S \ref{cycles-subsec}. Recall the cycle $Z_k(\tilde{x}_n)$ from \eqref{cycle-eq0}. To begin with, Zhang considers in \cite[\S 2.4]{Zhang-heights} the $k/2$-codimensional cycle 
\begin{equation} \label{W-tilde-eq}
W_k(\tilde{x}_n)\defeq \sum_{g\in S_{k-2}} \mathrm{sgn}(g)\cdot g^*\Bigl(Z_k(\tilde{x}_n)^{(k-2)/2}\Bigr)\in\CH^{k/2}(X/K_n)
\end{equation} 
(recall that $X=\tilde\E_N^{k-2}$), where $\mathrm{sgn}(g)$ is the sign of the permutation $g$. Recall that the geometric Heegner cycle $\tilde{\Gamma}_n$ defined in \eqref{cycle-eq1} belongs to $\Pi_B\Pi_\epsilon\cdot{\CH^{k/2}(X/K_n)}$. 

\begin{lemma} \label{lemma Heegner 1} 
The equalities 
\begin{enumerate}
\item $\displaystyle{\Pi_\epsilon\cdot\Bigl(Z_k(\tilde{x}_n)^{(k-2)/2}\Bigr)=\frac{W_k(\tilde{x}_n)}{(k-2)!}}$,
\vskip 2mm
\item $\displaystyle{\tilde{\Gamma}_n=\frac{\Pi_B\cdot W_k(\tilde{x}_n)}{(k-2)!}}$
\end{enumerate}
hold in $\Pi_B\Pi_\epsilon\cdot{\CH^{k/2}(X/K_n)}$.
\end{lemma} 

\begin{proof} Part (2) is an immediate consequence of part (1) and the definitions, so we only need to check part (1). Recall from \S \ref{motive-modular-form-subsubsec} that the restriction of $\epsilon$ to the subgroup $\mathfrak{S}_{k-2}$ of the group $\Gamma_{k-2}$ in \eqref{Gamma-k-2-eq} is the sign character, so the projector associated with this restriction is $\Pi_\epsilon'\defeq \frac{1}{(k-2)!}\sum_{g\in S_{k-2}} \mathrm{sgn}(g)g^*$. On the other hand, as in the proof of \cite[Lemma 2.4.3]{Zhang-heights}, $\bigl((\Z/N\Z)^2\rtimes\Z/2\Z\bigr)^{k-2}$ acts on $Z_k(\tilde{x}_n)^{(k-2)/2}$ via the restriction of $\epsilon$ to $\bigl((\Z/N\Z)^2\rtimes\Z/2\Z\bigr)^{k-2}$. Now $\epsilon$ is trivial on $(\Z/N\Z)^{2(k-2)}$ and the product map on $(\Z/2\Z)^{k-2}$. It follows that the projector $\Pi_\epsilon''$ associated with the restriction of $\epsilon$ to $\bigl((\Z/N\Z)^2\rtimes\Z/2\Z\bigr)^{k-2}$ acts trivially on $W_k(\tilde{x}_n)$. Finally, $\Pi_\epsilon=\Pi_\epsilon'\cdot\Pi_\epsilon''$, so there are equalities
\[ \Pi_\epsilon\cdot Z_k(\tilde{x}_n)^{(k-2)/2}=\Pi_\epsilon'\cdot Z_k(\tilde{x}_n)^{(k-2)/2}=\frac{W_k(\tilde{x}_n)}{(k-2)!}, \]  
as desired. \end{proof}

\subsubsection{Base change maps and trace maps} \label{base-change-trace-subsubsec}

Let $L_2/L_1$ be a Galois extension of number fields and set $\mathscr G\defeq\Gal(L_2/L_1)$. From here up to \S \ref{heegner-module-subsubsec}, let us adopt the convention that $\CH^n_\emptyset$ stands for $\CH^n$. Let $\star\in\{\emptyset,0,\arith\}$. Denote by 
\begin{equation} \label{base-change-eq2}
\iota_{L_1\to L_2}:\CH^{k/2}_\star(X/L_1)\longrightarrow\CH^{k/2}_\star(X/L_2)^{\mathscr G} 
\end{equation}
the base change map from $L_1$ to $L_2$, a special case of which was introduced in the proof of Proposition \ref{c_M(1)-prop}. Galois trace induces a map
\begin{equation} \label{base-change-trace-eq}
\tr_{L_2/L_1}:\CH^{k/2}_\star(X/L_2)\longrightarrow\CH^{k/2}_\star(X/L_1);
\end{equation}
the composition $\tr_{L_2/L_1}\circ\iota_{L_1\to L_2}$ is equal to the multiplication-by-$[L_2:L_1]$ map, so the kernel of $\iota_{L_1\to L_2}$ is torsion, annihilated by $[L_2:L_1]$. Extending scalars in \eqref{base-change-eq2} and \eqref{base-change-trace-eq} to any ring $R$ in which $[L_2:L_1]$ is invertible, we get mutually inverse maps
\begin{equation} \label{base-change-eq3}
\xymatrix@C=40pt{{\CH^{k/2}_\star(X/L_1)}_R\ar@<-0.7ex>[r]_-{\iota_{L_1\to L_2}}&\CH^{k/2}_\star(X/L_2)_R^{\mathscr G}\ar@<-0.7ex>[l]_-{\tr_{L_2/L_1}}}.
\end{equation}
Therefore, the extension of scalars to such an $R$ of the map in \eqref{base-change-eq2} is an isomorphism. In particular, when $R=F$ we obtain a map
\begin{equation} \label{base-change-eq4} 
\CH^{k/2}_\star(X/L_1)\longrightarrow\!\!\xymatrix@C=40pt{{\CH^{k/2}_\star(X/L_1)}_F\ar[r]^-{\iota_{L_1\to L_2}}_-\simeq&\CH^{k/2}_\star(X/L_2)_F^{\mathscr G}} 
\end{equation}
whose kernel is the torsion subgroup of $\CH^{k/2}_\star(X/L_1)$.

\begin{notation/convention} \label{notation-convention}
We identify any non-torsion element of $\CH^{k/2}_\star(X/L_1)$ with its image under the map in \eqref{base-change-eq4} and use the same symbol for both cycle classes. Conversely, we identify any element of $\CH^{k/2}_\star(X/L_2)_F^{\mathscr G}$ with its image in ${\CH^{k/2}_\star(X/L_1)}_F$ under the trace map $\tr_{L_2/L_1}$ in \eqref{base-change-eq3}. In particular, since the kernel of the obvious map
\[ \CH^{k/2}_\star(X/L_2)^{\mathscr G}\longrightarrow\CH^{k/2}_\star(X/L_2)_F^{\mathscr G} \] 
is the torsion subgroup of the left-hand side term, we shall not distinguish between a non-torsion element of $\CH^{k/2}_\star(X/L_2)^{\mathscr G}$ and its image in ${\CH^{k/2}_\star(X/L_1)}_F$ via $\tr_{L_2/L_1}$. In a similar fashion, with our usual notation in force, a non-torsion subgroup of $\CH^{k/2}_\star(X/L_2)$ injects into ${\CH^{k/2}_\star(X/L_2)}_{\cO_\p}$, so that we shall freely identify it with its image in this $\cO_\p$-module.
\end{notation/convention}

\begin{remark}
The fact that $\ker(\iota_{L_1\to L_2})$ is torsion is a special case of an analogous result for smooth projective varieties and arbitrary field extensions: see, \emph{e.g.}, \cite[Lemma (1A.3)]{bloch-lectures}, \cite[p. 238]{Ramakrishnan} for details (the statement in \cite{bloch-lectures} is given only for the Chow group $\CH^2$ of a surface, but the arguments in the proof work for all groups $\CH^i$ of a variety).
\end{remark}

\subsubsection{The Heegner module of level $N$} \label{heegner-module-subsubsec}
 
As in \S \ref{heegner-cycles-subsubsec}, let $\pi_N:X(N)\rightarrow X_0(N)$ be the canonical degeneracy map; as in \S \ref{motive-modular-form-subsubsec}, put $t_N\defeq\#\bigl(\Gamma_0(N)/\Gamma(N)\bigr)$, then write $\pi_N^*(x_n)=\sum_{i=1}^{t_N}\tilde{x}_{n,i}$. Set
\begin{equation} \label{w-eq}
W_k(x_n)\defeq \Pi_B\cdot W_k(\tilde{x}_n)=\frac{1}{t_N}\cdot\sum_{i=1}^{t_N}W_k(\tilde{x}_{n,i}). 
\end{equation} 
As in \S \ref{derivatives-subsubsec}, let $\mathcal G_1\defeq \Gal(K_1/K)$. The \emph{Heegner module $\Heeg_{K,N}$ of level $N$} is the subgroup of ${\CH^{k/2}_\arith(X/K_1)}$ generated by $T_m\cdot W_k({x}_1^\sigma)$ for all $\sigma\in\mathcal G_1$ and all Hecke operators $T_m$ with $(m,N)=1$ (the cycles $T_m\cdot W_k({x}_1^\sigma)$ do indeed lie in ${\CH^{k/2}_\arith(X/K_1)}$, \emph{cf.} \cite[\S 3.1]{Zhang-heights}): we view $\Heeg_{K,N}$ as a subgroup of $\Pi_B\Pi_\epsilon\cdot\CH^{k/2}_0(X/K_1)$, which in turn should be thought of as a subgroup of ${\CH^{k/2}_0(X/K_1)}_\Q$ (\emph{cf.} \S \ref{motivic-subsubsec}).

Fix a prime number $p$ such that $p\nmid2N$. As was summarized in \S \ref{base-change-trace-subsubsec}, if $\p$ is a prime of $F$ above $p$, then there are natural maps
\begin{equation} \label{Heeg-composition-eq}
\Heeg_{K,N}^{\mathcal G_1}\longmono\CH^{k/2}_\arith(X/K_1)^{\mathcal G_1}\longrightarrow\!\xymatrix@C=37pt{\CH^{k/2}_\arith(X/K_1)_{\cO_\p}^{\mathcal G_1}\ar[r]^-{\tr_{K_1/K}}&{\CH^{k/2}_\arith(X/K)}_{\cO_\p},} 
\end{equation}
where the leftmost arrow is the set-theoretic inclusion and the middle one is extension of scalars. Thus, by composition with \eqref{Heeg-composition-eq}, the $\p$-adic Abel--Jacobi map $\AJ_{K,\p}$ yields a map 
\begin{equation} \label{AJ Heegner}
\AJ_{K,\p}:\Heeg_{K,N}^{\mathcal G_1}\longrightarrow\Lambda_{\p}(K),
\end{equation} 
which will be denoted by the same symbol (\emph{cf.} Remark \ref{AJ-factorization-rem}). Let us define the cycle 
\begin{equation} \label{xi K}
\XX_K\defeq\sum_{\sigma\in\mathcal{G}_1}W_k(x_1^\sigma)=\tr_{K_1/K}\bigl(W_k(x_1)\bigr)\in\CH^{k/2}_\arith(X/K), 
\end{equation}
where the second equality can be checked by unfolding the definition of the Galois action on Heegner cycles. 

\begin{remark} \label{AJ-twist-rem}
At some point in this article, we will need to consider also the counterpart of $\AJ_{K,\p}$ with $f$ replaced by the quadratic twist $f^K$. We will denote this map by $\AJ_{f^K,K',\p}$, where $K'$ is a suitable imaginary quadratic field. In this case, we will use the symbol $\mathcal G'_1$ for the analogue of the Galois group $\mathcal G_1$.
\end{remark}

In the statement below, $y_{K,\p}\in\Lambda_{\p}(K)$ is the cycle introduced in \eqref{y_K-eq}. 
 
\begin{lemma} \label{lemma Heegner 3}
$\AJ_{K,\p}(\XX_K)=(k-2)!\cdot y_{K,\p}$. 
\end{lemma}

\begin{proof} Recall the cycle $\Gamma_{1,\p}$ in \eqref{cycle-eq3} and set
\[ \Gamma_{K,\p}\defeq\tr_{K_1/K}(\Gamma_{1,\p})\in\Pi_B\Pi_\epsilon\cdot{\CH^{k/2}_0(X/K)}_{\cO_\p}. \]
By part (2) of Lemma \ref{lemma Heegner 1}, there is an equality $W_k(x_1)=(k-2)!\cdot\Gamma_{1,\p}$, which implies that
\begin{equation} \label{xi-K-Gamma-eq}
\XX_K=(k-2)!\cdot\Gamma_{K,\p} 
\end{equation}
in ${\CH^{k/2}_\arith(X/K)}_{\cO_\p}$. Combining equality \eqref{xi-K-Gamma-eq}, the fact that $y_{1,\p}=\AJ_{K_1,\p}(\Gamma_{1,\p})$ and the commutativity of \eqref{AJ-cores-eq}, we obtain 
\[ \AJ_{K,\p}(\XX_K)=(k-2)!\cdot\cores_{K_1/K}(y_{1,\p})=(k-2)!\cdot y_{K,\p}, \] 
as claimed. \end{proof}

\subsection{Zhang's formula of Gross--Zagier type} \label{zhang-formula-subsec}

We review the main result of \cite{Zhang-heights}, which is a counterpart of the Gross--Zagier formula (\cite[Theorem 6.3]{GZ}) for higher (even) weight modular forms. 

\subsubsection{Zhang's cycles with coefficients in $\R$}

With $W_k(\tilde x_n)$ as in \eqref{W-tilde-eq}, let us consider the $k/2$-codimensional cycle with real coefficients 
\begin{equation} \label{s-eq}
S_k(\tilde{x}_n)\defeq c\cdot W_k(\tilde{x}_n) 
\end{equation}
on $X$, where $c\in\R$ is a positive constant such that the self-intersection of $S_k(\tilde{x}_n)$ on each fiber is equal to $(-1)^{(k-2)/2}$ (\emph{cf.} \cite[\S 2.4]{Zhang-heights}). Recall that we write $D_K$ for the discrminant of $K$. Since the self-intersection of $Z_k(\tilde{x}_n)$ is $-2D_K$ (\emph{cf.} the proof of \cite[Proposition 5.1]{Nek}), a direct computation shows that the self-intersection of $W_k(\tilde{x}_n)$ is 
$(k/2-1)!^{\,2}\cdot(k-2)!\cdot(-2D_K)^{k/2-1}$ (we warmly thank Congling Qiu for calculating this value for us); it follows that 
\begin{equation}\label{c}
c=\frac{1}{(k/2-1)!\cdot\sqrt{(k-2)!}\cdot\bigl(\sqrt{-2D_K}\bigr)^{k/2-1}}.
\end{equation}
With notation from \S \ref{heegner-modules-subsec}, as in \cite[\S 4.1]{Zhang-heights} we define 
\begin{equation} \label{S_k-eq}
S_k(x_n)\defeq\frac{1}{\sqrt{\deg(\pi_N)}}\cdot\sum_{i=1}^{t_N}S_k(\tilde{x}_{n,i})\in{\CH^{k/2}_\arith(X/K_n)}_\R,
\end{equation}
where the cycles $S_k(\tilde{x}_{n,i})$ are defined as in \eqref{s-eq}. As above, the fact that $S_k(x_n)$ belongs to ${\CH^{k/2}_\arith(X/K_n)}_\R$ follows from \cite[\S 3.1]{Zhang-heights}.

\begin{lemma} \label{lemma Heegner 2}
The equality
\[ W_k(x_n)=\frac{S_k({x}_n)}{c\cdot\sqrt{\deg(\pi_N)}} \]
holds in $\Pi_B\Pi_\epsilon\cdot{\CH^{k/2}(X/K_n)}_\R$. 
\end{lemma}

\begin{proof} Let $\{\gamma_1,\dots,\gamma_{t_N}\}$ be a set of representatives of $\Gamma_0(N)/\Gamma(N)$ and for all $i=1,\dots,t_N$ set $\tilde{x}_{n,i}\defeq \gamma^*_i(x_n)$. By definition, there are equalities 
\[ W_k(x_n)=\Pi_BW_k(\tilde{x}_n)=\frac{1}{\deg(\pi_N)}\cdot\sum_{i=1}^{t_N}\gamma_i^*W_k(\tilde{x}_n)=\frac{1}{\deg(\pi_N)}\cdot\sum_{i=1}^{t_N}W_k(\tilde{x}_{n,i}). \] 
Since each $W_k(\tilde{x}_{n,i})$ has self-intersection $c$, the result follows. \end{proof}

Now we can prove

\begin{proposition} \label{prop Heegner}
The equality
\[ S_k(\tilde{x}_n)=\sqrt{\frac{\deg(\pi_N)\cdot\binom{k-2}{k/2-1}}{{(-2D_K})^{k/2-1}}}\cdot\tilde{\Gamma}_n \]
holds in $\Pi_B\Pi_\epsilon\cdot{\CH^{k/2}(X/K_n)}_\R$. 
\end{proposition}

\begin{proof} Immediate by combining Lemmas \ref{lemma Heegner 1} and \ref{lemma Heegner 2} with the expression for $c$ in \eqref{c}. \end{proof}

Let us consider the $\R$-linear extension 
\[ {\langle\cdot,\cdot\rangle}_{\GS}:{\CH^{k/2}_\arith(X/K_1)}_\R\times{\CH^{k/2}_\arith(X/K_1)}_\R\longrightarrow\R \]
of the Gillet--Soul\'e height pairing from \S \ref{GS-subsec}. Moreover, let $S_k(x_n)$ be the cycle from \eqref{S_k-eq}. Following \cite[\S 0.1]{Zhang-heights}, define $V$ to be the $\R$-subspace of ${\CH^{k/2}_\arith(X/K_1)}_\R$ that is generated by $T_m\cdot S_k({x}_1^\sigma)$ for all $\sigma\in\mathcal G_1$ and all Hecke operators $T_m$ with $(m,N)=1$. It follows that $V=\Heeg_{K,N}\otimes_\Z\,\R$ and $S_k(\tilde{x}_1)\in V$. Let $V'$ be the quotient of $V$ by the null subspace for ${\langle\cdot,\cdot\rangle}_\mathrm{GS}$. Then $V'$ is a subquotient of $S_k(\Gamma_0(N))^{h_K}$, where $h_K\defeq \#\mathcal G_1$ is the class number of $K$ (\cite[Theorem 0.3.1]{Zhang-heights}). Clearly, ${\langle\cdot,\cdot\rangle}_\mathrm{GS}$ yields a height pairing, to be denoted in the same way, on $V'$. Choose an orthonormal basis $\{f=f_1,\dots,f_t\}$ of $V'$ with respect to the Petersson inner product ${(\cdot,\cdot)}_{\Gamma_0(N)}$, so that $V'$ splits into $f_j$-eigencomponents $V'_{f_j}$. Now define $s'_{k,f}(x_1^\sigma)$ to be the image of $S_k(x_1^\sigma)$ in $V'_f$ (where, as above, $\sigma\in\mathcal G_1$) and put 
\[ s_f'\defeq\sum_{\sigma\in\mathcal G_1}s_{k,f}'(x_1^\sigma)\in V'_f, \] 
so that $s'_f$ is the image of $\sum_{\sigma\in\mathcal G_1}S_k(x_1^\sigma)$ in $V'_f$. Notice that, in fact, $s'_f\in (V'_f)^{\mathcal G_1}$.

\begin{remark} \label{s_f-rem}
If ${\langle\cdot,\cdot\rangle}_\mathrm{GS}$ is non-degenerate on $V$, then $V'=V$ and 
\[ s'_f\in\bigl(\Heeg_{K,N}^{\mathcal G_1}\otimes_\Z\,\R\bigr)[\theta_f]\subset{\CH^{k/2}_\arith(X/K)}_\R[\theta_f]. \]
This fact will play a role in the proof of Proposition \ref{coro zhang}.
\end{remark}

\subsubsection{Zhang's formula}

In the statement below, $u_K\defeq \#\cO_K^\times/2$. There is a splitting
\begin{equation} \label{splitting-L-eq2}
L(f/K,s)=L(f,s)\cdot L(f^K,s),
\end{equation}
which implies that $r_\an(f/K)\geq1$ (see, \emph{e.g.}, \cite[p. 543]{BFH}).

\begin{theorem}[S.-W. Zhang] \label{zhang-thm}
$L'(f/K,k/2)=\begin{displaystyle}\frac{2^{2k-1}\pi^k{(f,f)}_{\Gamma_0(N)}}{(k-2)!u_K^2\sqrt{|D_K|}}\end{displaystyle}\cdot{\langle s_f',s_f'\rangle}_{\GS}$.
\end{theorem}

\begin{proof} With notation as in \cite{Zhang-heights}, this follows immediately from \cite[Corollary 0.3.2]{Zhang-heights} upon taking $\chi$ to be the trivial character. \end{proof}

Recall the map $\AJ_{K,\p}$ from \eqref{AJ Heegner}; by a slight abuse of terminology, it will be interpreted as the restriction to $\Heeg_{K,N}^{\mathcal G_1}$ of the map in \eqref{AJ-map*}. The following result is implicit in \cite{Zhang-heights}; since we shall use it later, we provide a complete proof of it. 

\begin{proposition} \label{coro zhang}
\begin{enumerate}
\item Assume that $\AJ_{K,\p}$ is injective on $\Heeg_{K,N}^{\mathcal G_1}$. If $r_\an(f/K)=1$, then $y_{K,\p}$ is not $\cO_\p$-torsion.  
\item Assume that ${\langle\cdot,\cdot\rangle}_{\GS}$ is non-degenerate on $\Heeg_{K,N}\otimes_\Z\,\R$. If $r_\an(f/K)>1$, then $y_{K,\p}$ is $\cO_\p$-torsion. 
\end{enumerate}
\end{proposition}

\begin{proof} For simplicity, set $\CH\defeq\CH^{k/2}_\arith(X/K)$. We first show (1). By Theorem \ref{zhang-thm}, $s_f'$ is non-zero because $r_\an(f/K)=1$, and so $\sum_{\sigma\in\mathcal{G}_1}S_n(x_1^\sigma)$ is non-zero as well. Thanks to Proposition \ref{prop Heegner}, $\sum_{\sigma\in\mathcal{G}_1}\tilde{\Gamma}_1^\sigma$ is non-zero in $\Pi_B\Pi_\epsilon\cdot(\CH\otimes_\Z\,\R)$, hence in $\Pi_B\Pi_\epsilon\cdot(\CH\otimes_\Z\,R)$. Let $\XX_K$ be as in \eqref{xi K}. Part (2) of Lemma \ref{lemma Heegner 1} implies that $\XX_K=(k-2)!\sum_{\sigma\in\mathcal{G}_1}\tilde{\Gamma}_1^\sigma$, so $\XX_K$ is non-torsion in $\Heeg_{K,N}^{\mathcal G_1}$. Finally, part (1) follows from Lemma \ref{lemma Heegner 3} and the injectivity of $\AJ_{K,\p}$ on $\Heeg_{K,N}^{\mathcal G_1}$, which we are assuming. 

Now we prove (2), which is more delicate. In light of Lemma \ref{lemma Heegner 3}, we need equivalently to show that $\AJ_{K,\p}(\XX_K)$ is $\cO_\p$-torsion in $\Lambda_{\p}(K)$. If $\mathcal F\in\bigl\{\R,\R\cap\bar\Q,\bar\Q_p\bigr\}$, then there is a decomposition
\begin{equation} \label{isotypic-splitting-eq}
\Pi_B\Pi_\epsilon\cdot(\CH\otimes_\Z\,\mathcal F)=\bigoplus_g\Pi_B\Pi_\epsilon\cdot(\CH\otimes_\Z\,\mathcal F)[\theta_g], 
\end{equation}
where $g$ varies over all normalized newforms in $S_k(\Gamma_0(N))$ and, with notation as in \eqref{isotypic-eq}, $\Pi_B\Pi_\epsilon\cdot\bigl(\CH\otimes_\Z\,\mathcal F\bigr)[\theta_g]$ denotes the $g$-isotypic submodule of $\Pi_B\Pi_\epsilon\cdot\bigl(\CH\otimes_\Z\,\mathcal F\bigr)$; see, \emph{e.g.}, \cite[pp. 656--657]{Nek2} for details (strictly speaking, \cite{Nek2} deals solely with $\mathcal F=\bar\Q_p$, but what one only needs is that the field $\mathcal F$ contains all Hecke eigenvalues of all normalized newforms; when $\mathcal F=\bar\Q_p$, we are implicitly using an embedding $\bar\Q\hookrightarrow\bar\Q_p$, which allows us to view complex algebraic numbers as elements of $\bar\Q_p$). It is straightforward to check that there are equalities
\begin{equation} \label{chow-eq1}
\Pi_B\Pi_\epsilon\cdot\bigl(\CH\otimes_\Z\,\bar\Q_p\bigr) =\Bigl(\Pi_B\Pi_\epsilon\cdot(\CH\otimes_\Z\,\cO_\p)\Bigr)\otimes_{\cO_\p}\!\bar\Q_p 
\end{equation}
and
\begin{equation} \label{chow-eq2}
\Pi_B\Pi_\epsilon\cdot\bigl(\CH\otimes_\Z\,\bar\Q_p\bigr)[\theta_g] =\Bigl(\Pi_B\Pi_\epsilon\cdot(\CH\otimes_\Z\,\cO_\p)[\theta_g]\Bigr)\otimes_{\cO_\p}\!\bar\Q_p 
\end{equation}
for all $g$ as above. Combining \eqref{chow-eq1} and \eqref{chow-eq2}, and using the fact that the Abel--Jacobi map $\AJ_{K,\p}$ is Hecke-equivariant and $\Lambda_{\p}(K)$ is $f$-isotypic, it follows that $\AJ_{K,\p}\otimes\id_{\bar\Q_p}$ factors through the $f$-isotypic component as 
\[ \AJ_{K,\p}\otimes\id_{\bar\Q_p}:\Pi_B\Pi_\epsilon\cdot\bigl(\CH\otimes_\Z\,\bar\Q_p\bigr)\longepi\Pi_B\Pi_\epsilon\cdot\bigl(\CH\otimes_\Z\,\bar\Q_p\bigr)[\theta_f]\longrightarrow\Lambda_{\p}(K)\otimes_{\cO_\p}\!\bar\Q_p, \]
where the first map is the projection induced by \eqref{isotypic-splitting-eq} with $\mathcal F=\bar\Q_p$. There is a commutative diagram
\begin{equation} \label{chow-commutative-eq}
\xymatrix@C=32pt@R=25pt{
\Pi_B\Pi_\epsilon\cdot\bigl(\CH\otimes_\Z\,\bar\Q_p\bigr)\ar@{->>}[r]^-{\pi_{1,f}}\ar@/^1.9pc/[rr]^-{\AJ_{K,\p}\otimes\,\id_{\bar\Q_p}}&
\Pi_B\Pi_\epsilon\cdot\bigl(\CH\otimes_\Z\,\bar\Q_p\bigr)[\theta_f]\ar[r] &\Lambda_{\p}(K)\otimes_{\cO_\p}\!\bar\Q_p\\
\Pi_B\Pi_\epsilon\cdot\bigl(\CH\otimes_\Z\,(\R\cap\bar\Q)\bigr)\ar@{->>}[r]^-{\pi_{2,f}}\ar@{^{(}->}[d]^-{\iota_2}\ar@{^{(}->}[u]^-{\iota_1}&\Pi_B\Pi_\epsilon\cdot\bigl(\CH\otimes_\Z\,(\R\cap\bar\Q)\bigr)[\theta_f]\ar@{^{(}->}[d]^-{\iota_3}\ar@{^{(}->}[u]\\
\Pi_B\Pi_\epsilon\cdot(\CH\otimes_\Z\,\R)\ar@{->>}[r]^-{\pi_{3,f}}&\Pi_B\Pi_\epsilon\cdot(\CH\otimes_\Z\,\R)[\theta_f]
}
\end{equation} 
in which the horizontal surjections are the projections induced by \eqref{isotypic-splitting-eq} and the vertical injections are given by extension of scalars. Adopting the same symbol for $\XX_K$ and for its natural image in $\Pi_B\Pi_\epsilon\cdot\bigl(\CH\otimes_\Z\,(\R\cap\bar\Q)\bigr)$, we want to show that 
\begin{equation} \label{xi-K-trivial-eq}
\bigl((\AJ_{K,\p}\otimes\id_{\bar\Q_p})\circ\iota_1\bigr)(\XX_K)=0.
\end{equation}
Since $r_\an(f/K)>1$, Theorem \ref{zhang-thm} and the non-degeneracy of ${\langle\cdot,\cdot\rangle}_\mathrm{GS}$ imply that $s_f'=0$ in $\Pi_B\Pi_\epsilon\cdot(\CH\otimes_\Z\,\R)[\theta_f]$ (\emph{cf.} Remark \ref{s_f-rem}). Comparing \eqref{w-eq}, \eqref{xi K} and \eqref{s-eq} gives the equality
\[ \iota_2(\XX_K)=\frac{\Pi_B}{c}\cdot\sum_{\sigma\in\mathcal G_1}S_k(x_1^\sigma) \]
in $\Pi_B\Pi_\epsilon\cdot(\CH\otimes_\Z\,\R)$, and then $\pi_{3,f}\bigl(\iota_2(\XX_K)\bigr)=(\Pi_B/c)\cdot s_f'=0$. Now the commutativity of the lower square in \eqref{chow-commutative-eq} and the injectivity of $\iota_3$ yield $\pi_{2,f}(\XX_K)=0$, while the commutativity of the upper square in \eqref{chow-commutative-eq} shows that $\pi_{1,f}\bigl(\iota_1(\XX_K)\bigr)=0$. This clearly implies \eqref{xi-K-trivial-eq}.

Keeping \eqref{chow-eq1} in mind, \eqref{xi-K-trivial-eq} shows that the natural image of $\AJ_{K,\p}(\XX_K)\in\Lambda_{\p}(K)$ in $\Lambda_{\p}(K)\otimes_{\cO_\p}\!\bar\Q_p$ is trivial. Finally, the map $\Lambda_{\p}(K)\otimes_{\cO_\p}\!F_\p\rightarrow\Lambda_{\p}(K)\otimes_{\cO_\p}\!\bar\Q_p$ is injective, so the image of $\AJ_{K,\p}(\XX_K)$ in $\Lambda_{\p}(K)\otimes_{\cO_\p}\!F_\p$ is trivial, which means that $\AJ_{K,\p}(\XX_K)$ is $\cO_\p$-torsion in $\Lambda_{\p}(K)$, as desired. \end{proof}

\begin{remark}
Unfortunately, while it is natural to impose a non-degeneracy condition like that in part (2) of Proposition \ref{coro zhang} when studying the arithmetic of Heegner cycles (see, \emph{e.g.}, \cite[Assumption 4.1]{Xue}), we are not aware of any result in this direction (except for weight $2$ modular forms, which are not considered in this paper).
\end{remark}

\subsection{Periods of modular forms} \label{section-PMF} 

We begin by connecting the periods $\Omega_\infty^{(\gamma)}$ from \S \ref{periodmap} to those appearing in the work of Vatsal (\cite{Vatsal}) and of Skinner--Urban (\cite{SU}). To do this, we clarify the relation between these periods and modular symbols. In what follows, we set
\begin{equation} \label{epsilon-eq}
\epsilon\defeq(-1)^{\frac{k-2}{2}}.
\end{equation}

\subsubsection{Modular symbols} \label{modular-symbols-subsubsec}

Let $R$ be a commutative ring, let $k\geq2$ be an integer, set $n\defeq k-2$ and let $L_n(R)\defeq \Sym^n(R)$ be the $R$-module of homogeneous polynomials of degree $n$ in the variables $X$ ad $Y$ with coefficients in $R$. We write $L_n(R)$ also for the corresponding locally constant sheaf on the open  modular curve $Y_\Gamma$ of level $\Gamma$, where $\Gamma\in\bigl\{\Gamma(N),\Gamma_1(N),\Gamma_0(N)\bigr\}$. 

Let 
\[ \Symb_{\Gamma}\bigl(L_n(R)\bigr)\defeq \Hom_{\Gamma}\bigl(\mathcal{D}_0,L_n(R)\bigr) \]
be the group of $\Gamma$-invariant $L_n(R)$-valued \emph{modular symbols}, where $\mathcal{D}_0$ is the group of degree $0$ divisors on $\PP^1(\Q)$ equipped with its left action of $R[\Sigma]$  
(\cite[Definition 4.6]{GS}). We also let 
\[ \Bound_{\Gamma}\bigl(L_n(R)\bigr)\defeq \Hom_{\Gamma}\bigl(\mathcal{D},L_n(R)\bigr) \] 
be the group of $\Gamma$-invariant $L_n(R)$-valued \emph{boundary symbols}, where $\mathcal{D}$ is the group of divisors on $\PP^1(\Q)$ equipped with the natural left action of $R[\Sigma]$. There is an exact sequence 
%\begin{equation} \label{symb}
\[ 0\longrightarrow H^0\bigl(\Gamma,L_n(R)\bigr)\longrightarrow\Bound_{\Gamma}\bigl(L_n(R)\bigr)\longrightarrow\Symb_{\Gamma}\bigl(L_n(R)\bigr)\longrightarrow H^1_\mathrm{par}\bigl(\Gamma,L_n(R)\bigr)\longrightarrow 0. \]
%\end{equation}
As before, $H_\mathrm{par}^1\bigl(\Gamma,L_n(R)\bigr)$ denotes parabolic cohomology of the open modular curve $Y(\Gamma)$ of level $\Gamma$ with coefficients in the locally constant sheaf associated with $L_n(R)$ (\emph{i.e.}, the image of the compact cohomology group $H^1_\mathrm{cpt}\bigl(\Gamma,L_n(R)\bigr)=H^1_\mathrm{cpt}\bigl(Y(\Gamma),L_n(R)\bigr)$ in $H^1\bigl(\Gamma,L_n(R)\bigr)=H^1\bigl(Y(\Gamma),L_n(R)\bigr)$). Using notation from \cite[\S 4]{GS}, the map $\Symb_{\Gamma}\bigl(L_n(R)\bigr)\rightarrow H^1_\mathrm{par} \bigl(\Gamma,L_n(R)\bigr)$ sends $\Phi$ to the cohomology class represented by the $1$-cocycle $\gamma\mapsto\Phi\bigl(\{\gamma(c)\}-\{c\}\bigr)$. 

Write $\mathfrak H_k(\Gamma)$ for the Hecke algebra acting on modular forms of level $\Gamma$ and weight $k$ (in particular, for $\Gamma=\Gamma(N)$ we recover the algebra $\mathfrak H_k(\Gamma(N))$ from \S \ref{subsecmot}). By \cite[Theorem 4.2]{GS}, there is an $\mathfrak H_k(\Gamma)$-equivariant isomorphism
\begin{equation} \label{symb-isom-eq}
\Symb_{\Gamma}\bigl(L_n(R)\bigr)\simeq H^1_\mathrm{cpt}\bigl(\Gamma,L_n(R)\bigr) 
\end{equation}
such that the action of complex conjugation on $H^1_\mathrm{cpt}\bigl(\Gamma,L_n(R)\bigr)$ corresponds to the action of the matrix $\iota\defeq \smallmat{-1}{0}{0}{1}$ on $\Symb_{\Gamma}\bigl(L_n(R)\bigr)$. The map 
\begin{equation} \label{Theta-eq}
\Theta:H^1 \bigl(\Gamma,L_n(R)\bigr)\longrightarrow\Symb_{\Gamma}\bigl(L_n(R)\bigr) 
\end{equation}
defined by $\omega\mapsto \eta_q\cdot\Phi$, where $\Phi$ is any lift of $\omega$, $T_q$ is the Hecke operator at a prime number $q\equiv1\pmod{Np}$ and 
\begin{equation} \label{eta-q-eq}
\eta_q\defeq T_q-(q+1),
\end{equation}
is independent of the choice of $\Phi$ because $\eta_q$ kills $\Bound_{\Gamma}\bigl(L_n(R)\bigr)$, \emph{i.e.}, $\eta_q\cdot x=0$ for all 
$x\in\Bound_{\Gamma}\bigl(L_n(R)\bigr)$ (\cite[\S 1.6]{Vatsal}). Furthermore, the map $\Theta$ is $\mathfrak H_k(\Gamma)$-equivariant and also equivariant for the action of complex conjugation on $H^1 \bigl(\Gamma,L_n(R)\bigr)$ and for the action of $\iota$ on $\Symb_{\Gamma}\bigl(L_n(R)\bigr)$. We also write
\begin{equation} \label{Theta-eq2}
\Theta:H^1 \bigl(\Gamma,L_n(R)\bigr)\longrightarrow H^1_\mathrm{cpt}\bigl(\Gamma,L_n(R)\bigr)
\end{equation}
for the composition of \eqref{symb-isom-eq} and \eqref{Theta-eq}.

If $R=\C$, then the modular symbol $\Phi_{f,\Gamma}$ associated with $f$ is given by 
\[ \Phi_{f,\Gamma}\bigl(\{a\}-\{b\}\bigr)\defeq2\pi i\cdot\int_b^af(z)\cdot(zX+Y)^{k/2-1}dz, \]
where the integral is computed along a geodesic path (with respect to the Poincar\'e metric) from $b$ to $a$ in the complex upper half-plane. Split the $\C$-vector space $\Symb_{\Gamma}\bigl(L_n(\C)\bigr)$ into $\pm$-eigenspaces $\Symb_{\Gamma}^\pm\bigl(L_n(\C)\bigr)$ for complex conjugation, then denote by $\Phi_{f,\Gamma}^\pm$ the projections of $\Phi_{f,\Gamma}$ to the respective eigenspaces. %Set $\Phi_f\defeq\Phi_{f,\Gamma_0(N)}$ and $\Phi_f^\pm\defeq\Phi^\pm_{f,\Gamma_0(N)}$.
Observe, in particular, that there is an isomorphism 
\begin{equation} \label{modular-isom-eq}
V_\B(-k/2)\otimes_F\C\simeq H^1 \bigl(\Gamma,L_n(\C)\bigr)[\theta_f]=\C\cdot\Phi_{f,\Gamma}^+\oplus\C\cdot\Phi_{f,\Gamma}^- 
\end{equation}
satisfying the following property: an element of $V_\B(-k/2)\otimes_F\C$ lies in the $+$-eigenspace for complex conjugation if and only if its image under \eqref{modular-isom-eq} lies in $\C\cdot\Phi_{f,\Gamma}^\epsilon$, with $\epsilon\in\{\pm1\}$ as in \eqref{epsilon-eq}. The comparison isomorphism between Betti and de Rham realizations gives rise to an isomorphism 
\[ \Comp_{\B,\dR}:\Symb_\Gamma\bigl(L_n(\C)\bigr)\overset\simeq\longrightarrow V_\dR\otimes_F\C \]
that is equivariant with respect to the action of complex conjugation. Define 
\begin{equation} \label{varphi-Gamma-eq}
\varphi_{f,\Gamma}\defeq\Comp_{\B,\dR}(\Phi_{f,\Gamma}). 
\end{equation}
%then set $\varphi_f\defeq\varphi_{f,\Gamma_0(N)}$ and $\varphi_f^\pm\defeq\varphi_{f,\Gamma_0(N)}^\pm$. 
Note that the generator $\omega_{f}$ of $\mathrm{Fil}^{k/2-1}(V_\dR)$ from \eqref{omega_f-eq} is sent to $(2\pi i)^{-1}\cdot\varphi_{f,\Gamma(N)}$ by the comparison isomorphism.

\subsubsection{The periods $\Omega_{f,\Gamma}^\pm$} \label{periods-subsubsec}

We need to extend the definition of periods given in \S \ref{periodmap}; we find it more convenient to work integrally throughout. For this, define 
\[ \T_{\mathrm{par},\Gamma}\defeq\im\Bigl(H^1 \bigl(\Gamma,L_n(\cO_F)\bigr) \overset{j}\longrightarrow H^1 \bigl(\Gamma,L_n(F)\bigr) \Bigr)\]
%
%\[ \tilde{T}_{\mathrm{par},\Gamma}\defeq\im\Bigl(H^1 \bigl(\Gamma,L_n(\cO_F)\bigr)[\theta_f]\overset{j}\longrightarrow H^1 \bigl(\Gamma,L_n(F)\bigr)[\theta_f]\Bigr),\quad T_{\B,\Gamma}\defeq\tilde{T}_{\B,\Gamma}(k/2), \]
where $j$ is the natural map. With the usual $\pm$-notation, pick $\delta_{f,\Gamma}^\pm\in\T_{\mathrm{par},\Gamma}^\pm[\theta_f]\smallsetminus\{0\}$ and set
\[ \eta_{f,\Gamma}^\pm\defeq(\Comp_{\B,\dR}\circ\,\Theta)\bigl(\delta^{\pm}_{f,\Gamma}\bigr); \]
% and note that, in particular, $\gamma_{f,\Gamma}^\pm\defeq(2\pi i)^{k/2}\cdot\delta_{f,\Gamma}^\pm$ is a non-zero element of $T_{\B,\Gamma}$. 
here $\Theta$ is the map in \eqref{Theta-eq} for $R=F$, which we view $\Theta$ as taking values in $\Symb_\Gamma\bigl(L_n(\C)\bigr)$ by means of the distinguished embedding $\iota_F:F\hookrightarrow\R$. Define periods $\Omega_{f,\Gamma}^\pm\in\C$ via the equality
\[ \varphi_{f,\Gamma}=\Omega_{f,\Gamma}^+\cdot\eta_{f,\Gamma}^++\Omega_{f,\Gamma}^-\cdot\eta_{f,\Gamma}^-, \]
where $\varphi_{f,\Gamma}$ was introduced in \eqref{varphi-Gamma-eq}.
%In addition, set also
%\begin{equation} \label{periodsdef}
%\Omega_{f,\Gamma}\defeq\Omega_{f,\Gamma}^\epsilon,\quad\Omega_{f}^\pm\defeq\Omega_{f,\Gamma_0(N)}^\pm. 
%\end{equation}
%In line with this notation, we put $\Omega_f\defeq\Omega_f^\epsilon$. 
Consider the module index
\[ \mathfrak{a}^\pm_{f,\Gamma}\defeq\bigl[\T_{\mathrm{par},\Gamma}^\pm:\delta_{f,\Gamma}^\pm\cdot\cO_F\bigr] \]
as defined, \emph{e.g.}, in \cite[p. 10]{frohlich}; by \cite[\S 3, Proposition 1, (ii)]{frohlich}, $\mathfrak{a}_{f,\Gamma}^\pm$ is a non-zero (integral) ideal of $\cO_F$. 

\begin{remark}
The periods $\Omega_{f,\Gamma}^\pm$ depend on our choice of $\delta^\pm_{f,\Gamma}$, but the products $\Omega_{f,\Gamma}^\pm\cdot\mathfrak{a}_{f,\Gamma}^\pm$, where we see $\mathfrak{a}_{f,\Gamma}^\pm$ as an $\cO_F$-submodule of $\R$ via the (set-theoretic) inclusion $\iota_F$, do not. 
\end{remark}

%Put $\gamma_f=(2\pi i )^{k/2}\cdot \eta_f$. 
%Define the period $\Omega^{(\gamma_f)}_{f}$ by the equation 
%$\omega_f=\Omega^{(\gamma_f)}_f\cdot\gamma_f$. 
%Now $\omega_f=(2\pi i)^{k-1}\varphi_f$, so we have 
%$(2\pi i )^{k-1}\varphi_f=(2\pi i )^{k/2}\Omega^{(\gamma_f)}_\infty\cdot\eta_f$, which shows that $\Omega_f^{(\gamma_f)}=(2\pi i)^{k/2}\Omega_f$.
%
%\[\eta_f(2\pi i )^{k/2}=\Omega^{\gamma_f}_\infty\cdot\omega_f^*.\] 
%Now $\omega_f=(2\pi i )^{k-1}\varphi_f$. So 
%\[\eta_f(2\pi i )^{k/2}=(2\pi i)^{k-1}\Omega^{\gamma_f}_\infty\cdot\varphi_f^*
%\Longrightarrow \eta_f=(2\pi i)^{k/2-1}\Omega^{\gamma_f}_\infty\cdot\varphi_f^*.\] 
%The period $\Omega_f$ is defined by 
%\[\eta_f=\Omega_f\cdot\varphi_f^*.\]
%So $\Omega_f=(2\pi i)^{k/2-1}\Omega^{\gamma_f}_\infty$, the opposite of what we want! 

\subsubsection{Comparison of periods} \label{comparison-subsubsec}

Now we relate the period $\Omega_\infty^{(\gamma_f)}$ to the periods introduced in \S\ref{periods-subsubsec}. For any $\sigma\in\Sigma$, replacing $f$ with $f^\sigma$ and $\iota_F$ with $\sigma$ in \S\ref{periods-subsubsec}, we obtain periods $\Omega_{f^\sigma,\Gamma}^\pm$ and $\cO_F$-submodules $\mathfrak{a}_{f^\sigma,\Gamma}^\pm$ of $\R$ ($\mathfrak{a}_{f^\sigma,\Gamma}^\pm$ is endowed with an $\cO_F$-module structure via $\sigma$). Briefly,  choose $\delta_{f,\Gamma}^\pm$ as in \S\ref{periods-subsubsec}, fix $\sigma\in\Sigma$ and define $\eta_{f^\sigma,\Gamma}^\pm\defeq(\Comp_{\B,\dR}\circ\,\Theta_\sigma)\bigl(\delta^{\pm}_{f,\Gamma}\bigr)$, where $\Theta_\sigma$ is the map obtained by composing $\Theta$ with $\sigma$. Define periods $\Omega_{f^\sigma,\Gamma}^\pm\in\C^\times$ via the equality $\varphi_{f^\sigma,\Gamma}=\Omega_{f^\sigma,\Gamma}^+\eta_{f^\sigma,\Gamma}^++\Omega_{f^\sigma,\Gamma}^-\eta_{f^\sigma,\Gamma}^-$. For $\epsilon=(-1)^{k/2-1}$ as in \eqref{epsilon-eq}, set 
\begin{itemize}
\item $\delta_{f,\Gamma}\defeq\delta_{f,\Gamma}^\epsilon$; 
\item $\Omega_{f^\sigma,\Gamma}\defeq\Omega_{f^\sigma,\Gamma}^\epsilon$ for each $\sigma\in\Sigma$;
\item $\mathfrak{a}_{f,\Gamma}\defeq\mathfrak{a}_{f,\Gamma}^\epsilon$.
\end{itemize}
The proposition below provides the comparison result we need.

\begin{proposition} \label{comparison-periods-prop}
Set $\gamma_f\defeq(2\pi i)^{k/2}\cdot(\Pi_B\Pi_\epsilon)\cdot\delta_{f,\Gamma(N)}\in T_\B^+\smallsetminus\{0\}$. The equality 
\begin{equation} \label{eqper}
\Omega_\infty^{(\gamma_f)}={C}\cdot\left((2\pi i)^{\frac{k-2}{2}}\cdot\Omega_{f^\sigma,\Gamma(N)}\right)_{\sigma\in \Sigma}
\end{equation} 
holds in $F_\infty^\times$ for some $C\in F^\times$ satisfying $\ord_\lambda(C)=0$ for all primes $\lambda$ of $F$ such that $\lambda\nmid N\cdot\mathfrak{a}_{f,\Gamma(N)}$.
\end{proposition}

\begin{proof} We use the argument in \cite[Lemma 4.1]{DSW}, which follows a duality argument from \cite[\S1.7]{deligne-valeurs}. Fix $\sigma\in\Sigma$ and denote by $\Omega_{\infty,\sigma}^{(\gamma_f)}$ the determinant of the comparison isomorphism 
\[ V_\B^\epsilon\otimes_{F,\sigma}\C\overset\simeq\longrightarrow\bigl(V_\dR/\Fil^0(V_\dR)\bigr)\otimes_{F,\sigma}\C \] 
computed with respect to the lattice $\cO_F\cdot\gamma_f$ and the image of $T_\dR$, which is generated by $\omega_{f^\sigma}$; thus, by definition, $\Omega_\infty^{(\gamma_f)}=\bigl(\Omega_{\infty,\sigma}^{(\gamma_f)}\bigr)_{\sigma\in\Sigma}$. In the notation of \cite{DSW}, we have $\Omega_{\infty,\sigma}^{(\gamma_f)}=\vol_\infty$ (observe that in \cite{DSW} the term $\vol_\infty$ is the determinant of the isomorphism 
\[ V_\B^\epsilon(-k/2)\otimes_{F,\sigma}\C\overset\simeq\longrightarrow\bigl(V_\dR/\Fil^0(V_\dR)\bigr)\otimes_{F,\sigma}\C \]
computed with respect to the lattice $\cO_F\cdot\delta_{f,\Gamma(N)}$ and the image of $T_\dR$ multiplied by $(2\pi i)^{k/2}$, and the last factor $(2\pi i)^{k/2}$ is taken into account directly by the twist isomorphism between $V_\B^\pm$ and $V_\B^\pm(-k/2)$). We view $\mathfrak{a}_{f^\sigma,\Gamma(N)}=\sigma(\mathfrak{a}_{f,\Gamma(N)})$ as an $\cO_F$-submodule of $\R$ via $\sigma$. By \cite[Lemma 4.1]{DSW}, there is an equality of sets 
\[ \Omega_{\infty,\sigma}^{(\gamma_f)}\cdot\mathfrak{a}_{f^\sigma,\Gamma(N)}^{-1}=\sigma(C_0)\cdot(2\pi i)^{k/2-1}\cdot\Omega_{f^\sigma,\Gamma(N)}\cdot\mathfrak{a}_{f^\sigma,\Gamma(N)} \] 
for a suitable $C_0\in F^\times$, independent of $\sigma$, with $\ord_\lambda(C_0)=0$ for all primes $\lambda$ of $F$ such that $\lambda\nmid N$ (note that the term $\Omega_\pm$ in \cite{DSW} is defined by comparison with $(2\pi i)^{-1}\varphi_{f^\sigma}^\pm$, which explains the power of $2\pi i$ appearing in \eqref{eqper}). Therefore, we get an equality 
\[ \Omega_{\infty,\sigma}^{(\gamma_f)}=\sigma(C)\cdot(2\pi i)^{k/2-1}\cdot\Omega_{f^\sigma,\Gamma(N)} \]
for some $C\in F^\times$ with the property that $\ord_\lambda(C)=0$ for all primes $\lambda$ of $F$ such that $\lambda\nmid N\cdot\mathfrak{a}_{f,\Gamma(N)}$, as desired. \end{proof}

\subsubsection{Special values and their algebraic parts} 
For each $\sigma$ set $\Omega_{f^\sigma}=\Omega_{f,\Gamma_0(N)}^\epsilon$ and denote $\Omega_f=\Omega_{f^{\iota_F}}$ (as before, $\epsilon$ is as in \eqref{epsilon-eq}). 
The special value of $L(f,s)$ at $s=k/2$ is described (see, \emph{e.g.}, \cite[Ch. I, \S7]{MTT}; \emph{cf.} also \cite[\S5]{DSW}) by the formula 
\begin{equation} \label{MTT}
L(f,k/2)=\frac{(2\pi)^{k/2}}{(k/2-1)!}\cdot\int_{0}^{\infty}f(it)t^{\frac{k-2}{2}}dt.
\end{equation}
Let us define the \emph{algebraic part} of $L(f,k/2)$ as
\begin{equation} \label{algebraic-part-eq}
L^\alg(f,k/2)\defeq\frac{L(f,k/2)}{(2\pi i)^{k/2-1}\cdot\Omega_{f}}; 
\end{equation}
it is well known that $L^\alg(f,k/2)$ belongs to $F$. The period $\Omega_f$ depends on $\delta_{f,\Gamma_0(N)}^\epsilon$, so the same is true of $L^\alg(f,k/2)$.

\subsubsection{Algebraic parts of special values and real embeddings}

The algebraic part $L^\alg(f,k/2)$ in \eqref{algebraic-part-eq} belongs to $F$, so it makes sense to consider $\sigma\bigl(L^\alg(f,k/2)\bigr)$ for $\sigma\in\Sigma$.

\begin{proposition} \label{speciavalueslemma} 
For each $\sigma\in\Sigma$, there is an equality 
\[ \sigma\bigl(L^\alg(f,k/2)\bigr)=L^\alg(f^\sigma,k/2). \]
\end{proposition}

\begin{proof} For $P\in L_n(R)$ we write $P(X,Y)=\sum_{j=0}^{k-2}r_j(P)X^jY^{k-2-j}$. 
From \eqref{MTT} we get
\[L^\alg(f^\sigma,k/2)=\frac{(-i)^{k/2}}{(2\pi i)\cdot (k/2-1)!\cdot\Omega_{f^\sigma}}\cdot\int_{0}^{\infty}f^\sigma(it)t^{\frac{k-2}{2}}dt.\]
Since $f^\sigma$ has real Fourier coefficients, the integral belongs to the $\epsilon$-eigenspace for complex conjugation. Thus, taking the definition of modular symbols into account, we obtain an equality
\[ L^\alg(f^\sigma,k/2)=\frac{r_{\frac{k-2}{2}}\Bigl(\varphi_{f^\sigma}\bigl(\{i\infty\}-\{0\}\bigr)\!\Bigr)}{(k/2-1)!\cdot\Omega_{f^\sigma}}. \]
Let $\eta_q$ be the Hecke element introduced in \eqref{eta-q-eq}. Since $\sigma\bigl(\theta_f(\eta_q)\bigr)=\theta_{f^\sigma}(\eta_q)$ by definition, it suffices to show that 
$\sigma\bigl(\Theta(\delta_f)\bigr)=\Theta(\delta_{f^\sigma})$. Thus, we are reduced to showing that the square
\[
\xymatrix{
H^1_\mathrm{cpt}\bigl(\Gamma_0(N),L_n(F)\bigr) \ar[d]^-\sigma & H^1_\B \bigl(\Gamma_0(N),L_n(F)\bigr) \ar[d]^-\sigma\ar[l]_-\Theta\\
H^1_\mathrm{cpt}\bigl(\Gamma_0(N),L_n(\C)\bigr) & H^1_\B \bigl(\Gamma_0(N),L_n(\C)\bigr) \ar[l]_-\Theta
}
\]
is commutative (here $\Theta$ stands for the map in \eqref{Theta-eq2}). This is an immediate consequence of the definition of $\Theta$ and the fact that $\sigma$ commutes with Hecke operators. \end{proof}

\subsubsection{A comparison of periods} \label{sec1.4.4}

Now we make our choice of $\delta_{f,\Gamma}^\pm$ more precise. Specifically, we choose elements
$\delta_{f,\Gamma}^\pm\in T_{\B,\Gamma}^\pm\smallsetminus\{0\}$ such that their natural images in 
$T_{\B,\Gamma}^\pm\otimes_\Z\cO_p$ generate this free $\cO_p$-module of rank $1$. At the cost of discarding finitely many primes $p$, one can proceed as follows. Take $\delta_{f,\Gamma}^\pm\in T_{\B,\Gamma}^\pm\smallsetminus\{0\}$ and recall the module index $\mathfrak{a}^\pm_{f,\Gamma}$ from \S \ref{periods-subsubsec}, which is an ideal of $\cO_F$ defined in terms of $\delta_{f,\Gamma}^\pm$. Let $\mathrm{N}\bigl(\mathfrak{a}_{f,\Gamma}^\pm\bigr)\defeq\#\bigl(\cO_F/\mathfrak{a}_{f,\Gamma}^\pm\bigr)$ be the norm of $\mathfrak{a}_{f,\Gamma}^\pm$. Basic properties of the module index (see, \emph{e.g.}, \cite[\S 3]{frohlich}) allow one to check that if $p\nmid\mathrm{N}\bigl(\mathfrak{a}_{f,\Gamma}^\pm\bigr)$, then the image of $\delta_{f,\Gamma}^\pm$ in $T_{\B,\Gamma}^\pm\otimes_\Z\cO_p$ generates $T_{\B,\Gamma}^\pm\otimes_\Z\cO_p$.

Thus, we assume that $p\nmid\mathrm{N}\bigl(\mathfrak{a}_{f,\Gamma(N)}^\pm\bigr)$; we want to compare the periods $\Omega_{f,\Gamma(N)}$, $\Omega_{f,\Gamma_1(N)}$, $\Omega_{f,\Gamma_0(N)}$ (here we are especially interested in comparing $\Omega_{f,\Gamma_1(N)}$ and $\Omega_{f,\Gamma_0(N)}$). Before doing this, we need to fix some more notation: for each prime $\p$ of $F$ above $p$, denote by $\cO_{(\p)}$ the localization of $\cO_F$ at $\p$, then set $\cO_{(p)}\defeq\prod_{\p|p}\cO_{(\p)}$.

\begin{proposition} \label{comparison}
The periods $\Omega_{f,\Gamma(N)}$, $\Omega_{f,\Gamma_1(N)}$, $\Omega_{f,\Gamma_0(N)}$ differ pairwise by multiplication by elements of $\cO_{({p})}^\times$. 
\end{proposition}

\begin{proof} Let $(\Gamma_1,\Gamma_2)\in\bigl\{(\Gamma(N),\Gamma_1(N)),(\Gamma(N),\Gamma_0(N)),(\Gamma_1(N),\Gamma_0(N))\bigr\}$. Recall that the $\C$-vector space $S_k(\Gamma_2)$ is isomorphic to the $\mathfrak H_k(\Gamma_1)$-submodule of $S_k(\Gamma_1)$ consisting of those forms on which $\Gamma_2/\Gamma_1$ acts via the trivial character. Therefore, there is a canonical map $\mathfrak H_k(\Gamma_1)\rightarrow\mathfrak H_k(\Gamma_2)$, so any $\mathfrak H_k(\Gamma_2)$-module is also equipped with a structure of $\mathfrak H_k(\Gamma_1)$-module by means of this map. There is a commutative diagram of $\mathfrak H_k(\Gamma_1)$-modules with exact rows
\[
\xymatrix{
\Bound_{\Gamma_2}\bigl(L_n(\cO_{({ p})})\bigr)\ar[r]\ar[d] & 
\Symb_{\Gamma_2}\bigl(L_n(\cO_{({ p})})\bigr)\ar[r] \ar[d]&
H^1_\B \bigl(\Gamma_2,L_n(\mathcal{O}_{({ p})})\bigr) \ar[d]\ar[r]&0\\
\Bound_{\Gamma_1}\bigl(L_n(\cO_{({ p})})\bigr)\ar[r] & 
\Symb_{\Gamma_1}\bigl(L_n(\cO_{({ p})})\bigr)\ar[r] &
H^1_\B \bigl(\Gamma_1,L_n(\cO_{({ p})})\bigr)\ar[r]&0
}
\] 
in which the vertical arrows are induced by restriction in cohomology. 

Let $\p$ be a prime of $F$ above $p$, denote by $\F_{\p}$ the residue field of $F$ at ${\p}$ and let $\Gamma\in\bigl\{\Gamma(N),\Gamma_1(N),\Gamma_0(N)\bigr\}$; there is a canonical map $\mathfrak H_k(\Gamma)\rightarrow\F_{\p}$, whose kernel will be denoted by $\mathfrak{m}_\Gamma$. If $M$ is an $\mathfrak H_k(\Gamma)$-module, then we write $M_{\mathfrak{m}_\Gamma}$ for the localization of $M$ at $\mathfrak{m}_\Gamma$. Then 
\[ \Bound_\Gamma\bigl(L_n(\cO_{({\p})})\bigr)_{\mathfrak{m}_\Gamma}=0 \]
because the action of $\mathfrak H_k(\Gamma)$ on boundary symbols is Eisenstein, so we get a commutative square of $\mathfrak H_k(\Gamma)$-modules
\begin{equation}\label{diagram-symb}
\xymatrix{
\Symb_{\Gamma_2}\bigl(L_n(\cO_{({\p})})\bigr)_{\mathfrak{m}_{\Gamma_2}}\ar[r]^-\simeq \ar[d]&
H^1_\B \bigl(\Gamma_2,L_n(\cO_{({\p})})\bigr)_{\mathfrak{m}_{\Gamma_2}} \ar[d]\\
\Symb_{\Gamma_1}\bigl(L_n(\cO_{({\p})})\bigr)_{\mathfrak{m}_{\Gamma_1}}\ar[r]^-\simeq &
H^1_\B \bigl(\Gamma_1,L_n(\cO_{({\p})})\bigr)_{\mathfrak{m}_{\Gamma_1}}
}
\end{equation} 
in which the horizontal maps are isomorphisms. Now we prove that the right vertical arrow is an isomorphism; to do this, we show that the map of free $\cO_{({\p})}$-modules 
\begin{equation} \label{star}
\Hom_{\Gamma_2}\bigl(\mathcal{D}_0,L_n(\cO_{({\p})})\bigr)_{\mathfrak{m}_{\Gamma_2}}\longrightarrow\Hom_{\Gamma_1}\bigl(\mathcal{D}_0,L_n(\cO_{({\p})})\bigr)_{\mathfrak{m}_{\Gamma_1}}\end{equation} 
is an isomorphism
(note that $\Hom_{\Gamma}\bigl(\mathcal{D}_0,L_n(\cO_{({\p})})\bigr)_{\mathfrak{m}_\Gamma}$ is a finitely generated torsion-free, and hence free, $\cO_{({\p})}$-module). The map $\Hom_{\Gamma_2}\bigl(\mathcal{D}_0,L_n(\cO_{({\p})})\bigr)\rightarrow\Hom_{\Gamma_1}\bigl(\mathcal{D}_0,L_n(\cO_{({\p})})\bigr)$ is injective, so \eqref{star} is injective, as localization is a flat operation. 
By Nakayama's lemma, it suffices to show that the map 
\begin{equation} \label{starFp}
\Hom_{\Gamma_2}\bigl(\mathcal{D}_0,L_n(\F_{\p})\bigr)[{\mathfrak{m}_{\Gamma_2}}]\longrightarrow \Hom_{\Gamma_1}\bigl(\mathcal{D}_0,L_n(\F_{\p})\bigr)[{\mathfrak{m}_{\Gamma_1}}]
\end{equation} 
is surjective. If this map is not surjective, then there is $\phi\in\Hom_{\Gamma_1}\bigl(\mathcal{D}_0,L_n(\F_{\p})\bigr)[{\mathfrak{m}_{\Gamma_1}}]$ on which $\Gamma_2/\Gamma_1$ acts via a non-trivial character $\varepsilon$; by Nakayama's lemma, there is a non-zero
$\Phi\in\Hom_{\Gamma_1}\bigl(\mathcal{D}_0,L_n(\cO_{({\p})})\bigr)_{\mathfrak{m}_{\Gamma_1}}[\varepsilon]$ mapping to $\phi$ under the canonical surjection $\cO_{({\p})}\twoheadrightarrow \F_{\p}$. Therefore, we obtain two modular symbols $\Phi_{f,\Gamma_1}$ and $\Phi$ in $\Symb_{\Gamma_1}\bigl(L_n(\cO_{({\p})})\bigr)_{\mathfrak{m}_{\Gamma_1}}$ that are distinct (because $\varepsilon$ is non-trivial) and share the eigenvalues for the action of $\mathfrak H_k(\Gamma_1)$. The images of $\Phi_{f,\Gamma_1}$ and $\Phi$ in $H^1_\B \bigl(\Gamma_1,L_n(\cO_{({\p})})\bigr)_{\mathfrak{m}_{\Gamma_1}}$ coincide because $f$ is a newform and then, in light of \eqref{diagram-symb}, we deduce that $\Phi=\Phi_{f,\Gamma_1}$: this contradiction proves the surjectivity of \eqref{starFp}, whence the surjectivity of \eqref{star}. Using \eqref{diagram-symb}, we conclude that there is an isomorphism of $\cO_{({\p})}$-modules 
\[ H^1_\B \bigl(\Gamma_2,L_n(\cO_{({\p})})\bigr)_{\mathfrak{m}_{\Gamma_2}}\simeq H^1_\B \bigl(\Gamma_1,L_n(\cO_{({\p})})\bigr)_{\mathfrak{m}_{\Gamma_1}}. \] 
Upon taking $\pm$-eigenspaces for complex conjugation and $\mathfrak{m}_{\Gamma_2}$- and
$\mathfrak{m}_{\Gamma_1}$-torsion submodules, respectively, we get an isomorphism 
\[ H^1_\B \bigl(\Gamma_2,L_n(\cO_{({\p})})\bigr)_{\mathfrak{m}_{\Gamma_2}}\!\bigl[\theta_{f,\Gamma_2},\pm\bigr]\simeq H^1_\B \bigl(\Gamma_1,L_n(\cO_{({\p})})\bigr)_{\mathfrak{m}_{\Gamma_1}}\!\bigl[\theta_{f,\Gamma_1},\pm\bigr] \] 
of free $\cO_{({\p})}$-modules of rank $1$, where $\theta_{f,\Gamma_2}$ and $\theta_{f,\Gamma_1}$ denote the two ring homomorphisms associated with $f$. Therefore, $\gamma_{f,\Gamma_1}^\pm$ and the image of $\gamma_{f,\Gamma_2}^\pm$ in $H^1_\B \bigl(\Gamma_1,L_n(\cO_{({\p})})\bigr)$ are both generators of these free $\cO_{({\p})}$-modules, which implies that the periods $\Omega_{f,\Gamma_2}$ and $\Omega_{f,\Gamma_1}$ differ by a unit of $\cO_{({\p})}$, as was to be shown. \end{proof}

\begin{remark}
When the modular form $f$ has weight $2$ (a case that we have excluded from the outset), Proposition \ref{comparison} is proved in \cite[Lemma 9.4]{SZ} by different arguments. More precisely, the proof of \cite[Lemma 9.4]{SZ} uses in a crucial way Eisenstein properties of the Shimura subgroup, which in our higher weight context are replaced by Eisenstein properties of modular symbols.  
\end{remark}

From here on, as in \S \ref{comparison-subsubsec}, we set $\delta_{f,\Gamma(N)}\defeq\delta_{f,\Gamma(N)}^\epsilon$; moreover, put $\mathfrak{a}_{f,\Gamma(N)}\defeq\mathfrak{a}_{f,\Gamma(N)}^\epsilon$ and, as in Proposition \ref{comparison-periods-prop}, define 
\[ \gamma_f\defeq(2\pi i)^{k/2}\cdot(\Pi_B\Pi_\epsilon)\cdot\delta_{f,\Gamma(N)}\in T_\B^+\smallsetminus\{0\}. \]
By what we noticed previously, if $p\nmid\mathrm{N}(\mathfrak a_{f,\Gamma(N)})$, then $\Comp_{\B,\et}(\gamma_f)$ generates $T_p^+$. In other words, notation being as in \S \ref{reformulation-subsubsec}, we know that
\begin{equation} \label{I-implication-eq}
p\nmid\mathrm{N}(\mathfrak a_{f,\Gamma(N)})\;\Longrightarrow\;\mathcal I_p(\gamma_f)=\cO_p.
\end{equation}
This implication will be used in the proof of our main results.

\begin{remark} \label{period-rem}
Although $T_\B$ and $\gamma_f\in T_\B^+$ are defined in terms of the congruence subgroup $\Gamma(N)$, Proposition \ref{comparison} ensures that, for our goals and arguments, we can equivalently work with the period $\Omega_{f,\Gamma_0(N)}$.
\end{remark}

\subsection{Choice of auxiliary imaginary quadratic fields} \label{imaginary-subsec}

In the remainder of the paper, we will need to fix auxiliary imaginary quadratic fields in a judicious way. For our purposes, we may restrict ourselves to $r_\an(\MM)\in\{0,1\}$.

\subsubsection{The $r_\an(\MM)=0$ case} \label{imaginary-0-subsubsec}

Assume that $r_\an(\MM)=0$. Let us consider the imaginary quadratic fields $K$ satisfying the following two conditions:
\begin{itemize}
\item the primes dividing $Np$ split in $K$;
\item $r_\an(f^K)=1$.
\end{itemize} 
Denote by $\mathscr I_0(f,p)$ the set of all such fields. By Lemma \ref{low-rank-lemma}, $r_\an(f)=0$, so $\varepsilon(f)=+1$, and then it follows from \cite[p. 543, Theorem, (i)]{BFH} (\emph{cf.} also \cite{MM-derivatives}) that $\mathscr I_0(f,p)\not=\emptyset$.

\subsubsection{The $r_\an(\MM)=1$ case} \label{imaginary-1-subsubsec}

Assume that $r_\an(\MM)=1$. Let us consider the imaginary quadratic fields $K$ satisfying the following two conditions:
\begin{itemize}
\item the primes dividing $Np$ split in $K$;
\item $r_\an(f^K)=0$.
\end{itemize} 
Denote by $\mathscr I_1(f,p)$ the set of all such fields. By Lemma \ref{low-rank-lemma}, $r_\an(f)=1$, so $\varepsilon(f)=-1$, and then $\mathscr I_1(f,p)\not=\emptyset$ by \cite[p. 543, Theorem, (ii)]{BFH} (\emph{cf.} also \cite{Waldspurger}).

\subsection{Rationality conjecture for $\MM$ in analytic rank $0$}  \label{ratconj-0-subsec} 

Here we prove (under suitable assumptions) the rationality conjecture (Conjecture \ref{ratconj}) when $r_\an(\MM)=0$. 

\subsubsection{Nekov\'a\v{r}'s theorem} \label{nekovar-subsubsec}

Let $L$ be a number field. Following \cite{Nek}, we define the Shafarevich--Tate group of $\MM$ over $L$ at $\p$ \emph{\`a la} Nekov\'a\v{r} via the short exact sequence of $\cO_\p$-modules
\begin{equation} \label{selmer-nek-eq}
0\longrightarrow\Lambda_\p(L)\otimes_{\cO_\p}\!(F_\p/\cO_\p)\longrightarrow H^1_f(L,A_\p)\longrightarrow \Sha_\p^{\Nek}(L,\MM)\longrightarrow0 
\end{equation}
(see, \emph{e.g.}, \cite[\S 2.4]{LV} for details on the leftmost non-trivial map, which will be tacitly regarded as a set-theoretic inclusion). Let us also define the Shafarevich--Tate group of $\MM$ over $L$ at $p$ \emph{\`a la} Nekov\'a\v{r} by setting
\begin{equation} \label{sha-nek-eq}
\Sha_p^{\Nek}(L,\MM)\defeq\bigoplus_{\p\mid p}\Sha_\p^{\Nek}(L,\MM), 
\end{equation}
where the direct sum ranges over all primes of $F$ above $p$.

\begin{remark}
The group $\Sha_\p^{\Nek}(L,\MM)$ should be thought of as a higher weight counterpart of the classical Shafarevich--Tate group of an abelian variety, which is given by the recipe ``Selmer group modulo rational points''.
\end{remark}

As before, let $K$ be an imaginary quadratic field in which all the prime divisors of $Np$ split and let $\p$ be a prime of $F$ above $p$. The theorem below is a higher weight analogue of a well-known result of Kolyvagin for Mordell--Weil and Shafarevich--Tate groups of elliptic curves (see, \emph{e.g.}, \cite[Theorem 1.3]{Gross}).

\begin{theorem}[Nekov\'a\v{r}] \label{nekovar-thm}
If $y_{K,\p}$ is not torsion, then
\begin{enumerate}
\item $\Lambda_\p(K)\otimes_\Z\Q=F_\p\cdot y_{K,\p}$;
\item $\Sha_\p^{\Nek}(K,\MM)$ is finite;
\item $\corank_{\cO_\p}\!H^1_f(K,A_\p)=1$.
\end{enumerate}
\end{theorem}

\begin{proof} This is \cite[Theorem 13.1]{Nek} (\emph{cf.} also \cite[Ch. II, (6.5)]{Nek2}). \end{proof}

As an immediate consequence of Theorem \ref{nekovar-thm} and \eqref{sha-nek-eq}, if $y_{K,\p}$ is not torsion for each $\p\,|\,p$, then $\Sha_p^{\Nek}(K,\MM)$ is finite.

\subsubsection{Comparing Shafarevich--Tate groups}

Given a number field $L$, there is an inclusion $\Lambda_\p(L)\otimes_{\cO_\p}\!(F_\p/\cO_\p)\subset H^1_f(L,A_\p{)}_\divv$, which induces a surjection $\Sha_\p^{\Nek}(L,\MM)\twoheadrightarrow\Sha_\p^{\BKK}(L,\MM)$ of $\cO_\p$-modules. To our knowledge, no finer, general comparison between $\Sha_\p^{\BKK}(L,\MM)$ and $\Sha_\p^{\Nek}(L,\MM)$ is available in the literature.

Now take an imaginary quadratic field $K$ as above. The next result offers an alternative description of $\Sha_\p^{\BKK}(K,\MM)$ in an important special case.

\begin{proposition} \label{selmer-alternative-prop}
If $y_{K,\p}$ is not torsion, then $\Sha_\p^{\BKK}(K,\MM)=\Sha_\p^{\Nek}(K,\MM)$.
\end{proposition}

\begin{proof} If $y_{K,\p}$ is not torsion, then, by Theorem \ref{nekovar-thm}, both $\Lambda_\p(K)\otimes_{\cO_\p}\!(F_\p/\cO_\p)$ and $H^1_f(K,A_\p)$ have corank $1$ over $\cO_\p$, so $\Lambda_\p(K)\otimes_{\cO_\p}\!(F_\p/\cO_\p)$ is the maximal $p$-divisible submodule of $H^1_f(K,A_\p)$, whence the claim of the proposition. \end{proof}

\subsubsection{On Conjecture \ref{orderconj}, $r_\an(\MM)=0,1$}

The next result, which is basically a consequence of Theorems \ref{zhang-thm} and \ref{nekovar-thm}, establishes (conditionally on certain assumptions on $p$-adic Abel--Jacobi maps and regulators) the ``algebraic rank = analytic rank'' conjecture (Conjecture \ref{orderconj}) over $\Q$ when $r_\an(\MM)\in\{0,1\}$. We use notation from \S \ref{imaginary-subsec}.

\begin{theorem}[Nekov\'a\v{r}, S.-W. Zhang] \label{nekovar-thm2}
Assume that $r_\an(\MM)\in\{0,1\}$ and that 
\begin{enumerate}
\item there exists $K\in\mathscr I_{r_\an(\MM)}(f,p)$ such that $\AJ_{K,\p}$  is injective on $\Heeg_{K,N}^{\mathcal G_1}$ for some $\p\,|\,p$;
\item the $p$-part of Conjecture \ref{regpconj} over $\Q$ holds true. 
\end{enumerate}
Then $r_\alg(\MM)=r_\an(\MM)$.
\end{theorem}

\begin{proof} Assume that $r_\an(\MM)=1$. Fix $K\in\mathscr I_1(f,p)$ and a prime $\p$ of $F$ above $p$ for which condition (1) holds. It follows from splitting \eqref{splitting-L-eq2} that $r_\an(f/K)=1$, hence, by part (1) of Proposition \ref{coro zhang}, $y_{K,\p}$ is not $\cO_\p$-torsion. As explained in the proof of \cite[Theorem 5.26]{Vigni}, an analysis of the action of complex conjugation on $y_{K,\p}$ combined with Theorem \ref{nekovar-thm} shows that $\Lambda_\p(K)\otimes_{\cO_\p}F_\p$ and $\Lambda_\p(\Q)\otimes_{\cO_\p}F_\p$ are both $1$-dimensional over $F_\p$ and that $\Sha_\p^{\Nek}(K,\MM)$ and $\Sha_\p^{\Nek}(\Q,\MM)$ are both finite (\emph{cf.} also \cite[Proposition 5.25]{Vigni}). It follows from \eqref{selmer-nek-eq} that $\corank_{\cO_\p}\!H^1_f(\Q,A_\p)=1$, and then, since we are assuming condition (2), $r_\alg(\MM)=1$ by Corollary \ref{upsilon-onto-coro}. 

Now assume that $r_\an(\MM)=0$. Fix $K\in\mathscr I_0(f,p)$ and a prime $\p$ of $F$ above $p$ for which condition (1) holds. As in the previous case, it follows from \eqref{splitting-L-eq2} that $r_\an(f/K)=1$, hence, by part (1) of Proposition \ref{coro zhang}, $y_{K,\p}$ is not $\cO_\p$-torsion. As in the proof of \cite[Theorem 7.4]{Vigni}, an analysis of complex conjugation acting on $y_{K,\p}$ allows one to show that $\Lambda_{\p}(\Q)\otimes_{\cO_\p}F_\p$ is trivial. On the other hand, it follows from Theorem \ref{nekovar-thm} and \cite[Proposition 5.25]{Vigni} that $\Sha_\p^{\Nek}(\Q,\MM)$ is finite. Therefore, $\corank_{\cO_\p}\!H^1_f(\Q,A_\p)=0$, and then, since we are assuming condition (2), $r_\alg(\MM)=0$ by Corollary \ref{upsilon-onto-coro}. \end{proof}

\begin{remark} \label{injectivity-AJ-rem}
In the statement of Theorem \ref{nekovar-thm2} and elsewhere in this paper, we need to impose injectivity assumptions on $p$-adic Abel--Jacobi maps, which are natural to ask for if one wants to pass from information on analytic ranks to results on algebraic ranks by combining Theorems \ref{zhang-thm} and \ref{nekovar-thm}. While it is a ``folklore'' conjecture that such maps are always injective, it is worth emphasizing that, in the present article, conditions of this kind are only exploited, in the guise of part (1) of Proposition \ref{coro zhang}, to use Theorem \ref{nekovar-thm} and prove Conjecture \ref{orderconj} and the finiteness of $\Sha_p^{\Nek}(\Q,\MM)$ in the special cases we are interested in. In particular, if one is willing to assume the validity of Conjecture \ref{orderconj} and the finiteness of $\Sha_p^{\Nek}(\Q,\MM)$ in low rank situations, then the aforementioned conditions can be dispensed with.
\end{remark}

%\begin{remark}
%Under the assumptions of Theorem \ref{nekovar-thm2}, the arguments used in the proof of Proposition \ref{selmer-alternative-prop} show that there is an equality $\Sha_\p^{\BKK}(\Q,\MM)=\Sha_\p^{\Nek}(\Q,\MM)$.
%\end{remark}

\subsubsection{Proof of Conjecture \ref{ratconj}, $r_\an(\MM)=0$}

Now we can prove the main result of this subsection. 

\begin{theorem} \label{ratconj0thm}
Assume that $r_\an(\MM)=0$ and that 
\begin{enumerate}
\item there exists $K\in\mathscr I_0(f,p)$ such that $\AJ_{K,\p}$  is injective on $\Heeg_{K,N}^{\mathcal G_1}$ for some $\p\,|\,p$;
\item the $p$-part of Conjecture \ref{regpconj} over $\Q$ holds true. 
\end{enumerate}
Then Conjecture \ref{ratconj} is true. 
\end{theorem}

\begin{proof} Assume that $r_\an(\MM)=0$; this implies, by Lemma \ref{low-rank-lemma}, that $r_\an(f^\sigma)=0$ for all $\sigma\in\Sigma$. Thanks to Theorem \ref{nekovar-thm2}, $r_\alg(\MM)=0$, so $\Reg(\MM)=1$. Let $\Omega_\infty$ be the period in \S \ref{reformulation-subsubsec}; by Proposition \ref{rationality-prop}, to prove the theorem we can equivalently show that $L^*(\MM,0)\big/\Omega_\infty\in F^\times$. Namely (\emph{cf.} Remark \ref{deligne-rem2}), we need to show that there exists $\mathscr{L}_\MM\in F^\times$ such that 
\begin{equation} \label{iota-sigma-eq}
\iota_\Sigma(\mathscr{L}_\MM)=\frac{L^*(\MM,0)}{\Omega_\infty}.
\end{equation}
By definition of $\iota_\Sigma$, and in light of \eqref{eqper}, equality \eqref{iota-sigma-eq} means that $\sigma(\mathscr{L}_\MM)=L^\mathrm{alg}(f^\sigma,k/2)$ for all $\sigma\in \Sigma$. By Proposition \ref{speciavalueslemma}, the element $\mathscr{L}_\MM\defeq L^\mathrm{alg}(f,k/2)\in F^\times$ does the job, and we are done. \end{proof}

As explained in \S \ref{ratconj2-subsubsec}, Theorem \ref{ratconj0thm} shows that, under the specified assumptions, the analytic rank $0$ case of Conjecture \ref{ratconj2} holds true as well.

\subsection{$p$-TNC for $\MM$ in analytic rank $0$} \label{SUsec} \label{p-adic-subsec}

Recall that $D_F$ is the discriminant of $F$, the ideal $\mathfrak a_{f,\Gamma(N)}$ of $\cO_F$ was introduced at the end of \S \ref{sec1.4.4} and $c_f=[\cO_F:\cO_f]$. We work under the following list of conditions on the pair $(f,p)$.

\begin{assumption} \label{ass-0}
\begin{enumerate}
\item $p\nmid 6ND_F\mathrm{N}(\mathfrak a_{f,\Gamma(N)})c_f$;
\item $k\equiv 2\pmod{2(p-1)}$;
\item $a_p(f)\in\cO_\p^\times$;
\item $a_p(f)\not\equiv1\pmod{\p}$ for each prime $\p\,|\,p$; 
\item $\rho_{p}$ has big image;  
\item $\bar{\rho}_{p}$ is irreducible;
\item $N\geq 3$ and there exists a prime $\ell$ dividing $N$ exactly such that $\rho_{\p}$ is ramified at $\ell$ for each prime $\p\,|\,p$.
\end{enumerate}
\end{assumption}

Our purpose is to prove the $p$-part of the TNC for $\MM$ when $r_\an(\MM)=0$. Unless otherwise stated, from here on we make our choice of periods as in \S \ref{sec1.4.4}. 

\subsubsection{$p$-adic $L$-functions}

Let $\Q_\infty$ be the cyclotomic $\Z_p$-extension of $\Q$, define 
\[ \Gamma\defeq\Gal(\Q_\infty/\Q)\simeq\Z_p \] 
and write $\Lambda_\p\defeq\cO_\p[\![{\Gamma}]\!]$ for the Iwasawa algebra of $\Gamma$ with coefficients in $\cO_\p$. For every $n\in\N$ let $\Q_n$ be the subfield of $\Q_\infty$ such that $\Gal(\Q_n/\Q)\simeq\Z/p^n\Z$; in particular, $\Q_0=\Q$. For any prime $v$ of $\Q_n$ denote by $\Q_{n,v}$ the completion of $\Q_n$ at $v$ and write $I_{n,v}=I_{\Q_{n,v}}$ for the corresponding inertia group. Finally, let $\chi_\cyc:\Gamma\rightarrow\Z_p^\times$ be the $p$-adic cyclotomic character.

Let $\Omega_{f}^\pm\defeq\Omega^{\pm}_{f,\Gamma_0(N)}\in\C^\times$ be the periods from \S \ref{periods-subsubsec}. In \cite{SU}, the periods $\Omega_{f,\Gamma_1(N)}^\pm$ are considered instead, but the choice of $\Omega_{f,\Gamma_0(N)}^\pm$ is equivalent for us since, by Proposition \ref{comparison}, these complex numbers differ by a $p$-adic unit. Let $\bar\Q\hookrightarrow\bar\Q_p$ be an embedding corresponding to $\p$, which allows us to consider $L^\alg(f,k/2)\in F$ as a $p$-adic number in $\bar\Q_p$; here recall that $L^\mathrm{alg}(f,k/2)=L(f,k/2)\big/(2\pi i)^{k/2-1}\cdot\Omega_f$, where $\Omega_f=\Omega_f^\epsilon$ and $\epsilon$ is the sign of $(-1)^{k/2-1}$. 

Let $\mathcal{L}_{f,\p}\in\Lambda$ be the cyclotomic $p$-adic $L$-function of $f$ and $\p$ constructed in \cite[Theorem 16.2]{Kato} and \cite[\S3.4.4]{SU} (\emph{cf.} also \cite{MTT}). Adopting the conventions in \cite[Theorem 16.2]{Kato}, there is an interpolation formula
\begin{equation} \label{p-adic L-function}
\mathcal{L}_{f,\p}(\chi_\cyc^{k/2})=\Bigg(1-\frac{p^{\frac{k-2}{2}}}{\alpha}\Bigg)^{\!\!2}\cdot\biggl(\frac{k-2}{2}\biggr)!\cdot L^\mathrm{alg}(f,k/2). 
\end{equation}
Since $k>2$, the multiplicative factor $\Bigl(1-\frac{p^{\frac{k-2}{2}}}{\alpha}\Bigr)^2$ is a unit of $\cO_{(\p)}$. Therefore, there is an equality of $\cO_\p$-ideals
\begin{equation}\label{L}
\Bigl(\mathcal{L}_{f,\p}(\chi_\cyc^{k/2})\Bigr)=\Bigl((k/2-1)!\cdot L^\mathrm{alg}(f,k/2)\Bigr).
\end{equation}

\subsubsection{Comparing the periods of $f$ and $f^K$}

As in Remark \ref{motive-f^K-rem}, let $f^K$ be the twist of $f$ by the Dirichlet character $\epsilon_K$ attached to $K$. Of course, the Hecke field of $f^K$ is $F$.

\begin{lemma} \label{lemma4.4}
\begin{enumerate}
\item For any choice of $\delta_{f}^\pm$ as in \S \ref{periods-subsubsec}, the equality $\Omega_{f^K}^\mp=\sqrt{D_K}\cdot\Omega_f^\pm$ holds up to elements of $F^\times$. 
\item For any choice of $\delta_{f}^\pm$ as in \S \ref{sec1.4.4}, the equality $\Omega_{f}^\pm=\Omega^\mp_{f^K}$ holds up to $p$-adic units. 
\end{enumerate}
\end{lemma}

\begin{proof} For part (1), we adapt (with minor changes) the proof of \cite[Lemma 9.6]{SZ}. For any ring $R$ as before, let us define a homomorphism $H^1\bigl(\Gamma_0(N),L_n(R)\bigr)\rightarrow H^1\bigl(\Gamma_0(ND^2_K),L_n(R)\bigr)$ of $R$-modules by 
\[ \varphi\longmapsto\Biggl(\gamma\mapsto\sum_{a\in(\Z/D_K\Z)^\times}\epsilon_K(a)\cdot\varphi\Bigl(\!\smallmat{1}{-a/D_K}{0}{1}\cdot\gamma\cdot\smallmat{1}{a/D_K}{0}{1}\!\Bigr)\!\Biggr). \]  
For $R=\C$, $\omega_f^\pm$ is mapped to $\tau(\epsilon_K)\cdot\omega_{f^K}^\mp$, where $\tau(\epsilon_K)$ is the Gauss sum of $\epsilon_K$. For $R=F$, $\delta_{f}^\pm$ is mapped to an $F^\times$-multiple of $\delta_{f^K}^\mp$. Therefore, 
$\Omega_{f^K}^\pm$ is an $F^\times$-multiple of $\Omega_f^\mp/\tau(\epsilon_K)$, and the conclusion follows because $\tau(\epsilon_K)^2=D_K$. 

Part (2) is a consequence of the interpolation formulas satisfied by the $p$-adic $L$-functions associated with $f$ and $f^K$. With notation as in \cite[\S 3.4.4]{SU}, the sign of the period in the interpolation formula for $\mathcal{L}_{f,\psi\epsilon_K}(\phi)$ at an integer $0\leq m\leq k-2$ is $\mathrm{sgn}\bigl((-1)^m\psi(-1)\epsilon_K(-1)\bigr)$, while if $\mathbf{1}$ is the trivial character, then the sign of the period in the interpolation formula for $\mathcal{L}_{f^K,\mathbf{1}}(\phi)$ at $m$ is $\mathrm{sgn}\bigl((-1)^m\psi(-1)\bigr)$. On the other hand, $\mathcal{L}_{f,\psi\epsilon_K}(\phi)$ and $\mathcal{L}_{f^K,\mathbf{1}}(\phi)$ differ by a $p$-adic unit; since $\epsilon_K(-1)=-1$ because $K$ is imaginary, it follows that the two periods have opposite signs. \end{proof}

\begin{remark}
In the case of elliptic curves, the analogue of part (2) of Lemma \ref{lemma4.4} is proved, \emph{e.g.}, in 
\cite[Lemma 9.6]{SZ} with a different method. 
\end{remark}

\subsubsection{Greenberg's Selmer group}

Let $\eta_p:G_{\Q_p}=\Gal(\bar\Q_p/\Q_p)\rightarrow\bar\Q_p^\times$ denote the unramified character of $G_{\Q_p}$ taking arithmetic Frobenius to $a_p(f)$. Since $f$ is ordinary at $\p$ and $V_\p$ is (equivalent to) the self-dual twist of the $\p$-adic representation attached to $f$, there is a short exact sequence of $G_{\Q_p}$-modules 
\[ 0\longrightarrow V_\p^+\longrightarrow V_\p\longrightarrow V_\p^-\longrightarrow0 \] 
such that $V^+_\p$ and $V^-_\p$ are $1$-dimensional $F_\p$-vector spaces, and $G_{\Q_p}$ acts on $V^+_\p$ and $V^-_\p$ through $\eta_p^{-1}\chi_\cyc^{k/2}$ and $\eta_p\chi_\cyc^{-k/2+1}$, respectively. Define $T^+_\p\defeq T_\p\cap V^+_\p$, $A^+_\p\defeq V^+_\p/T^+_\p$, $A^-_\p\defeq A_\p/A^+_\p$ and set
\[ H^1_\ord(\Q_{n,v},A_\p)\defeq\ker\Bigl(H^1(\Q_{n,v},A_\p)\longrightarrow H^1(I_{\Q_{n,v}},A^-_\p)\Bigr). \]
For every $n\in\N$, define the \emph{$\p$-primary Greenberg Selmer group of $f$ over $\Q_n$} to be
\begin{equation} \label{selmer-f-eq}
\Sel_\p(f/\Q_n)\defeq\ker\biggl(H^1(\Q_n,A_\p)\longrightarrow \prod_{v\nmid p}\frac{H^1(\Q_{n,v},A_\p)}{H^1_\mathrm{ur}(\Q_{n,v},A_\p)}\times\prod_{v\mid p}\frac{H^1(\Q_{n,v},A_\p)}{H^1_\mathrm{ord}(\Q_{n,v},A_\p)}
\biggr), 
\end{equation}
where $v$ denotes a prime of $\Q_n$. By \cite[Proposition 4.2]{Ochiai}, for all $n$ there is an injection $H^1_f(\Q_n,A_\p)\hookrightarrow \Sel_\p(f/\Q_n)$ with finite cokernel whose order is bounded independently of $n$. As in \cite{SU}, let us consider the \emph{$\p$-primary Greenberg Selmer group of $f$ over $\Q_\infty$} given by
\[ S_\p=\Sel_\p(f/\Q_\infty)\defeq\varinjlim_n\Sel_\p(f/\Q_n), \]
where the direct limit is taken with respect to restriction maps. By \cite[Theorem 17.4, (1)]{Kato} (\emph{cf.} also \cite[Theorem 3.15]{SU}), the $\Lambda_\p$-module $S_\p$ is cotorsion, \emph{i.e.}, the Pontryagin dual $X_\p\defeq\Hom(S_\p,F_\p/\cO_\p)$ of $S_\p$ is torsion over $\Lambda_\p$; we write $(\mathcal{F}_{f,\p})\subset\Lambda_\p$ for the characteristic ideal of $X_\p$.

\subsubsection{Proof of $p$-TNC, $r_\an(\MM)=0$}

The following theorem is a special case of a result for Selmer groups over finite abelian extensions of $\Q$. 

\begin{theorem}[Kato] \label{kato-thm}
If $r_\an(f)=0$, then $H^1_f(\Q,A_\p)$ is finite.
\end{theorem}

\begin{proof} This is \cite[Theorem 14.2, (2)]{Kato} for $K=\Q$. \end{proof}

Now we prove the $p$-part of the Tamagawa number conjecture for $\MM$ when $r_\an(\MM)=0$; in fact, we prove, more generally, a rank $0$ counterpart of Theorem B. As will be clear, our proof builds crucially on work of Kato and of Skinner--Urban. Recall that Assumption \ref{ass-0} is in force.

\begin{theorem} \label{skinner-urban-thm}
Assume that $r_\an(\MM)=0$ and that 
\begin{enumerate}
\item there exists $K\in\mathscr I_0(f,p)$ such that $\AJ_{K,\p}$  is injective on $\Heeg_{K,N}^{\mathcal G_1}$ for some $\p\,|\,p$;
\item the $p$-part of Conjecture \ref{regpconj} over $\Q$ holds true. 
\end{enumerate}
Then the following results hold:
\begin{enumerate}
\item[(a)] $r_\alg(\MM)=0$; 
\item[(b)] $\Sha_\p^{\BKK}(\Q,\MM)=\Sha_\p^{\Nek}(\Q,\MM)$ for each $\p\,|\,p$ satisfying (1);  
\item[(c)] the $p$-part of Conjecture \ref{TNC} is true. 
\end{enumerate}
Furthermore, if condition (1) holds for each $\p\,|\,p$, then $\Sha_p^{\BKK}(\Q,\MM)=\Sha_p^{\Nek}(\Q,\MM)$.
\end{theorem}

\begin{proof} Part (a) was already proved in Theorem \ref{nekovar-thm2}. Let $\p\,|\,p$ be a prime of $F$ satisfying condition (1). As explained in the proof of Theorem \ref{nekovar-thm2}, $\Lambda_\p(\Q)\otimes_{\cO_\p}F_\p$ is trivial, and then $\Lambda_\p(\Q)\otimes_{\cO_\p}(F_\p/\cO_\p)$ is trivial as well. Thus, $\Sha_\p^{\Nek}(\Q,\MM)=H^1_f(\Q,A_\p)$. On the other hand, $H^1_f(\Q,A_\p)=\Sha_\p^{\BKK}(\Q,\MM)$ because, by Theorem \ref{kato-thm}, $H^1_f(\Q,A_\p)$ is finite: this proves part (b) and (by taking direct sums over all $\p\,|\,p$) the last statement.

We prove part (c) by proving equality \eqref{eqBK} in Theorem \ref{motivesthm}. To start with, observe that Conjecture \ref{ratconj} is true in this case, by Theorem \ref{ratconj0thm}. Thus, all the assumptions in Theorem \ref{motivesthm} are verified (\emph{cf.} Remark \ref{non-degeneracy-rem}). 

By Lemma \ref{low-rank-lemma}, $r_\an(f)=0$. Let $\mathbf{1}$ denote the trivial character. We first show that 
\begin{equation} \label{IMC}
\bigl(\mathcal{F}_{f,\p}(\mathbf{1})\bigr)=\bigl(\mathcal{L}_{f,\p}(\chi_\cyc^{k/2})\bigr)
\end{equation}
as ideals of $\cO_\p$. Define the $\Lambda_\p$-torsion module
\[ X_\p(k/2)\defeq\Hom\bigl(S_\p(-k/2),F_\p/\cO_\p\bigr) \]
and write $\mathcal{C}=\mathrm{Char}\bigl(X_\p(k/2)\bigr)$ for its characteristic power series, which does not depend on $k$ (\emph{cf.} \cite[Proposition 17.2]{Kato}). By \cite[Theorem 17.4]{Kato} and \cite[Theorem 3.29]{SU}, there is an equality $(\mathcal{C})=(\mathcal{L}_{f,\p})$, so $\mathcal{C}\bigl(\chi_\cyc^{k/2}\bigr)$ and $\mathcal{L}_{f,\p}\bigl(\chi_\cyc^{k/2}\bigr)$ generate the same ideal of $\cO_\p$ (here, for an element $x\in\Lambda$ and a character $\chi:\Gamma\rightarrow\bar\Q_p^\times$, we set $x(\chi)\defeq\chi(x)$). In order to show \eqref{IMC}, it therefore remains to note that $\mathcal{C}\bigl(\chi_\cyc^{k/2}\bigr)$ and $\mathcal{F}_{f,\p}(\mathbf{1})$ generate the same ideal of $\cO_\p$, which follows easily by taking into consideration the Tate twist in the relevant definitions (see, \emph{e.g.}, \cite[Lemma 1.2]{Rubin-ES}). 

Combining \eqref{p-adic L-function} and \eqref{IMC}, we obtain the equivalences
\[ L(f,k/2)\neq0\Longleftrightarrow\mathcal{L}_{f,\p}\bigl(\chi_\cyc^{k/2}\bigr)\neq0\Longleftrightarrow \mathcal{F}_{f,\p}(\mathbf{1})\neq0. \]
Since $r_\an(f)=0$, we deduce from Theorem \ref{kato-thm} that $H^1_f(\Q,A_\p)$ is finite. A generalization to our setting of the arguments in \cite[Section 4]{Greenberg} (see \cite{LV-Iwasawa}) shows that 
\begin{equation} \label{greenberg-eq2.2}
\mathcal{I}_{\cO_\p}\bigl(\cO_\p/\mathcal{F}_{f,\p}(\mathbf{1})\cdot\cO_\p\bigr)=\mathcal{I}_{\cO_\p}\bigl(S_\p^\Gamma/(\#S_\p)_\Gamma\bigr)=\mathcal{I}_{\cO_\p}\bigl(H^1_f(\Q,A_\p)\bigr)\cdot\prod_{\ell\mid N}\mathrm{Tam}_\ell^{(\p)}(\MM).
\end{equation}
Combining \eqref{L}, \eqref{IMC} and \eqref{greenberg-eq2.2}, we see that for each $\p\,|\,p$ there is an equality 
\begin{equation} \label{L-eq}
\Bigl((k/2-1)!\cdot L^\mathrm{alg}(f,k/2)\Bigr)=\mathcal{I}_{\cO_\p}\bigl(H^1_f(\Q,A_\p)\bigr)\cdot\prod_{\ell\mid N}\mathrm{Tam}_\ell^{(\p)}(\MM)
\end{equation}
of (fractional) $\cO_\p$-ideals. On the other hand, by Proposition \ref{speciavalueslemma} (\emph{cf.} the proof of Theorem \ref{ratconj0thm}), $\iota_\Sigma\bigl((k/2-1)!\cdot L^\mathrm{alg}(f,k/2)\bigr)=(k/2-1)!\cdot L^*(\MM,0)\big/\Omega^{(\gamma_f)}_\infty$. Thus, with the convention from Remark \ref{deligne-rem2}, one can replace $(k/2-1)!\cdot L^\mathrm{alg}(f,k/2)$ with $(k/2-1)!\cdot L^*(\MM,0)\big/\Omega^{(\gamma_f)}_\infty$ in \eqref{L-eq}. Since $H^1_f(\Q,A_p)=\Sha_p^{\BKK}(\Q,\MM)$ and \eqref{L-eq} holds for each $\p\,|\,p$, we obtain an equality
\begin{equation} \label{L-eq2}
\biggl(\frac{(k/2-1)!\cdot L^*(\MM,0)}{\Omega_\infty}\biggr)=\mathcal{I}_{\cO_p}\bigl(\Sha_p^{\BKK}(\Q,\MM)\bigr)\cdot\prod_{\ell\mid N}\mathrm{Tam}_\ell^{(p)}(\MM)
\end{equation}
of (fractional) $\cO_p$-ideals. Now notice that, by part (1) of Proposition \ref{no torsion lemma}, $H^1(\Q,T_\p{)}_\mathrm{tors}$ is trivial for each $\p\,|\,p$, so $H^1(\Q,T_p{)}_\mathrm{tors}$ is trivial. Furthermore, $H^0(G_S,A_p)$ is trivial by \cite[Lemma 3.10, (2)]{LV} (\emph{cf.} also the proof of \cite[Lemma 2.4]{LV-kyoto}). This shows that $\Tors_p(\MM)=\cO_p$ (\emph{cf.} \S \ref{pid-prod-subsubsec}). By \eqref{Tam-p-eq}, $\Tam_{p}^{(p)}(\MM)=\cO_p$, while $\Tam_\infty^{(p)}(\MM)=\cO_p$ by \eqref{Tam-infty-eq} and $\mathcal I_p(\gamma_f)=\cO_p$ by \eqref{I-implication-eq}. Finally, by Theorem \ref{nekovar-thm2}, $r_\alg(\MM)=0$, so $\Reg(\MM)=1$. Therefore, equality \eqref{L-eq2} coincides with equality \eqref{eqBK} (or, better, with its equivalent form \eqref{eqBKbis} in Remark \ref{BK-bis-rem}) in our setting, which completes the proof. \end{proof}

\begin{remark} 
We want to explain how to relax an assumption in \cite{LV-Iwasawa} so that the results in \cite{LV-Iwasawa} can be safely applied in our current setting. In \cite{LV-Iwasawa}, one requires (\cite[Assumption 6(b)]{LV-Iwasawa}) that $H^0(I_v,A^-)=0$, which is too restrictive for the applications described above. This condition is used in the proof of \cite[Lemma 5.4]{LV-Iwasawa} to show that $H^1\bigl(G_v^\mathrm{un},H^1(I_v,T^-{)}_\mathrm{tors}\bigr)=0$. In that argument, we identify $H^1(I_v,T^-{)}_\mathrm{tors}$ with the largest cotorsion quotient of $H^0(I_v,A^-)$ and then conclude using $H^0(I_v,A^-)=0$. It turns out that we can relax the assumption $H^0(I_v,A^-)=0$ and still show that $H^1\bigl(G_v^\mathrm{un},H^1(I_v,T^-{)}_\mathrm{tors}\bigr)$ is trivial, arguing as follows. As observed above, the group $G_v^\mathrm{un}$ acts on $H^0(I_v,A^-)$ via the unramified character taking the geometric Frobenius to the unit root $\alpha$ of the Hecke polynomial. In the case of weight $2$, we have $H^0(I_v,A^-)=A^-$, hence the  largest cotorsion quotient of $H^0(I_v,A^-)$ is trivial (as $A^-$ is divisible) and we conclude that $H^1\bigl(G_v^\mathrm{un},H^1(I_v,T^-{)}_\mathrm{tors}\bigr)$ is trivial using the argument in \cite{LV-Iwasawa}. If the weight is bigger than $2$, then $H^0(I_v,A^-)=A^-[p^n]$ for some integer $n$, due to the fact that inertia acts on $A^-$ via the $(1-k/2)$-th power of the cyclotomic character. Since $a_p(f)\not\equiv 1\pmod{\mathfrak{p}},$ we also have $\alpha\not\equiv1\pmod{\mathfrak{p}}$. Thus, 
$H^1\bigl(G_v^\mathrm{un},H^1(I_v,T^-{)}_\mathrm{tors}\bigr)$ is isomorphic to $A^-[p^n]\big/(\alpha-1)\cdot A^-[p^n]$. Now $\alpha-1$ is invertible modulo $\mathfrak{p}$, so $(\alpha-1)\cdot A^-[p^n]=A^-[p^n]$ and $H^1\bigl(G_v^\mathrm{un},H^1(I_v,T^-{)}_\mathrm{tors}\bigr)$ is trivial.
\end{remark}

\subsection{Kolyvagin's conjecture and Shafarevich--Tate groups} \label{p-primary-Sha-subsec} 

As usual, let $\p$ be a prime of $F$ above $p$. We gather some results on $\Sha_\p^{\BKK}(K,\MM)$, where $K$ is, as usual, an imaginary quadratic field in which all the prime factors of $Np$ split. 

\subsubsection{A structure theorem for $\Sha_\p^{\BKK}(K,\MM)$}

Following \cite[\S 4.2]{Masoero}, for every integer $M\geq1$ we define $\tilde{S}_1(M)$ to be the set of prime numbers $\ell$ such that
\begin{itemize}
\item $\ell\nmid Np$;
\item $\ell$ is inert in $K$;
\item $\p^M\,|\,\ell+1$.
\end{itemize}
Denote by $\tilde{S}_n(M)$ the set of square-free products of $n$ primes in $\tilde{S}_1(M)$ (here $\tilde{S}_0(M)\defeq \{1\}$) and define 
\[ \tilde{S}(M)\defeq \bigcup_{n\in\N}\tilde{S}_n(M). \]
For every integer $m\geq1$, let $\mathscr M(m)$ be the set of integers $M\geq1$ such that $m\in\tilde{S}(M)$. 

As in \cite[\S7.1]{Masoero}, we also consider the set $S_1(M)$ of prime numbers $\ell$ such that
\begin{itemize}
\item $\ell\nmid Np$;
\item $\ell$ is inert in $K$;
\item $\p^M\,|\,a_\ell(f)$, $\p^M\,|\,\ell+1$;
\item $\p^{M+1}\nmid \ell+1\pm a_\ell(f)$.
\end{itemize}
Write ${S}_n(M)$ for the set of square-free products of $n$ primes in ${S}_1(M)$ (with ${S}_0(M)\defeq \{1\}$) and define 
\[ {S}(M)\defeq \bigcup_{n\in\N}{S}_n(M). \] 
With notation as in \S \ref{kolyvagin-classes-subsec}, for every integer $n\geq1$ set 
\[ \omega(n)\defeq \max\bigl\{M\in\mathscr M(n)\mid {[z_{n,\p}]}_M=0\bigr\}, \]
where we put $\omega(n)\defeq \infty$ if this maximum does not exist. Finally, for every $r\in\N$ set
\[ M_r\defeq \min\bigl\{\omega(n)\mid n\in S_r(\omega(n)+1)\bigr\}, \]
with $M_r\defeq \infty$ if $\omega(n)=\infty$. 

\begin{lemma} \label{lemma-masoero2} 
$M_0$ is finite and $M_r\geq M_{r+1}\geq 0$ for all $r\in\N$. 
\end{lemma}

\begin{proof} This is \cite[Lemma 7.4]{Masoero}. \end{proof}
 
Set $M_\infty\defeq \inf\bigl\{M_r\mid r\in\N\bigr\}$ and note that Conjecture \ref{kolyvagin-conj} is equivalent to the statement that $M_\infty< \infty$. Let $\Sha_\p^{\BKK,\pm}(f/K)$ be the $\pm$-eigenspaces of $\Gal(K/\Q)$ acting on $\Sha_\p^{\BKK}(K,\MM)$ and let $\varepsilon\in\{\pm\}$ be the sign of the root number of $f$ (\emph{cf.} \S \ref{Lfunctsec}). 

In the next theorem, which is a higher weight analogue of the main result of \cite[\S 5]{McCallum}, we let $N_i\defeq M_{i-1}-M_i$ for all $i\geq1$. Observe that, by Lemma \ref{lemma-masoero2}, $N_i\geq0$ for all $i$.  

\begin{theorem}[Nekov\'a\v{r}, Masoero] \label{structure sha}
If $y_{K,\p}$ is not torsion, then there are isomorphisms
\begin{align}
\Sha_\p^{\BKK,-\varepsilon}(f/K)&\simeq\bigl(\cO_\p/p^{N_1}\cO_\p\bigr)^2\oplus
\bigl(\cO_\p/p^{N_3}\cO_\p\bigr)^2\oplus\dots\\
\intertext{and}
\Sha_\p^{\BKK,\varepsilon}(f/K)&\simeq\bigl(\cO_\p/p^{N_2}\cO_\p\bigr)^2\oplus
\bigl(\cO_\p/p^{N_4}\cO_\p\bigr)^2\oplus\dots
\end{align}
of $\cO_\p$-modules. 
\end{theorem}

\begin{proof} By Proposition \ref{selmer-alternative-prop}, $\Sha_\p^{\BKK}(K,\MM)=\Sha_\p^{\Nek}(K,\MM)$. Thus, the theorem is essentially \cite[Theorem 7.3]{Masoero}, the only difference being that in \cite{Masoero} one is interested in the structure of $\Sha_\p^{\Nek}(K,\MM)$ as an abelian group, whereas here we are looking at it as an $\cO_\p$-module. \end{proof}

\subsubsection{Some consequences on $\Sha_\p^{\BKK}(K,\MM)$ and $y_{K,\p}$}

The next theorem is a consequence of Theorem \ref{kolyvagin-main-thm}, which asserts the validity of Conjecture \ref{strong-kolyvagin-conj}, and Theorem \ref{structure sha}; it will play a key role in the proof of our results on the $p$-part of the TNC for $\MM$ in analytic rank $1$.

\begin{theorem} \label{teorema-masoero}
If $y_{K,\p}$ is not $\cO_\p$-torsion, then $\length_{\cO_\p}\bigl(\Sha_\p^{\BKK}(K,\MM)\bigr)=2M_0$.
\end{theorem}

\begin{proof} As in the proof of \cite[Corollary 7.11]{Masoero}, it follows from Theorem \ref{structure sha} that 
\begin{equation} \label{sha-M-eq}
\length_{\cO_\p}\bigl(\Sha_\p^{\BKK}(K,\MM)\bigr)={2(M_0-M_\infty)}. 
\end{equation}
On the other hand, Theorem \ref{kolyvagin-main-thm} says that $c_1(f,n)\neq0$ for a suitable $n\in\Lambda_{\Kol}(f)$, which implies that $M_\infty=0$. The theorem is then a consequence of equality \eqref{sha-M-eq}. \end{proof}

For later reference, we prove

\begin{proposition} \label{lemma-masoero1} 
If $y_{K,\p}$ is not $\cO_\p$-torsion, then $M_0=\length_{\cO_\p}\bigl(\Lambda_{\p}(K)/y_{K,\p}\cdot\cO_\p\bigr)$. 
\end{proposition}

\begin{proof} By definition, $M_0$ is the largest integer $M\geq1$ such that ${[z_{1,\p}]}_M=0$. Equivalently, $M_0$ is the largest integer $M\geq1$ such that $z_{1,\p}\in\p^M\Lambda_{\p}(K_1)$. Now denote by $M'_0$ the length of $\Lambda_{\p}(K)/y_{K,\p}\cdot\cO_\p$ as an $\cO_\p$-module, which can be described as the largest integer $M\geq1$ such that  $y_{K,\p}\in\p^M\Lambda_{\p}(K)$. We want to show that $M_0=M_0'$. 

As explained in the proof of Proposition \ref{c_M(1)-prop}, restriction gives an $\cO_\p$-linear injection
\begin{equation} \label{res-length-eq}
\res_{K_1/K}:\Lambda_{\p}(K)\longmono{\Lambda_{\p}(K_1)}^{\mathcal G_1}\subset\Lambda_{\p}(K_1) 
\end{equation}
such that, by \eqref{y_K-z_1-eq}, $\res_{K_1/K}(y_{K,\p})=z_{1,\p}$. It follows immediately that $M_0'\leq M_0$. On the other hand, part (1) of Proposition \ref{no torsion lemma} ensures that $\Lambda_{\p}(K_1)$ is torsion-free, so \eqref{res-length-eq} induces for every integer $M\geq1$ an $\cO_\p$-linear injection
\[ \Lambda_{\p}(K)\big/\p^M\Lambda_{\p}(K)\longmono\Lambda_{\p}(K_1)\big/\p^M\Lambda_{\p}(K_1). \]
This shows that $M_0\leq M'_0$, and the proof is complete. \end{proof}
%
%\subsection{Congruence ideals} \label{CIandTF}
%
%As in \S \ref{realizations-subsec}, $\theta_f:\mathfrak{H}_k(\Gamma(N))\rightarrow\cO_F$ is the ring homomorphism attached to $f$ from the Hecke algebra $\mathfrak{H}_k(\Gamma(N))$ acting on modular forms of weight $k$ and level $\Gamma_0(N)$. The \emph{congruence ideal of $f$} is the ideal 
%\[ \eta_{f}\defeq\theta_f\Bigl(\mathrm{Ann}_{\mathfrak{H}_k(\Gamma(N))}\bigl(\ker(\theta_f)\bigr)\!\Bigr) \]
%of $\cO_F$. We also set  
%%\begin{equation} \label{deflambda}
%\[ \lambda_f\defeq\frac{(f,f)}{\pi^2\cdot\Omega_f^+\cdot\Omega_f^-}\in F. \]
%%\end{equation} 
%Then $\lambda_f$ is well defined up to $p$-adic units. By \footnote{Inserire referenza a Hsieh, Triple} 
%\cite[Lemma 12.1]{SU} (\emph{cf.} also \cite[Theorem 0.1]{Hida-Modules} and, for the weight $2$ case, \cite[Proposition 6.7]{PW}), $\eta_f=(\lambda_f)$ as ideals of $\cO_{(p)}$.
%
%\begin{remark} Let $m$ be the smallest integer such that $\varpi:J_p\rightarrow T_p$ is the multiplication by $p^m$ when restricted to $T_p$; if $\eta_f$ is a $p$-adic unit, then $m=0$.
%Our main results are stated under this condition; when $\val_\p(\eta_f)>1$ for some prime $\p\mid p$, we expect that the degree $p^m$ of the parametrisation appears in Conjecture \ref{TNC} 
%and the index-comparison between lattices. 
%\end{remark}

\subsection{Rationality conjecture for $\MM$ in analytic rank $1$} \label{ratconj-1-subsec}

Our goal here is to prove the rationality conjecture (Conjecture \ref{ratconj}) when $r_\an(\MM)=1$. Notation from Remark \ref{AJ-twist-rem} and \S \ref{imaginary-subsec} is in force.

\begin{theorem}\label{ratconj1thm}
Assume that $r_\an(\MM)=1$ and that there exists $K\in\mathscr I_1(f,p)$ such that
\begin{enumerate}
\item there is $K'\in\mathscr I_0(f^K,p)$ with $\AJ_{f^K,K',\p}$ injective on $\Heeg_{K',N}^{\mathcal G'_1}$ for some $\p\,|\,p$;
\item the $p$-part of Conjecture \ref{regpconj} for $\MM(f^K)$ over $\Q$ holds true. 
\end{enumerate}
Then Conjecture \ref{ratconj} is true. 
\end{theorem} 

\begin{proof} Assume that $r_\an(\MM)=1$; this implies, by Lemma \ref{low-rank-lemma}, that $r_\an(f^\sigma)=1$ for all $\sigma\in\Sigma$. Let $\mathscr B$ be a basis of $H^1_\mot(\Q,\MM)$ over $F$. By definition of the embedding $\iota_\Sigma$ in Remark \ref{deligne-rem2}, and bearing \eqref{reg-sigma-eq} and \eqref{eqper} in mind, we need to show that there exists $\mathscr{L}'\in F^\times$ such that 
\begin{equation} \label{ratconj1}
\sigma(\mathscr{L}')=\frac{L'(f^\sigma,k/2)}{(2\pi i)^{k/2-1}\cdot\Omega_{f^\sigma}\cdot\Reg_\mathscr{B}^\sigma(\MM)}\end{equation}
for all $\sigma\in\Sigma$. 

Choose an imaginary quadatic field $K\in\mathscr I_1(f,p)$ satisfying conditions (1) and (2); then $r_\an(f/K)=1$. Set $g\defeq f^K$; again by Lemma \ref{low-rank-lemma}, $r_\an(g^\sigma)=0$ for all $\sigma\in\Sigma$. Theorem \ref{zhang-thm} ensures that for each $\sigma\in\Sigma$ there is an equality 
\[ \frac{L'(f^\sigma/K,k/2)\cdot(k-2)!\cdot u_K^2\cdot\sqrt{|D_K|}}{2^{2k-1}\cdot\pi^k\cdot{(f^\sigma,f^\sigma)}_{\Gamma_0(N)}\cdot{\langle s_{f^\sigma}',s_{f^\sigma}'\rangle}_\mathrm{GS}}=1. \] 
Thanks to the factorization in \eqref{splitting-L-eq2} for $L(f^\sigma/K,s)$ and the fact that $g^\sigma=(f^\sigma)^K$, we can write 
\[ \frac{L'(f^\sigma,k/2)\cdot(k-2)!\cdot u_K^2\cdot\sqrt{|D_K|}\cdot i^{k/2-1}\cdot\Omega_{g^\sigma}}{2^{3k/2}\cdot\pi^{k/2+1}\cdot(f^\sigma,f^\sigma)_{\Gamma_0(N)}\cdot\langle s_{f^\sigma}',s_{f^\sigma}'\rangle_\mathrm{GS}}=\Bigl(L^\mathrm{alg}(g^\sigma,k/2)\Bigr)^{-1}. \] 
Properties (1) and (2) ensure that we can apply Theorem \ref{ratconj0thm} to $g$. Therefore, there exists $\widetilde{\mathscr{L}}'\in F^\times$ such that 
\begin{equation} \label{rat-conj1}
\sigma\bigl(\widetilde{\mathscr{L}}'\bigr)=\frac{L'(f^\sigma,k/2)\cdot\sqrt{|D_K|}\cdot i^{k/2-1}\cdot\Omega_{g^\sigma}}{\pi^{k/2+1}\cdot{(f^\sigma,f^\sigma)}_{\Gamma_0(N)}\cdot{\langle s_{f^\sigma}',s_{f^\sigma}'\rangle}_\mathrm{GS}}\end{equation}
for all $\sigma\in\Sigma$. Comparing \eqref{ratconj1} and \eqref{rat-conj1}, we see that it suffices to show that there exists $\mathscr A\in F^\times$ such that
\begin{equation} \label{rat-conj2}
\sigma(\mathscr A)=\frac{\Omega_{f^\sigma}\cdot\Omega_{g^\sigma}\cdot\sqrt{|D_K|}}{\pi^2\cdot{(f^\sigma,f^\sigma)}_{\Gamma_0(N)}}\cdot\frac{\Reg_\mathscr{B}^\sigma(\MM)}{{\langle s_{f^\sigma}',s_{f^\sigma}'\rangle}_\mathrm{GS}} 
\end{equation}
for all $\sigma\in\Sigma$, for then \eqref{ratconj1} follows with $\mathscr{L}'\defeq\widetilde{\mathscr L}'\big/\mathscr A$. To prove \eqref{rat-conj2}, we deal with the two factors on the right-hand side separately. By Lemma \ref{lemma4.4}, $\Omega_f\Omega_{g}\sqrt{|D_K|}=i\Omega_f^+\Omega_f^-$ up to elements of $F^\times$; furthermore, $i\Omega_f^+\Omega_f^-\big/\pi^2{(f,f)}_{\Gamma_0(N)}$ belongs to $F^\times$ and satisfies the equality $\sigma\bigl(i\Omega_f^+\Omega_f^-\big/\pi^2{(f,f)}_{\Gamma_0(N)}\bigr)=i\Omega_{f^\sigma}^+\Omega_{f^\sigma}^-\big/\pi^2{(f^\sigma,f^\sigma)}_{\Gamma_0(N)}$ for all $\sigma\in\Sigma$ (see, \emph{e.g.}, \cite[\S1.4]{Hsieh-triple}). This ensures that $\Omega_f\Omega_g\sqrt{|D_K|}\big/\pi^2{(f,f)}_{\Gamma_0(N)}$ belongs to $F^\times$ and $\sigma\bigl(\Omega_f\Omega_g\sqrt{|D_K|}\big/\pi^2{(f,f)}_{\Gamma_0(N)}\bigr)=\Omega_{f^\sigma}\Omega_{g^\sigma}\sqrt{|D_K|}\big/\pi^2{(f^\sigma,f^\sigma)}_{\Gamma_0(N)}$ for all $\sigma\in\Sigma$.

As for the other term, Proposition \ref{prop Heegner} implies that $\Reg_\mathscr{B}^{\iota_F}(\MM)\big/{\langle s_f',s_f'\rangle}_\mathrm{GS}$ also belongs to $F^\times$ and satisfies $\sigma\bigl(\Reg_\mathscr{B}^{\iota_F}(\MM)\big/{\langle s_f',s_f'\rangle}_\mathrm{GS}\bigr)=\Reg_\mathscr{B}^\sigma(\MM)\big/{\langle s_{f^\sigma}',s_{f^\sigma}'\rangle}_\mathrm{GS}$ for all $\sigma\in\Sigma$. Summing up, the element
\[ \mathscr A\defeq\frac{\Omega_f\cdot\Omega_g\cdot\sqrt{|D_K|}}{\pi^2\cdot{(f,f)}_{\Gamma_0(N)}}\cdot\frac{\Reg_\mathscr{B}^{\iota_F}(\MM)}{{\langle s_f',s_f'\rangle}_\mathrm{GS}}\in F^\times \]
satisfies \eqref{rat-conj2}, which concludes the proof. \end{proof}

As observed in \S \ref{ratconj2-subsubsec}, Theorem \ref{ratconj0thm} also shows that, under the assumptions appearing in its statement, the analytic rank $1$ case of Conjecture \ref{ratconj2} holds true.

\subsection{$p$-TNC for $\MM$ in analytic rank $1$}

We prove the main result of this paper, which says that, under the assumptions in \S \ref{main-assumptions-subsubsec} and those directly described in the statement of Theorem \ref{ThmTNC}, the $p$-part of the TNC for $\MM$ is true when $r_\an(\MM)=1$.

\subsubsection{Splitting ${\CH^{k/2}_{\arith}(X/L)}_F$ over the Hecke algebra}

With notation as in \S \ref{isotypic-subsubsec}, let us write $\theta_{f,F}:\mathfrak{H}_k(\Gamma_0(N))_F\twoheadrightarrow F$ for the (surjective) $F$-linear extension of $\theta_f$ and denote by $\mathrm{Ann}_{\mathfrak{H}_k(\Gamma_0(N))_F}\!\bigl(\ker(\theta_{f,F})\bigr)$ the annihilator ideal of $\ker(\theta_{f,F})$ in $\mathfrak{H}_k(\Gamma_0(N))_F$. The \emph{congruence ideal of $\theta_{f,F}$} is the ideal of $F$ given by 
\[ \eta_{\theta_{f,F}}\defeq\theta_{f,F}\Bigl(\mathrm{Ann}_{\mathfrak{H}_k(\Gamma_0(N))_F}\!\bigl(\ker(\theta_{f,F})\bigr)\!\Bigr). \]
As explained, \emph{e.g.}, in \cite[p. 250]{lenstra}, it turns out that $\ker(\theta_{f,F})\cap\mathrm{Ann}_{\mathfrak{H}_k(\Gamma_0(N))_F}\!\bigl(\ker(\theta_{f,F})\bigr)=\{0\}$: this follows from the fact that $\mathfrak{H}_k(\Gamma_0(N))_F$ is (trivially) flat over $F$. In particular, $\eta_{\theta_{f,F}}=F$ and the (tautological) short exact sequence of $\mathfrak{H}_k(\Gamma_0(N))_F$-modules
\[ 0\longrightarrow\ker(\theta_{f,F})\longrightarrow\mathfrak{H}_k(\Gamma_0(N))_F\xrightarrow{\theta_{f,F}}F\longrightarrow 0 \] 
splits canonically (actually, in \cite{lenstra} it is assumed that the the counterpart of $\theta_{f,F}$ takes values in a complete DVR, but this property is not necessary for the conclusion above to hold). Thus, there is an identification (or, rather, a canonical isomorphism) of $\mathfrak{H}_k(\Gamma_0(N))_F$-modules
\begin{equation} \label{lenstra-bis}
\mathfrak{H}_k(\Gamma_0(N))_F=\ker(\theta_{f,F})\oplus F.
\end{equation}
Now let $L$ be a number field. The splitting in \eqref{lenstra-bis} yields a splitting 
\begin{equation} \label{CH-split-bis-eq}
\resizebox{.9\hsize}{!}{${\CH^{k/2}_{\arith}(X/L)}_F=\Bigl({\CH^{k/2}_{\arith}(X/L)}_F\otimes_{\mathfrak{H}_k(\Gamma_0(N))_F}\!F\Bigr)\textstyle{\bigoplus}\Bigl({\CH^{k/2}_{\arith}(X/L)}_F\otimes_{\mathfrak{H}_k(\Gamma_0(N))_F}\!\ker(\theta_{f,F})\Bigr).$}
\end{equation}
On the other hand, there is an identification
\begin{equation} \label{mot-splitting-eq}
H^1_\mot(L,\MM)={\CH^{k/2}_{\arith}(X/L)}_F\otimes_{\mathfrak{H}_k(\Gamma_0(N))_F}\!F.
\end{equation}
Combining \eqref{CH-split-bis-eq} and \eqref{mot-splitting-eq}, we get a canonical surjection
\begin{equation} \label{CH-mot-proj-eq}
\Pi_{\MM,L}:{\CH^{k/2}_{\arith}(X/L)}_F\longepi H^1_\mot(L,\MM)
\end{equation}
of $\mathfrak{H}_k(\Gamma_0(N))_F$-modules. In order not to make our notation heavier than necessary, if $\mathcal R$ is any subring of $F$, then we shall use the same symbol for the obvious natural map induced on ${\CH^{k/2}_{\arith}(X/L)}_{\mathcal R}$ by $\Pi_{\MM,L}$.

\subsubsection{Assumptions} \label{main-assumptions-subsubsec}

We work under the following assumption on the form $f$ and the prime $p$; we freely use notation from previous sections.

\begin{assumption} \label{ass AJ} 
\begin{enumerate}
%\item The form $f$ has no complex multiplication.
%\vskip 1mm
\item Assumptions \ref{ass} and \ref{ass-0} are satisfied by $(f,\p)$ for all $\p\,|\,p$. 
%\vskip 1mm
%\item Assumption \ref{ass-p-iso} holds true, \emph{i.e.}, $f$ is $p$-isolated. 
%\vskip 1mm
%\item \label{AJ-condition} $\AJ_{K,\p}$  is injective on $\Heeg_{K,N}^{\mathcal G_1}$ for every imaginary quadratic field $K$ in which the primes dividing $Np$ split and for all $\p\,|\,p$.
%\vskip 1mm
\item \label{reg-condition} Conjecture \ref{cong-int} holds true for $L=\Q$, \emph{i.e.}, the $p$-adic regulator 
\[ \reg_p:H^1_\mot(\Q,\MM{)}_\text{$p$-int}\longrightarrow H^1_f(\Q,T_p) \] 
from \eqref{AJ-map-p} is an isomorphism of $\cO_p$-modules.
%\vskip 1mm
%\item \label{height-condition} The height pairing ${\langle\cdot,\cdot\rangle}_{\GS}$ on $H^1_\mot(L,\MM)$ is non-degenerate for $L=\Q$ and for $L$ an imaginary quadratic field in which the primes dividing $Np$ split. 
\end{enumerate}
\end{assumption}

These hypotheses could certainly be relaxed, but we prefer to state here assumptions that sit well in the general framework of Section \ref{secmot}. Notice that condition \eqref{reg-condition} in Assumption \ref{ass AJ} implies the validity of Conjecture \ref{regpconj} for the number fields specified above.

\begin{remark}
Since $N$ is assumed to be square-free, $f$ has no complex multiplication.
\end{remark}

%\begin{remark} 
%Conditions (1) and (2) in Assumption \ref{ass AJ} are satisfied for all $p$ except possibly a finite number of them. The content of Assumption (3), assuming Conjecture \ref{regpconj}, is that the map \eqref{AJ-map-p} is surjective, a condition that, given Conjecture \eqref{conj-int}, we expect to be true for all $p$ except possibly a finite number.  
%, and, by Conjecture \ref{conj-cong}, is actually related to congruences between modular forms. To check that $(\lambda_p)=\mathcal{O}_p$ for all $p$ except possibly a finite number, note that the periods $\Omega_f^\pm$ depend on the choice of $\delta_f^\pm$ only up to $p$-adic units, and therefore the quotient $\lambda_p$ depends on $p$ only up to $p$-adic units. Since $\lambda_p$ is an element of $F$, it is a $p$-adic unit for all $p$ except possibly a finite number. 
%Condition (4) could be relaxed by requiring that $\reg_p$ and $\reg_{K,p}$ are isomorphisms for all imaginary quadratic fields $K$ appearing in \cite[p. 543, Theorem]{BFH}. 
%One could also replace this condition with the injectivity of the $p$-adic regulator on $\Heeg_{K,N}$ for the same set of imaginary quadratic fields, but the statement of the following theorem would become less clear. 
%\end{remark}

\subsubsection{Tamagawa ideals} 

Recall that $\MM(f^K)$ denotes the motive of the twist $f^K$ of $f$, which is a newform of level $ND_K^2$, as $(N,D_K)=1$. In the statement below, $\ell$ is a prime number.

\begin{proposition} \label{lemmatam}
\begin{enumerate}
\item If $\ell\,|\,N$, then $\Tam_\ell^{(p)}(\MM)=\cO_p$. 
\item If $\ell\,|\,ND_K$, then $\Tam_\ell^{(p)}\bigl(\MM(f^K)\bigr)=\cO_p$. 
\end{enumerate}
\end{proposition}

\begin{proof} Let us prove part (1). Take a prime number $\ell\,|\,N$. Since $\ell$ divides $N$ exactly (which is true because $N$ is assumed to be square-free), the action of the inertia subgroup $I_\ell$ of $G_{\Q_\ell}$ acting on $A_p$ is given, up to isomorphism, by a matrix $\bigl(\begin{smallmatrix}0&c\\0&0\end{smallmatrix}\bigr)$ for a suitable $c\neq0$ (see, \emph{e.g.}, \cite[\S 12.4.4.2 and Lemma 12.4.5, (ii)]{Nek-Selmer}). Thus, $A_p^{I_\ell}$ is isomorphic to $F_p/\cO_p$, hence divisible, so $H^1_\mathrm{unr}(\Q_\ell,A_p)=0$ and part (1) follows from part (1) of Proposition \ref{Tamv}. 

Now we turn to part (2) (\emph{cf.} also \cite[Corollary 9.2]{SZ}). Let us write $A_p^K$ for the analogue of $A_p$ relative to $f^K$; there is an isomorphism $A^K_p\simeq A_p$ of $G_K$-modules. Take a prime number $\ell\,|\,ND_K$, let $\lambda$ be the unique prime of $K$ above $\ell$ and let $I_\ell$ (respectively, $I_\lambda$) be the inertia subgroup of $G_{\Q_\ell}$ (respectively, $G_{K_\lambda}$); note that the residue field of $K$ at $\lambda$ is $\F_\ell$. If $\ell\,|\,N$, then one can argue as in the proof of part (1). Suppose that $\ell\,|\,D_K$. The prime $\ell$ ramifies in $K$, so $I_\ell=I_\lambda$. Since $p\neq2$, we may split $H^1_\mathrm{unr}(K_\lambda,A_p)=H^1(\F_\ell,A_p^{I_\ell})$ into the direct sum of its eigenspaces for the action of $\mathcal{G}\defeq\Gal(K_\lambda/\Q_\ell)$; the eigenspace on which $\mathcal{G}$ acts trivially (respectively, as $-1$) is $H^1_\mathrm{unr}(K_\lambda,A_p)^{\mathcal{G}}$ (respectively, is isomorphic to 
$H^1_\mathrm{unr}(K_\lambda,A_p^K)^ {\mathcal{G}}$). Therefore, the canonical map 
\[ H^1_\mathrm{unr}(\Q_\ell,A_p)\oplus H^1_\mathrm{unr}\bigl(\Q_\ell,A^K_p\bigr)\longrightarrow H^1_\mathrm{unr}(K_\lambda,A_p) \] 
is an isomorphism and, since $\mathcal{I}\bigl(H^1_\mathrm{unr}(K_\lambda,A_p)\bigr)=\mathcal{I}\bigl(H^1_\mathrm{unr}(\Q_\ell,A^K_p)\bigr)=\cO_p$ by part (1), we deduce that 
$\mathcal{I}\bigl(H^1_\mathrm{unr}(\Q_\ell,A^K_p)\bigr)=\mathrm{Tam}_\ell^{(p)}\bigl(\MM(f^K)\bigr)=\cO_p$ as well, concluding the proof. \end{proof}

\subsubsection{Comparison of periods} 

Since $f$ is $p$-isolated (\emph{cf.} Assumption \ref{ass AJ}), by \cite[Theorem 0.1]{Hida-Modules} there is an equality 
\begin{equation} \label{comp-per}
\Biggl(\frac{\pi^2\cdot{(f,f)}_{\Gamma_0(N)}}{\Omega_f\cdot\Omega_{f^K}}\Biggr)=\cO_p
\end{equation}
of (fractional) $\cO_p$-ideals. See also \cite[\S1.4]{Hsieh-triple} for details.

\subsubsection{A distinguished $\Q$-rational cycle} \label{distinguished-subsubsec}

Recall the cycle $\XX_K\in\Heeg_{K,N}^{\mathcal G_1}$ introduced in \eqref{xi K} and the map
\[
\Pi_{\MM,\Q}:{\CH^{k/2}_\arith(X/\Q)}_F\longrightarrow H^1_\mot(\Q,\MM) 
\]
from \eqref{CH-mot-proj-eq} with $L=\Q$. It turns out that, in the setting we are concerned with, $\XX_K$ is $\Q$-rational, as we prove in the following proposition. We use notation from \S \ref{imaginary-1-subsubsec}.

\begin{proposition} \label{Gamma-nonzero-lemma}
Assume that $r_\an(\MM)=1$ and that $K\in\mathscr I_1(f,p)$ has the property that $\AJ_{K,\p}$ is injective on $\Heeg_{K,N}^{\mathcal G_1}$ for some $\p\,|\,p$. Then 
\begin{enumerate}
\item $\XX_K\in{\CH^{k/2}_\arith(X/\Q)}_F$;
\item $\Pi_{\MM,\Q}(\XX_K)\not=0$.
\end{enumerate}
\end{proposition}

\begin{proof} Let $\p$ be a prime of $F$ above $p$ such that $\AJ_{K,\p}$ is injective on $\Heeg_{K,N}^{\mathcal G_1}$. Since $K\in\mathscr I_1(f,p)$, we know that $r_\an(f/K)=1$. Part (1) of Proposition \ref{coro zhang} tells us that $y_{K,\p}$ is not $\cO_\p$-torsion. Furthermore, an analysis of the action of complex conjugation on $y_{K,\p}$ shows that $y_{K,\p}\in\Lambda_\p(K)^{\Gal(K/\Q)}$ (see, \emph{e.g.}, the proof of \cite[Theorem 5.27]{Vigni}). On the other hand, $\AJ_{K,\p}(\XX_K)=(k-2)!\cdot y_{K,\p}$ by Lemma \ref{lemma Heegner 3}, so $\XX_K$ is not torsion. In addition, the map $\AJ_{K,\p}$ is $\Gal(K/\Q)$-equivariant, so the injectivity of $\AJ_{K,\p}$ on $\Heeg_{K,N}^{\mathcal G_1}$ (which we are assuming) allows us to conclude that 
\[ \begin{split}
   \XX_K\in\Bigl(\Heeg_{K,N}^{\mathcal G_1}\Bigr)^{\!\Gal(K/\Q)}&\subset\Bigl(\CH^{k/2}_\arith(X/K_1)^{\mathcal G_1}\Bigr)^{\!\Gal(K/\Q)}\\&=\CH^{k/2}_\arith(X/K_1)^{\Gal(K_1/\Q)}.
   \end{split} \]
Let $\mathcal K\in\{F,F_\p\}$. By what we explained in \S \ref{base-change-trace-subsubsec} (\emph{cf.} Notation/Convention \ref{notation-convention}), the cyclic subgroup of $\CH^{k/2}_\arith(X/K_1)^{\Gal(K_1/\Q)}$ generated by $\XX_K$ injects into $\CH^{k/2}_\arith(X/K_1)^{\Gal(K_1/\Q)}_{\mathcal K}$: we will not distinguish between $\XX_K$ and its image in this vector space; moreover, we identify this element (hence $\XX_K$ as well) with its image in ${\CH^{k/2}_\arith(X/\Q)}_{\mathcal K}$. With this convention in force, this shows, in particular, that $\XX_K\in{\CH^{k/2}_\arith(X/\Q)}_F$. There is a commutative diagram
\begin{equation} \label{AJ-res-diagram-eq}
\xymatrix@C=40pt@R=30pt{{\CH^{k/2}_\arith(X/\Q)}_F\ar[r]\ar@{^{(}->}[d]^-{\iota_{\Q\to K}}&{\CH^{k/2}_\arith(X/\Q)}_{F_\p}\ar@{^{(}->}[d]^-{\iota_{\Q\to K}}\ar[r]^-{\AJ_{\Q,\p}}&\Lambda_\p(\Q)\otimes_{\cO_\p}F_\p\ar@{^{(}->}[d]^-{\res_{K/\Q}}\\ {\CH^{k/2}_\arith(X/K)}_F\ar[r]&{\CH^{k/2}_\arith(X/K)}_{F_\p}\ar[r]^-{\AJ_{K,\p}}&\Lambda_\p(K)\otimes_{\cO_\p}F_\p}
\end{equation}
in which, as before, the left and middle vertical arrows are the base change maps from \eqref{base-change-eq2} and the unlabelled maps are extensions of scalars. By a slight abuse of notation, we write $\AJ_{\Q,\p}$ (respectively, $\AJ_{K,\p}$) also for the composition of the two upper (respectively, lower) horizontal maps. Since $y_{K,\p}$ is not $\cO_\p$-torsion, diagram \eqref{AJ-res-diagram-eq} shows that $\AJ_{\Q,\p}(\XX_K)\not=0$ in $H^1_f(\Q,V_\p)$. Finally, triangle \eqref{CH-p-triangle-eq} with $L=\Q$ and $\star=\p$ induces a commutative diagram
\begin{equation} \label{CH-p-triangle-eq2}
\xymatrix@R=30pt@C=25pt{{\CH^{k/2}_{\arith}(X/\Q)}_F\ar[r]\ar[rd]^-{\Pi_{\MM,\Q}}&{\CH^{k/2}_{\arith}(X/\Q)}_{F_\p}\ar[rr]^-{\AJ_{\Q,\p}}\ar@{->>}[rd]^-{\Pi_{\MM,\Q,\p}}&& H^1_f(\Q,V_\p)\\
&H^1_\mot(\Q,\MM)\ar[r]&H^1_\mot(\Q,\MM)_\p\ar[ru]^-{\reg_\p}} 
\end{equation}
in which $H^1_\mot(\Q,\MM)_\p$ is the $F_\p$-vector space from \eqref{H^1-mot-star-eq} and the unlabelled horizontal maps are extensions of scalars. Since $\AJ_{\Q,\p}(\XX_K)\not=0$ in $H^1_f(\Q,V_\p)$, diagram \eqref{CH-p-triangle-eq2} gives $\Pi_{\MM,\Q}(\XX_K)\not=0$. \end{proof}

From here on, set 
\begin{equation} \label{Gamma-MM-eq}
\YY_\MM\defeq\frac{1}{(k-2)!}\cdot\Pi_{\MM,\Q}(\XX_K)\in H^1_\mot(\Q,\MM)\smallsetminus\{0\}.
\end{equation}
Set $\YY_{\MM,K}\defeq\iota_{\Q\to K}(\YY_\MM)\in H^1_\mot(K,\MM)\smallsetminus\{0\}$. As a consequence of triangle \eqref{CH-p-triangle-eq} with $L=K$ and Lemma \ref{lemma Heegner 3}, there is an equality
\begin{equation} \label{Gamma-MM-eq2}
\reg_{K,\p}(\YY_{\MM,K})=y_{K,\p}
\end{equation}
for each prime $\p$ of $F$ above $p$. 

The next result establishes, for each $\p\,|\,p$, an isomorphism between $\Lambda_\p(\Q)$ and $\Lambda_\p(K)^{\Gal(K/\Q)}$.

\begin{proposition} \label{res-iso-prop}
For each prime $\p$ of $F$ above $p$, restriction induces an isomorphism
\[ \res_{K/\Q}:\Lambda_\p(\Q)\overset\simeq\longrightarrow\Lambda_\p(K)^{\Gal(K/\Q)} \]
of $\cO_\p$-modules.
\end{proposition}

\begin{proof} Let $\p$ be a prime of $F$ above $p$. Since $[K:\Q]=2$ and $p$ is odd, base change gives an isomorphism
\[ \iota_{\Q\to K}:{\CH^{k/2}_0(X/\Q)}_{\cO_\p}\overset\simeq\longrightarrow\CH^{k/2}_0(X/K)_{\cO_\p}^{\Gal(K/\Q)} \]
of $\cO_\p$-modules (\emph{cf.} \S \ref{base-change-trace-subsubsec}). On the other hand, there is a commutative square
\[
\xymatrix@C=40pt@R=32pt{{\CH^{k/2}_0(X/\Q)}_{\cO_\p}\ar@{->>}[r]^-{\AJ_{\Q,\p}}\ar[d]^-{\iota_{\Q\to K}}_-\simeq&\Lambda_\p(\Q)\ar@{^{(}->}[d]^-{\res_{K/\Q}}\\ \CH^{k/2}_0(X/K)_{\cO_\p}^{\Gal(K/\Q)}\ar@{->>}[r]^-{\AJ_{K,\p}}&\Lambda_\p(K)^{\Gal(K/\Q)}}
\]
in which the bottom horizontal arrow is surjective by \cite[Lemma 5.8]{Vigni} and the injectivity of $\res_{K/\Q}$ follows as in the proof of Proposition \ref{c_M(1)-prop}. We conclude that $\res_{K/\Q}$ is necessarily surjective, hence an isomorphism. \end{proof}

Under the assumptions of Proposition \ref{Gamma-nonzero-lemma}, we know that $y_{K,\p}\in\Lambda_\p(K)^{\Gal(K/\Q)}$. Thus, in light of Proposition \ref{res-iso-prop}, we define
\begin{equation} \label{y-Q-eq}
y_{\Q,\p}\defeq\res_{K/\Q}^{-1}(y_{K,\p})\in\Lambda_\p(\Q).
\end{equation}
In particular, $y_{\Q,\p}$ is not $\cO_\p$-torsion. Observe that the square 
\[
\xymatrix@C=40pt@R=32pt{H^1_\mot(\Q,\MM)_\p\ar@{->>}[r]^-{\reg_\p}\ar[d]^-{\iota_{\Q\to K}}_-\simeq&\Lambda_\p(\Q)\otimes_{\cO_\p}\!F_\p\ar[d]^-{\res_{K/\Q}}_-\simeq\\ H^1_\mot(K,\MM)_\p^{\Gal(K/\Q)}\ar@{->>}[r]^-{\reg_{K,\p}}&\Lambda_\p(K)^{\Gal(K/\Q)}\otimes_{\cO_\p}\!F_\p}
\]
is commutative, so equality \eqref{Gamma-MM-eq2} guarantees that
\begin{equation} \label{Gamma-MM-eq3}
\reg_\p(\YY_\MM)=y_{\Q,\p}
\end{equation}
for each prime $\p$ of $F$ above $p$. 

\subsubsection{Proof of Theorem B} \label{main-proof-subsubsec}

Retaining the assumptions of Proposition \ref{Gamma-nonzero-lemma}, now we state a technical result that will be used in the proof of our main theorem. We fix a prime $\p$ of $F$ above $p$.

\begin{lemma} \label{lengths-lemma}
$\length_{\cO_\p}\bigl(\Lambda_{\p}(\Q)/y_{\Q,\p}\cdot\cO_\p\bigr)=\length_{\cO_\p}\bigl(\Lambda_{\p}(K)/y_{K,\p}\cdot\cO_\p\bigr)$.
\end{lemma}

\begin{proof} Since $\res_{K/\Q}(y_{\Q,\p})=y_{K,\p}$ by \eqref{y-Q-eq}, one can proceed \emph{mutatis mutandis} as in the proof of Proposition \ref{lemma-masoero1}. \end{proof}

We restate Theorem B in a precise form; as before, we employ notation from Remark \ref{AJ-twist-rem} and \S \ref{imaginary-subsec}. Recall that Assumption \ref{ass AJ} is in force.

\begin{theorem} \label{ThmTNC}
Assume that $r_\an(\MM)=1$ and that there exists $K\in\mathscr I_1(f,p)$ such that
\begin{enumerate}
\item $\AJ_{K,\p}$ is injective on $\Heeg_{K,N}^{\mathcal G_1}$ for all $\p\,|\,p$;
\item there is $K'\in\mathscr I_0(f^K,p)$ with $\AJ_{f^K,K',\p}$ injective on $\Heeg_{K',N}^{\mathcal G'_1}$ for some $\p\,|\,p$;
\item the $p$-part of Conjecture \ref{regpconj} for $\MM(f^K)$ over $\Q$ holds true. 
\end{enumerate}
Moreover, assume that
\begin{enumerate}
\item[(4)] the $p$-part of Conjecture \ref{regpconj} over $\Q$ holds true.
\end{enumerate}
Then the following results hold:
\begin{enumerate}
\item[(a)] $r_\alg(\MM)=1$; 
\item[(b)] $\Sha_p^{\BKK}(\Q,\MM)=\Sha_p^{\Nek}(\Q,\MM)$;  
\item[(c)] the $p$-part of Conjecture \ref{TNC} is true. 
\end{enumerate}
\end{theorem}

\begin{proof} Part (a) was already proved in Theorem \ref{nekovar-thm2}. On the other hand, the arguments used in the proof of Proposition \ref{selmer-alternative-prop} show that there is an equality $\Sha_\p^{\BKK}(\Q,\MM)=\Sha_\p^{\Nek}(\Q,\MM)$ for each $\p\,|\,p$ (\emph{cf.} also the proof of Theorem \ref{nekovar-thm2}), and then part (b) follows upon taking direct sums over all such $\p$.

As we did in rank $0$ in Theorem \ref{skinner-urban-thm}, we prove part (c) by checking equality \eqref{eqBK} in Theorem \ref{motivesthm}. First of all, the existence of $K\in\mathscr I_1(f,p)$ satisfying properties (2) and (3) guarantees, by Theorem \ref{ratconj1thm}, that Conjecture \ref{ratconj} is true in this case. Thus, keeping property (4) in mind and noting that Theorem \ref{zhang-thm} implies that the regulator we will be working with is non-zero (\emph{cf.} below), all the assumptions in Theorem \ref{motivesthm} are verified (\emph{cf.} Remark \ref{non-degeneracy-rem}). As a further preliminary observation (implicit in the proofs of parts (a) and (b)), notice that if we fix $K\in\mathscr I_1(f,p)$ with the properties in the statement of the theorem and let $\p$ be a prime of $F$ above $p$, then $r_\an(f/K)=1$ and, thanks to property (1) and part (1) of Proposition \ref{coro zhang}, $y_{K,\p}$ is not torsion. 

By Theorem \ref{nekovar-thm2}, $r_\alg(\MM)=1$. Let $\YY_\MM\in H^1_\mot(\Q,\MM)\smallsetminus\{0\}$ be as in \eqref{Gamma-MM-eq}. Thus, $\mathscr B\defeq\{\YY_\MM\}$ is an $F$-basis of $H^1_\mot(\Q,\MM)$. Let $y_{\Q,\p}\in\Lambda_\p(\Q)$ be as in \eqref{y-Q-eq} and recall from \eqref{Gamma-MM-eq3} that $\reg_\p(\YY_\MM)=y_{\Q,\p}$. Recall the $\p$-adic regulators
\[ \reg_\p:H^1_\mot(\Q,\MM{)}_\star\longrightarrow H^1_f(\Q,M_\p) \] 
from \eqref{pp-reg-eq} if $(\star,M)=(\p,V)$ or from \eqref{AJ-map} if $(\star,M)=(\text{$\p$-int},T)$. The surjectivity of $\Pi_{\MM,\Q,\p}$ implies that $\im(\reg_\p)=\Lambda_\p(\Q)$. It follows that condition \eqref{reg-condition} in Assumption \ref{ass AJ} yields an equality $\Lambda_\p(\Q)=H^1_f(\Q,T_\p)$ of free $\cO_\p$-modules of rank $1$; once we view them in $H^1_f(\Q,V_\p)$, these two $\cO_\p$-lattices coincide with $\uH^1_f(\Q,T_\p)$.

For each $\p\,|\,p$, pick $\varpi_\p\in\p\smallsetminus(\p^2\cup\bigcup_{\p'|p,\,\p'\not=\p}\p')$; in particular, $\varpi_\p$ is a uniformizer at $\p$. Let us write
\begin{equation} \label{H^1-y-eq}
H^1_f(\Q,T_\p)\big/y_{\Q,\p}\cdot\cO_\p\simeq\cO_\p\big/(\p\cO_\p)^{f_\p}=\cO_\p\big/(\varpi_\p^{f_\p}\cO_\p) 
\end{equation}
for some $f_\p\in\N$, then set $\varpi_p\defeq\prod_{\p|p}\varpi_\p^{f_\p}\in\cO_F\smallsetminus\{0\}$. Now define 
\[ y_{\Q,p}\defeq{(y_{\Q,\p})}_{\p|p}\in\bigoplus_{\p|p}H^1_f(\Q,T_\p)=H^1_f(\Q,T_p) \]
and $\hat y_{\Q,p}\defeq\varpi^{-1}_p\cdot y_{\Q,p}\in H^1_f(\Q,T_p)$. Observe that, by construction, $\{\hat y_{\Q,p}\}$ is an $\cO_p$-basis of $H^1_f(\Q,T_p)=\uH^1_f(\Q,T_p)$. Furthermore, with notation as in \S \ref{linear-algebra-lattices-subsubsec}, $\reg_p(\YY_\MM)=y_{\Q,p}$, so that $\tilde{\mathscr B}=\{y_{\Q,p}\}$. Therefore, there is an equality
\begin{equation} \label{A-varpi-eq}
\mathtt{A}_{\tilde{\mathscr B}}=(\varpi_p).
\end{equation}
Recall from \S \ref{distinguished-subsubsec} the element $\YY_{\MM,K}=\iota_{\Q\to K}(\YY_\MM)\in H^1_\mot(K,\MM)\smallsetminus\{0\}$. By Proposition \ref{prop Heegner} and our choice of $p$, there are equalities
\[ \begin{split}
   {\langle s_f',s_f'\rangle}_{\GS}&=\frac{{\deg(\pi_N)}\cdot{\binom{k-2}{k/2-1}}}{{(-2D_K})^{k/2-1}}\cdot{\big\langle\YY_{\MM,K},\YY_{\MM,K}\big\rangle}_{\GS,\iota_F}^K\\
   &=\frac{{\deg(\pi_N)}\cdot{\binom{k-2}{k/2-1}}}{{(-2D_K})^{k/2-1}}\cdot{\big\langle\YY_\MM,\YY_\MM\big\rangle}_{\GS,\iota_F}^K, 
   \end{split} \]
where the superscripts indicate that the pairings are taken with respect to the ground field $K$. On the other hand, ${\langle\YY_\MM,\YY_\MM\rangle}_{\GS,\iota_F}^K=2{\langle\YY_\MM,\YY_\MM\rangle}_{\GS,\iota_F}^\Q$, where the pairing on the right is taken relative to the ground field $\Q$ (see, \emph{e.g.}, \cite[\S 3.1.4]{BGS}). Thus, since $p\nmid 2D_K\deg(\pi_N)$, we can replace the term ${\langle s_f',s_f'\rangle}_{\GS}$ in Theorem \ref{zhang-thm} with $\binom{k-2}{k/2-1}\cdot{\langle\YY_\MM,\YY_\MM\rangle}_{\GS,\iota_F}^\Q$. Observe that Theorem \ref{zhang-thm} ensures that ${\langle\YY_\MM,\YY_\MM\rangle}_{\GS,\iota_F}^\Q\not=0$ because $r_\an(f/K)=1$; since $\Reg_{\mathscr B}^{\iota_F}(\MM)={\langle\YY_\MM,\YY_\MM\rangle}_{\GS,\iota_F}^\Q$, we have $\Reg_{\mathscr B}^{\iota_F}(\MM)\not=0$.

Combining Theorem \ref{zhang-thm} with the period comparison formula \eqref{comp-per} yields, for each prime $\p$ of $F$ above $p$, an equality
\begin{equation} \label{eq7.4}
\Biggl(\frac{\bigl((k/2-1)!\bigr)^2\cdot L'(f/K,k/2)}{\pi^{k-2}\cdot\Omega_f\cdot\Omega_{f^K}\cdot\Reg_{\mathscr B}^{\iota_F}(\MM)}\Biggr)=\cO_\p
\end{equation}
of fractional $\cO_\p$-ideals. Since equality \eqref{eq7.4} holds for each $\p\,|\,p$, we get an equality
\begin{equation} \label{proofeq2}
\begin{split}
\Biggl(\frac{\bigl((k/2-1)!\bigr)^2\cdot L'(f/K,k/2)}{\pi^{k-2}\cdot\Omega_f\cdot\Omega_{f^K}\cdot\Reg^{\iota_F}_\mathscr{B}(\MM)}\Biggr)&=(\varpi^{-2}_p)\cdot(\varpi^2_p)\cdot\cO_p\\
&=\bigl(\det(\mathtt{A}_{\tilde{\mathscr{B}}})\bigr)^{-2}\cdot\Bigl(\mathcal{I}_{\cO_p}\bigl(H^1_f(\Q,T_p)\big/y_{\Q,p}\cdot\cO_p\bigr)\!\Bigr)^2\\
&=\bigl(\det(\mathtt{A}_{\tilde{\mathscr{B}}})\bigr)^{-2}\cdot\mathcal{I}_{\cO_p}\bigl(\Sha_p^{\BKK}(K,\MM)\bigr)
\end{split}
\end{equation}
of fractional $\cO_p$-ideals, where the second equality follows from \eqref{H^1-y-eq} and \eqref{A-varpi-eq}, while the third, in light of Definition \ref{Sha-def} and Lemma \ref{lengths-lemma}, is a consequence of Theorem \ref{teorema-masoero} and Proposition \ref{lemma-masoero1}. 

The $p$-adic Galois representation attached to $f^K$ is the twist of $\rho_p$ by $\epsilon_K$, so Assumption \ref{main-assumption} holds for $f^K$. Thus, by Lemma \ref{lemma4.4}, Theorem \ref{skinner-urban-thm} and Proposition \ref{lemmatam}, there is an equality 
\begin{equation}\label{proofeq3}
\Biggl(\frac{(k/2-1)!\cdot L(f^K,k/2)}{(2\pi i)^{k/2-1}\cdot\Omega_{f^K}}\Biggr)=\mathcal{I}_{\cO_p}\Bigl(\Sha_p^{\BKK}\bigl(\Q,\MM(f^K)\bigr)\!\Bigr)
\end{equation}
of fractional $\cO_p$-ideals. Combining the factorization 
%\begin{equation}\label{proofeq1}
\[ L'(f/K,k/2)=L'(f,k/2)\cdot L(f^K,k/2) \]
%\end{equation}
and formula \eqref{comp-per} with equalities \eqref{proofeq2} and \eqref{proofeq3}, we obtain an equality
\[ \begin{split}
   \biggl(\frac{(k/2-1)!\cdot L'(f,k/2)}{(2\pi i)^{k/2-1}\cdot\Omega_f\cdot\Reg_\mathscr{B}^{\iota_F}(\MM)}\biggr)=&\;\bigl(\det(\mathtt{A}_{\tilde{\mathscr{B}}})\bigr)^{-2}\cdot\mathcal{I}_{\cO_p}\bigl(\Sha_p^{\BKK}(K,\MM)\bigr)\\&\;\cdot\mathcal{I}^{-1}_{\cO_p}\Bigl(\Sha_p^{\BKK}\bigl(\Q,\MM(f^K)\bigr)\!\Bigr). 
   \end{split} \]
Furthermore, the splitting $\Sha_p^{\BKK}(K,\MM)=\Sha_p^{\BKK}(\Q,\MM)\oplus\Sha_p^{\BKK}\bigl(\Q,\MM(f^K)\bigr)$ of Shafarevich--Tate groups, which is a consequence of an analogous decomposition of Selmer groups (see, \emph{e.g.}, \cite[Proposition 6.2]{LV-Pisa}), induces an equality 
\begin{equation} \label{eq-final1}
\biggl(\frac{(k/2-1)!\cdot L'(f,k/2)}{(2\pi i)^{k/2-1}\cdot\Omega_f\cdot\Reg_\mathscr{B}^{\iota_F}(\MM)}\biggr)=\bigl(\det(\mathtt{A}_{\tilde{\mathscr{B}}})\bigr)^{-2}\cdot\mathcal{I}_{\cO_p}\bigl(\Sha_p^{\BKK}(\Q,\MM)\bigr).
\end{equation}
of fractional $\cO_p$-ideals.

Now we compare \eqref{eq-final1} with equality \eqref{eqBK} in Theorem \ref{motivesthm}. Thanks to Theorem \ref{ratconj1thm}, we already know that 
\[ \frac{L^*(\MM,0)}{\Omega_\infty\cdot\Reg_\mathscr{B}(\MM)}=\biggl(\frac{L'(f^\sigma,k/2)}{(2\pi i)^{k/2-1}\cdot\Omega_{f^\sigma}\cdot\Reg_\mathscr{B}^\sigma(\MM)}\biggr)_{\!\sigma\in\Sigma} \] 
belongs to $F^\times$ in the sense of Remark \ref{deligne-rem2}; more explicitly, there is an equality
\begin{equation} \label{eq-final2}
\iota_\Sigma\biggl(\frac{L'(f,k/2)}{(2\pi i)^{k/2-1}\cdot\Omega_f\cdot\Reg_\mathscr{B}^{\iota_F}(\MM)}\biggr)=\frac{L^*(\MM,0)}{\Omega_\infty\cdot\Reg_\mathscr{B}(\MM)}. 
\end{equation}
Combining \eqref{eq-final1} and \eqref{eq-final2}, we get an equality 
\begin{equation} \label{eq-final3}
\biggl(\frac{(k/2-1)!\cdot L^*(\MM,0)}{\Omega_\infty\cdot\Reg_\mathscr{B}(\MM)}\biggr)=\bigl(\det(\mathtt{A}_{\tilde{\mathscr{B}}})\bigr)^{-2}\cdot\mathcal{I}_{\cO_p}\bigl(\Sha_p^{\BKK}(\Q,\MM)\bigr)
\end{equation}
of fractional $\cO_p$-ideals. As we explained in the proof of Theorem \ref{skinner-urban-thm}, $\Tors_p(\MM)=\cO_p$. Moreover, if $\ell\,|\,N$, then $\Tam_\ell^{(p)}(\MM)=\cO_p$ by part (1) of Proposition \ref{lemmatam}, while $\Tam_p^{(p)}(\MM)=\cO_p$ by \eqref{Tam-p-eq} and $\Tam_\infty^{(p)}(\MM)=\cO_p$ by \eqref{Tam-infty-eq}. Finally, $\mathcal I_p(\gamma_f)=\cO_p$ by \eqref{I-implication-eq}. Therefore, equality \eqref{eq-final3} coincides with \eqref{eqBK} (or, rather, with its equivalent form \eqref{eqBKbis} in Remark \ref{BK-bis-rem}) in our setting, and the proof of the theorem is complete. \end{proof}

\section{On the structure of Selmer groups} \label{structure-sec}

As an application of Theorem \ref{kolyvagin-main-thm}, we deduce results on the structure of Selmer groups of modular forms. As will be apparent, these results basically follow from \cite{Masoero}.

Fix a newform $f$ with Hecke field $F$ and a prime $\p$ of $F$ above $p$ that satisfy Assumption \ref{main-assumption}. For any number field $L$, let
\begin{equation} \label{r-corank-eq}
r_\p(f/L)\defeq\corank_{\cO_\p}\!H^1_f(L,A_\p)
\end{equation}
be the corank of $H^1_f(L,A_\p)$ over $\cO_\p$. Set also $r_\p(f)\defeq r_\p(f/\Q)$.

Throughout this section, $K$ is an imaginary quadratic field in which all the prime factors of $Np$ split.

\subsection{Vanishing order of $\kappa_{f,\infty}$}

Let $\kappa_{f,\infty}$ be the Kolyvagin set attached to $f$, $\p$, $K$ from \eqref{Kolsys}. By Theorem \ref{kolyvagin-main-thm}, $\kappa_{f,\infty}\neq\{0\}$. For every $n\in\Lambda_{\Kol}(f)=\Lambda_{\Kol}(f,\p,K)$, denote by $\nu(n)$ the number of prime factors of $n$ and let $M(n)$ be the Kolyvagin index of $n$ that was introduced in \eqref{M(n)-eq}. 

\begin{definition} \label{vanishing-def}
The \emph{vanishing order} of $\kappa_{f,\infty}$ is
\[ \nu_\infty\defeq \min\bigl\{\nu(n)\mid\text{$n\in\Lambda_{\Kol}(f)$ and $c_M(f,n)\not=0$ for some $M\leq M(n)$}\bigr\}\in\N. \]
\end{definition}

The following result will be used in the proof of Theorem \ref{main-higher-rank}.

\begin{proposition} \label{nu-prop}
If $y_{K,\p}$ is $\cO_\p$-torsion, then $\nu_\infty\geq1$.
\end{proposition}

\begin{proof} By part (2) of Proposition \ref{no torsion lemma}, the $\cO_\p$-module $\Lambda_\p(K)$ is free, so $y_{K,\p}=0$. It follows from Proposition \ref{c_M(1)-prop} that $c_M(f,1)=0$ for all $M$, whence $\nu_\infty\geq1$. \end{proof}

As before, let $\varepsilon(f)\in\{\pm1\}$ be the root number of $f$. It is convenient to consider the sign
\begin{equation} \label{varepsilon-infty-eq}
\varepsilon_\infty\defeq\mathrm{sign}\bigl(\varepsilon(f)\cdot(-1)^{\nu_\infty+1}\bigr)\in\{\pm\},
\end{equation}
which will appear in Theorem \ref{main-vanishing-thm}. 

\subsection{A structure theorem for $H^1_f(K,A_\p)$}

From now on, let $H^1_f(K,A_\p)^\pm$ denote the $\pm1$-eigenspaces of complex conjugation acting on $H^1_f(K,A_\p)$ and write
\[ r^\pm_\p(f/K)\defeq \corank_{\cO_\p}H^1_f(K,A_\p)^\pm \]
for the corresponding coranks over $\cO_\p$. Observe that
\begin{equation} \label{r-sum-eq}
r_\p(f/K)=r^+_\p(f/K)+r^-_\p(f/K). 
\end{equation}

\subsubsection{A lemma on $p^m$-torsion}

The next auxiliary result will be used in the proof of the structure theorem for $H^1_f(K,A_\p)$ (Theorem \ref{main-vanishing-thm}).

\begin{lemma} \label{lemma9.3} 
For all $m\in\N$ there is a Galois-equivariant identification 
\[ H^1_f\bigl(K,A_\p[p^m]\bigr)=H^1_f(K,A_\p)[p^m]. \]  
\end{lemma} 

\begin{proof} By \cite[Lemma 2.4, (1)]{LV-kyoto}, $H^0(K,A_\p)=0$, so the inclusion $A_\p[p^m]\hookrightarrow A_\p$ induces an identification 
\begin{equation} \label{torsion-eq}
H^1(K,A_\p)[p^m]=H^1\bigl(K,A_\p[p^m]\bigr).
\end{equation}
By definition, $H^1_f\bigl(K,A_\p [p^m]\bigr)$ consists of the elements of $H^1\bigl(K,A_\p[p^m]\bigr)$ whose image in $H^1(K,A_\p)$ lies in $H^1_f(K,A_\p)$, and then the lemma follows from \eqref{torsion-eq}. \end{proof}

\subsubsection{Structure theorem}

Recall that $p$, which is unramified in $F$, is a uniformizer for $\cO_\p$. As in \S \ref{p-primary-Sha-subsec}, let $\varepsilon\in\{\pm\}$ be the sign of the root number of $f$. Write 
\[ H^1_f(K,A_\p)^\pm\simeq (F_\p/\mathcal{O}_\p)^{r_\p^\pm(f/K)}\oplus\mathcal X_\p^\pm \] 
where $\mathcal X_\p^\pm$ is a finite $\cO_\p$-module, then introduce splittings
\begin{align}
\mathcal X_\p^{-\varepsilon}&\simeq(\cO_\p/p^{n_1}\cO_\p)^2\oplus(\cO_\p/p^{n_3}\cO_\p)^2\oplus\dots\\
\intertext{and}
\mathcal X_\p^{\varepsilon}&\simeq(\cO_\p/p^{n_2}\cO_\p)^2\oplus(\cO_\p/p^{n_4}\cO_\p)^2\oplus\dots
\end{align}
of $\cO_\p$-modules. Finally, let the integers $N_i\in\N$ be defined as in \S \ref{p-primary-Sha-subsec} (\emph{cf.} Theorem \ref{structure sha}) and let $\varepsilon_\infty\in\{\pm\}$ be the sign from \eqref{varepsilon-infty-eq}. 

\begin{theorem} \label{main-vanishing-thm} 
\begin{enumerate}
\item $r_\p^{\varepsilon_\infty}(f/K)=\nu_\infty+1$ and $r_\p^{-\varepsilon_\infty}(f/K)\leq\nu_\infty$.
\item $\nu_\infty=\max\bigl\{r^+_\p(f/K),r^-_\p(f/K)\bigr\}-1$.
\item $n_i=N_i$ for all $i>\nu_\infty+1$. 
\item $0\leq\nu_\infty-r_\p^{-\varepsilon_\infty}(f/K)\equiv0\pmod{2}$.
\end{enumerate}
\end{theorem}

\begin{proof} Parts (1) and (3) are consequences of the techniques exploited in the proof of \cite[Theorem 8.4]{Masoero}, using the decomposition as $\cO_\p$-modules rather than as groups; details are left to the reader. Part (2) is a restatement of (1), since $\max\bigl\{r^+_\p(f/K),r^-_\p(f/K)\bigr\}=r_\p^{\varepsilon_\infty}(f/K)$, by (1). By the second statement in (1), it remains to show the congruence in (4). We write $r=r_\p^{-\varepsilon_\infty}(f/K)$ to simplify the notation. As above, consider the decomposition 
\[ H^1_f\bigl(K,A_\p [p^M]\bigr)=H^1_f\bigl(K,A_\p [p^M]\bigr)^+\oplus H^1_f\bigl(K,A_\p [p^M]\bigr)^- \] 
under the action of $\Gal(K/\Q)$. By Lemma \ref{lemma9.3}, the invariants of $H^1_f\bigl(K,A_\p [p^M]\bigr)^{-\varepsilon_\infty}$ are those of $\mathcal{X}_{f,\p}^{-\varepsilon_\infty}$ shifted by $r$ terms all equal to $M$ coming from the divisible subgroup $(F_\p/\cO_\p)^r$ of $\Sel_\p(f/K)^{-\varepsilon_\infty}$. Thus, by (3), the last $\nu_\infty+1+r$ invariants of $H^1_f\bigl(K,A_\p [p^M]\bigr)^{-\varepsilon_\infty}$ are in even number. The existence of a Flach--Cassels pairing on $H^1_f\bigl(K,A_\p [p^M]\bigr)^{-\varepsilon_\infty}$ that is alternating and non-degenerate (\cite{Flach}, \cite[Section 6]{Masoero}) ensures that the total number of the invariants of $H^1_f\bigl(K,A_\p [p^M]\bigr)^{-\varepsilon_\infty}$ is even, and therefore $\nu_\infty+r$ is even. \end{proof}

\section{Parity results} \label{parity-sec}

We prove a $p$-parity result for modular forms (\S \ref{parity-subsec}) and then deduce from it part (1) of Theorem D (\S \ref{thmc1}). We fix throughout a newform $f$ with Hecke field $F$ and a prime $\p$ of $F$ above $p$ satisfying Assumption \ref{main-assumption}. 

\subsection{A $\p$-parity result} \label{parity-subsec}

Let $r_\p(f)$ be defined as in \eqref{r-corank-eq} and, as in \S \ref{Lfunctsec}, let $\varepsilon(f)$ be the root number of $f$. The following is a $\p$-parity result for $f$. 
%
%\begin{conjecture}[$\p$-parity conjecture for $f$] \label{parity-conj}
%$(-1)^{r_\p(f)}=\varepsilon(f)$.
%\end{conjecture}
%
%\begin{remark}
%The analogue of Conjecture \ref{parity-conj} for a large class of elliptic curves and, more generally, Hilbert modular forms of parallel weight has been proved by Nekov\'a\v{r} (\cite{Nekovar-parity}, \cite{Nekovar-growth}, \cite{Nekovar-parity2}).
%\end{remark}
%
%As a corollary of Theorem \ref{main-vanishing-thm}, we deduce some cases of Conjecture \ref{parity-conj} (before proved with different methods in \cite{Nekovar-growth}). 

\begin{theorem}\label{parity} 
$(-1)^{r_\p(f)}=\varepsilon(f)$. 
\end{theorem}

\begin{proof} Choose an imaginary quadratic field $K$ in which all the prime factors of $Np$ split. Combining the inflation-restriction exact sequence 
\[ \begin{split}
   0\longrightarrow H^1\bigl(\Gal(K/\Q),A_\p (K)\bigr)&\longrightarrow H^1(\Q,A_\p)\\
   &\longrightarrow H^1(K,A_\p)^{\Gal(K/\Q)}
    \longrightarrow H^2\bigl(\Gal(K/\Q),A_\p (K)\bigr)
   \end{split}
  \]
and the triviality of $A_\p(K)$ from \cite[Lemma 2.4, (1)]{LV-kyoto} gives an identification
\begin{equation} \label{K-plus-eq}
H^1(\Q,A_\p)=H^1(K,A_\p)^+.
\end{equation}
By keeping track of local conditions, one can then check that \eqref{K-plus-eq} induces an identification
\begin{equation} \label{plus-selmer-eq}
H^1_f(\Q,A_\p)=H^1_f(K,A_\p)^+. 
\end{equation} 
The desired equality follows from Theorem \ref{main-vanishing-thm} by an easy combinatorial argument. \end{proof}

\begin{remark}
The analogue of Theorem \ref{parity} for a large class of elliptic curves and, more generally, Hilbert modular forms of parallel weight has been proved by Nekov\'a\v{r} (\cite{Nekovar-parity}, \cite{Nekovar-growth}, \cite{Nekovar-parity2}).
\end{remark}

\subsection{Proof of part (1) of Theorem D} \label{thmc1}

Recall the set $\Sigma$ of real (equivalently, complex) embeddings of $F$. For all $\sigma\in\Sigma$, the representation of $G_\Q$ attached to $f$ and $\p$ is equivalent (over $\bar\Q_p$) to the representation of $G_\Q$ attached to $f^\sigma$ and the prime $\sigma(\p)$ of the Hecke field $\sigma(F)$ of $f^\sigma$. In particular, $r_\p(f)=r_{\sigma(\p)}(f^\sigma)$ for all $\sigma\in \Sigma$. Then, by Theorem \ref{parity}, $\varepsilon(f^\sigma)$ is constant as $\sigma$ varies in $\Sigma$, which means, by \eqref{r-parity-eq}, that the parity of $r_\an(f^\sigma)$ is constant as $\sigma$ varies in $\Sigma$. In light of equality \eqref{r-min-eq} and Theorem \ref{parity}, we get the congruence
\[ r_\p(f)\equiv r_\mathrm{an}(\MM)\pmod{2}. \]
Now, by Corollary \ref{upsilon-onto-coro}, $r_\p(f)=r_\alg(\MM)$, and the proof is complete.\hfill\qedsymbol

\section{Converse theorems} \label{converse-sec}

We prove $p$-converse theorems for modular forms (\S \ref{converse-subsec}) and then deduce from them part (2) of Theorem D (\S \ref{thmc2}). We fix throughout a newform $f$ with Hecke field $F$ and a prime $\p$ of $F$ above $p$ satisfying Assumption \ref{main-assumption}. 

\subsection{$\p$-converse theorems} \label{converse-subsec}

For any number field $L$, define the $F_\p$-vector space
\[ X_{\p}(L)\defeq\Lambda_{\p}(L)\otimes_{\Z}\Q=\Lambda_{\p}(L)\otimes_{\cO_\p}\!F_\p. \]

\subsubsection{Results over $K$}

The next result is a higher weight counterpart of the algebraic part of \cite[Theorem 1.3]{zhang-selmer}.

\begin{theorem} \label{main-converse-thm} 
Let $K$ be an imaginary quadratic field in which all the prime factors of $Np$ split. If $r_\p(f/K)=1$, then
\begin{enumerate}
\item $y_{K,\p}$ is not $\cO_\p$-torsion;
\item $\dim_{F_\p}\bigl(X_{\p}(K)\bigr)=1$;
\item $\Sha_\p^{\Nek}(K,\MM)$ is finite.
\end{enumerate}
\end{theorem}

\begin{proof} If (1) holds, then (2) and (3) follow from Theorem \ref{nekovar-thm}, so we need to show only (1). If $r_\p(f/K)=1$, then equality \eqref{r-sum-eq} implies that
\[ \max\bigl\{r^+_\p(f/K),r^-_\p(f/K)\bigr\}=1. \] 
By part (2) of Theorem \ref{main-vanishing-thm}, this is equivalent to $\nu_\infty=0$, \emph{i.e.}, $c_M(f,1)\not=0$ for some $M\geq1$. On the other hand, Proposition \ref{c_M(1)-prop} says that $c_M(f,1)=\iota_{K,M}\bigl({[y_{K,\p}]}_M\bigr)$, so we surmise that $y_{K,\p}\neq0$. By part (2) of Proposition \ref{no torsion lemma}, the $\cO_\p$-module $\Lambda_\p(K)$ is free of finite rank, hence $y_{K,\p}$ is not torsion, as was to be shown. \end{proof}

For another higher weight converse to the Kolyvagin--Gross--Zagier theorem, the reader is referred to \cite[Theorem 2]{wang}.

Recall that if $K$ is a field as in Theorem \ref{main-converse-thm}, then $r_\an(f/K)\geq1$. The following is a $\p$-converse result over $K$.

\begin{corollary} \label{main-converse-coro}
Let $K$ be an imaginary quadratic field in which all the prime factors of $Np$ split. Assume that ${\langle\cdot,\cdot\rangle}_{\GS}$ is non-degenerate on $\Heeg_{K,N}\otimes_\Z\,\R$. If $r_\p(f/K)=1$, then $r_\an(f/K)=1$.
\end{corollary}

\begin{proof} By part (1) of Theorem \ref{main-converse-thm}, $y_{K,\p}$ is not $\cO_\p$-torsion. Since we are assuming that ${\langle\cdot,\cdot\rangle}_{\GS}$ is non-degenerate on $\Heeg_{K,N}\otimes_\Z\,\R$, the claim follows from part (2) of Proposition \ref{coro zhang}. \end{proof}

Focusing now on the case where the base field is $\Q$, we can prove an analogue in higher weight of the algebraic part of \cite[Theorem 1.4, (i)]{zhang-selmer}.

\begin{theorem} \label{main-selmer-thm} 
If $r_\p(f)=1$, then 
\begin{enumerate}
\item $\dim_{F_\p}\bigl(X_\p(\Q)\bigr)=1$;
\item $\Sha_\p^{\Nek}(\Q,\MM)$ is finite.
\end{enumerate}
\end{theorem}

\begin{proof} Since $r_\p(f)=1$, it follows from Theorem \ref{parity} that $\varepsilon(f)=-1$. Choose an imaginary quadratic field $K$ such that 
\begin{itemize}
\item all the prime factors of $Np$ split in $K$;
\item $r_\an(f^K)=0$.
\end{itemize} 
The existence of such a $K$ is guaranteed by \cite[p. 543, Theorem, (ii)]{BFH}. Let $A_\p^K$ be the analogue for $f^K$ of the $\cO_\p$-module $A_\p$ associated with $f$. By Theorem \ref{kato-thm} with $f^K$ in place of $f$, the Selmer group $H^1_f\bigl(\Q,A^K_\p\bigr)$ is finite, so $r_\p(f^K)=0$. It can be checked (see, \emph{e.g.}, the proof of \cite[Proposition 6.2]{LV-Pisa}) that there is a canonical identification
\begin{equation} \label{minus-selmer-eq}
H^1_f\bigl(\Q,A^K_\p\bigr)=H^1_f(K,A_\p)^-.
\end{equation}
Combining \eqref{r-sum-eq}, \eqref{plus-selmer-eq} and \eqref{minus-selmer-eq}, we obtain
\begin{equation} \label{r-f-f^K-eq}
r_\p(f/K)=r_\p(f)+r_\p(f^K)=1, 
\end{equation}
and then Theorem \ref{main-converse-thm} ensures that $\dim_{F_\p}\bigl(X_{f,\p}(K)\bigr)=1$ and $\Sha_\p^{\Nek}(K,\MM)$ is finite. Finally, part (1) and part (2) of the theorem follow from \cite[Theorem 5.26]{Vigni} and \cite[Proposition 5.25]{Vigni}, respectively. \end{proof}

\subsubsection{Assumption {\normalfont (\texttt{GS})} and results over $\Q$} \label{p-converse-Q-subsubsec}

To complete the picture, we prove a $\p$-converse result over $\Q$ that can be regarded as a higher weight counterpart of \cite[Theorem A]{Skinner}, \cite[Theorem A]{Ven} and \cite[Theorem 1.4, (i)]{zhang-selmer}. To do this, with notation as in \S \ref{imaginary-1-subsubsec}, we need to introduce hypotheses concerning, in particular, the non-degeneracy of the Gillet--Soul\'e height pairings:
\begin{itemize}
\item[(\texttt{GS})] there is $K\in\mathscr I_1(f,p)$ such that ${\langle\cdot,\cdot\rangle}_{\GS}$ is non-degenerate on $\Heeg_{K,N}\otimes_\Z\,\R$.
\end{itemize}

\begin{theorem} \label{main-converse-Q-thm}
If $r_\p(f)=1$ and $(\mathtt{GS})$ holds, then $r_\an(f)=1$.
\end{theorem}

\begin{proof} On the one hand, Theorem \ref{kato-thm} gives $r_\p(f^K)=0$, and then $r_\p(f/K)=1$ by the first equality in \eqref{r-f-f^K-eq}. Therefore, we can apply Corollary \ref{main-converse-coro} to deduce that $r_\an(f/K)=1$. On the other hand, factorization \eqref{splitting-L-eq2} yields the equality 
\[ r_\an(f/K)=r_\an(f)+r_\an(f^K), \]
whence $r_\an(f)=1$. \end{proof}

\begin{remark} \label{K-rem}
If we knew that Gillet--Soul\'e height pairings are non-degenerate (at least on the $\R$-vector space $\Heeg_{K,N}\otimes_\Z\,\R$, or in full generality, as predicted by the conjectures in \cite{Beilinson}, \cite{bloch-height}, \cite{GS-2}), then Corollary \ref{main-converse-coro} and Theorem \ref{main-converse-Q-thm} would become unconditional. Unfortunately, non-degeneracy results of this kind appear to lie well beyond the scope of currently available techniques.
\end{remark}

\begin{remark}
Recently, Burungale and Tian proved a $p$-converse result for CM elliptic curves over $\Q$ at good ordinary primes $p$ (\cite[Theorem 1.2]{BT}). It would be desirable to obtain $\p$-converse theorems (possibly conditional, like Theorem \ref{main-converse-Q-thm}, on the non-degeneracy of suitable height pairings) for higher weight CM newforms.
\end{remark}

\subsection{Proof of part (2) of Theorem D} \label{thmc2} 

Recall that we are assuming that
\begin{itemize}
\item Conjecture \ref{regpconj} holds true for $p$ and $\Q$;
\item condition (\texttt{GS}) from \S \ref{p-converse-Q-subsubsec} is satisfied.
\end{itemize}
With notation as above, if $r_\alg(\MM)=1$, then $r_\p(f)=1$ by Corollary \ref{upsilon-onto-coro} (with $\star=\p$), so Theorem \ref{main-converse-Q-thm} gives $r_\an(f)=1$. By Lemma \ref{low-rank-lemma}, $r_\an(\MM)=1$, as desired.\hfill\qedsymbol

\section{Higher rank results} \label{higher-sec}

In this final section, we collect higher rank results for $f$ and its motive $\MM$. As before, we require throughout that $f$ and the prime $\p$ of $F$ above $p$ satisfy Assumption \ref{main-assumption}. 

\subsection{Higher rank results for $f$} \label{higher-subsec}

We begin with results on $f$, in particular on the invariant $r_\p(f)$ from \eqref{r-corank-eq}. We use notation from \S \ref{imaginary-0-subsubsec}.

\subsubsection{Assumption {\normalfont (\texttt{reg})}}

According to the sign of the root number of $f$, we will need to assume one of two different sets of hypotheses. The first is (\texttt{GS}) from \S \ref{p-converse-Q-subsubsec}, whereas the second takes care, in addition, of the injectivity (at least on Heegner modules) of $\p$-adic regulators over imaginary quadratic fields:
\begin{itemize}
\item[(\texttt{reg})] there is $K\in\mathscr I_0(f,p)$, of discriminant $D_K$, such that 
%\vskip 1mm
\begin{itemize}
\item ${\langle\cdot,\cdot\rangle}_{\GS}$ is non-degenerate on $\Heeg_{K,N}\otimes_\Z\,\R$,
%\vskip 1mm
\item $\reg_{f^K,K'_1,\p}$ is injective on $\Heeg_{K',ND_K^2}$ for every imaginary quadratic field $K'$ in which all the prime factors of $ND_Kp$ split.
\end{itemize}
\end{itemize}
Here $\reg_{f^K,K'_1,\p}$ is the counterpart of the $\p$-adic regulator $\reg_{K'_1,\p}$ relative to the motive of the twist $f^K$ of $f$. As before (\emph{cf.} Remark \ref{K-rem}), note that the results in \cite{BFH} guarantee that if $\varepsilon(f)=-1$, then there is always an imaginary quadratic field satisfying the first two conditions in (\texttt{reg}). 

\begin{remark}
As will become clear, the reason why in (\texttt{reg}) we consider $\Heeg_{K',ND_K^2}$ is that, since $(N,D_K)=1$, the level of $f^K$ is $ND_K^2$.
\end{remark}

\subsubsection{Higher rank results for $f$}

The following result is a higher weight analogue of \cite[Theorem 1.4, (ii)]{zhang-selmer}; in \S \ref{3-4-subsec} we shall deduce from it analogous results for $\MM$.

\begin{theorem} \label{main-higher-rank} 
Assume that either
\begin{itemize}
\item $\varepsilon(f)=-1$ and $(\mathtt{GS})$ holds
\end{itemize}
or
\begin{itemize}
\item $\varepsilon(f)=+1$ and $(\mathtt{reg})$ holds.
\end{itemize}
If $r_\an(f)>1$, then
\[ r_\p(f)\in\biggl\{2n+\frac{1-\varepsilon(f)}{2}\;\,\Big|\;\,n\in\Z_{\geq1}\biggr\}. \]
In particular, $r_\p(f)\geq\displaystyle{\frac{5-\varepsilon(f)}{2}}$.
\end{theorem}

\begin{proof} We need to show that 
\begin{itemize}
\item $r_\p(f)\geq 3$ if $\varepsilon(f)=-1$, 
%\vskip 1mm
\item $r_\p(f)\geq 2$ if $\varepsilon(f)=+1$. 
\end{itemize}
Assume first that $\varepsilon(f)=-1$ and (\texttt{GS}) holds. By Theorem \ref{parity}, $r_\p(f)$ is odd. If $r_\p(f)=1$, then $r_\an(f)=1$ by Theorem \ref{main-converse-Q-thm}, so $r_\p(f)\in\{2n+1\mid n\geq1\}$.

Assume now that $\varepsilon(f)=+1$ and (\texttt{reg}) holds; recall the vanishing order $\nu_\infty$ of $\kappa_{f,\infty}$ introduced in Definition \ref{vanishing-def}. By Theorem \ref{parity}, $r_\p(f)$ is even. Choose an imaginary quadratic field $K$ satisfying (\texttt{reg}). Then $r_\an(f/K)>1$, so part (2) of Proposition \ref{coro zhang} guarantees that $y_{K,\p}$ is $\cO_\p$-torsion. It follows from Proposition \ref{nu-prop} that $\nu_\infty\geq1$; equivalently, by part (2) of Theorem \ref{main-vanishing-thm}, we obtain
\begin{equation} \label{r-higher-eq}
\max\bigl\{r^+_\p(f/K),r^-_\p(f/K)\bigr\}\geq2. 
\end{equation}
On the other hand, since $r_\an(f^K)=1$, by \cite[p. 543, Theorem, (ii)]{BFH} there exists an imaginary quadratic field $K'$ (which we fix) such that
\begin{itemize}
\item all the prime factors of $ND_Kp$ split in $K'$;
\item $r_\an(f^K/K')=1$.
\end{itemize}
Using the injectivity of $\reg_{f^K,K'_1,\p}$ on $\Heeg_{K,ND_K^2}$, we apply Theorem \ref{nekovar-thm} to $f^K$ and obtain, in particular, $r_\p(f^K/K')=1$. Finally, reasoning as, \emph{e.g.}, in the proof of \cite[Theorem 5.27]{Vigni}, we get $r_\p(f^K)=1$, and then the equality $r_\p^-(f/K)=r_\p(f^K)$ combined with \eqref{r-higher-eq} gives $r_\p(f)=r^+_\p(f/K)\in\{2n\mid n\geq1\}$. \end{proof}

\begin{remark}
We sketch an alternative approach to the $\varepsilon(f)=+1$ part of Theorem \ref{main-higher-rank} that does not use the injectivity on $\Heeg_{K,ND_K^2}$ of the $\p$-adic regulator, but relies instead on certain conjectural (non-)vanishing properties of $p$-adic $L$-functions of modular forms. Namely, let $g$ be a weight $k\geq4$ newform on $\Gamma_0(M)$ with $p\nmid M$ and let $L_p(g,s)$ be the $p$-adic $L$-function of $g$ in the sense of Mazur--Tate--Teitelbaum (\cite[Ch. I, \S 13]{MTT}; for simplicity, we suppress dependence of $L_p(g,s)$ on the ``allowable $p$-root for $g$'', \emph{cf.} \cite[Ch. I, \S 12]{MTT}). We normalize $L_p(g,s)$ as in \cite[p. 430]{GS}. It is conjectured in \cite[Ch. I, \S 16]{MTT} that, in our non-exceptional setting (\emph{cf.} \cite[Ch. I, \S 15]{MTT}), $\ord_{s=k/2}L_p(g,s)=r_\an(g)$. Here we just assume the following implication:
\begin{equation} \label{implication-eq}
r_\an(g)=1\;\Longrightarrow\;\ord_{s=\frac{k}{2}}L_p(g,s)=1.
\end{equation}
With notation as in the proof of Theorem \ref{main-higher-rank}, rather than using the last condition in (\texttt{reg}), we combine the equality $r_\an(f^K)=1$ and \eqref{implication-eq} to get $\ord_{s=k/2}L_p(f^K,s)=1$. Then, keeping Corollary \ref{upsilon-onto-coro} in mind, \cite[Theorem C, (2)]{Nek2} gives $r_\p(f^K)=1$, and by \eqref{r-higher-eq} we conclude that $r_\p(f)\in\{2n\mid n\geq1\}$.
\end{remark}

\subsection{Proof of parts (3) and (4) of Theorem D}  \label{3-4-subsec}

For the sake of clarity, we restate the result we want to prove on the motive $\MM$.

\begin{theorem} \label{main-higher-rank-motive}
Assume that either
\begin{itemize}
\item $r_\an(\MM)$ is odd and $(\mathtt{GS})$ holds
\end{itemize}
or
\begin{itemize}
\item $r_\an(\MM)$ is even and $(\mathtt{reg})$ holds.
\end{itemize}
Furthermore, assume that Conjecture \ref{regpconj} holds true for $p$ and $\Q$.

If $r_\an(\MM)>1$, then
\[ r_\alg(\MM)\in\Biggl\{2n+\frac{1-(-1)^{r_\an(\MM)}}{2}\;\,\Big|\;\,n\in\Z_{\geq1}\Biggr\}. \]
In particular, $r_\alg(\MM)\geq\displaystyle{\frac{5-(-1)^{r_\an(\MM)}}{2}}$.
\end{theorem}

\begin{proof} To begin with, note that, by Lemma \ref{low-rank-lemma}, $r_\mathrm{an}(\MM)>1$ if and only if $r_\mathrm{an}(f)>1$. Moreover, recall from \S \ref{thmc1} that $r_\an(\MM)$ and $r_\an(f)$ have the same parity; in other words, $(-1)^{r_\an(\MM)}=\varepsilon(f)$. Finally, since we are assuming Conjecture \ref{regpconj} for $p$ and $\Q$, Corollary \ref{upsilon-onto-coro} gives $r_\alg(\MM)=r_\p(f)$, and then the desired result follows from Theorem \ref{main-higher-rank}. \end{proof}

\appendix

\section{Determinants of projective modules} \label{determinants-subsec}

We sketch the basic elements of the theory of determinants of projective modules over commutative rings, putting a focus on modules over (products of finitely many) principal ideal domains. In doing so, we follow \cite[\S 2.1]{Kato-Iwasawa} and \cite[Lecture 1, \S 5]{Kings} quite closely.

\subsection{Determinants over commutative rings}

Let $R$ be a commutative ring. Denote by $\mathcal{L}_R$ the category of isomorphism classes of graded invertible $R$-modules: the objects of $\mathcal L_R$ are pairs $(L,r)$ consisting of an invertible (\emph{i.e.}, projective, rank $1$) $R$-module $L$ and a function $r:\Spec(R)\rightarrow\Z$, which is to be thought of as a grading, that is locally constant for the Zariski topology, while morphisms between two objects $(L,r)$ and $(M,s)$ are trivial if $r\neq s$ and isomorphisms of $R$-modules $L\rightarrow M$ otherwise. One defines a product in $\mathcal{L}_R$ by 
\begin{equation} \label{product-invertible-eq}
(L,r)\cdot(M,s)\defeq(L\otimes_RM,r+s), 
\end{equation}
and then $(R,0)$ is the unit object for this product. The inverse of a pair $(L,r)$ is $(L,r)^{-1}\defeq (L^*,-r)$, where $L^*\defeq \Hom_R(L,R)$ is the $R$-linear dual of $L$; more precisely, there is a canonical isomorphism
\[ (L,r)\cdot(L^*,-r)\simeq(R,0) \]
induced by the usual evaluation map $L\otimes_R\Hom_R(L,R)\rightarrow R$. The monoidal category $\mathcal L_R$ is equipped with a modified commutativity constraint involving a sign that depends on the grading (see, \emph{e.g.}, \cite[\S 2.5]{BF}, \cite[Definition 1.27]{Kings}). If $f:R\rightarrow R'$ is a ring homomorphism, then one defines a functor 
\[ \mathcal{L}_R\longrightarrow\mathcal{L}_{R'} \]  
by the recipe $(L,r)\mapsto (L\otimes_{R}R',r\circ f^*)$ on objects and in the obvious way (\emph{i.e.}, by extension of scalars) on morphisms, where $f^*:\Spec(R')\rightarrow\Spec(R)$ is the map induced by $f$ by (contravariant) functoriality. 

Let us write $\Proj_R^\mathrm{fg}$ for the category of finitely generated, projective $R$-modules. For every object $M$ of $\Proj_R^\mathrm{fg}$ let 
\[ \rk_R(M):\Spec(R)\longrightarrow\Z \]
be the rank function attached to $M$, which maps $\p$ to $\rk_{R_\p}(M_\p)$ and is locally constant with respect  to the Zariski topology on $\Spec(R)$ (see, \emph{e.g.}, \cite[Ch. I, Corollary 2.2.2]{Weibel}); notice that $M_\p$ is projective, hence free, of finite rank over the local ring $R_\p$. If $\rk_R(M)$ is constant, then $\wedge_R^{\rk_R(M)}M$ denotes the usual $\rk_R(M)$-th exterior power of $M$ over $R$. If $\rk_R(M)$ is not constant, then the definition of $\wedge_R^{\rk_R(M)}M$ is more delicate: see, \emph{e.g.}, \cite[Definition 1.3.1]{popescu} or \cite[p. 21]{Weibel}. In both cases, $\wedge_R^{\rk_R(M)}M$ is an invertible $R$-module, called the \emph{determinant of $M$}; see, \emph{e.g.}, \cite[Part 3]{Flach-tamagawa} for motivation for retaining the rank information. There is a (covariant) functor $\Proj_R^\mathrm{fg}\rightarrow \mathcal{L}_R$ defined by 
\[ M\longmapsto\Det_R(M)\defeq\Biggl(\,\bigwedge_R^{\rk_R(M)}\!M,\,\rk_R(M)\!\Biggr) \] 
on objects and in the obvious fashion on morphisms. Notice that $\Det_R(0)=(R,0)$. For notational convenience, for every object $M$ of $\Proj_R^\mathrm{fg}$ we set $\Det_R^{-1}(M)\defeq\Det_R(M)^{-1}$.

\begin{caveat}
In this paper, we frequently use expressions like ``$N$ is an $R$-submodule of $\Det_R(M)$'', meaning that $N$ is a submodule of the $R$-module underlying $\Det_R(M)$. A similar interpretation must be given to statements like ``The set $\star$ is a basis of $\Det_R(M)$ over $R$''.
\end{caveat}

It turns out that $\Det_R$ is multiplicative on short exact sequences: if we are given an exact sequence 
\[ 0\longrightarrow K\longrightarrow P\longrightarrow C\longrightarrow0 \]
in $\Proj_R^\mathrm{fg}$, then  
\begin{equation} \label{det-multiplicativity-eq}
\Det_R(P)\simeq\Det_R(K)\cdot\Det_R(C), 
\end{equation}
where the product on the right is defined as in \eqref{product-invertible-eq}. Moreover, if there is a short exact sequence of $R$-modules 
\begin{equation} \label{det1-eq}
0\longrightarrow P\longrightarrow Q\longrightarrow T\longrightarrow0
\end{equation}
with $P,Q$ objects of $\Proj_R^\mathrm{fg}$, then we define $\Det_R(T)\defeq\Det_R(Q)\cdot\Det^{-1}_R(P)$. One can check that $\Det_R(T)$ is independent of the choice of an exact sequence as in \eqref{det1-eq}.

\begin{remark/notation} \label{rem-not}
Suppose that $M$ is free of finite rank $r$ over $R$. If $\{m_1,\dots,m_r\}$ is an $R$-basis of $M$, then $\{m_1\wedge\dots\wedge m_r\}$ is an $R$-basis of $\Det_R(M)$. In this case, we denote by $(m_1\wedge\dots\wedge m_r)^{-1}$ the dual element of $m_1\wedge\dots\wedge m_r$, so that $\bigl\{(m_1\wedge\dots\wedge m_r)^{-1}\bigr\}$ is an $R$-basis of $\Det_R^{-1}(M)$. Furthermore, for every $n\geq1$, the natural pairing
\[ \bigwedge^nM\times\bigwedge^nM^*\longrightarrow R,\quad(m_1\wedge\dots\wedge m_n,\ell_1\wedge\dots\wedge\ell_n)\longmapsto\det\bigl(\ell_i(m_j)\bigr)  \]
is perfect, so it induces a canonical isomorphism 
\begin{equation} \label{noetherian-eq}
\bigwedge^nM^*\simeq\biggl(\bigwedge^nM\biggr)^{\!*}
\end{equation}
of $R$-modules, which can be regarded as an identification. With standard notation, the basis $\bigl\{(m_1\wedge\dots\wedge m_r)^{-1}\bigr\}$ of $\Det_R^{-1}(M)$ corresponds, under the identification \eqref{noetherian-eq}, to the basis $\{m_1^*\wedge\dots\wedge m_r^*\}$ of $\Det_R(M^*)$. Finally, there is obviously an equality $\rk_R(M)=\rk_R(M^*)$ of rank functions, so it follows that $\Det^{-1}_R(M)$ and $\Det_R(M^*)$ have the same underlying $R$-module but opposite rank functions: this is the reason why we use different symbols for the bases of $\Det_R^{-1}(M)$ and of $\Det_R(M^*)$ that are built out of a given basis of $M$ over $R$. It will be useful to keep this observation in mind when, later in this article, we will be computing with determinants of modules. 
\end{remark/notation}

\begin{remark}
Assume that $R$ is noetherian (which is always the case in the main body of the article). If $M$ is an object of $\Proj_R^\mathrm{fg}$, then $M^*$ is an object of $\Proj_R^\mathrm{fg}$. Since $R$ is noetherian, $M$ is finitely presented over $R$, so for every $\p\in\Spec(R)$ there is a canonical isomorphism $(M^*)_\p\simeq M_\p^*$ of $R_\p$-modules, where the right hand term is the $R_\p$-linear dual of $M_\p$ (see, \emph{e.g.}, \cite[p. 52, Corollary]{matsumura}). Thus, there is an equality $\rk_R(M)=\rk_R(M^*)$ of rank functions.
\end{remark}

\subsection{Determinants over PID's} \label{pid-subsubsec}

Let $R$ be a principal ideal domain, write $\mathrm{frac}(R)$ for its fraction field and let $T$ be a torsion $R$-module. Choose a resolution
\[ 0\longrightarrow P\overset{\phi}\longrightarrow Q\longrightarrow T\longrightarrow 0 \]
of $T$ with objects $P,Q$ of $\Proj_R^\mathrm{fg}$. Since $R$ is a PID, $P$ and $Q$ are free over $R$; furthermore, $T$ being torsion forces the ranks of $P$ and $Q$ to be equal. It turns out that 
\[ \Det_R(T)=\det(\phi)^{-1}\cdot R\subset\mathrm{frac}(R), \]
where $\det(\phi)$ is computed with respect to fixed bases of $P$ and $Q$. Equivalently, $\Det_R(T)$ is the inverse of the ideal that is generated by the product of the elementary divisors of the $R$-module $T$. In this setting, we usually write $\mathcal{I}_R(T)$ in place of $\Det_R(T)$ to stress the fact that $\Det_R(T)$ is a fractional ideal of $R$; we also set $\mathcal{I}^{-1}_R(T)\defeq\Det^{-1}_R(T)$. In particular, $\mathcal I_R\bigl((0)\bigr)=R$. See, \emph{e.g.}, \cite[Example 1.31]{Kings} for details in the $R=\Z_p$ case.

\begin{remark/notation} \label{rem-not2}
Let $r$ be the rank of $P$ and $Q$ and let $\{p_1,\dots,p_r\}$ (respectively, $\{q_1,\dots,q_r\}$) be a basis of $P$ (respectively, $Q$) over $R$. Keeping Remark/Notation \ref{rem-not} in mind, there is an equality 
\begin{equation} \label{rem-not-eq1}
\Det_R(Q)\cdot\Det_R^{-1}(P)=R\cdot(q_1\wedge\dots\wedge q_r)\cdot(p_1\wedge\dots\wedge p_r)^{-1},
\end{equation} 
where $(q_1\wedge\dots\wedge q_r)\cdot(p_1\wedge\dots\wedge p_r)^{-1}$ is just a shorthand for $(q_1\wedge\dots\wedge q_r)\otimes(p_1\wedge\dots\wedge p_r)^{-1}$. Finally, combining \eqref{det-multiplicativity-eq} and \eqref{rem-not-eq1} yields a natural isomorphism 
\[ \Det_R(T)\simeq R\cdot(q_1\wedge\dots\wedge q_r)\cdot(p_1\wedge\dots\wedge p_r)^{-1}, \]
which will often be viewed as an identification.
\end{remark/notation}

\subsection{Determinants over products of PID's} \label{pid-prod-subsubsec}

Let $R_1,\dots,R_n$ be PID's and consider their product $R\defeq\prod_{i=1}^nR_i$. For $i=1,\dots,n$ let $e_i\in R$ be the idempotent corresponding to $R_i$, so that there is an identification $R_i=e_iR$. Let $T$ be an $R$-module. For $i=1,\dots,n$ set $T_i\defeq e_iT$; equivalently, $T_i=T\otimes_RR_i$. The $R$-submodule $T_i$ of $T$ is naturally an $R_i$-module and there is a canonical identification
\[ T=\bigoplus_{i=1}^nT_i \]
of $R$-modules. Assume now that $T$ is finite; of course, this is tantamount to $T_i$ being finite for every $i=1,\dots,n$. Let $Q(R)$ be the total quotient ring of $R$ and set
\begin{equation} \label{det-splitting-eq}
\Det_R(T)\defeq\prod_{i=1}^n\Det_{R_i}(T_i)\subset\prod_{i=1}^n\mathrm{frac}(R_i)=Q(R),
\end{equation}
where $\Det_{R_i}(T_i)$ is defined as in \S \ref{pid-subsubsec}. As in the $n=1$ case treated above, we also set $\mathcal{I}_R(T)\defeq\Det_R(T)$ and $\mathcal{I}^{-1}_R(T)\defeq\Det^{-1}_R(T)$. In particular, $\mathcal I_R\bigl((0)\bigr)=R$. Finally, a \emph{fractional $R$-ideal} will be, by definition, a product of the form $I=I_1\times\dots\times I_n$ where $I_j$ is a fractional $R_j$-ideal for $j=1,\dots,n$; in this case, $\mathcal I_R(I)=\prod_{j=1}^n\mathcal I_{R_j}(I_j)$.

\begin{remark}
In the applications we have in mind, the $R_i$ will be discrete valuation rings (namely, they will be the completions of $\cO_F$ at prime ideals above a fixed prime number $p$), so $R$ will be regular. 
\end{remark}

\subsection{Determinants and base change}

If $R\rightarrow S$ is a ring homomorphism and $P$ is an object of $\Proj_R^\mathrm{fg}$, then $P\otimes_RS$ is an object of $\Proj_S^\mathrm{fg}$ and there is a base change isomorphism 
\begin{equation} \label{det-basechange-eq}
\Det_R(P)\otimes_RS\simeq\Det_S(P\otimes_RS) 
\end{equation}
of $S$-modules.

\subsection{Determinants of complexes}

Let $\Proj^\bullet_R$ denote the category of complexes of $R$-modules that are quasi-isomorphic to a bounded complex of $R$-modules in $\Proj_R^\mathrm{fg}$. For an object $C^\bullet$ in $\Proj^\bullet_R$, fix a quasi-isomorphism $\tilde{C}^\bullet\rightarrow C^\bullet$ with a bounded complex $\tilde{C}^\bullet$ of modules in $\Proj_R^\mathrm{fg}$ and define 
\begin{equation} \label{complex-determinant-eq}
\Det_R(C^\bullet)\defeq \prod_{i\in\Z}\Det_R^{-1}\bigl(\tilde{C}^i\bigr)
\end{equation}
(\cite[Definition 1.29]{Kings}). If an object $C^\bullet$ of $\Proj_R^\bullet$ has the property that $H^j(C^\bullet)$ is an object of $\Proj_R^\mathrm{fg}$ for all $j$, then there is an isomorphism 
\[ \Det_R(C^\bullet)\simeq\prod_{j\in\Z}\Det_R^{(-1)^j}\bigl(H^j(C^\bullet)\bigr) \]
of $R$-modules. 

\section{Remarks on Pontryagin duals} \label{pontryagin-appendix}

We gather some facts about Pontryagin duals that are used in the main body of the paper. These results are consequences of well-known properties of duals of local fields; however, for several of them we could not find a convenient reference in the literature, so we decided to collect them here for the reader's benefit.

\subsection{Generalities on Pontryagin duals}

Let $\TT\defeq\{z\in\C\mid |z|=1\}\simeq\R/\Z$ be the unit circle in the complex plane, viewed as a subgroup of $\C^\times$. The \emph{Pontryagin dual} of a locally compact, Hausdorff topological abelian group $G$ is the group
\begin{equation} \label{pontryagin-def-eq2}
G^\wedge\defeq\Hom_\cont(G,\TT) 
\end{equation}
of characters of $G$, \emph{i.e.}, continuous homomorphisms from $G$ to $\TT$, where $\TT$ is equipped with its natural complex topology (equivalently, the quotient topology of $\R/\Z$). In turn, $G^\wedge$ can be endowed with the compact-open topology. This definition of Pontryagin dual does not coincide, in general, with the one that was given in \S \ref{local-finite-subsec} for $\Z_p$-modules; however, it is well known that if $G$ is profinite, then the image of any continuous homomorphism $G\rightarrow\TT$ is finite (take, \emph{e.g.}, $n=1$ in \cite[Proposition 2.2]{hida-MFGC}), so $G^\wedge=\Hom_\cont(G,\Q/\Z)$. Of course, this equality is true also if $G$ is a torsion abelian group. Thus, if $G$ is either 
\begin{itemize}
\item a $\Z_p$-module that is a torsion abelian group
\end{itemize}
or
\begin{itemize}
\item a profinite $\Z_p$-module,
\end{itemize}
then $G^\wedge=\Hom_\cont(G,\Q_p/\Z_p)$. In other words, with notation as in \eqref{pontryagin-def-eq}, $G^\wedge=G^\vee$. 

\begin{remark} \label{pontryagin-rem} \label{pontryagin-rem2}
If $G$ is a profinite $\Z_p$-module, then every element of $\Hom_\cont(G,\Q_p/\Z_p)$ is $\Z_p$-linear. Conversely, if $G$ is a finitely generated $\Z_p$-module, then a $\Z_p$-linear homomorphism $G\rightarrow\Q_p/\Z_p$ is always continuous, so $G^\wedge=G^\vee=\Hom_{\Z_p}(G,\Q_p/\Z_p)$ in this case. More generally, suppose that $\KK$ is a finite extension of $\Q_p$ with valuation ring $\mathscr O$. As explained, \emph{e.g.}, in \cite[\S 2.9.1, \S 2.9.2]{Nek-Selmer}, if $G$ is a (co)finitely generated $\mathscr O$-module, then $G^\wedge=G^\vee\simeq\Hom_{\mathscr O}(G,\KK/\mathscr O)$.
\end{remark}

\subsection{Pontryagin duals of finite extensions of $\Q_p$} \label{local-fields-duals-subsec}

Let $\KK$ be a finite extension of $\Q_p$. Recall the definition of $\KK^\wedge$ from \eqref{pontryagin-def-eq2}. Fix a compatible system ${(\zeta_{p^n})}_{n\geq1}$ of $p$-power roots of unity: $\zeta_{p^n}\in\bar\Q$ is a primitive $p^n$-root of unity such that $\zeta_{p^{n+1}}^p=\zeta_{p^n}$ for all $n\geq1$; the standard choice is $\zeta_{p^n}\defeq e^{2\pi i/p^n}$ for all $n\geq1$. Let us define the (non-trivial) standard character $\chi_0\in\Q_p^\wedge$ by  
\begin{equation} \label{chi-0-eq}
\chi_0\Biggl(\sum_{k=-n}^{+\infty}a_kp^k\Biggr)\defeq\prod_{k=-n}^{-1}\zeta_{p^{|k|}}^{a_k}; 
\end{equation}
here $n\in\N$ and $a_k\in\Z$ for all $k\geq-n$. More succinctly, with the choice of roots of unity specified above, $\chi_0$ is the composition
\[ \chi_0=\Bigl(\Q_p\longepi\Q_p/\Z_p\longmono\Q/\Z\xrightarrow{e^{2\pi i(\cdot)}}\TT\Bigr). \]
Now fix a non-trivial $\varphi\in\KK^\wedge$. It turns out that for every $\psi\in\KK^\wedge$ there exists a unique $a_\varphi(\psi)\in\KK$ such that $\psi(x)=\varphi\bigl(a_\varphi(\psi)\cdot x\bigr)$ for all $x\in\KK$ and the map
\begin{equation} \label{pontryagin1} 
\KK^\wedge\overset\simeq\longrightarrow\KK,\quad\psi\longmapsto a_\varphi(\psi)
\end{equation}
 is an isomorphism of topological groups. See, \emph{e.g.}, \cite[\S 8.3, Proposition 1]{hida-elementary} and \cite[Ch. 7, Exercise 1]{RV} for details. 

\begin{remark}
The (non-trivial) standard character of $\KK$ is $\chi_0\circ\tr_{\KK/\Q_p}$, with $\chi_0$ as in \eqref{chi-0-eq}.
\end{remark}

Let $\KK^\vee$ be defined as in \eqref{pontryagin-def-eq}.

\begin{proposition} \label{pontryagin2}
There is an isomorphisms $\KK^\vee\simeq\KK^\wedge$ of topological groups. 
\end{proposition}

\begin{proof} Observe that $\Z_p$ is contained in the kernel of the character $\chi_0$ from \eqref{chi-0-eq}, so there is an induced map
\begin{equation} \label{pontryagin3}
\bar\chi_0:\Q_p/\Z_p\longrightarrow\TT,\quad[x]\longmapsto\chi_0(x), 
\end{equation}
where $[x]$ is the image of $x\in\Q_p$ in $\Q_p/\Z_p$. One can show that $\bar\chi_0$ yields an isomorphism of groups between $\Q_p/\Z_p$ and the subgroup $\Bmu_{p^\infty}\subset\TT$ of $p$-power roots of unity. In light of this fact, from here on we shall view the map $\bar\chi_0$ from \eqref{pontryagin3} as an isomorphism
\begin{equation} \label{pontryagin3-bis}
\bar\chi_0:\Q_p/\Z_p\overset\simeq\longrightarrow\Bmu_{p^\infty}. 
\end{equation}
Given $\varphi\in\KK^\vee$, set $\varphi^\wedge\defeq\bar\chi_0\circ\varphi\in\KK^\wedge$. On the other hand, given $\chi\in\KK^\wedge$, note that $\im(\chi)\subset\Bmu_{p^\infty}$ (the proof of \cite[\S 8.3, Proposition 1]{hida-elementary}, which treats the $\KK=\Q_p$ case, carries over \emph{verbatim} to our setting), so we can define $\chi^\vee\defeq\bar\chi_0^{-1}\circ\chi\in\KK^\vee$. Clearly, $\varphi\mapsto\varphi^\wedge$ and $\chi\mapsto\chi^\vee$ are topological group homomorphisms $\KK^\vee\rightarrow\KK^\wedge$ and $\KK^\wedge\rightarrow\KK^\vee$ that are inverse to each other. \end{proof}

Since we regard the character $\chi_0\in\Q_p^\wedge$ described in \eqref{chi-0-eq} as canonical, we shall tacitly view the isomorphism $\KK^\vee\simeq\KK^\wedge$ as an identification. In light of \eqref{pontryagin1}, it follows that for every non-trivial $\varphi\in\KK^\wedge$ there is an isomorphism of topological groups 
\begin{equation} \label{pontryagin4}
a_\varphi:\KK^\vee\overset\simeq\longrightarrow\KK.
\end{equation} 
Namely, for every $\psi\in\KK^\vee$ there is a unique $a_\varphi(\psi)\in\KK$ such that $\psi^\wedge(x)=\varphi\bigl(a_\varphi(\psi)\cdot x\bigr)$ for all $x\in\KK$; equivalently, $\psi(x)=\varphi^\vee\bigl(a_\varphi(\psi)\cdot x\bigr)$ for all $\psi\in\KK^\vee$ and $x\in\KK$, where $\varphi^\vee\in\KK^\vee$ corresponds to $\varphi$ under the isomorphism of Proposition \ref{pontryagin2}.  

In the special case $\KK=\Q_p$ we shall take $\varphi=\chi_0$ and set $a\defeq a_{\chi_0}$; here notice that $\chi_0^\vee$ is given by
\[ %\begin{split}
\chi_0^\vee\Biggl(\sum_{k=-n}^{+\infty}a_kp^k\Biggr)
%&
=
\bar\chi_0^{-1}\Biggl(\prod_{k=-n}^{-1}\zeta_{p^{|k|}}^{a_k}\Biggr)\\
%&
=
\Biggl[\sum_{k=-n}^{-1}a_kp^k\Biggr],
%\end{split}
\]
where, as above, $[x]$ is the class of $x\in\Q_p$ in $\Q_p/\Z_p$ and $\bar\chi_0$ is the isomorphism from \eqref{pontryagin3-bis}. In other words, $\chi_0^\vee$ is just the canonical projection $\Q_p\twoheadrightarrow\Q_p/\Z_p$.

It is convenient to introduce the following notation: given $\psi\in\KK^\wedge$ and $c\in\KK$, we define $c\cdot\psi\in\KK^\wedge$ by $(c\cdot\psi)(x)\defeq\psi(cx)$ for all $x\in\KK$. This endows $\KK^\wedge$ with a $\KK$-vector space structure, and an analogous definition can be given for $\KK^\vee$. In particular, the image of $\psi\in\KK^\vee$ under the isomorphism $a_\varphi$ in \eqref{pontryagin4} can be described by requiring $a_\varphi(\psi)\in\KK$ to satisfy $\psi=a_\varphi(\psi)\cdot\varphi^\vee$.

\begin{proposition} \label{pontryagin-alt-prop}
There is an isomorphism $\KK^\vee\simeq\Hom_{\mathscr O}(\KK,\KK/\mathscr O)$ of topological groups.
\end{proposition}

\begin{proof} Let $\varpi$ be a uniformizer for $\mathscr O$, so that $\KK=\varinjlim_{n\in\Z}\varpi^n\mathscr O$. For every $n\in\Z$, the $\mathscr O$-module $\varpi^n\mathscr O$ is (topologically) free of rank $1$, so Remark \ref{pontryagin-rem2} ensures that there is an isomorphism $(\varpi^n\mathscr O)^\vee\simeq\Hom_{\mathscr O}(\varpi^n\mathscr O,\KK/\mathscr O)$ of topological groups. It follows that there are isomorphisms of topological groups
\[ \KK^\vee=\varprojlim\nolimits_{n\in\Z}(\varpi^n\mathscr O)^\vee\simeq\varprojlim\nolimits_{n\in\Z}\Hom_{\mathscr O}(\varpi^n\mathscr O,\KK/\mathscr O)=\Hom_{\mathscr O}(\KK,\KK/\mathscr O), \]
as was to be shown. \end{proof}

From here on, we view the isomorphism provided by Proposition \ref{pontryagin-alt-prop} as an identification $\KK^\vee=\Hom_{\mathscr O}(\KK,\KK/\mathscr O)$. 

\begin{remark} \label{vector-pontryagin-rem}
The isomorphisms in \eqref{pontryagin1} and Proposition \ref{pontryagin2} are $\KK$-linear, so $\KK^\wedge$ and $\KK^\vee$ are $\KK$-vector spaces of dimension $1$.
\end{remark}

\subsection{On Pontryagin duals of $\KK$-vector spaces} \label{vector-duals-subsec}

Let $\chi_\KK:\KK\twoheadrightarrow\KK/\mathscr O$ be the canonical projection and let $\bX_\KK\defeq\mathscr O\chi_\KK$ be the $\mathscr O$-submodule of $\KK^\vee$ generated by $\chi_\KK$.

\begin{lemma} \label{pontryagin9}
An element of $\KK^\vee$ is trivial on $\mathscr O$ if and only if it belongs to $\bX_\KK$.
\end{lemma}

\begin{proof} It is obvious that $\chi_\KK$ is trivial on $\mathscr O$, so every element of $\bX_\KK$ is trivial on $\mathscr O$. Now pick $\psi\in\KK^\vee$. Thanks to isomorphism \eqref{pontryagin1} and Proposition \ref{pontryagin2}, there exists a unique $a_\KK(\psi)\in\KK$ such that $\psi=a_\KK(\psi)\cdot\chi_\KK$. On the other hand, the square
\begin{equation} \label{square-pontryagin-eq}
\xymatrix@R=30pt@C=30pt{\KK\ar@{->>}[d]\ar[r]^-{a_\KK^{-1}}_-{\simeq}&\KK^\vee\ar@{->>}[d]\\
                \KK/\mathscr O\ar[r]^-\simeq&\mathscr O^\vee}
\end{equation}
is commutative. If $\psi$ is trivial on $\mathscr O$, then $\psi$ has trivial image in $\mathscr O^\vee$, and the commutativity of \eqref{square-pontryagin-eq} ensures that the image of $a_\KK(\psi)$ in $\KK/\mathscr O$ is trivial as well, \emph{i.e.}, $a_\KK(\psi)\in\mathscr O$. This shows that $\psi$ belongs to $\bX_\KK$. \end{proof}

Now let $V$ be a $\KK$-vector space of finite dimension, say $r$. Let $\mathscr B=\{v_1,\dots,v_r\}$ be a basis of $V$ over $\KK$; it induces topological isomorphisms $V\simeq\KK^r$ and
\begin{equation} \label{dual-iso-eq}
V^\vee\simeq(\KK^\vee)^r. 
\end{equation}
Analogously to what was done in \S \ref{local-fields-duals-subsec} for $\KK^\wedge$ and $\KK^\vee$, one can endow $V^\vee$ with a natural $\KK$-vector space structure, and then $V^\vee$ is $r$-dimensional over $\KK$ (\emph{cf.} Remark \ref{vector-pontryagin-rem}).

\begin{remark}
There is an isomorphism $V^\vee\simeq V^\wedge$ of topological groups, which we can view as a canonical identification.
\end{remark}

For each $i\in\{1,\dots,r\}$, write $v_i^\vee$ for the element of $V^\vee$ corresponding under isomorphism \eqref{dual-iso-eq} to the element of $(\KK^\vee)^r$ with all components equal to $0$ except the $i$-th that is equal to $\chi_\KK$. Equivalently, the elements $v_1,\dots,v_r$ give rise to the dual basis $\{v_1^*,\dots,v_r^*\}$ of the $\KK$-linear dual of $V$ by the recipe $v_i^*(v_j)\defeq\delta_{ij}$, where $\delta_{ij}$ is the Kronecker delta: composing the $\KK$-linear maps $v_i^*$ with $\chi_\KK$, we obtain the elements $v_i^\vee$ of the Pontryagin dual $V^\vee$.

\begin{lemma} \label{pontryagin-indep-lemma}
The elements $v_1^\vee,\dots,v_r^\vee$ are linearly independent over $\mathscr O$.
\end{lemma}

\begin{proof} Suppose there is an equality
\begin{equation} \label{lin-dep-eq2}
a_1v_1^\vee+\dots+a_rv_r^\vee=0
\end{equation} 
with $a_1,\dots,a_r\in\mathscr O$ and there is $j$ such that $a_j\neq0$; let $\nu_j\in\N$ be the valuation of $a_j$. Let $\varpi\in\mathscr O$ be a uniformizer and notice that $\chi_\KK$ is $\mathscr O$-linear. Then
\[ \bigl(a_1v_1^\vee+\dots+a_rv_r^\vee\bigr)\bigl(\varpi^{-(\nu_j+1)}v_j\bigr)=\chi_\KK\bigl(a_j\,\varpi^{-(\nu_j+1)}\bigr)\neq0, \]
where the inequality on the right is a consequence of the fact that $a_j\,\varpi^{-(\nu_j+1)}\notin\mathscr O$. This contradicts \eqref{lin-dep-eq2}. \end{proof}

Since $\KK$ is the quotient field of $\mathscr O$, it follows from Lemma \ref{pontryagin-indep-lemma} that $v_1^\vee,\dots,v_r^\vee$ are linearly independent over $\KK$ as well. Keeping in mind that $V^\vee$ is an $r$-dimensional $\KK$-vector space, it makes sense to give

\begin{definition} \label{dual-basis-pontryagin-def}
The \emph{Pontryagin dual basis} of $\mathscr B$ is the basis $\mathscr B^\vee\defeq\bigl\{v_1^\vee,\dots,v_r^\vee\bigr\}$ of $V^\vee$ as a $\KK$-vector space.
\end{definition}

Now let $T_{\mathscr B}$ be the $\mathscr O$-lattice in $V$ spanned by $\mathscr B$ and write $\Xi_{\mathscr B}$ for the $\mathscr O$-lattice in $V^\vee$ spanned by $\mathscr B^\vee$.

\begin{proposition} \label{pontryagin10}
An element of $V^\vee$ belongs to $\Xi_{\mathscr B}$ if and only if it is trivial on $T_{\mathscr B}$.
\end{proposition}

\begin{proof} Immediate from Lemma \ref{pontryagin9}. \end{proof}

\bibliographystyle{amsplain}
\bibliography{Iwasawa}

\end{document}